\documentclass[letterpaper,10pt]{article}

\usepackage{enumitem}

\usepackage[letterpaper,left=1in,right=1in,top=1in,bottom=1in]{geometry}
\newcommand{\myauthor}{}
\newcommand{\mytitle}{Scheiderer motives}

\title{Scheiderer motives and equivariant higher topos theory}

\author{}
 \author{Elden Elmanto, Jay Shah}
\date{\today}
 
\usepackage[pdfstartview=FitH,
            pdfauthor={\myauthor},
            pdftitle={\mytitle},
            colorlinks,
            linkcolor=reference,
            citecolor=citation,
            urlcolor=e-mail,
            backref]{hyperref}

\usepackage{amsmath}
\usepackage{amscd}
\usepackage{amsbsy}
\usepackage{amssymb}
\usepackage{verbatim}
\usepackage{eufrak}
\usepackage{eucal}
\usepackage{microtype}
\usepackage{hyperref}
\usepackage{mathrsfs}
\usepackage{amsthm}
\usepackage{stmaryrd}
\usepackage[all,cmtip]{xy}
\usepackage{tikz}
\usetikzlibrary{matrix,arrows}
\usepackage[abbrev,lite,alphabetic]{amsrefs}
\usepackage{dsfont}
\usetikzlibrary{matrix,arrows}
\usepackage{tikz-cd}

\usepackage{graphicx} % for '\rotatebox' macro
\newcommand\myrot[1]{\mathrel{\rotatebox[origin=c]{#1}{$\Rightarrow$}}}
\newcommand\NEarrow{\myrot{45}}

\newcommand\SWarrow{\myrot{-135}}
\newcommand\SEarrow{\myrot{-45}}

\usepackage{fancyhdr}
\usepackage{makeidx}
\makeindex
\newcommand{\df}[1]{{\bf #1} \index{#1}}

\usepackage{color}
\definecolor{todo}{rgb}{1,0,0}
\definecolor{conditional}{rgb}{0,1,0}
\definecolor{e-mail}{rgb}{0,.40,.80}
\definecolor{reference}{rgb}{.20,.60,.22}
\definecolor{mrnumber}{rgb}{.80,.40,0}
\definecolor{citation}{rgb}{0,.40,.80}

\let\overlinedmarginpar\marginpar
\renewcommand\marginpar[1]{\-\overlinedmarginpar[\raggedleft\footnotesize #1]%
{\raggedright\footnotesize #1}}

\makeatletter
\newcommand{\xRightarrow}[2][]{\ext@arrow 0359\Rightarrowfill@{#1}{#2}}
\makeatother

\newcommand{\res}{\mathrm{res}}
\newcommand{\leftnat}[1]{\vphantom{#1}_{\natural}\mskip-1mu{#1}}

\newcommand{\DK}{\mathrm{DK}}
\renewcommand{\big}{\mathrm{big}}

\newcommand{\Ascr}{\mathcal{A}}

\newcommand{\Escr}{\mathcal{E}}
\newcommand{\Fscr}{\mathcal{F}}

\newcommand{\Oscr}{\mathcal{O}}

\newcommand{\Xscr}{\mathcal{X}}
\newcommand{\Yscr}{\mathcal{Y}}
\newcommand{\Uscr}{\mathcal{U}}
\newcommand{\Zscr}{\mathcal{Z}}
\newcommand{\Lscr}{\mathcal{L}}
\newcommand{\Kscr}{\mathcal{K}}
\newcommand{\A}{\mathrm{A}}
\newcommand{\B}{\mathrm{B}}
\newcommand{\C}{\mathrm{C}}
\newcommand{\D}{\mathrm{D}}

\renewcommand{\H}{\mathrm{H}}

\renewcommand{\1}{\mathds{1}}
\renewcommand{\AA}{\mathds{A}}
\newcommand{\CC}{\mathds{C}}

\newcommand{\FF}{\mathds{F}}
\newcommand{\GG}{\mathds{G}}

\newcommand{\NN}{\mathds{N}}
\newcommand{\PP}{\mathds{P}}
\newcommand{\QQ}{\mathds{Q}}
\newcommand{\RR}{\mathds{R}}

\newcommand{\TT}{\mathds{T}}
\newcommand{\ZZ}{\mathds{Z}}

\newcommand{\Fr}{\mathrm{Fr}}

\newcommand{\triv}{\mathrm{triv}}

\newcommand{\op}{\mathrm{op}}
\newcommand{\ex}{\mathrm{ex}}
\newcommand{\lex}{\mathrm{lex}}

\newcommand{\fib}{\mathrm{fib}}

\newcommand{\Sing}{\mathrm{Sing}}

\newcommand{\cart}{\mathrm{cart}}
\newcommand{\inv}[1]{[\tfrac{1}{#1}]}
\newcommand{\sad}{\mathrm{sad}}
\newcommand{\cocart}{\mathrm{cocart}}
\newcommand{\pinv}{\inv{p}}

\newcommand{\Spc}{\mathrm{Spc}}
\newcommand{\Sp}{\mathrm{Sp}}

\newcommand{\Fin}{\mathrm{Fin}}
\newcommand{\FinGpd}{\mathrm{FinGpd}}

\newcommand{\CMon}{\mathrm{CMon}}

\newcommand{\Mack}{\mathrm{Mack}}
\newcommand{\Exc}{\mathrm{Exc}}
\newcommand{\Sets}{\mathrm{Sets}}

\newcommand{\Gpd}{\mathrm{Gpd}}

\newcommand{\Top}{\mathrm{Top}}

\newcommand{\Nm}{\mathrm{Nm}}

\newcommand{\xto}{\xrightarrow}

\newcommand{\ev}{\mathrm{ev}}

\DeclareMathOperator{\Gal}{Gal}

\DeclareMathOperator{\id}{id}

\newcommand{\cd}{\mathrm{cd}}

\renewcommand{\geq}{\geqslant}
\renewcommand{\leq}{\leqslant}

\newcommand{\stab}{\mathrm{stab}}
\newcommand{\Span}{\mathrm{Span}}

\newcommand{\pt}{\mathrm{pt}}
\newcommand{\Sect}{\mathrm{Sect}}
\newcommand{\Be}{\mathrm{Be}}
\newcommand{\fin}{\mathrm{fin}}
\newcommand{\Tate}{\mathrm{Tate}}
\newcommand{\Set}{\mathrm{Set}}
\newcommand{\Cat}{\mathrm{Cat}}

\newcommand{\Topoi}{\mathscr{T}\mathrm{op}}
\newcommand{\RTop}{\Topoi^R}
\newcommand{\LTop}{\Topoi^L}
\newcommand{\PrL}{\mathrm{Pr}^{\mathrm{L}}}

\newcommand{\comp}{{}^{{\kern -.5pt}\wedge}_{p}}
\newcommand{\comtwo}{{}^{{\kern -.5pt}\wedge}_{2}}
\newcommand{\comfin}{{}^{{\kern -.5pt}\wedge}}
\newcommand{\comrho}{{}^{{\kern -.5pt}\wedge}_{\rho}}

\newcommand{\mot}{\mathrm{mot}}
\DeclareMathOperator{\Fun}{Fun}

\DeclareMathOperator{\Hom}{Hom}
\newcommand{\Map}{\mathrm{Map}}

\DeclareMathOperator{\Pre}{PShv}
\DeclareMathOperator{\Shv}{Shv}

\newcommand{\lax}{\mathrm{lax}}
\newcommand{\rlax}{\mathrm{rlax}}

\newcommand{\Mod}{\mathrm{Mod}}

\newcommand{\CAlg}{\mathrm{CAlg}}

\newcommand{\ret}{\mathrm{r\acute{e}t}}
\newcommand{\Sm}{\mathrm{Sm}}
\newcommand{\QP}{\mathrm{QP}}
\newcommand{\Et}{\mathrm{\acute{E}t}}

\newcommand{\SH}{\mathrm{SH}}

\newcommand{\eff}{\mathrm{eff}}

\newcommand{\Sch}{\mathrm{Sch}}

\newcommand{\ft}{\mathrm{ft}}
\DeclareMathOperator*{\colim}{colim}

\newcommand{\et}{\mathrm{\acute{e}t}}

\newcommand{\Nis}{\mathrm{Nis}}
\newcommand{\pre}{\mathrm{pre}}

\DeclareMathOperator{\Spec}{Spec}
\DeclareMathOperator{\Sper}{Sper}

\theoremstyle{plain}
\newtheorem{theorem}{Theorem}[section]
\newtheorem*{theorem*}{Theorem}
\newtheorem{lemma}[theorem]{Lemma}

\newtheorem{proposition}[theorem]{Proposition}

\newtheorem{corollary}[theorem]{Corollary}
\newtheorem*{corollary*}{Corollary}

\theoremstyle{plain}

\theoremstyle{definition}

\newtheoremstyle{named}{}{}{\itshape}{}{\bfseries}{.}{.5em}{#1 \thmnote{#3}}
\theoremstyle{named}

\newcommand{\Th}{\mathrm{Th}}

\theoremstyle{definition}
\newtheorem{definition}[theorem]{Definition}
\newtheorem{warning}[theorem]{Warning}

\newtheorem{notation}[theorem]{Notation}

\newtheorem{example}[theorem]{Example}
\newtheorem*{example*}{Example}

\newtheorem{question}[theorem]{Question}
\newtheorem*{question*}{Question}
\newtheorem{construction}[theorem]{Construction}

\newtheorem{remark}[theorem]{Remark}

\newcommand{\hh}{h_{\diamond}}

\begin{document}

\tikzcdset{arrow style=tikz, diagrams={>=stealth}}

\maketitle

% ; if $X$ is the spectrum of a real closed field, then this recovers the $\infty$-category $\Sp^{C_2}$ of genuine $C_2$-spectra
\begin{abstract} We give an algebro-geometric interpretation of $C_2$-equivariant stable homotopy theory by means of the $b$-topology introduced by Claus Scheiderer in his study of $2$-torsion phenomena in \'etale cohomology. To accomplish this, we first revisit and extend work of Scheiderer on equivariant topos theory by functorially associating to a $\infty$-topos $\Xscr$ with $G$-action a presentable stable $\infty$-category $\Sp^G(\Xscr)$, which recovers the $\infty$-category $\Sp^G$ of genuine $G$-spectra when $\Xscr$ is the terminal $G$-$\infty$-topos. Given a scheme $X$ with $\tfrac{1}{2} \in \Oscr_X$, our construction then specializes to produce an $\infty$-category $\Sp^{C_2}_b(X)$ of ``$b$-sheaves with transfers'' as $b$-sheaves of spectra on the small \'etale site of $X$ equipped with certain transfers along the extension $X[i] \rightarrow X$; if $X$ is the spectrum of a real closed field, then $\Sp^{C_2}_b(X)$ recovers $\Sp^{C_2}$. On a large class of schemes, we prove that, after $p$-completion, our construction assembles into a premotivic functor satisfying the full six functors formalism. We then introduce the $b$-variant $\SH_b(X)$ of the $\infty$-category $\SH(X)$ of motivic spectra over $X$ (in the sense of Morel-Voevodsky), and produce a natural equivalence of $\infty$-categories $\SH_b(X) \comp \simeq \Sp^{C_2}_b(X) \comp$ through amalgamating the \'etale and real \'etale motivic rigidity theorems of Tom Bachmann. This involves a purely algebro-geometric construction of the $C_2$-Tate construction, which may be of independent interest. Finally, as applications, we deduce a ``$b$-rigidity'' theorem, use the Segal conjecture to show \'etale descent of the $2$-complete $b$-motivic sphere spectrum, and construct a parametrized version of the $C_2$-Betti realization functor of Heller-Ormsby.
\end{abstract}


    \paragraph{Mathematics Subject Classification 2010.}
    \href{http://www.ams.org/mathscinet/msc/msc2010.html?t=14Fxx&btn=Current}{14F42},
    \href{http://www.ams.org/mathscinet/msc/msc2010.html?t=14Pxx&btn=Current}{14P10},
    \href{http://www.ams.org/mathscinet/msc/msc2010.html?t=55Pxx&btn=Current}{55P91}.

\tableofcontents

\section{Introduction}

\subsection{Motivation from algebraic geometry: compactifying the \'etale topos of a real variety} Grothendieck introduced the \'etale topology into algebraic geometry in order to construct a cohomology theory $H^{*}_{\et}$ for algebraic varieties that resembles the theory of singular cohomology $H^{*}$ for topological spaces. The analogy between the two theories is closest when working over a separably closed field $k$. For instance, when the base scheme is the field $\CC$ of complex numbers, the {\bf Artin comparison theorem} \cite{sga4-3}*{Th\'eorem\`e XI.4.4} states that given a complex variety $X$ and a finite abelian group $M$, there is a canonical isomorphism $$ H^{*}_{\et}(X;M) \cong H^{*}(X(\CC);M),$$
where we endow $X(\CC)$ with the complex analytic topology. More generally, if we let $\cd_{\ell}(X_{\et})$ denote the \'etale $\ell$-cohomological dimension of a $k$-variety $X$, then we have the inequality $\cd_{\ell}(X_{\et}) \leq 2 \dim(X)$ for all primes $l$ \cite{sga4-3}*{Corollaire X.4.3}, as predicted by Artin's theorem when $k = \CC$.

% Let us restrict our attention to the case where $k$ is a real closed field, so that $G = C_2$ is finite;
If we no longer suppose that $k$ is separably closed, we now expect properties of the absolute Galois group $G = \Gal(\overline{k}/k)$ to  manifest in the \'etale cohomology of $k$-varieties; correspondingly, the role of singular cohomology on the topological side should be replaced by $G$-equivariant Borel cohomology $H^{*}_G$. For instance, suppose that $k = \RR$ is the field of real numbers, so that $G = C_2$. Then parallel to Artin's theorem, Cox proved that for a real variety $X$ and a finite abelian group $M$, there is a canonical isomorphism between \'etale cohomology and (Borel-style) $C_2$-equivariant cohomology: $$H^{*}_{\et}(X;M) \cong H^{*}_{C_2}(X(\CC);M),$$
using the $C_2$-action on $X(\CC)$ given by complex conjugation \cite{cox}. In particular, we may now have $\cd_2(X_{\et}) = \infty$ as a consequence of the infinite $2$-cohomological dimension of $C_2$ itself (in topological terms, the infinite mod $2$ cohomology of $\RR P^{\infty}$). More precisely, if we let $X(\CC)/C_2$ to be the quotient of the topological space $X(\CC)$ with the analytic topology by its natural $C_2$-action, then Cox deduces a long-exact sequence of cohomology groups \cite{cox}*{Proposition 1.2}
$$ \cdots \to H^k(X(\CC)/C_2, X(\RR); M) \to H^k_{\et}(X;M) \to H^k_{C_2}(X(\RR);M) = H^k(X(\RR) \times B C_2; M) \to \cdots  $$
from which it follows that we have an isomorphism
$$ H^n_{\et}(X; \ZZ/2) \xto{\cong} \bigoplus_{i=0}^{\dim(X)} H^i(X(\RR); \ZZ/2) = H^{*}(X(\RR); \ZZ/2) $$
for all $n > 2 \dim(X)$. Furthermore, these results extend in a straightforward way from $\RR$ to an arbitrary real closed field if one instead uses ($C_2$-equivariant) semialgebraic cohomology \cite{scheiderer}*{Chapter 15}.

In his book \cite{scheiderer}, Scheiderer generalized the Cox exact sequence to one involving abelian \'etale sheaves $A$ over a base scheme $S$ on which $2$ is invertible, in which the role of $H^{\ast}(X(\RR);M)$ is replaced by {\bf real \'etale cohomology} . His theorems to this effect operate at the level of topoi, from which the exact sequence is a direct corollary. In picturesque terms, the main idea is to ``compactify the (small) \'etale topos $\widetilde{S_{\et}}$ by gluing in the (small) real \'etale topos $\widetilde{S_{\ret}}$ at $\infty$'', where the notion of gluing in question is that of a {\bf recollement} of topoi \cite{sga4-1}*{D\'efinition 9.1.1}; see also \cite{quigley-shah}*{\S 1-2}, \cite{higheralgebra}*{\S A.8}, and \cite{barwick-glasman} for treatments of this subject in the higher categorical setting.

In more detail, Scheiderer first describes a generic gluing procedure that receives as input any topos $\Xscr$ with $C_p$-action and outputs the ``genuine toposic $C_p$-orbits'' $\Xscr_{C_p}$; in this paper, we recapitulate his construction in the setting of $\infty$-topoi (\S\ref{sect:gen-stab}). His key insight was then that for the $C_2$-topos $\widetilde{S[i]_{\et}}$, the formation of toposic $C_2$-orbits yields $\widetilde{S_{\et}}$ by Galois descent, the formation of toposic $C_2$-fixed points yields $\widetilde{S_{\ret}}$, and with respect to these identifications, the gluing functor $\rho: \widetilde{S_{\et}} \to \widetilde{S_{\ret}}$ identifies with real \'etale sheafification. Therefore, the glued topos $(\widetilde{S[i]_{\et}})_{C_2}$ is equivalent to $\widetilde{S_b}$, where the {\bf $b$-topology} on the small \'etale site is defined to be the intersection of the \'etale and real \'etale topologies. By recollement theory, $\widetilde{S_b}$ thus contains $\widetilde{S_{\et}}$ as an open subtopos with closed complement $\widetilde{S_{\ret}}$.

Apart from its role in obtaining the Cox exact sequence, the $b$-topos turns out to have many excellent formal properties; in particular, Scheiderer proves that $\cd_2(S_b)$ is either $\cd_2(S[i]_{\et})$ or $\cd_2(S[i]_{\et})+1$ \cite{scheiderer}*{Corollary 7.18}. This result leads us to think of $\widetilde{S_b}$ as a suitable compactification of $\widetilde{S_{\et}}$ when the \emph{virtual} \'etale $2$-cohomological dimension of $S$ is finite.

\subsection{Motivation from algebraic topology: compactifying the topos of equivariant sheaves}

Let $G$ be a finite group and $X$ a topological space with properly discontinuous $G$-action. The $\infty$-topos $\Shv_G(X)$ of $G$-equivariant sheaves of spaces\footnote{In this paper, we will always use `space' as a synonym for Kan complex, i.e., $\infty$-groupoid, as opposed to a topological space.} on $X$ generally fails to possess good finiteness properties if $G$ does not act freely. To explain, consider the simple example of $X = \ast$ with trivial $G$-action, so that $$\Shv_G(X) \simeq \Spc^{BG} = \Fun(BG,\Spc),$$ where $BG$ is the one-object groupoid with $G$ as its endomorphisms and $\Spc$ is the $\infty$-category of spaces. Since the formation of homotopy fixed points fails to commute with homotopy colimits in general, we see that the unit is not a compact object in $\Spc^{BG}$.

On the other hand, we have the homotopy theory for $G$-spaces given by taking a category of (nice) topological spaces with $G$-action and inverting those morphisms that induce weak homotopy equivalences on all fixed points. Let $\Spc_G$ denote the resulting $\infty$-category, whose objects we henceforth term $G$-spaces (as opposed to spaces with $G$-action). By Elmendorf's theorem \cite{elmendorf}, we have an equivalence $$\Spc_G \simeq \Fun(\Oscr_G^{\op}, \Spc)$$ of the $\infty$-category of $G$-spaces with presheaves of spaces on the orbit category $\Oscr_G$. In $\Spc_G$, taking $G$-fixed points amounts to evaluation on the orbit $G/G$, which \emph{does} preserve all colimits. Therefore, the unit in $\Spc_G$ is a compact object, in contrast to $\Spc^{BG}$. Moreover, right Kan extension along the inclusion $BG \subset \Oscr_G^{\op}$ (as the full subcategory on the orbit $G/1$) yields a fully faithful embedding $\Spc^{BG} \to \Spc_G$, which presents $\Spc^{BG}$ as an open subtopos of $\Spc_G$. We are thus entitled to view $\Spc_G$ as a compactification of $\Spc^{BG}$. Indeed, if $G = C_p$, then Scheiderer's construction applied to the $\infty$-topos $\Spc$ with trivial $C_p$-action yields $\Spc_{C_p}$ as its genuine toposic $C_p$-orbits. In particular, the $b$-$\infty$-topos of a real closed field identifies with $\Spc_{C_2}$.

Let us now pass to stabilizations. We then have the $\infty$-category of {\bf Borel $G$-spectra} $\Sp^{BG} = \Fun(BG, \Sp)$, and the $\infty$-category of {\bf naive $G$-spectra} $\Fun(\Oscr_G^{\op}, \Sp)$. As is well-known in equivariant stable homotopy theory, the latter $\infty$-category fails to be an adequate homotopy theory of $G$-spectra as it does not possess enough dualizable objects (for instance, one does not have equivariant Atiyah duality in this $\infty$-category). To correct this deficiency, one passes to the $\infty$-category $\Sp^G$ of {\bf genuine $G$-spectra}, which is obtained from $\Spc_G$ by inverting the real representation spheres. In $\Sp^G$, the unit remains compact; indeed, the compact objects now coincide with the dualizable objects. Finally, we still have a fully faithful embedding $\Sp^{BG} \to \Sp^G$ as a right adjoint, whose essential image consists of the {\bf cofree} genuine $G$-spectra.

\subsection{What is this paper about?}

Our first goal in this paper is to unify the above two perspectives on compactifications of $\infty$-topoi in the stable regime. In other words, having seen that the $b$-topos is analogous to the $\infty$-topos $\Spc_{C_2}$ of $C_2$-spaces, we aim to answer the following:

\begin{question} What is the analogue of the passage from $G$-spaces to genuine $G$-spectra for the $b$-topos?
\end{question}

% ; the ``official" construction is stated at the beginning of \S\ref{genuine-six}
To this end, we will construct a stable $\infty$-category of {\bf $b$-sheaves of spectra with transfers} (Definition~\ref{dfn:OfficialBSheavesWithTransfers}) on a scheme $S$, which we denote by $\Sp^{C_2}_b(S)$. This $\infty$-category agrees with $\Sp^{C_2}$ when $S = \Spec k$ for $k$ a real closed field.

Our second goal is to then relate our construction to a variant of Morel-Voevodsky's stable $\AA^1$-homotopy theory of schemes $\SH(S)$, in which the role of the Nisnevich topology is replaced by that of the finer $b$-topology (on the big site $\Sm_S$ of smooth schemes over $S$). In fact, given any topology $\tau$ finer than the Nisnevich topology, define the full subcategory $\SH_{\tau}(S) \subset \SH(S)$ of {\bf $\tau$ motivic spectra}\footnote{Note that we consider the hypercomplete version of the theory.} to be the local objects with respect to the class of morphisms
\[ \{\Sigma^{p,q}\Sigma^{\infty}_+U_{\bullet} \rightarrow \Sigma^{p,q}\Sigma^{\infty}_+T: p,q \in \ZZ, \: U_{\bullet} \rightarrow T\,\text{is the \v{C}ech nerve of a $\tau$-hypercover $U \rightarrow T$ in $\Sm_S$}\}. \]
% The theory of Bousfield localization then applies to produce the localization adjunction
% $$\adjunct{L_{\tau}}{\SH(X)}{\SH_{\tau}(X)}{i_{\tau}}.$$

We wish to relate $\SH_b(S)$. To this end, two recent \emph{rigidity theorems} of Bachmann allow us to isolate the $\infty$-categories of real \'etale and ($p$-complete) \'etale motivic spectra in terms of the corresponding $\infty$-categories of sheaves of spectra on the \emph{small} site.

\begin{theorem}[\cite{bachmann-ret}] \label{thm:BachmannRealEtaleRigidity} Suppose $S$ is a scheme of finite dimension. Then the morphism of sites $$ (\Et_S, \ret) \to (\Sm_S, \ret) $$ induces an equivalence\footnote{For our formulation of Bachmann's theorem, we use in addition the hypercompleteness of the real \'etale site (see Appendix~\ref{subsec:hyp}).}
\begin{equation} \label{eq:x-ret}
\Sp(\widetilde{S_{\ret}}) \simeq \SH_{\ret}(S).
\end{equation}
Under this equivalence, the localization functor $L_{\ret}$ is given by inverting the element $\rho \in \pi_{-1,-1}(S^0)$ defined by the unit $-1 \in \Oscr^{\times}(S)$.
\end{theorem}

\begin{remark} \label{rem:shret} In \cite{bachmann-ret}, the schemes involved are assumed to be noetherian. Using continuity for the functor $S \mapsto \Sp(\widetilde{S}_{\ret})$ (see Lemma~\ref{lem:approx}) and continuity for $\SH$, which persists for $\SH_{\ret}$ (by the same argument with real \'etale covers as in Lemma~\ref{lem:approx}), we have removed this hypothesis in the statement of the above theorem. This is implicit in Bachmann's formulation for the \'etale case below, which we cite verbatim.
\end{remark}

For the next result, we denote by $\Sch^{p\mbox{-}\fin}_B$ the category of schemes that are ``locally $p$-\'etale finite" (see \cite{bachmann-et}*{Definition 5.8}) over a base scheme $B$.

\begin{theorem}[\cite{bachmann-et}] \label{thm:BachmannEtaleRigidity} Suppose $S \in \Sch^{p\mbox{-}\fin}_{\ZZ[\tfrac{1}{p}]}$. Then the morphism of sites $$ (\Et_S, \et) \to (\Sm_S, \et) $$ induces an equivalence after $p$-completion
\begin{equation} \label{eq:x-et}
\Sp(\widehat{S_{\et}})\comp \simeq \SH_{\et}(S)\comp,
\end{equation}
where we consider \emph{hypercomplete} \'etale sheaves.
\end{theorem}

\begin{remark} We do not expect Theorem~\ref{thm:BachmannEtaleRigidity} to hold integrally, in view of the equivalence of Nisnevich and \'etale motives after rationalization.
\end{remark}

Roughly speaking, we will amalgamate Bachmann's two theorems to construct an equivalence $$\Sp^{C_2}_b(S)\comp \simeq \SH_b(S)\comp$$ under the same hypotheses as Theorem~\ref{thm:BachmannEtaleRigidity}.

\begin{remark} In particular, we obtain a fully faithful embedding of $\Sp^{C_2}_b(S)\comp$ into $\SH(S)\comp$. In \cite{behrens-shah}, Behrens and the second author showed that the right adjoint to $C_2$-Betti realization gives a fully faithful embedding $(\Sp^{C_2})\comp \hookrightarrow \SH^{\text{cell}}(\RR)\comp$ of $p$-complete genuine $C_2$-spectra into $p$-complete \emph{cellular} real motivic spectra. Our theorems here thus specialize to a non-cellular version of this result when $S = \Spec \RR$  (upon identifying $b$-localization with $C_2$-Betti realization).
\end{remark}

\subsection{Main results}
Let us now state our main results in greater detail. Roughly, this paper divides into two parts: \S\ref{sect:gen-stab}-\ref{sect:six} concerns equivariant higher topos theory and its application to establishing the formalism of six operations, while \S\ref{sec:field-case}-\ref{sec:apps} builds a connection with motivic homotopy theory in the $b$-topology. The starting point for equivariant higher topos theory is the following definition.

\begin{definition} Let $G$ be a finite group. A {\bf $G$-$\infty$-topos} is a functor $BG \to \RTop$, where $\RTop$ is the $\infty$-category of $\infty$-topoi and geometric morphisms thereof.
\end{definition}

Given a $G$-$\infty$-topos $\Xscr$, we will construct its genuine stabilization $\Sp^G(\Xscr)$ in two stages (see \S\ref{sect:gen-stab} for details):

\begin{enumerate} \item We define the {\bf toposic genuine $G$-orbits} $\Xscr_{G}$ as a certain lax colimit in $\RTop$ (Construction~\ref{cons:orb}); this recovers Scheiderer's construction when $G = C_p$ (Example~\ref{ex:b}). Ranging over all subgroups $H \leq G$, these $\infty$-topoi assemble into a presheaf 
\[
\Xscr_{(-)}: (\Oscr_G)^{\op} \to \LTop \simeq (\RTop)^{\op},
\] which constitutes an example of a {\bf $G$-$\infty$-category}. 
\item Given any $G$-$\infty$-category $\C$ that admits finite $G$-limits, Nardin has constructed a suitable candidate for its genuine stabilization $\Sp^G(\C)$ as the $\infty$-category of $G$-commutative monoids (or parametrized Mackey functors) in the fiberwise stabilization of $\C$. We apply his construction to define $\Sp^G(\Xscr) = \Sp^G(\Xscr_{(-)})$.
\end{enumerate}

If $G = C_p$, then we also have an alternative description of $\Sp^{C_p}(\Xscr)$ that generalizes Glasman's description of genuine $C_p$-spectra as a recollement obtained by gluing along the $C_p$-Tate construction $(-)^{tC_p}: \Sp^{B C_p} \to \Sp$ \cite{glasman}.

\begin{theorem}[Theorem \ref{prop:recoll-stab}]  Let $\Xscr$ be a $C_p$-$\infty$-topos and let $\Xscr_{\hh C_p}$ and $\Xscr^{\hh C_p}$ denote the toposic colimit and limit of the $C_p$-action on $\Xscr$ (i.e., as taken in $\RTop$ instead of $\Cat_{\infty}$). Then there exists a functor $$\Theta^{\Tate}: \Sp(\Xscr_{\hh C_p}) \to \Sp(\Xscr^{\hh C_p})$$ whose right-lax limit is canonically equivalent to $\Sp^{C_p}(\Xscr)$.
\end{theorem}

One payoff of the recollement perspective is the construction of a symmetric monoidal structure on $\Sp^{C_p}(\Xscr)$, which we explain in~\S\ref{sec:symmon}. This symmetric monoidal structure will be essential for our subsequent construction of a new six functors formalism when we specialize to our main example of a $G$-$\infty$-topos: the $C_2$-$\infty$-topos $\widehat{S[i]_{\et}}$ of hypercomplete \'etale sheaves of spaces on $S[i]$ for a scheme $S$ on which $2$ is invertible.

%\begin{remark} Although this is not difficult, we do not construct the `smash product' symmetric monoidal structure on $\Sp^G(\Xscr)$ in this paper. However, since the functor $\Theta_{\Tate}$ canonically lifts to the structure of a lax symmetric monoidal functor, if $G = C_p$ then we may equip $\Sp^{C_p}(\Xscr)$ with the structure of a symmetric monoidal $\infty$-category using Theorem~\ref{prop:recoll-stab}.
%\end{remark}
%

\begin{definition}[Definition~\ref{dfn:OfficialBSheavesWithTransfers}] \label{def:b-shvtr} Let $\Sp^{C_2}_b(S) = \Sp^{C_2}(\widehat{S[i]}_{\et})$ be the $\infty$-category of {\bf hypercomplete $b$-sheaves of spectra with transfers}.
\end{definition}

Building upon Bachmann's results, we endow $\Sp^{C_2}_b(-)\comp$ with the structure of a premotivic functor satisfying the full six functors formalism.

\begin{theorem} [Theorem~\ref{construct:genuine}] For any prime $p$, there is a six functors formalism
\[
(\Sp_b^{C_2})\comp: (\Sch^{\ft}_{\ZZ[\tfrac{1}{p}, \tfrac{1}{2}]})^{\op} \rightarrow \CAlg(\PrL_{\infty,\stab}),
\]
which assigns to a real closed field the $p$-completion of the $\infty$-category of genuine $C_2$-spectra $\Sp^{C_2}$.
\end{theorem}

That $(\Sp_b^{C_2})\comp$ assembles into a full six functors formalism crucially relies on the maneuver of genuine stabilization, which guarantees a sufficient supply of dualizable objects; see Definition~\ref{def:enough} for the precise dualizibility assumption that we require. In particular, the Tate motive $\1(1)$ is invertible in $(\Sp_b^{C_2})\comp$, even though the construction of $(\Sp_b^{C_2})\comp$ as a premotivic functor does not \emph{a priori} enforce this (cf.  Lemma~\ref{lem:duals}). 

Next, we relate our construction to $b$-motivic spectra. In order to facilitate the comparison, we first express the $C_2$-Tate construction in purely algebro-geometric terms.

% \label{thm:mot-v-tate}
\begin{theorem}[Theorem~\ref{thm:mot-v-tate}] Let $k$ be a real closed field. Under the equivalences $(\Sp^{BC_2})\comp \stackrel{\simeq}{\rightarrow} \SH_{\et}(k)\comp$ and $(\Sp)\comp \stackrel{\simeq}{\rightarrow} \SH_{\ret}(k)\comp$, we have a canonical equivalence of functors
\[
(-)^{tC_2} \simeq L_{\ret}i_{\et}:  \SH_{\et}(k)\comp \rightarrow \SH_{\ret}(k)\comp.
\]
\end{theorem}

To prove Theorem~\ref{thm:mot-v-tate}, we extend the construction of the $C_2$-Betti realization functor to a real closed field in \S\ref{betti-c2} by means of semialgebraic topology, and then prove that its right adjoint is fully faithful after $p$-completion. 

Finally, we obtain the following identification:

% Let $S$ be a finite dimensional base scheme such that $\tfrac{1}{2} \in \Oscr_S$. Let $p$ be a prime such that $S$ is locally $p$-\'etale finite and $\tfrac{1}{p} \in S$, then there is a canonical equivalence of stable $\infty$-categories
\begin{theorem} [Theorem~\ref{cor:sch-1}] Let $S$ be a scheme such that $\tfrac{1}{2} \in \Oscr_S$. Then there is a canonical strong symmetric monoidal functor
\[ C_S: \Sp^{C_2}_b(S) \to \SH_b(S), \]
such that if $S$ is locally $p$-\'etale finite and $\tfrac{1}{p} \in S$, then $C_S$ is an equivalence after $p$-completion.
\end{theorem}

The proof of Theorem~\ref{cor:sch-1} follows at once from the identification of Theorem~\ref{thm:mot-v-tate} after we construct a transformation $\Theta^{\Tate} \Rightarrow \Theta^{\mot} = L_{\ret}i_{\et}$ that is suitably stable under base change. 

We conclude with three applications of our results. First, in \S\ref{b-rigid}, we discuss a version of rigidity for the $b$-topology which, unlike the \'etale topology, is not quite an identification of the stable motivic category with sheaves of spectra on the small site. Indeed, as Theorem~\ref{cor:sch-1} indicates, one needs to adjoin more transfers. In \S\ref{realization}, we describe yet another realization functor out of $\SH$, which in this case is a parametrized version of the $C_2$-Betti realization functor of Heller-Ormsby. Lastly, in \S\ref{segal}, we apply Segal's conjecture for the $C_2$-equivariant sphere to deduce a surprising result: the $b$-sphere satisfies \'etale descent after $2$-completion. We do not know a way to see this result without appealing to equivariant stable homotopy theory.

% equivariance in motivic homotopy theory
% The group of equivariance is always the cyclic group of order 2 regarded as the automorphisms of the Galois extension $X[i] \rightarrow X$ (we assume that $\tfrac{1}{2} \in \Oscr_X$). 
\begin{remark} In \cite{hoyois-sixops}, a full six functors formalism has been constructed for equivariant motivic homotopy theory where the group of equivariance is a rather general linear algebraic group. One should view \emph{loc. cit.} as extending motivic homotopy theory from schemes to quotient stacks. In contrast, this paper addresses a different sort of equivariance, namely, \emph{Galois equivariance}. In the context of motivic homotopy theory, this was first studied by Heller and Ormsby in \cite{heller-ormsby}, where the authors established a connection between $C_2$-equivariant homotopy theory and motivic homotopy theory over the real numbers. In this light, one contribution of the present paper is to vastly extend the line of inquiry in \emph{loc. cit.} to the parametrized setting. In particular, we believe that the current paper provides an abstract framework via the six functors formalism for studing Real algebraic cycles (in the style of \cite{etalemotives}*{\S 7.1} and \cite{cisinski-london}). In a future work, we will elaborate on how our formalism is a natural home for various Real cycle class maps recently introduced by various authors in \cite{realcycle1} and \cite{realcycle2}.
\end{remark}

\subsection{Conventions and notation}

\begin{enumerate}

\subsubsection{Algebraic geometry}
\item By convention, all schemes that appear in this paper are quasicompact and quasiseparated (qcqs). If $S$ is a base scheme, we write $\Sch'_S$ to indicate some full subcategory of schemes over $S$ (e.g., finite type $S$-schemes). We add the superscript $\mathrm{fin.dim}$ (resp. $\mathrm{noeth}$) to various categories of schemes to indicate the intersection with finite dimensional (resp. noetherian) schemes.
\item For a ring $A$, resp. scheme $X$, let $\Sper A$, resp. $X_r$ denote the real spectrum, regarded as a topological space in the usual way (cf. \cite{scheiderer}*{0.4}).
\item If $X$ is a scheme and $\tau$ is topology on $\Et_X$, then we write $\widetilde{X_{\tau}}$ for the $\infty$-topos of $\tau$-sheaves of spaces on the small \'etale site of $X$ (e.g., $\widetilde{X}_{\et}$). Abusively, we write $\widetilde{X}_{\pre}$ for the $\infty$-topos of presheaves on $\Et_X$. We also write $\widehat{X_{\tau}}$ for its hypercompletion, and we have a localization $(-)^h: \widetilde{X_{\tau}} \rightarrow \widehat{X_{\tau}}$.
% , and likewise consider $\Shv_{\tau_U}(C, \Escr)$, $\Shv_{\tau_Z}(C, \Escr)$
\item Given a presentable $\infty$-category $\Escr$, we write $\Shv_{\tau}(\C, \Escr)$ for the $\infty$-category of $\Escr$-valued sheaves on $(\C,\tau)$. We also suppress the decoration $\Escr$ if $\Escr = \Spc$ is the $\infty$-category of spaces.
\item If $\C$ is a small $\infty$-category and $X \in \C$, then we denote by $h_X \in \Pre(\C)$ the presheaf corresponding to $X$ under the Yoneda embedding.
\item Let $\C$ be an $\infty$-category and $R \subset \C_{/X}$ be a sieve. By the correspondence in \cite[Proposition 6.2.2.5]{htt}, this determines and is determined by a monomorphism $R \hookrightarrow h_X$.
\subsubsection{(Higher) category theory}
\item Let $\Cat_{\infty}$ and $\widehat{\Cat}_{\infty}$ denote the $\infty$-category of small, resp. large $\infty$-categories, and let $\PrL_{\infty} \subset \widehat{\Cat}_{\infty}$ denote the subcategory of presentable $\infty$-categories and colimit-preserving functors thereof. 
\item Given a geometric morphism $f^{\ast}: \Xscr \rightleftarrows \Yscr: f_{\ast}$ of $\infty$-topoi, we will abuse notation and also write $$f^{\ast}: \Sp(\Xscr) \rightleftarrows \Sp(\Yscr): f_{\ast}$$ for the induced adjunction on stabilizations. Note that since both $f^{\ast}$ and $f_{\ast}$ are left-exact, they are computed via postcomposition by the unstable $f^{\ast}$ and $f_{\ast}$ at the level of spectrum objects.
\item If $\C$ is a stable presentable symmetric monoidal $\infty$-category, we write $\C\comp$ for the $p$-completion of $\C$ (i.e., the Bousfield localization at $\1/p$). We have a localization functor $L_p: \C \rightarrow \C\comp$ and a fully faithful right adjoint $i_p: \C\comp \rightarrow \C$. We use \cite{mnn-descent} as our basic reference for this material.
\item Let $\mathrm{Top}$ be the ordinary category of topological spaces. Then there is a functor $\Sing: \Top \rightarrow \Spc$ which sends a topological space $X$ to $|\Hom(\Delta^{\bullet}, X)|$, where $\Delta^n$ is the standard $n$-simplex.
\item Given a functor $f:S \rightarrow \widehat{\Cat}_{\infty}$, we write $\int f \rightarrow S^{op}$ for the associated cartesian fibration \cite{htt}*{\S3.2}. 
\item Given a cartesian fibration $X \rightarrow S^{\op}$, we write $X^{\vee} \rightarrow S$ for the dual cocartesian fibration \cite{dualizing}.
\item Following standard terminology in parametrized higher category theory, a {\bf $G$-$\infty$-category} is a cocartesian fibration $\C \rightarrow \Oscr_G^{\op}$, and a {\bf $G$-functor} $\C \rightarrow \D$ is a functor that preserves cocartesian edges.\footnote{To avoid confusion with the notion of a $G$-$\infty$-topos, it would perhaps be more appropriate to call such objects $\Oscr_G$-$\infty$-categories and $\Oscr_G$-functors, but we will stick to this terminology.} We refer to \cite{quigley-shah}*{Appendix A} for a quick primer on the theory of $G$-$\infty$-categories. Let $\widehat{\Cat}_{\infty, \Oscr_G}$ denote the $\infty$-category of $G$-$\infty$-categories (so $\widehat{\Cat}_{\infty, \Oscr_G} \simeq \Fun(\Oscr_G^{\op}, \widehat{\Cat}_{\infty}))$).
\item Since $\widehat{\Cat}_{\infty, \Oscr_G}$ is cartesian-closed, it has an internal hom: given $G$-$\infty$-categories $\C$ and $\D$, we let $\underline{\Fun}_G(\C,\D)$ denote the $G$-$\infty$-category of $G$-functors $\C \to \D$. The second author proffered an explicit construction of this internal hom at the level of marked simplicial sets in \cite{jay-thesis}*{\S 3}.
\item $\underline{\Fun}_G(\C,\D)$ serves as a parametrized enhancement of the $\infty$-category $\Fun_G(\C,\D)$ of $G$-functors $\C \to \D$. In general, if a given construction admits some sort of parametrized enhancement, then we will distinguish between the two possibilities by means of this `underline' notation.
\item We adopt the conventions of \cite{quigley-shah}*{\S1} for recollements. In particular, if $\Xscr$ is an $\infty$-category, then a recollement on $\Xscr$ is specified by a pair $(\Uscr, \Zscr)$ where $\Uscr$ is the open part and $\Zscr$ is the closed part, and we have the recollement adjunctions 
\[ \begin{tikzcd}[row sep=4ex, column sep=6ex, text height=1.5ex, text depth=0.5ex]
\Uscr \ar[shift right=1,right hook->]{r}[swap]{j_{\ast}} & \Xscr \ar[shift right=2]{l}[swap]{j^{\ast}} \ar[shift left=2]{r}{i^{\ast}} & \Zscr \ar[shift left=1,left hook->]{l}{i_{\ast}}
\end{tikzcd} \]
in which the composite $i^{\ast} j_{\ast}$ is said to be the {\bf gluing functor} of the recollement. Conversely, given a left-exact functor $\phi: \Uscr \rightarrow \Zscr$, the {\bf right-lax limit} $$\Xscr = \Uscr \times_{\phi, \Zscr, \ev_1} \Fun(\Delta^1, \Zscr) = \Uscr \overrightarrow{\times} \Zscr$$ admits a recollement given by $(\Uscr, \Zscr)$ with gluing functor equivalent to $\phi$. Note also that in the stable setting a recollement is completely determined by a localizing and colocalizing stable subcategory \cite{barwick-glasman}.
\subsubsection{Motivic homotopy theory}
% We let $\H(X)$ resp. $\H(X)_{\bullet}$ be the unstable (pointed) motivic homotopy $\infty$-category, whose objects are called (pointed) motivic spaces.
\item We denote by $\TT$ the pointed motivic space $\TT:=\frac{\AA^1}{\AA^1 \setminus 0}$. We have a canonical equivalence $\TT \xrightarrow{\simeq} \PP^1$ of pointed motivic spaces.
\item Let $\tau$ be a Grothendieck topology and $X$ a scheme. We will denote by $$\H_{\tau}(X), \: \H_{\tau}(X)_{\bullet}, \: \SH^{S^1}_{\tau}(X), \: \SH_{\tau}(X)$$ the unstable, pointed, $S^1$-stable and $\TT$-stable motivic $\infty$-categories defined with respect to the hypercomplete $\tau$-topology. By definition, these are the localizations of the usual motivic categories $$\H(X), \: \H(X)_{\bullet}, \: \SH^{S^1}(X), \: \SH(X)$$ at $\tau$-hypercovers (or desuspensions thereof). See \cite{elso}*{\S2} for details.
\item We remark on the choice to work with the hypercomplete versions of motivic homotopy theory:
\begin{itemize}
\item Bachmann's theorems in \cite{bachmann-et} regarding \'etale motivic homotopy theory only hold after hypercompletion.
\item In Appendix~\ref{subsec:hyp}, we prove that the real \'etale topos of a scheme of finite Krull dimension is hypercomplete. Thus, the hypercomplete and non-hypercomplete versions of real \'etale motivic homotopy theory coincide.
 \end{itemize}
\end{enumerate}

\subsection{Acknowledgements} Special thanks go to Tom Bachmann for interesting discussions related to this work and for setting the stage for it by writing \cite{bachmann-ret} and \cite{bachmann-et}. We would also like to thank Joseph Ayoub, Clark Barwick, Mark Behrens, Dori Ali-Bejleri, Denis-Charles Cisinski, Peter Haine, Marc Hoyois, Paul Arne \O stv\ae r, James Tao and Claus Scheiderer for their interest and helpful discussions related to this paper. We also thank an anonymous referee for correcting some mathematical errors and suggestions which improved the exposition of this paper. The authors were partially supported by the National Science Foundation under grants DMS-1440140 and DMS-1547292. Part of this work was also completed while the first author was in residence at the Mathematical Sciences Research Institute for the ``Derived Algebraic Geometry'' program in spring 2019. Finally, both authors were together at the University of Chicago as undergraduates many years ago --- we would like to thank Peter May and the Topology group there for originally sparking our interest in this subject.

\section{Genuine stabilization of $G$-$\infty$-topoi} \label{sect:gen-stab}

As we recalled in the introduction, one constructs the stable $\infty$-category $\Sp^G$ of {\bf genuine $G$-spectra} so as to have a good theory of compact and dualizable objects in equivariant stable homotopy theory. Traditionally, to do this one $\otimes$-inverts inside $\Spc_{G \bullet}$ the set $\{ S^V \}$ of (finite-dimensional) real representation spheres (see \cite{lewis} and \cite{bachmann-hoyois}*{\S9.2} for a modern treatment). More relevant for our purposes are two more recent and distinct approaches to constructing $\Sp^G$:

\begin{enumerate}
\item Let $\Mack_G = \Fun^{\times}(\Span(\Fin_G), \Sp)$ denote the $\infty$-category of {\bf spectral Mackey functors}. A theorem of Guillou-May shows that $\Mack_G$ is equivalent to $\Sp^G$ \cite{guillou-may}; see also \cite{denis-stab} for a proof of this result in the language of the current paper. This perspective was elaborated further in \cite{barwick} and was later shown by Nardin in \cite{denis-stab} to be an instance of a more general procedure that, given a $G$-$\infty$-category, outputs a {\bf $G$-stable $G$-$\infty$-category} in the sense of \cite{param}.
\item Given $E \in \Sp^G$, the associated Mackey functor evaluates on an orbit $G/H$ to the {\bf categorical fixed points} spectrum $E^H$. The Guillou-May theorem thus says that a genuine $G$-spectrum can be understood diagramatically in terms of its categorical fixed points. These categorical fixed points can in turn be described via {\bf geometric fixed points} and data coming from {\bf Borel equivariant homotopy theory} (see \cite{barwick}*{B.7}). For example, if $G = C_p$ is the cyclic group of order $p$, then we have the pullback square
\[
\begin{tikzcd}
E^{C_p} \ar{r} \ar{d} & E^{\Phi C_p} \ar{d}\\ 
E^{hC_p} \ar{r} & E^{tC_p},
\end{tikzcd}
\]
where $E^{tC_p}$ is the cofiber of the additive norm map $E_{hC_p} \rightarrow E^{hC_p}$. From this point of view, one may reconstruct genuine $G$-spectra in terms of its ``Borel pieces" (see \cite{glasman} and \cite{naive}).
\end{enumerate}

The goal of this section is generalize this picture from $G$-spaces to $G$-$\infty$-topoi.

\begin{definition} Let $G$ be a finite group. A {\bf $G$-$\infty$-topos} $\Xscr$ is a $G$-object in $\RTop$, i.e., a functor
\[ \Xscr: BG \rightarrow \RTop. \]
\end{definition}

Given a $G$-$\infty$-topos $\Xscr$, we will abuse notation and also refer to its underlying $\infty$-topos as $\Xscr$.

\begin{example} \label{ex:g-top-ex} Let $X$ be a topological space with a $G$-action through continuous maps. We may form the $\infty$-category $\Shv(X)$ of sheaves of spaces on $X$, which is naturally a $G$-$\infty$-topos. For instance, if $X= \ast$ is the one-point space (with necessarily trivial $G$-action), then $\Shv(X) \simeq \Spc$ with trivial $G$-action, which is the terminal object in the $\infty$-category $(\RTop)^{BG} = \Fun(BG, \RTop)$ of $G$-$\infty$-topoi.
\end{example}

In this section, we will explain how to functorially associate to a $G$-$\infty$-topos $\Xscr$ a certain $\infty$-category $\Sp^G(\Xscr)$ of ``$G$-spectrum objects'' in $\Xscr$ (Definition \ref{construct:denis}), which we will regard as its {\bf genuine stabilization}. For the terminal $G$-$\infty$-topos $\Spc$, this construction by design specializes to $\Mack_G$ (Example \ref{ex:g-sp}), so we are entitled to think of $\Sp^G(\Xscr)$ as a generalization of (1) wherein we adopt spectral Mackey functors as our preferred \emph{definition} of $G$-spectra. Specializing to the case $G = C_p$, we then introduce the toposic generalization of the $C_p$-Tate construction (Definition \ref{def:stable-glue}) and prove a reconstruction theorem for $\Sp^{C_p}(\Xscr)$ (Theorem \ref{prop:recoll-stab}) that generalizes (2). We conclude by applying our reconstruction theorem to endow $\Sp^{C_p}(\Xscr)$ with a symmetric monoidal structure (Construction \ref{con:MonoidalStructureGenuineStabilization}).
% This is a two-step procedure: 1) we associate to $\Xscr$ a $G$-$\infty$-category $\underline{\Xscr}_{G}$, and 2) we plug $\underline{\Xscr}_{G}$ into the $G$-stabilization machine in \cite{denis-stab}.

\subsection{Toposic homotopy fixed points, homotopy orbits and genuine orbits}

Suppose that $\C$ is an $\infty$-category that admits all (small) limits and colimits, so that we have the adjunctions
\[
(-)^{\triv}: \C \rightleftarrows \C^{BG}: (-)^{hG} = \lim_{BG}, \qquad  (-)_{hG} = \colim_{BG}: \C^{BG} \rightleftarrows \C: (-)^{\triv}.
\]
The functors $(-)^{hG}$ and $(-)_{hG}$ are called the {\bf homotopy fixed points} and {\bf homotopy orbits}, respectively.

Now let $\C = \RTop$, which admits all (small) limits and colimits in view of \cite[\S6.3]{htt}. We are then interested in the homotopy fixed points and orbits as taken in $\RTop$, as opposed to $\widehat{\Cat}_{\infty}$. To distinguish between the two possibilities, we introduce the following modified notation/definition.

\begin{definition} Let $\Xscr$ be a $G$-$\infty$-topos. Then we let
\[ \Xscr_{\hh G} = \colim_{BG} \Xscr, \qquad \Xscr^{\hh G} = \lim_{BG} \Xscr\]
denote the {\bf toposic homotopy orbits} and {\bf toposic homotopy fixed points} of $\Xscr$, i.e., the colimit and limit over $BG$ as formed in $\RTop$. We will also use the terms {\bf homotopy orbits topos} and {\bf homotopy fixed topos} to refer to $\Xscr_{\hh G}$ and $\Xscr^{\hh G}$.

By definition, these $\infty$-topoi come with the canonical geometric morphisms
\[ \nu^{\ast}: \Xscr^{\hh G} \rightleftarrows \Xscr: \nu_{\ast}, \qquad \pi^{\ast}: \Xscr \rightleftarrows \Xscr_{\hh G}: \pi_{\ast}, \]
which are moreover $G$-equivariant with respect to the trivial $G$-actions on $\Xscr_{\hh G}$ and $\Xscr^{\hh G}$.
\end{definition}

For brevity, we will also make use of the following notation.

\begin{notation} Suppose $\Xscr$ is a $G$-$\infty$-topos. We let $$\underline{\Xscr} \rightarrow BG$$ denote the cartesian fibration associated to the functor $$\Xscr^{\op}: BG^{\op} \rightarrow \LTop \subset \widehat{\Cat}_{\infty}.$$
\end{notation}

\subsubsection{Toposic homotopy orbits}

Recall that for any simplicial set $S$, we have a natural equivalence 
\begin{equation} \label{eq:prlr}
(-)^{\dagger}:[S, \RTop] \stackrel{\simeq}{\rightarrow} [S^{\op}, \LTop]
\end{equation}
that transports a diagram
\begin{equation*}
\Fscr:S \rightarrow \RTop \qquad \left( f: s \rightarrow t \right) \mapsto \left( f_*:\Fscr(s) \rightarrow \Fscr(t) \right)
\end{equation*}
to the diagram
\begin{equation*}
\Fscr^{\dagger}: S \rightarrow \LTop \qquad \left( f: t \rightarrow s \right) \mapsto \left( f^*:\Fscr(t) \rightarrow \Fscr(s) \right),
\end{equation*}
where $(f^*,f_*)$ defines a geometric morphism of $\infty$-topoi \cite{htt}*{Corollary 6.3.1.8}. For clarity, we will always decorate functors into $\LTop$ by a dagger.

The following facts will then enable us to compute colimits in $\RTop$ using the equivalence \eqref{eq:prlr}:
\begin{enumerate}
\item Colimits in $\RTop$ are computed as limits of the corresponding diagram in $\LTop$.
\item The forgetful functor $\LTop \subset \widehat{\Cat}_{\infty}$ creates limits \cite{htt}*{Proposition 6.3.2.3}.
\item Given a diagram $f: K \to \widehat{\Cat}_{\infty}$, we have a natural equivalence \cite{htt}*{Corollary 3.3.3.2} $$\lim_K f \simeq \Sect^{\text{cart}}_{K^{\op}}( \int f).$$ 
\end{enumerate}

It follows that for a $G$-$\infty$-topos $\Xscr$, we have natural equivalences $$\Xscr_{\hh G} \simeq \Xscr^{h G} \simeq \Sect_{BG}(\underline{\Xscr}).$$

\begin{example} \label{exmp:trivial} Suppose that $\Xscr$ is a $G$-$\infty$-topos where the action is trivial. Then $$\Xscr_{\hh G} \simeq \Xscr^{hG} \simeq \Fun(BG, \Xscr).$$ In other words, taking toposic homotopy orbits computes the $\infty$-category of  ``Borel-equivariant" $G$-objects in $\Xscr$.
\end{example}

We also introduce a construction to record the functoriality of the toposic homotopy orbits ranging over subgroups $H \leq G$. 

\begin{construction} \label{construct:borel-orb}  Suppose $\Xscr: BG \to \RTop$ is a $G$-$\infty$-topos. Let $$\Xscr_{\hh(-)}: \Oscr_G \to \RTop$$ be the left Kan extension of $\Xscr$ along the inclusion $BG \subset \Oscr_G$. Under the equivalence \eqref{eq:prlr}, $\Xscr_{\hh(-)}$ corresponds to the functor $$\Xscr_{\hh(-)}^{\dagger}: \Oscr_G^{\op} \to \widehat{\Cat}_{\infty}, \qquad G/H \mapsto \Fun_{/BG}(BH,\underline{\Xscr}), $$ with functoriality given by restriction in the source (using the action groupoid functor $\Oscr_G \to \Gpd$ that sends $G/H$ to $BH$). We denote the cocartesian fibration associated to $\Xscr_{\hh(-)}^{\dagger}$ by
\[
\underline{\Xscr}_{\hh G} \rightarrow \Oscr_G^{\op}.
\]
\end{construction}

In particular, $\underline{\Xscr}_{\hh G}$ constitutes an example of a $G$-$\infty$-category.

\begin{remark} Given a map $f: G/H \to G/K$ of $G$-orbits, the restriction functor $$f^{\ast}: \Fun_{/BG}(BK,\underline{\Xscr}) \to \Fun_{/BG}(BH,\underline{\Xscr})$$ admits both left and right adjoints $f_!$ and $f_{\ast}$ given by relative left and right Kan extension along $BH \to BK$ (cf. \cite{quigley-shah}*{\S2.2.1} for the theory of relative Kan extensions along a general functor).
\end{remark}

\begin{example} \label{exmp:et-cover} Suppose that $X$ is a scheme such that $\frac{1}{2} \in \Oscr_X$. Then $X[i] \rightarrow X$ is a $C_2$-Galois cover and thus the $\infty$-topos $\widetilde{X[i]}_{\et}$ acquires a canonical $C_2$-action. The homotopy orbits topos of $\widetilde{X[i]}_{\et}$ is, by Galois descent, equivalent to the \'etale topos of $X$, so we have:
\[
(\widetilde{X[i]}_{\et})_{\hh C_2} \simeq (\widetilde{X[i]}_{\et})^{hC_2} \simeq \widetilde{X}_{\et}.
\]
\end{example}

\subsubsection{Toposic homotopy fixed points}

Since limits in $\RTop$ are not computed at the level of the underlying $\infty$-categories, it is usually difficult to describe $\Xscr^{\hh G}$ explicitly. However, this task has been accomplished by Scheiderer in our main examples of interest at the level of $1$-topoi, as we now recall. Let $\RTop_1$ denote the $(2,1)$-category of $1$-topoi and geometric morphisms thereof, and let a \df{$G$-topos} be a $G$-object in $\RTop_1$. Note for the following two examples that Scheiderer's notion of the fixtopos of a $G$-topos (\cite{scheiderer}*{Definition 10.15}) corresponds to taking the limit over $BG$ in $\RTop_1$ (cf. \cite{scheiderer}*{10.15.1}; Scheiderer formulates his constructions in the setting of fibered topoi).

\begin{example} \label{ex:involute} Suppose that $X$ is a topological space with $G$-action where $G$ acts properly discontinuously on $X$, and consider the $G$-topos $\Xscr = \Shv(X, \Set)$. Then by \cite{scheiderer}*{Proposition 13.2}, we have a canonical equivalence of $1$-topoi
\[
\Shv(X, \Set)^{\hh G} \simeq \Shv(X^G, \Set),
\]
where $X^G$ is the subspace of $G$-fixed points in $X$ and the limit is taken in $\RTop_1$. Moreover, the geometric morphism $(\nu^{\ast}, \nu_{\ast})$ is the one induced by the inclusion of spaces $X^G \subset X$.
\end{example}

%\footnote{We use the fact $X$ has the homotopy type of a finite CW-complex to conclude that the $\infty$-topoi $\Shv(X^G)$ is hypercomplete. Otherwise the statement is true after hypercompletion; see Remark~\ref{rem:1-top}.},

% and consider $\Xscr = \widetilde{X[i]}_{\et}$
\begin{example} \label{ex:ret} Suppose that $X$ is a scheme with $1/2 \in \Oscr_X$, and consider the $C_2$-topos $\Xscr = \Shv_{\et}(\Et_{X[i]}, \Set)$. Then by \cite{scheiderer}*{Theorem 11.1.1}, we have a canonical equivalence of $1$-topoi
\[
\Shv_{\et}(\Et_{X[i]}, \Set)^{\hh C_2} \simeq \Shv_{\ret}(\Et_{X}, \Set),
\]
where the limit is taken in $\RTop_1$.
\end{example}

% In both Examples \ref{ex:involute} and \ref{ex:ret}, Scheiderer proves the relevant theorem at the level of the underlying $1$-topoi.
We next indicate how to deduce equivalences at the level of $\infty$-topoi (we deal with Example \ref{ex:ret}; the argument for Example \ref{ex:involute} is similar).

\begin{theorem} \label{thm:exret} Suppose that $X$ is a scheme with $1/2 \in \Oscr_X$. Then there is a canonical equivalence
\[
\widetilde{X[i]}_{\et}^{\hh C_2} \simeq \widetilde{X}_{\ret}.
\]
\end{theorem}

% Denote by $\RTop_1$ the 1-category of ordinary topoi.
% Scheiderer proves the theorem as \cite{scheiderer}*{Theorem 11.1.1} at the level of sheaves of sets.
\begin{proof} As we recall in Lemma~\ref{lem:fin-lim}, there is a limit-preserving fully faithful functor $\RTop_1 \rightarrow \RTop_{\infty}$ that sends a $1$-topos to its associated $1$-localic $\infty$-topos, such that for any $1$-topos $\Xscr \simeq \Shv_{\tau}(X, \Sets)$ presented as the $1$-category of sheaves of sets on a site $(X, \tau)$ with finite limits, the associated $1$-localic $\infty$-topos is $\Shv_{\tau}(X)$. Since the functor $\RTop_1 \rightarrow \RTop_{\infty}$ preserves limits, the equivalence of Example \ref{ex:ret} then implies the result at the level of sheaves of spaces.
\end{proof}

As we did with toposic homotopy orbits, we also introduce a construction that will record the functoriality of the toposic homotopy fixed points.

\begin{construction} \label{construct:borel-fix} Suppose $\Xscr: BG \to \RTop$ is a $G$-$\infty$-topos. Let
\[
\Xscr^{\hh(-)}:  \Oscr_G^{\op} \rightarrow \RTop
\] 
be the right Kan extension of $\Xscr$ along the inclusion $BG \rightarrow \Oscr_G^{\op}$. Note that by the formula for right Kan extension, we have that $\Xscr^{\hh(-)}$ evaluates to $\Xscr^{\hh H} = \lim_{BH} \Xscr$ on $G/H$. We denote the cartesian fibration associated to $(\Xscr^{\hh(-)})^{\dagger}$ by
\[
\underline{\Xscr}^{\hh G} \rightarrow \Oscr^{\op}_G.
\]
\end{construction}

\begin{remark} \label{rem:caution-cart} We caution the reader that $\underline{\Xscr}^{\hh G}$ is not usually a cocartesian fibration over $\Oscr^{\op}_G$, as its \emph{cartesian} functoriality is that given by the \emph{left} adjoints of a geometric morphism.
\end{remark} 
\subsubsection{Toposic genuine orbits}

In addition to the toposic homotopy orbits and homotopy fixed points, we will need the topos-theoretic counterpart of ``genuine orbits". This is the most subtle of the three constructions and is one of the key innovations of this paper. First, we need a technical lemma.

\begin{lemma} \label{lem:sect} Let $S$ be a small $\infty$-category and $\Fscr: S \rightarrow \RTop$ a functor.
\begin{enumerate}
\item The $\infty$-category of sections
\[
\Sect_S(\int \Fscr^{\dagger})
\]
is an $\infty$-topos.
\end{enumerate}
Now suppose $\phi:T \rightarrow S$ is a functor of small $\infty$-categories.
\begin{enumerate} \setcounter{enumi}{1} \item The restriction functor
\[
\phi^*:\Sect_S(\int \Fscr^{\dagger}) \rightarrow \Sect_T((\int \Fscr^{\dagger}) \times_S T) \simeq \Sect_T(\int (\Fscr \phi)^{\dagger})
\]
preserves finite limits and (small) colimits. In other words, $\phi^*$ is the left adjoint of a geometric morphism of $\infty$-topoi.
\item Suppose that $\phi$ satisfies the following assumption: for any $x \in S$, the simplicial set $T \times_S S^{x/}$ is finite. Then the right adjoint $\phi_{\ast}$ is computed as relative right Kan extension\footnote{For us, the term ``Kan extension'' always refers to the concept of \emph{pointwise} Kan extension.} along $\phi$. 
\item Suppose that $\phi: T \to S$ is (equivalent to) a cocartesian fibration. Then $\phi^*$ admits a left adjoint $\phi_!$, computed as relative left Kan extension along $\phi$. Moreover, given a object $x: T \to \int \Fscr^{\dagger}$ in $\Sect_T(\int (\Fscr \phi)^{\dagger})$, $\phi_! x$ is computed by taking colimits fiberwise, i.e., for all $s \in S$, we have an equivalence
\begin{equation} \label{eq:shriek-formula}
(\phi_! x)(s) \simeq \colim (x|_{T_s}: T_s \rightarrow \Fscr(s)).
\end{equation}
\end{enumerate}
\end{lemma}

\begin{proof} \begin{enumerate}[leftmargin=*]
\item We have a topos fibration (in the sense of \cite{htt}*{Definition 6.3.1.6}) $$\int \Fscr \simeq (\int \Fscr^{\dagger})^{\vee} \rightarrow S^{\op}.$$ By \cite{quigley-shah}*{(1.12)}, we thus have an equivalence
\[
\Sect_S(\int \Fscr^{\dagger}) \simeq \Fun_{/S^{\op}}^{\mathrm{cocart}}(\mathrm{Tw}(S^{\op}), (\int \Fscr^{\dagger})^{\vee}) \simeq \Fun_{/S^{\op}}^{\mathrm{cocart}}(\mathrm{Tw}(S^{\op}), \int \Fscr).
\]
The right hand side is then an $\infty$-topos by the criterion of \cite{htt}*{Proposition 5.4.7.11}.
\end{enumerate}
\begin{enumerate} \setcounter{enumi}{1}
\item Since the cartesian pullback functors of $\int \Fscr^{\dagger}$ preserve finite limits and all colimits, finite limits and all colimits in $\Sect_S(\int \Fscr^{\dagger})$ are computed pointwise. It follows that $\phi^*$ preserves finite limits and all colimits.

\item The content of the claim is that $\phi_*$ is computed by the pointwise formula for the relative right Kan extension. For this, it suffices to show that given a diagram
\[
\begin{tikzcd}
T \ar{r}{x} \ar{d}{\phi} & \int (\Fscr \phi)^{\dagger} \ar{r}{\rho} & \int \Fscr^{\dagger} \ar{d}{p}\\
S \ar[dashed]{urr}[swap]{\phi_{\ast} x}  \ar{rr}[swap]{=} & & S,
\end{tikzcd}
\]
there exists a dotted lift $\phi_{\ast} x$ that is a $p$-right Kan extension of $\rho x$ along $\phi$ in the sense of \cite{quigley-shah}*{Definition 2.8}. Given our hypotheses on the functor $\phi$, this follows by \cite{quigley-shah}*{Remark 2.9}. 

\item This is immediate from the dual of \cite{quigley-shah}*{Corollary 2.11}.
\end{enumerate}
\end{proof}

\begin{definition} \label{def:laxcolim} Suppose we have a diagram $\Fscr: S \rightarrow \RTop$. The {\bf toposic lax colimit} of $\Fscr$ is the $\infty$-topos
\[
\lax.\colim_S \Fscr = \Sect_{S}(\int \Fscr^{\dagger}).
\]
\end{definition}

\begin{remark}\label{rem:sect} By \cite{ghn}*{Proposition 7.1}, $\Sect_{S}(\int \Fscr^{\dagger})$ computes the oplax limit of $\Fscr^{\dagger}$ (as a functor into $\widehat{\Cat}_{\infty}$). This justifies the terminology of Definition \ref{def:laxcolim}.
\end{remark}

\begin{lemma} \label{lem:functoriality} \begin{enumerate}[leftmargin=*]
\item The toposic lax colimit construction assembles into a functor
\[
\lax.\colim: (\Cat)_{\infty/\RTop}^{\op} \rightarrow \widehat{\Cat}_{\infty}.
\]
\end{enumerate}
\begin{enumerate} \setcounter{enumi}{1}
\item Suppose that $F:K \rightarrow (\Cat)_{\infty/\RTop}^{\op}$ is a functor such that for any $i \rightarrow j \in K$ and $x \in F(j)$, the simplicial set $F(i) \times_{F(j)} F(j)^{x/}$ is finite. Then there exists a canonical lift:
\[
\begin{tikzcd}[column sep=8ex]
K \ar[dashed, "\lax.\colim"]{r} \ar[swap]{d}{F} & \LTop  \ar{d} \\
(\Cat)_{\infty/\RTop}^{\op} \ar{r}{\lax.\colim} & \widehat{\Cat}_{\infty}.
\end{tikzcd}
\]
\end{enumerate}
\end{lemma}

\begin{proof}  This follows immediately from Lemma~\ref{lem:sect}.
\end{proof}

%\begin{proof} Now, the functor $\int \Fscr^{\dagger} \rightarrow S$ is a topos fibration by \cite{htt}*{Proposition 6.3.1.7}, hence the space of sections is an $\infty$-topos by \cite{htt}*{Lemma 6.3.3.2}. The fact that the induced functor on section categories is a geometric morphism follows from the fact that for any $s \in S$ the evaluation functor $\Sect_{S^{\op}}(\int \Fscr^{\dagger}) \rightarrow \Fscr(s)$ is a geometric morphism.
%\end{proof}

% The lax colimit functor and its functoriality is used in the next construction. 

\begin{construction} \label{cons:orb} Let $\Xscr: BG \rightarrow \RTop$ be a $G$-$\infty$-topos. Consider the functor
\[
\Oscr_G \rightarrow(\Cat_{\infty})_{/\RTop}, \qquad G/H \mapsto \left( \Xscr^{\hh(-)}|_{(\Oscr^{\op}_G)_{(G/H)/}}: (\Oscr^{\op}_G)_{(G/H)/} \rightarrow \RTop \right),
\]
satisfying the hypotheses of Lemma~\ref{lem:functoriality}. Then the formation of toposic lax colimits yields a functor
\[
\Xscr^{\dagger}_{(-)}: \Oscr_G^{\op} \rightarrow \LTop,
\]
whose value 
\[
\Xscr_G = \Sect_{\Oscr_G^{\op}}(\underline{\Xscr}^{\hh G}),
\] on $G/G$ is the {\bf toposic genuine orbits} of $\Xscr$. We denote the associated \emph{bicartesian} fibration over $\Oscr_G^{\op}$ by
\[
\underline{\Xscr}_{G} \rightarrow \Oscr_G^{\op}.
\]
Given a $G$-equivariant geometric morphism $f_*: \Yscr \to \Xscr$, we have an induced functor $f^*: \underline{\Xscr}^{\hh G} \to \underline{\Yscr}^{\hh G}$ over $\Oscr_G^{\op}$ that preserves cartesian edges, and hence postcomposition induces a functor upon taking sections 
\[
f^*: \Xscr_G \to \Yscr_G.
\]
Moreover, since colimits and finite limits are computed fiberwise in toposic lax colimits, $f^*$ preserves colimits and finite limits and is thus the left adjoint of a geometric morphism.

As $f^*$ is given by postcomposition of sections, it is compatible with pullback of sections, and hence we obtain an induced functor $f^*: \underline{\Xscr}_G \to \underline{\Yscr}_G$. The assignment $\Xscr \mapsto \underline{\Xscr}_G$ thereby assembles to a functor valued in $G$-$\infty$-categories
\[ \Fun(BG, \LTop) \to \widehat{\Cat}_{\infty, \Oscr_G}. \]
\end{construction}

% Note that the assignment $\Xscr \mapsto \underline{\Xscr}_G$ is functorial in the \emph{left} adjoint of a geometric morphism.

\begin{remark} \label{rem:dror-farjoun} Suppose that $X$ is a CW complex equipped with a $G$-action (in the $1$-categorical sense) with at least one fixed point. Then by a result of Dror-Farjoun, for any subgroup $H \leq G$ we have an equivalence
\[
\Sing_{\bullet}(X/H) \simeq \colim_{K \in \Oscr_{\H}} \Sing_{\bullet}(X^K),
\]
where $X^K \subset X$ is the subspace of $K$-fixed points \cite{dror-farjoun}*{Chapter~4, Lemma~A.3}.  In this light, we may view Construction~\ref{cons:orb} as a categorification of this formula for the genuine orbits of a $G$-space.
\end{remark}

In the context of $S$-$\infty$-categories (i.e., cocartesian fibrations over $S$), the theory of parametrized limits and colimits over $S$ was comprehensively studied in the second author's thesis \cite{jay-thesis}. When specialized to $S = \Oscr_G^{\op}$, the resulting notions of admitting all {\bf finite $G$-(co)products} or all {\bf (finite) $G$-(co)limits} unwind to the following more explicit definition \cite{jay-thesis}*{Proposition 5.11 and Corollary 12.15}.\footnote{The assertion that the existence of all \emph{finite} $G$-(co)limits in the sense of \cite{jay-thesis} (where the finiteness condition is on the diagram $G$-$\infty$-category) is equivalent to the below formulation follows as a consequence of the technique of reduction to the Grothendieck construction as explained in \cite{quigley-shah}*{A.12}.}

\begin{definition} \label{def:g-products} A functor 
\[
\Fscr: \Fin_G^{\op} \rightarrow \widehat{\Cat}_{\infty}, \qquad (U \rightarrow V) \mapsto (f^*:\Fscr(V) \rightarrow \Fscr(U)),
\] is said to {\bf admit finite $G$-products} (resp. {\bf $G$-coproducts}) if
\begin{enumerate}
\item $\Fscr$ is a right Kan extension of its restriction to $\Oscr_G^{\op}$. In other words, given an orbit decomposition $U \simeq \coprod_{i} U_i$ of a finite $G$-set $U$, the canonical map $\Fscr(U) \to \prod_i \Fscr(U_i)$ is an equivalence.
%\item for all $T \in \Fin_G$, the $\infty$-category $\Fscr(T)$ has all finite products and for any map $f: T \rightarrow S$ in $\Fin_G$ the functor $f^*:\Fscr(S) \rightarrow \Fscr(T)$ preserves products.
\item For any map $f: U \rightarrow V$ of finite $G$-sets, the functor $f^*:\Fscr(V) \rightarrow \Fscr(U)$ admits a right (resp. left) adjoint
\[
f_{\ast} = \prod_f: \Fscr(U) \rightarrow \Fscr(V) \qquad (\text{resp. } f_! = \coprod_f: \Fscr(U) \rightarrow \Fscr(V)).
\]
\item $\Fscr$ satisfies the {\bf Beck-Chevalley condition}: for any pullback diagram
\[
\begin{tikzcd}
U' \ar{d}{f'} \ar{r}{g'} & U \ar{d}{f}\\
V' \ar{r}{g} & V
\end{tikzcd}
\]
of finite $G$-sets, the commutative diagram
\[
\begin{tikzcd}
\Fscr(V) \ar{d}{f^{\ast}} \ar{r}{g^{\ast}} & \Fscr(V') \ar{d}{f'^{\ast}} \\
\Fscr(U) \ar{r}{g'^{\ast}} & \Fscr(U')
\end{tikzcd}
\]
is right (resp. left) adjointable, i.e., the exchange transformation
\[ f^* g_* \Rightarrow g'_* f'^*, \quad \text{resp.} \quad f'_! g'^* \Rightarrow g^* f_! \]
is an equivalence.
\end{enumerate}
We then say that $\Fscr$ {\bf admits all (finite) $G$-limits} (resp. {\bf admits all (finite) $G$-colimits}) if
\begin{enumerate}
\item $\Fscr$ admits finite $G$-products (resp. finite $G$-coproducts).
\item For every finite $G$-set $U$, $\Fscr(U)$ admits (finite) limits (resp. (finite) colimits).
\item For every map $f: U \to V$ of finite $G$-sets, the functor $f^*$ is preserves all (finite) limits (resp. preserves all (finite) colimits).
\end{enumerate}
Finally, we also say the same for the $G$-$\infty$-category corresponding to $\Fscr|_{\Oscr_G^{\op}}$.
\end{definition}

\begin{remark} \label{rem:ReallyObviousRemark} In Definition~\ref{def:g-products}, if $\Fscr$ admits \emph{both} finite $G$-coproducts and finite $G$-products and $\Fscr(U)$ is cocomplete and complete for all $G$-orbits $U$ (e.g., presentable), then $\Fscr$ admits all $G$-colimits and $G$-limits.
\end{remark}

\begin{remark} \label{rem:ExtendedToposConstructions} The functors $\Xscr^{\dagger}_{\hh(-)}$ and $\Xscr_{(-)}^{\dagger}$ extend in a natural way from $\Oscr_G^{\op}$ to $\Fin_G^{\op}$. For $\Xscr^{\dagger}_{\hh(-)}$, let $\FinGpd$ be the $1$-category of finite groupoids, let $\omega: \Fin_G \to \FinGpd_{/BG}$ be the action groupoid functor $U \mapsto U//G$, and define the extension
\[ \Xscr^{\dagger}_{\hh(-)}: \Fin_G^{\op} \xrightarrow{\omega} \FinGpd_{/BG}^{\op} \rightarrow \LTop, \quad U \mapsto \Xscr_{\hh U} = \Fun_{/BG}(U//G, \underline{\Xscr}).  \]
For $\Xscr_{(-)}^{\dagger}$, let $\underline{U} \to \Oscr_{G}^{\op}$ be the left fibration classified by $\Hom_{\Fin_G}(-,U)$ (covariantly functorial in maps of finite $G$-sets) and define the extension
\[ \Xscr_{(-)}^{\dagger}: \Fin_G^{\op} \rightarrow \LTop, \quad U \mapsto \Xscr_U = \Fun_{/\Oscr_G^{\op}}(\underline{U}, \underline{\Xscr}^{\hh G}). \]
Clearly, these extensions are right Kan extensions of their restrictions to $\Oscr_G^{\op}$.
\end{remark}

To proceed, we need to show that the $G$-$\infty$-categories $\underline{\Xscr}_{\hh G}$ and $\underline{\Xscr}_G$ admit all $G$-colimits and $G$-limits. By Remark~\ref{rem:ReallyObviousRemark}, it suffices to show the existence of finite $G$-coproducts and finite $G$-products.

\begin{proposition} \label{prop:genuine} Let $\Xscr$ be a $G$-$\infty$-topos. Then the functors
\[
\Xscr^{\dagger}_{\hh(-)}: \Fin_G^{\op} \rightarrow \LTop \subset \widehat{\Cat}_{\infty}
\]
and 
\[
\Xscr_{(-)}^{\dagger}: \Fin_G^{\op} \rightarrow \LTop \subset \widehat{\Cat}_{\infty}
\]
defined as in Remark~\ref{rem:ExtendedToposConstructions} admit finite $G$-coproducts and finite $G$-products.
\end{proposition}

To prove this result, we need a few categorical preliminaries on the Beck-Chevalley condition.

\begin{lemma} \label{lem:bc1} Suppose that we have a diagram of $\infty$-categories
\[ \begin{tikzcd}[row sep=4ex, column sep=4ex, text height=1.5ex, text depth=0.25ex]
I \times_J K \ar{r}{\psi'} \ar{d}{\phi'} & I \ar{d}{\phi} \ar{r}{F} & C \ar{d}{p} \\
K \ar{r}{\psi} & J \ar{r} & S
\end{tikzcd} \]
in which $p$ is a cartesian fibration and the lefthand square is homotopy cartesian. Consider the class of simplicial sets $$\Ascr = \{ I \times_J J^{y/} : y \in J \},$$
and suppose that the fibers of $C$ admit $\Ascr$-indexed limits and the cartesian pullback functors of $C$ preserve $\Ascr$-indexed limits. Then the $p$-right Kan extensions $\phi_{\ast} F$ and $\phi'_{\ast} (F \circ \psi')$ exist, and we may consider the exchange transformation
$$ \chi: \psi^{\ast} \phi_{\ast} F \to \phi'_{\ast} \psi'^{\ast} F. $$
Furthermore, if the map $\theta_x$ of Lemma \ref{lem:BeckChevalley} is right cofinal for all $x \in K$, then $\chi$ is an equivalence.
\end{lemma}

\begin{proof} The first claim follows from the pointwise existence criterion for $p$-right Kan extensions \cite{quigley-shah}*{Remark 2.9}. The second then follow immediately from the definitions.
\end{proof}

% Lemma~\ref{lem:bc1} tells us that in order for the exchange transformation $\chi$ to be equivalence, we need to find conditions on left square to ensure that $\theta_x$ is right cofinal. This is addressed by the next lemma.

\begin{lemma} \label{lem:BeckChevalley} Suppose we have a homotopy pullback square of $\infty$-categories
\[ \begin{tikzcd}[row sep=4ex, column sep=4ex, text height=1.5ex, text depth=0.25ex]
I \times_J K \ar{r}{\psi'} \ar{d}{\phi'} & I \ar{d}{\phi} \\
K \ar{r}{\psi} & J.
\end{tikzcd} \]
For $x \in K$, consider the induced map
$$ \theta_x: I \times_J K \times_K K_{x/} \to I \times_J J_{\psi(x)/}.$$
\begin{enumerate}
\item Suppose that $\psi$ is equivalent to a left fibration and the induced functor $K_{x/} \to J_{\psi(x)/}$ has weakly contractible fibers. Then $\theta_x$ is right cofinal. In particular, if $\psi: K \to J$ is equivalent to a map of corepresentable left fibrations $f^{\ast}: S_{v/} \simeq (S_{v/})_{f/} \to S_{u/}$ for a morphism $f: u \to v$ in an $\infty$-category $S$, then $\theta_x$ is right cofinal.
\item  Suppose that $\phi^{\op}: I^{\op} \to J^{\op}$ is smooth \cite{htt}*{Definition 4.1.2.9} (for example, if $\phi$ is equivalent to a cartesian fibration \cite{htt}*{Proposition 4.1.2.15}). Then $\theta_x$ is right cofinal.
\end{enumerate}
\end{lemma}
\begin{proof} \begin{enumerate}[leftmargin=*] \item By (the dual of) Joyal's version of Quillen's theorem A \cite{htt}*{Theorem 4.1.3.1}, it suffices to check that for any object $(a,\alpha) \in I \times_J J_{\psi(x)/}$, the pullback
$$ (I \times_J K \times_K K_{x/}) \times_{I \times_J J_{\psi(x)/}} (I \times_J J_{\psi(x)/})_{/(a,\alpha)} $$
is weakly contractible. Note that we have a pullback square
\begin{equation} \label{eq:pull}
\begin{tikzcd}[row sep=4ex, column sep=4ex, text height=1.5ex, text depth=0.25ex]
I \times_J K \times_K K_{x/} \ar{r}{\theta_x} \ar{d} & I \times_J J_{\psi(x)/} \ar{d} \\
K_{x/} \ar{r} & J_{\psi(x)/}.
\end{tikzcd} 
\end{equation}
Since $\psi$ was assumed to be (equivalent to) a left fibration, the same holds for $\theta_x$. Now, left fibrations are smooth by \cite{htt}*{Proposition 4.1.2.15}. Since the inclusion of the final object $$\{(a,\alpha) \}  \hookrightarrow (I \times_J J_{\psi(x)/})_{/(a,\alpha)}$$ is a left cofinal inclusion and hence right anodyne \cite{htt}*{Proposition 4.1.1.3(4)}, we have that the induced map $$(I \times_J K \times_K K_{x/}) \times_{I \times_J J_{\psi(x)/}} \{(a,\alpha) \} \rightarrow (I \times_J K \times_K K_{x/}) \times_{I \times_J J_{\psi(x)/}}  (I \times_J J_{\psi(x)/})_{/(a,\alpha)}$$ continues to be left cofinal since right anodyne maps are preserved under base change along a smooth morphism \cite{htt}*{Proposition 4.1.2.8}.

But now, from the cartesian diagram~\eqref{eq:pull}, we have an equivalence $$(K_{x/})_\alpha \simeq  (I \times_J K \times_K K_{x/}) \times_{I \times_J J_{\psi(x)/}} \{(a,\alpha) \}.$$ Thus, it further suffices to check that for any object $\alpha: \psi(x) \to y$ in $J_{\psi(x)/}$, the fiber $(K_{x/})_\alpha$ is weakly contractible, which is true by assumption.

Finally, suppose that $\psi$ is equivalent to a functor $f^{\ast}: S_{v/} \to S_{u/}$ for a morphism $f: u \to v$ in $S$, and let $[g: v \to w] \in S_{v/}$ be any object. Then the induced functor $$(S_{v/})_{g/} \xto{\simeq} (S_{x/})_{(g \circ f)/} \simeq S_{w/}$$ is an equivalence, verifying the criterion of (1). 
% \end{enumerate}
% \begin{enumerate} \setcounter{enumi}{1} 
\item We may factor $\theta_x$ as
\[ \begin{tikzcd}
I \times_J K \times_K \{ x \} \ar{r}{\cong} \ar{d}{\iota} & I \times_J \{ \psi(x) \} \ar{d}{\iota'} \\
I \times_J K \times_K K_{x/} \ar{r}{\theta_x} & I \times_J J_{\psi(x)/}.
\end{tikzcd} \]

Since $\phi^{\op}$ is smooth and $\{x \} \to K_{x/}$ is a right cofinal inclusion, we deduce that $\iota$ is right cofinal, using now the dual of \cite{htt}*{Proposition 4.1.2.8}. Likewise, $\iota'$ is right cofinal. The right cancellative property of right cofinal maps \cite[Proposition 4.1.1.3(2)]{htt} then shows that $\theta_x$ is right cofinal.
\end{enumerate}
\end{proof}

% Because the left adjoint in a geometric morphism is left-exact, the cartesian fibration $\underline{\Xscr}^{\hh G} \to \Oscr_{G}^{\op}$ satisfies the relevant existence hypotheses for the relative right Kan extension along a map $\underline{U} \to \underline{V}$ induced by a map of finite $G$-sets $f: U \to V$, where $\underline{U} \to \Oscr_G^{\op}$ is the left fibration classified by $\Hom_{\Fin_G}(-,U)$. Therefore, the

\begin{proof} [Proof of Proposition~\ref{prop:genuine}]We first verify the claim for $\Xscr_{(-)}^{\dagger}$. Let $f: U \to V$ be any map of finite $G$-sets and also denote by $f$ the induced map of $G$-$\infty$-categories $\underline{U} \to \underline{V}$. By Lemma~\ref{lem:sect}.3 and Lemma~\ref{lem:sect}.4, the coinduction and induction functors $f_{\ast}, f_!: \Xscr_U \to \Xscr_V$ exist and are computed by relative right and left Kan extension along $f$. It remains to check the Beck-Chevalley condition. For this, suppose we have a pullback square of finite $G$-sets
\[ \begin{tikzcd}[row sep=4ex, column sep=4ex, text height=1.5ex, text depth=0.25ex]
W \times_V U \ar{r}{g'} \ar{d}{f'} & U \ar{d}{f} \\
W \ar{r}{g} & V.
\end{tikzcd} \]
Without loss of generality, we may suppose that $U, V, W$ are orbits. Then the exchange transformation $f^* g_* \Rightarrow g'_* f'^*$ is an equivalence by Lemma~\ref{lem:bc1} and Lemma~\ref{lem:BeckChevalley}.1, and the exchange transformation $f'_! g'^* \Rightarrow g^* f_!$ is an equivalence by the dual of Lemma~\ref{lem:BeckChevalley}.2.

The claim for $\Xscr_{\hh(-)}$ follows by the same argument where $\Oscr_G^{\op}$ is replaced by $BG$, and we use instead Lemma~\ref{lem:BeckChevalley}.2 for the coinduction functors and its dual for the induction functors.
\end{proof}

We are now prepared to establish functoriality of the assignment $\Xscr \mapsto \underline{\Xscr}_G$ in the \emph{right} adjoint of a $G$-equivariant geometric morphism. For the following proposition, recall the notion of a relative adjunction \cite{higheralgebra}*{\S7.3.2}. Given a relative adjunction $L \dashv R$ of $G$-$\infty$-categories, we say that $L \dashv R$ is a \emph{$G$-adjunction} if both $L$ and $R$ preserve cocartesian edges over $\Oscr_G^{\op}$ \cite{jay-thesis}*{\S8}.

\begin{proposition} \label{prp:InducedAdjunctionParamTopOrbits} Let $f_{\ast}: \Xscr \to \Yscr$ be a $G$-equivariant geometric morphism of $G$-$\infty$-topoi and consider the $G$-functor $f^{\ast}: \underline{\Yscr}_G \to \underline{\Xscr}_G$. Then for any map of $G$-orbits $\alpha: U \to V$, the exchange transformation
\[ \chi: \alpha^{\ast} f_{\ast} \Rightarrow f_{\ast} \alpha^{\ast} \]
is an equivalence. Therefore, the fiberwise pushforward functors $f_*$ assemble into a $G$-functor $f_*: \underline{\Xscr}_G \to \underline{\Yscr}_G$ that is $G$-right adjoint to $f^*$.

Similarly, we have a $G$-adjunction $f^*: \underline{\Yscr}_{\hh G} \rightleftarrows \underline{\Xscr}_{\hh G}: f_*$.
\end{proposition}
\begin{proof} By the general theory of relative adjunctions, the functor $f_*: \underline{\Xscr}_G \to \underline{\Yscr}_G$ exists as a relative right adjoint to $f^*$ \cite{higheralgebra}*{Proposition~7.3.2.6}, but does not necessarily preserve cocartesian edges. To check that it does amounts to verifying that $\chi$ is an equivalence for all morphisms $\alpha$ in $\Oscr_G$.

By Proposition~\ref{prop:genuine}, the induction functor $\alpha_!$ exists. Therefore, it suffices to show the adjoint exchange transformation $\chi': \alpha_! f^{\ast} \Rightarrow f^{\ast} \alpha_!$ is an equivalence. But this is evident in view of the fiberwise colimit description of $\alpha_!$~\eqref{eq:shriek-formula} as well as the definition of $f^{\ast}$ as given by postcomposition of sections.

Finally, the analogous claim for $\underline{\Xscr}_{\hh G}$ is similar but easier.
\end{proof}

In this way, the assignment $\Xscr \mapsto \underline{\Xscr}_G$ assembles into a functor
\[ \underline{(-)}_G: \Fun(BG, \RTop) \to \widehat{\Cat}_{\infty, \Oscr_G}. \]

\subsection{Genuine stabilization: the Mackey approach} \label{section:MackeyApproach}

The parametrized toposic genuine orbits construction is designed so as to feed into the machinery of parametrized $\infty$-categories. In particular, we wish to now deploy Nardin's $G$-stabilization construction. To explain this notion, first recall from \cite{higheralgebra}*{\S1.4} that we have a functor
\[
\Sp: \widehat{\Cat}^{\lex}_{\infty} \rightarrow \widehat{\Cat}^{\ex}_{\infty,\stab}
\]
that assigns to an $\infty$-category $\C$ with finite limits the stable $\infty$-category $\Sp(\C)$ of {\bf spectrum objects} in $\C$, which is the universal stable $\infty$-category equipped with a left-exact functor $\Omega^{\infty}: \Sp(\C) \rightarrow \C$.

\begin{definition} Let $p: \C \to S$ be a cocartesian fibration that straightens to a functor $S \to \Cat_{\infty}^{\lex}$. We then write $\underline{\Sp}(\C) \to S$ for the {\bf fiberwise stabilization} of $p$, and $\Omega^{\infty}: \underline{\Sp}(\C) \to \C$ for the induced functor over $S$ (which preserves cocartesian edges). If $p$ straightens to a functor valued in $\widehat{\Cat}^{\ex}_{\infty,\stab}$ (or equivalently, if $\Omega^{\infty}$ is an equivalence), then we say that $\C$ is {\bf fiberwise stable}.
\end{definition}

Along with fiberwise stability, we need the additional concept of {\bf $G$-semiadditivity} to formulate $G$-stability.

\begin{definition} \label{def:semiadditive} Let $\C$ be a $G$-$\infty$-category, and write $\Fscr:\Fin_G^{\op} \rightarrow \widehat{\Cat}_{\infty}$ for the functor corresponding to $\C$. Suppose that $\C$ admits finite $G$-products and coproducts and is pointed.\footnote{In other words, for each $U \in \Fin_G$, the canonical comparison map from the initial to the terminal object in $\Fscr(U)$ is an equivalence.} Then $\C$ is {\bf $G$-semiadditive} if for all $f: U \rightarrow V$ in $\Fin_G$, the canonical comparison transformation of \cite{denis-stab}*{Construction 5.2}
\[
\coprod_f (-) \Rightarrow \prod_f (-)
\]
is an equivalence.
\end{definition}

\begin{definition} \label{def:stable} Let $\C$ be a pointed $G$-$\infty$-category that admits finite $G$-products and $G$-coproducts. We say that $\C$ is {\bf $G$-stable} if it is fiberwise stable and $G$-semiadditive.
\end{definition}

% The next construction produces a $G$-stable $G$-$\infty$-category from a $G$-$\infty$-category. 
Let $\widehat{\Cat}_{\infty,\Oscr_G}^{\lex}$ be the $\infty$-category of $G$-$\infty$-categories that admit finite $G$-limits and $G$-left-exact $G$-functors, and let $\widehat{\Cat}_{\infty,\Oscr_G}^{\ex,\stab}$ be the $\infty$-category of $G$-stable $G$-$\infty$-categories and $G$-exact $G$-functors. We now recall Nardin's construction of the $G$-stabilization functor 
\[
\underline{\Sp}^G:\widehat{\Cat}_{\infty,\Oscr_G}^{\lex} \rightarrow \widehat{\Cat}_{\infty,\Oscr_G}^{\ex,\stab}.
\] 
\begin{construction} \label{rem:nardin} Suppose that $\C$ is a $G$-$\infty$-category that admits finite $G$-limits, and let $\C_H = \C \times_{\Oscr_G^{\op}} \Oscr_H^{\op}$ denote the restriction of $\C$ to a $H$-$\infty$-category.\footnote{We implicitly use the identification $\Oscr_H^{\op} \simeq (\Oscr_G^{\op})^{(G/H)/}$.} Let $\underline{\Fin}_{G \ast}$ denote the $G$-category of finite pointed $G$-sets \cite{denis-stab}*{Definition~4.12}, whose fiber over $G/H$ is equivalent to $\Fin_{H \ast}$. Define the $G$-$\infty$-category of {\bf $G$-commutative monoids} in $\C$ to be the full $G$-subcategory
\[
\underline{\CMon}^G(\C) \subset \underline{\Fun}_G(\underline{\Fin}_{G \ast}, \C)
\]
whose fiber $\CMon^H(\C_H)$ over $G/H$ is the full subcategory on those $H$-functors $\underline{\Fin}_{H \ast} \to \C_H$ that send finite $H$-coproducts to finite $H$-products \cite{denis-stab}*{Definition~5.9}. $\underline{\CMon}^G(\C)$ is then $G$-semiadditive \cite{denis-stab}*{Proposition~5.8}.

We also have an alternative model for $\underline{\CMon}^G(\C)$, given by a parametrized variant of the technology of Mackey functors. To define this, we need the {\bf effective Burnside $G$-$\infty$-category} $\underline{\A}^{\eff}(\Fin_{G})$ \cite{denis-stab}*{Definition~4.10}, whose fiber over $G/H$ is equivalent to the effective Burnside $\infty$-category\footnote{In fact, $\A^{\eff}(\Fin_H)$ is a $(2,1)$-category.} $\A^{\eff}(\Fin_H)$ of finite $H$-sets (formed by taking spans in $\Fin_H$). One may check that $\underline{\A}^{\eff}(\Fin_{G})$ admits finite $G$-products and $G$-coproducts and is moreover $G$-semiadditive. Define the $G$-$\infty$-category of {\bf $G$-Mackey functors} in $\C$ to be the full $G$-subcategory
\[
\underline{\Mack}^G(\C) = \underline{\Fun}^{\times}_{G}(\underline{\A}^{\eff}(\Fin_{G}), \C) \subset \underline{\Fun}_{G}(\underline{\A}^{\eff}(\Fin_{G}), \C)
\]
whose fiber $\Mack^H(\C_H)$ over $G/H$ is the full subcategory on those $H$-functors $\A^{\eff}(\Fin_H) \to \C_H$ that preserve finite $H$-products. We then have an equivalence \cite{denis-stab}*{Theorem~6.5} 
\[
\underline{\Mack}^G(\C) \xrightarrow{\simeq} \underline{\CMon}^G(\C)
\]
implemented by restriction along a certain $G$-functor $\underline{\Fin}_{G*} \to \underline{\A}^{\eff}(\Fin_{G})$.

In either case, we write $U$ for the forgetful $G$-functor given by evaluation along the cocartesian section $\Oscr_G^{\op}$ that selects either $G/G_+$ in $\Fin_{G*}$ or $G/G$ in $\A^{\eff}(\Fin_G)$ (so $U$ is compatible with the stated equivalence). $U: \underline{\CMon}^G(\C) \to \C$ is then universal among $G$-left-exact functors from a $G$-semiadditive $G$-$\infty$-category \cite{denis-stab}*{Corollary~5.11.1}, and likewise for $U: \underline{\Mack}^G(\C) \to \C$.\footnote{In the cited reference, Nardin more generally proves that $U$ is universal among $G$-semiadditive functors $\D \to \C$ from pointed $G$-$\infty$-categories $\D$ that admit finite $G$-coproducts.} In particular, if $\C$ is $G$-semiadditive to begin with, then $U$ is an equivalence \cite{denis-stab}*{Proposition 5.11}.

Finally, define the {\bf $G$-stabilization} $\underline{\Sp}^G(\C)$ to be \cite{denis-stab}*{Definition 7.3}
\[
\underline{\Sp}^G(\C) = \underline{\Sp}(\underline{\CMon}^G(\C)).
\]
Since $\underline{\Sp}$ preserves the property of $G$-semiadditivity \cite{denis-stab}*{Lemma 7.2}, the $G$-$\infty$-category $\underline{\Sp}^G(\C)$ is $G$-stable. Define $\Omega^{\infty}: \underline{\Sp}^G(\C) \to \C$ to be the composite
\[
\underline{\Sp}^G(\C) = \underline{\Sp}(\underline{\CMon}^G(\C)) \xrightarrow{\Omega^{\infty}} \underline{\CMon}^G(\C) \xrightarrow{U} \C.
\]
Then $\Omega^{\infty}$ is universal among $G$-left-exact $G$-functors to $\C$ from $G$-stable $G$-$\infty$-categories \cite{denis-stab}*{Theorem~7.4}. For later use, we also observe by the above facts that the forgetful $G$-functor implements an equivalence
\begin{equation} \label{eq:denis-observe}
U: \underline{\CMon}^G(\underline{\Sp}^G(\C)) \xrightarrow {\simeq} \underline{\Sp}^G(\C).
\end{equation}

We now elaborate on the functoriality of genuine stabilization. Suppose $\D$ is another $G$-$\infty$-category that admits finite $G$-limits, and let $F: \C \to \D$ be a $G$-left-exact $G$-functor. Then via postcomposition we have induced $G$-left-exact $G$-functors
\begin{align*}
F: \underline{\CMon}^G(\C) \to \underline{\CMon}^G(\D) &, \qquad F: \underline{\Mack}^G(\C) \to \underline{\Mack}^G(\D), \\
F: \underline{\Sp}(\C) \to \underline{\Sp}(\D) &, \qquad F: \underline{\Sp}^G(\C) \to \underline{\Sp}^G(\D)
\end{align*}
that commute with $U$ and $\Omega^{\infty}$.
\end{construction}

Nardin constructs $\underline{\Sp}^G$ via the two-step procedure of $G$-semiadditivization and fiberwise stabilization. The next lemma shows that we may perform these two operations in either order. Though simple, it will be crucial at various points of the paper.

\begin{lemma} \label{lem:interchange} Let $\C$ be a $G$-$\infty$-category that admits finite $G$-limits. Then there is a canonical equivalence of $G$-$\infty$-categories
\[ \underline{\CMon}^G (\underline{\Sp}(\C)) \simeq \underline{\Sp}(\underline{\CMon}^G(\C)). \]
\end{lemma}
\begin{proof} We already know that $\underline{\Sp}$ of a $G$-semiadditive $G$-$\infty$-category is again $G$-semiadditive (hence $G$-stable). It remains to show that given a fiberwise stable $G$-$\infty$-category $\D$ that admits finite $G$-limits, $\CMon^G(\D)$ is again stable (and thus $\underline{\CMon}^G(\D)$ is fiberwise stable). To see this, we use the fully faithful embedding
\[ \CMon^G(\D) \simeq \Fun_G^{\times}(\underline{\A}^{\eff}(\Fin_G),\D) \subset \Fun_G(\underline{\A}^{\eff}(\Fin_G),\D). \]
Given any $G$-$\infty$-category $K$, in view of our assumption on $\D$, finite limits and colimits in $\Fun_G(K,\D)$ exist and are computed fiberwise, hence $\Fun_G(K,\D)$ is stable. It thus suffices to show that given a finite diagram $p$ of $G$-functors $\underline{\A}^{\eff}(\Fin_G) \to \D$ that preserve finite $G$-products, both the limit and colimit of $p$ as computed in $\Fun_G(\underline{A}^{\eff}(\Fin_G),\D)$ is again a $G$-functor that preserves finite $G$-products. But for any map $f: U \to V$ of $G$-orbits, our assumption is that $f_{\ast}: \D_U \to \D_V$ is an exact functor, so the claim follows.

We then have comparison functors $$\alpha: \underline{\CMon}^G(\underline{\Sp}(\C)) \to \underline{\Sp}(\underline{\CMon}^G(\C)), \quad \beta: \underline{\Sp}(\underline{\CMon}^G(\C)) \to  \underline{\CMon}^G(\underline{\Sp}(\C))$$ obtained via the universal properties of $\underline{\CMon}^G$ as $G$-semiadditivization and $\underline{\Sp}$ as fiberwise stabilization. These are easily checked to be mutually inverse equivalences, using again the universal properties.
\end{proof}

We next consider the presentable situation.

\begin{lemma} \label{lem:SimplePresentabilityCheck} Let $\C$ be a $G$-$\infty$-category that admits finite $G$-products and is fiberwise presentable. Then $\CMon^G(\C)$, resp. $\Mack^G(\C)$ is an accessible localization of $\Fun_G(\Fin_{G\ast}, \C)$, resp. $\Fun_G(\underline{\A}^{\eff}(\Fin_{G}), \C)$. Therefore, $\CMon^G(\C)$ and $\Sp^G(\C)$ are presentable.
\end{lemma}
\begin{proof} Note that under our hypotheses, $\C \to \Oscr_G^{\op}$ is a presentable fibration \cite{htt}*{Definition~5.5.3.2}. We prove the first claim; the second will follow by an identical argument (or one may use \cite{denis-stab}*{Theorem~6.5}). By \cite{htt}*{Proposition~5.4.7.11}, for any small $G$-$\infty$-category $K$, $\Fun_G(K,\C)$ is presentable. Since right adjoints preserve all limits, it is clear that $i: \CMon^G(\C) \subset \Fun_G(\underline{\Fin}_{G \ast}, \C)$ is closed under limits. Now let $\kappa$ be a regular cardinal such that for all maps $f: U \to V$ in $\Oscr_G$, $f_*: \C_U \to \C_V$ is $\kappa$-accessible. Then $i$ is likewise $\kappa$-accessible. We conclude that $i$ is the inclusion of an accessible localization of a presentable $\infty$-category, and thus $\CMon^G(\C)$ is presentable \cite{htt}*{Remark~5.5.1.6}. Since $\Sp$ of a presentable $\infty$-category is again presentable \cite{higheralgebra}*{Proposition 1.4.4.4}, we further conclude that $\Sp^G(\C)$ is presentable.
\end{proof}

\begin{notation} \label{ntn:sad} In Lemma~\ref{lem:SimplePresentabilityCheck}, denote the localization functor by $L_{\sad}$.
\end{notation}

\begin{remark} Suppose $\C$ admits finite $G$-products and is fiberwise presentable, so $\underline{\Sp}^G(\C)$ is fiberwise presentable by Lemma~\ref{lem:SimplePresentabilityCheck}. Then the $G$-functor $\Omega^{\infty}: \underline{\Sp}^G(\C) \to \C$ admits a $G$-left adjoint $\Sigma^{\infty}_+: \C \to \underline{\Sp}^G(\C)$ by \cite{higheralgebra}*{Proposition~7.3.2.11}. Moreover, if $\C$ admits finite $G$-coproducts (and thus all $G$-colimits), then $\Sigma^{\infty}_+$ preserves finite $G$-coproducts (and any $G$-colimit) as a $G$-left adjoint \cite{jay-thesis}*{Corollary~8.7}. Similarly, the forgetful $G$-functor $U:\underline{\CMon}^G(\C) \to \C$ admits a $G$-left adjoint $\Fr$.
\end{remark}

We obtain additional functoriality for $\underline{\Sp}^G$, etc., in the presentable setting.

\begin{proposition} \label{prp:PresentableFunctorialityOfGStabilization} Suppose $\C$ and $\D$ be $G$-$\infty$-categories that admit finite $G$-products and are fiberwise presentable, and let $F: \C \rightleftarrows \D: R$ be a $G$-adjunction. Then we obtain induced $G$-adjunctions
\begin{align*} F: \underline{\CMon}^G(\C) \rightleftarrows \underline{\CMon}^G(\D): R &, \qquad F: \underline{\Mack}^G(\C) \rightleftarrows \underline{\Mack}^G(\D): R \\
F: \underline{\Sp}(\C) \rightleftarrows \underline{\Sp}(\D): R &, \qquad F: \underline{\Sp}^G(\C) \rightleftarrows \underline{\Sp}^G(\D): R
\end{align*}
such that $F$ commutes with $\Sigma^{\infty}_+$ and $\Fr$, and $R$ commutes with $\Omega^{\infty}$ and $U$.
\end{proposition}
\begin{proof} Since $\Sp$ is given as a functor $\Sp: \PrL_{\infty} \to \PrL_{\infty,\stab}$, we already know the assertion regarding fiberwise stabilization. By \cite{jay-thesis}*{Corollary~8.3}, for any $G$-$\infty$-category $K$ (and not using our hypotheses on $\C$ and $\D$), we obtain an induced $G$-adjunction via postcomposition
\[ F : \underline{\Fun}_G(K,\C) \rightleftarrows \underline{\Fun}_G(K, \D) : R.
\]
Since the right $G$-adjoint $R$ preserves $G$-commutative monoids (or $G$-Mackey functors), we may define the left $G$-adjoint $F$ at the level of $\underline{\CMon}^G$ as postcomposition followed by $L_{\sad}$. In this manner, we obtain the desired $G$-adjunctions. Finally, we already noted the claim about $R$ and $\Omega^{\infty}, U$ above, and the other one follows by adjunction.
\end{proof}

Let us now return to our study of $G$-$\infty$-topoi.

\begin{definition} \label{construct:denis} 
Let $\Xscr$ be a $G$-$\infty$-topos. Then by Proposition~\ref{prop:genuine}, $\underline{\Xscr}_G$ admits finite $G$-limits, so we may define its {\bf parametrized genuine stabilization} to be the $G$-stable $G$-$\infty$-category
\[
\underline{\Sp}^G(\Xscr) = \underline{\Sp}^G(\underline{\Xscr}_G),
\]
and its {\bf genuine stabilization} to be the fiber $\Sp^G(\Xscr) = \underline{\Sp}^G(\Xscr)_{G/G}$. We also refer to objects in $\Sp^G(\Xscr)$ as {\bf $G$-spectrum objects} in $\Xscr$.
\end{definition}

Note that given a $G$-equivariant geometric morphism $f_{\ast}: \Xscr \to \Yscr$ of $G$-$\infty$-topoi, by Proposition~\ref{prp:InducedAdjunctionParamTopOrbits} and Proposition~\ref{prp:PresentableFunctorialityOfGStabilization} we have an induced $G$-adjunction
\[ f^*: \underline{\Sp}^G(\Yscr) \rightleftarrows \underline{\Sp}^G(\Xscr): f_*. \]

\begin{example}\label{ex:g-sp} Consider the terminal $G$-$\infty$-topos $\Spc$. Then the right Kan extension $\Spc^{\hh(-)}: \Oscr_G^{\op} \rightarrow \RTop$ is the constant functor at $\Spc$ because limits of terminal objects are again terminal objects. The toposic genuine orbits are then given by

\[
 \Spc_H = \lax.\colim(\Spc^{\hh(-)}|_{(\Oscr^{\op}_G)_{G/H/}}) = \Sect_{(\Oscr^{\op}_G)_{G/H/}}(\Spc \times (\Oscr^{\op}_G)_{G/H/}) \simeq \Fun(\Oscr_H^{\op}, \Spc).
\]
Thus, Construction~\ref{cons:orb} produces the usual $G$-$\infty$-category $\underline{\Spc}_{G}$ of $G$-spaces \cite{denis-stab}*{Definition 2.7}. Definition~\ref{construct:denis} then yields the $G$-$\infty$-category $\underline{\Sp}^G$ of $G$-spectra, whose fiber over $G/H$ is the usual $\infty$-category $\Sp^H$ of genuine $H$-spectra.
\end{example}

The next lemma shows that the genuine stabilization of $\underline{\Xscr}_{\hh G}$ is given already by fiberwise stabilization.

\begin{lemma} \label{lem:borel-computation}  Let $\Xscr$ be a $G$-$\infty$-topos. Then
\begin{enumerate}
\item The fiberwise stabilization of $\underline{\Xscr}_{\hh G}$ is in addition $G$-stable, i.e., it is $G$-semiadditive. 
\item There are canonical equivalences of stable $\infty$-categories
\[
\Sp^G(\underline{\Xscr}_{\hh G}) \simeq \Sp(\Xscr_{\hh G}) \simeq \Sp(\Xscr^{hG}) \simeq \Sp(\Xscr)^{h G}.
\]
\end{enumerate}
\end{lemma}
\begin{proof} The second claim is equivalent to saying that the genuine stabilization of $\underline{\Xscr}_{\hh G}$ is given by the fiberwise stabilization. Therefore, it suffices to verify the first claim. To verify $G$-semiadditivity of $\underline{\Sp}(\underline{\Xscr}_{\hh G})$, note first that the functors $$\pi_*, \pi_!: \Xscr \rightarrow \Xscr_{\hh G} \simeq \Sect_{BG}(\underline{\Xscr})$$ that are right and left adjoint to the restriction functor $\pi^*$ are computed by the pointwise formula for relative right and left Kan extension as the $G$-indexed product and coproduct (for instance, given $x \in \Xscr$, we have $(\pi_{\ast} x)(\ast) \simeq \prod_{G} g^{\ast} x$). Since finite products and coproducts coincide in a stable $\infty$-category, we deduce that the canonical comparison map $\pi_! \rightarrow \pi_\ast$ is an equivalence after stabilization. The same reasoning holds for the left and right adjoints to the restriction $f^{\ast}$ induced by any map $f: U \to V$ of finite $G$-sets, so we may conclude.
\end{proof}

\begin{example} Let $\Xscr$ be the terminal $G$-$\infty$-topos $\Spc$. Then $\underline{\Sp}^G(\underline{\Xscr}_{\hh G})$ is the $G$-$\infty$-category of Borel $G$-spectra, whose fiber over $G/H$ is equivalent to $\Fun(BH, \Sp)$.
\end{example}

\subsection{A Borel approach for cyclic groups of prime order}  \label{sect:borel-approach}

In this subsection, we undertake a \emph{Borel} approach to genuine stabilization for $G$-$\infty$-topoi in the sense of \cite{glasman} or \cite{naive} (see also \cite{quigley-shah}*{Theorem~3.44}). However, we will not fully generalize the stratification theory of \cite{glasman} to accommodate $G$-$\infty$-topoi for an arbitrary finite group $G$, instead restricting our attention to the simplest case of the cyclic group $C_p$ of prime order $p$. In fact, the case $p=2$ is the only case of interest for our applications in algebraic geometry. We begin with the unstable situation.

Consider the $G$-equivariant geometric morphism $\nu_*:\Xscr^{\hh G} \rightarrow \Xscr$ of Construction~\ref{construct:borel-fix}. Taking its toposic homotopy orbits, we obtain the geometric morphism
\[
\nu(G)_*:(\Xscr^{\hh G})_{\hh G} \rightarrow \Xscr_{\hh G}.
\]
Since $\Xscr^{\hh G}$ has trivial $G$-action, by Example~\ref{exmp:trivial} we have a canonical equivalence of $\infty$-topoi $$(\Xscr^{\hh G})_{\hh G} \simeq \Fun(BG, \Xscr^{\hh G}).$$

\begin{definition} \label{def:glue} Let $\Xscr$ be a $C_p$-$\infty$-topos. The {\bf unstable gluing functor} is
\[
\theta = (-)^{hC_p} \circ \nu(C_p)^*: \Xscr_{\hh C_p} \rightarrow \Fun(BC_p, \Xscr^{\hh C_p}) \rightarrow \Xscr^{\hh C_p}.
\]
\end{definition}

Note that $\theta$ is in general only a left-exact functor of $\infty$-topoi.

\begin{example} \label{ex:et-to-ret} Suppose that $X$ is a scheme with $\frac{1}{2} \in \Oscr_X$. Using the identification of Theorem~\ref{thm:exret} and Example~\ref{exmp:et-cover}, the unstable gluing functor takes the form
\[
\theta: \widetilde{X}_{\et} \rightarrow \widetilde{X}_{\ret}.
\]
This functor is in turn computed to be \cite{scheiderer}*{Lemma~6.4.2(b)}
\[
\widetilde{X}_{\et} \xrightarrow{i_{\et}} \widetilde{X}_{\pre} \xrightarrow{L_{\ret}} \widetilde{X}_{\ret}.
\]
Define the {\bf $b$-topology} on the small site $\Et_X$ to be the intersection of the \'etale and real \'etale topologies \cite{scheiderer}*{Definition~2.3}. Scheiderer identifies the right-lax limit of $L_{\ret} i_{\et}$ (which he calls $\rho$) with the $\infty$-topos\footnote{Scheiderer works at the level of $1$-topoi, but the statement for $\infty$-topoi follows by the same reasoning, as we show in Appendix~\ref{app:sh-top}.} $\widetilde{X}_b$ \cite{scheiderer}*{Proposition~2.6.1}. As we recall in Appendix~\ref{app:sh-top}, the key point is that the other composite $L_{\et} i_{\ret}$ is equivalent to the constant functor at the terminal sheaf (Example~\ref{ex:recoll-b} and Lemma~\ref{lem:asymmetry} with $C = \Et_X$ and $\tau = b$, $\tau_U = \et$, $\tau_Z = \ret$).
\end{example}

The next proposition justifies the terminology of Definition~\ref{def:glue}.

\begin{proposition} \label{prop:recoll} Let $\Xscr$ be a $C_p$-$\infty$-topos. Then the right-lax limit of the unstable gluing functor $\theta$ is equivalent to $\Xscr_{C_p}$. Equivalently, we have a recollement
\[ \begin{tikzcd}[row sep=4ex, column sep=6ex, text height=1.5ex, text depth=0.25ex]
\Xscr_{\hh C_p} \ar[shift right=1,right hook->]{r}[swap]{j_{*}} & \Xscr_{C_p} \ar[shift right=2]{l}[swap]{j^{*}} \ar[shift left=2]{r}{i^{*}} & \Xscr^{\hh C_p} \ar[shift left=1,left hook->]{l}{i_{*}}
\end{tikzcd} \]
in which the gluing functor $i^{*} j_{*}$ is equivalent to $\theta$.
\end{proposition}

\begin{proof} Let $\pi: \Oscr_{C_p}^{\op} \to \Delta^1$ be the unique functor that sends $C_p/C_p$ to $0$ and $C_p/1$ to $1$. Let $p: {\underline{\Xscr}}^{\hh C_p} \to \Oscr_{C_p}^{\op}$ denote the cartesian fibration of Construction~\ref{construct:borel-fix}, so that $\Xscr_{C_p} = \Sect(\underline{\Xscr}^{\hh C_p})$. Note that $(\underline{\Xscr}^{\hh C_p})_0 \simeq \Xscr^{\hh C_p}$ and $(\underline{\Xscr}^{\hh C_p})_1 \simeq \underline{\Xscr}$. We may readily verify the existence hypotheses of \cite{quigley-shah}*{Proposition 2.4}, so we have an induced recollement decomposing $\Xscr_{C_p}$ in terms of $\Xscr_{\hh C_p}$ and $\Xscr^{\hh C_p}$. It remains to identify the gluing functor $i^{\ast} j_{\ast}$ of this recollement as $\theta$. For this, we observe that by the discussion prior to \cite[Proposition 2.4]{quigley-shah}, given an object $f: B C_p \to \underline{\Xscr} \subset \underline{\Xscr}^{\hh C_p}$ of $\Xscr_{\hh C_p}$, $i^{\ast} j_{\ast} (f)$ is computed as value of the $p$-limit diagram $\overline{f}$ on the cone point $v$:
\[ \begin{tikzcd}[row sep=4ex, column sep=4ex, text height=1.5ex, text depth=0.25ex]
B C_p \ar{r}{f} \ar{d} & \underline{\Xscr}^{\hh C_p} \ar{d}{p} \\
\Oscr_{C_p}^{\op} \cong (B C_p)^{\lhd} \ar{r} \ar{ru}{\overline{f}} & \Oscr_{C_p}^{\op}.
\end{tikzcd} \]
We now digress to explain how to compute such $p$-limits. In general, suppose given a diagram
\[ \begin{tikzcd}[row sep=4ex, column sep=6ex, text height=1.5ex, text depth=0.25ex]
K \ar{r}{f} \ar{d} & C \ar{d}{p} \\ 
K^{\lhd} \ar{r}{\overline{pf}} \ar[dotted]{ru}{\overline{f}} & S
\end{tikzcd} \]
in which $p$ is a cartesian fibration, and let $x = \overline{p f}(v)$ for $v$ the cone point. Supposing that $C_x$ admits $K$-indexed limits, we then have the dotted functor $\overline{f}$ given by sending $v$ to the limit of $g: K \to C_x$, where $g$ is the $p$-cartesian pullback of $f$ to the fiber over $x$. In more detail, we may construct $\overline{f}$ via the following recipe:
\begin{itemize}
    \item Let $\Fun(\Delta^1,C)_{\cart} \to \Fun(\Delta^1,S) \times_S C$ be the trivial fibration given by the dual of \cite{jay-thesis}*{Lemma 2.22} that sends a $p$-cartesian edge $[a \to b]$ to $([p(a) \to p(b)], b)$, let $\sigma$ be a choice of a section and let 
    \[
    P = \ev_0 \circ \sigma: \Fun(\Delta^1,S) \times_S C \to C
    \] be the resulting choice of cartesian pullback functor. Let $r: K \times \Delta^1 \to K^{\lhd}$ be the unique functor that restricts to the trivial functor $K \times \{0\} \to \ast$ and to the identity on $K \times \{1\}$, let $q: K \to \Fun(\Delta^1,S)$ be the functor adjoint to $K \times \Delta^1 \xto{r} K^{\lhd} \xto{\overline{pf}} S$, and let $$F = (q,f): K \to \Fun(\Delta^1,S) \times_S C.$$ Then $P \circ F: K \to C$ factors through $C_x$. Let $g: K \to C_x$ denote this functor, which is the $p$-cartesian pullback of $f$ to the fiber over $x$. Using that $C_x$ admits $K$-indexed limits, let $\overline{g}: K^{\lhd} \to C_x$ denote the extension of $g$ to a limit diagram.

    Next, let $h: K \times \Delta^1 \to C$ be the functor adjoint to $\sigma \circ F: K \to \Fun(\Delta^1,C)_{\cart}$ and consider the diagram
\[ \begin{tikzcd}[row sep=4ex, column sep=6ex, text height=1.5ex, text depth=0.25ex]
K \times \Delta^1 \cup_{K \times \{0\} } K^{\lhd} \times \{ 0\} \ar{r}{(h,\overline{g})} \ar{d} & C \ar{d}{p} \\
(K \times \Delta^1)^{\lhd} \ar{r} \ar[dotted]{ru}[swap]{\overline{h}} & S.
    \end{tikzcd} \]
    By \cite{htt}*{Lemma 2.1.2.3}, the lefthand inclusion is inner anodyne, so the dotted lift $\overline{h}$ exists. Finally, we set: 
    \[\overline{f} = \overline{h}|_{(K \times \{1\})^{\lhd}}.
    \]
\end{itemize}

By the proof of \cite{htt}*{Corollary 4.3.1.11}, $\overline{f}$ is then a $p$-limit diagram if and only if for all morphisms $(\alpha: y \to x = \overline{p f}(v))$ in $S$ (defining the pullback functor $\alpha^{\ast}: C_x \to C_y$), $\alpha^{\ast} \overline{g}$ is a limit diagram in $C_y$. In particular, if $x$ is an initial object in $S$, then the latter condition is vacuous and $\overline{f}$ is necessarily a $p$-limit diagram. Now suppose that $S \simeq S_1^{\lhd}$ and $C_x$ admits $S_1$-indexed limits. Let $K = S_1$ and $\overline{p f} = \id$. Then the content of \cite{quigley-shah}*{Proposition 2.4} in this case is that $\Fun_{/S}(S,C)$ is the right-lax limit of the gluing functor
\[ \Fun_{/S_1}(S_1, C \times_S S_1) \simeq \Fun_{/S}(S_1, C) \xrightarrow{P_{\ast}} \Fun(S_1, C_x) \xrightarrow{\lim} C_x, \]
where $P_{\ast}$ denotes the functorial assignment $f \mapsto g$ above that is determined by the choice of $P$ (which is only ambiguous up to contractible choice).

Returning to our situation of interest, observe that
$$P: \Fun(\Delta^1,\Oscr^{\op}_{C_p}) \times_{\Oscr^{\op}_{C_p}} \underline{\Xscr}^{\hh C_p} \to \underline{\Xscr}^{\hh C_p}$$
restricts to
$$P': \underline{\Xscr} \simeq BC_p \times_{\Oscr_{C_p}^{op}} \underline{\Xscr}^{\hh C_p} \to \Xscr^{\hh C_p},$$
where we identify the full subcategory of $\Fun(\Delta^1,\Oscr^{\op}_{C_p})$ on those arrows with source $C_p/C_p$ and target $C_p/1$ with $BC_p$. We may then identify $P_{\ast}$ with postcomposition by $P'$.

Note that the $C_p$-equivariant functor $\nu^{\ast}: \Xscr \to \Xscr^{\hh C_p}$ passes under unstraightening to the functor $$(P', p|_{\underline{\Xscr}}): \underline{\Xscr} \to \Xscr^{\hh C_p} \times B C_p$$ over $BC_p$. Therefore, the $C_p$-homotopy fixed points\footnote{Recall here that the left adjoint of the geometric morphism given under formation of toposic homotopy orbits in $\RTop$ is computed by taking homotopy fixed points in $\widehat{\Cat}_{\infty}$.} $\nu(C_p)^{\ast}: \Xscr_{\hh C_p} \to \Fun(BC_p, \Xscr^{\hh C_p})$ is computed by postcomposition by $P'$ in terms of the descriptions of the domain and codomain $\infty$-categories as sections. We conclude that $(-)^{h C_p} \circ P_{\ast}$ identifies with $\theta = (-)^{h C_p} \circ \nu(C_p)^{\ast}$, which shows that $i^\ast j_{\ast} \simeq \theta$. 
% $\overline{f}(v) \simeq (\nu(C_p)^{\ast}(f))^{h C_p}$ naturally in $f$
\end{proof}

\begin{remark} In view of Proposition~\ref{prop:recoll}, $\Xscr_{C_p}$ corresponds to Scheiderer's notion of {\bf quotient topos} \cite{scheiderer}*{Definition~14.2}, which he only defines for $G = C_p$. The toposic genuine orbits construction may therefore be viewed as a generalization of Scheiderer's construction to the case of an arbitrary finite group.
\end{remark}

\begin{example} \label{ex:b} Suppose that $X$ is a scheme with $\frac{1}{2} \in \Oscr_X$. Then in view of Example~\ref{ex:et-to-ret} and Proposition~\ref{prop:recoll}, we have an equivalence
\[
\widetilde{X}_b \simeq (\widetilde{X[i]}_{\et})_{C_2}.
\]
\end{example}

We can also understand the functoriality of toposic genuine orbits in terms of the unstable gluing functor. For the next proposition, recall from \cite{quigley-shah}*{1.7} that given a lax commutative square (with $\phi$ left-exact)
\[ \begin{tikzcd}[row sep=4ex, column sep=4ex, text height=1.5ex, text depth=0.5ex]
\Uscr \ar{r}{\phi} \ar{d}[swap]{f^*} \ar[phantom]{rd}{\SWarrow} & \Zscr \ar{d}{f^*} \\
\Uscr' \ar{r}[swap]{\phi} & \Zscr',
\end{tikzcd} \]
we obtain an induced functor $f^*: \Xscr = \Uscr \overrightarrow{\times} \Zscr \to \Xscr' = \Uscr' \overrightarrow{\times} \Zscr'$ upon taking right-lax limits horizontally. Moreover, $f^*$ is a (lax) {\bf morphism of recollements} $(\Uscr, \Zscr) \to (\Uscr',\Zscr')$, that is, $f^*$ sends $j^*$, resp. $i^*$-equivalences to $j'^*$, resp. $i'^*$-equivalences \cite{quigley-shah}*{Definition~1.1}. Conversely, if $f^*$ is a morphism of recollements, then $f^*$ induces functors $f^*: \Uscr \to \Uscr'$ and $f^*: \Zscr \to \Zscr'$ that commute with $f^*: \Xscr \to \Xscr'$ and the recollement left adjoints, and we thereby obtain an exchange transformation $$\chi: f^* \phi \simeq f^* i^* j_* \Rightarrow \phi f^* \simeq i'^* j'_* f^*.$$ These two constructions are inverse equivalent to each other \cite{higheralgebra}*{Proposition~A.8.8}.

\begin{proposition} \label{prp:UnstableRecollementFunctoriality} Let $f_{\ast}: \Xscr \to \Yscr$ be a $C_p$-equivariant geometric morphism of $C_p$-$\infty$-topoi. Then the induced functor $f^{\ast}: \Yscr_{C_p} \to \Xscr_{C_p}$ is a morphism of the recollements as defined in Proposition~\ref{prop:recoll}. Therefore, we obtain an exchange transformation $\chi: f^* \theta \Rightarrow \theta f^*$ such that $f^*$ is canonically equivalent to the right-lax limit of $\chi$. Moreover, $\chi$ factors as 
\[ \begin{tikzcd}[row sep=6ex, column sep=8ex, text height=1.5ex, text depth=0.5ex]
\Yscr_{\hh C_p} \ar{r}{\nu(C_p)^*} \ar{d}[swap]{(f^*)_{\hh C_p}} \ar[phantom]{rd}{\simeq \SWarrow} & \Fun(B C_p, \Yscr^{\hh C_p}) \ar{d}{f^*} \ar{r}{(-)^{h C_p}} \ar[phantom]{rd}{\SWarrow} & \Yscr^{\hh C_p} \ar{d}{(f^*)^{\hh C_p}} \\
\Xscr_{\hh C_p} \ar{r}[swap]{\nu(C_p)^*} & \Fun(B C_p, \Xscr^{\hh C_p}) \ar{r}[swap]{(-)^{h C_p}} & \Xscr^{\hh C_p},
\end{tikzcd} \]
where the lefthand invertible transformation is toposic homotopy $C_p$-orbits of $(f^*)^{\hh C_p} \nu^* \simeq \nu^* f^*$ and the righthand transformation is adjoint to $(f^*)^{\hh C_p} \delta \simeq \delta (f^*)^{\hh C_p}$ for the constant diagram functor $\delta$.
\end{proposition}
\begin{proof} The first assertion is clear since $f^{\ast}$ is given by postcomposition of sections and $j^*, i^*$ are given by restriction of sections. We thus obtain $\chi$ as just explained. The factorization of $\chi$ then follows from a straightforward diagram chase upon unpacking the factorization $i^* j_* \simeq (-)^{hC_p} \nu(C_p)^*$ of Proposition~\ref{prop:recoll}.
\end{proof}

We next aim to construct a \emph{stable} gluing functor that recovers the genuine stabilization $\Sp^{C_p}(\Xscr)$. To begin with, let us denote the stabilization of the functor $\theta$ by
\begin{equation} \label{eq:small-theta}
\Theta = \Sp(\theta) \simeq (-)^{h C_p} \circ \nu(C_p)^{\ast}: \Sp(\Xscr_{\hh C_p}) \rightarrow \Fun(B C_p,  \Sp( \Xscr^{\hh C_p}) ) \rightarrow \Sp( \Xscr^{\hh C_p}).
\end{equation}
Given Proposition~\ref{prop:recoll}, it is easy to see that the right-lax limit of $\Theta$ obtains $\Sp(\Xscr_{C_p})$ (and not $\Sp^{C_p}(\Xscr)$ in general) -- we record this observation as Lemma~\ref{lem:stabilizationOfRecollement}. To instead obtain the gluing functor for $\Sp^{C_p}(\Xscr)$, we must replace the homotopy fixed points by the Tate construction, as in the following definition.

\begin{definition} \label{def:stable-glue} Let $\Xscr$ be a $C_p$-$\infty$-topos. The {\bf stable gluing functor} is
\[
\Theta^{\Tate} = (-)^{tC_p} \circ \nu(C_p)^*: \Sp(\Xscr_{\hh C_p}) \rightarrow \Fun(BC_p, \Sp(\Xscr^{\hh C_p})) \rightarrow \Sp(\Xscr^{\hh C_p}).
\]
\end{definition}

Here is the first main theorem of the paper. 

\begin{theorem} \label{prop:recoll-stab} Let $\Xscr$ be a $C_p$-$\infty$-topos and let $\Sp^{\Tate}_{C_p}(\Xscr)$ denote the right-lax limit of the stable gluing functor $\Theta^{\Tate}$. We have an equivalence of $\infty$-categories
$$\Sp^{\Tate}_{C_p}(\Xscr) \simeq \Sp^{C_p}(\Xscr).$$
\end{theorem}

\begin{question} \label{quest:arb} What is the analogue of Theorem~\ref{prop:recoll-stab} for an arbitrary finite group?
\end{question}

The rest of this section will be devoted to a proof of Theorem~\ref{prop:recoll-stab}. To begin with, we need a better understanding of how $\underline{\Sp}^{C_p}(-)$ interacts with recollements. Fixing notation, consider the following functors obtained from Constructions~\ref{construct:borel-orb} and~\ref{cons:orb} respectively:
\[
\Xscr_{\hh(-)}: \Oscr^{\op}_{C_p} \rightarrow \LTop, \qquad (C_p/1 \rightarrow C_p/C_p) \mapsto (\Xscr \xleftarrow{\pi^*} \Xscr_{\hh C_p}),
\]
\[
\Xscr_{(-)}: \Oscr^{\op}_{C_p} \rightarrow \LTop, \qquad (C_p/1 \rightarrow C_p/C_p) \mapsto (\Xscr \xleftarrow{\overline{\pi}^*} \Xscr_{C_p}).
\]
We also have the following functor:
\[
\Xscr^c: \Oscr^{\op}_{C_p} \rightarrow \LTop, \qquad (C_p/1 \rightarrow C_p/C_p) \mapsto (* \leftarrow \Xscr^{\hh C_p}),
\]
where $*$ indicates the initial $\infty$-topos, i.e., the category $\Delta^0$. The ``$c$" decoration indicates that $\Xscr^c$ wants to be the closed part of a fiberwise recollement on $\Xscr_{(-)}$. To justify this, we need the following lemmas.

\begin{lemma} \label{lem:j-open} For any finite group $G$, there is a canonical adjunction 
\[ j^{\ast}:{\Xscr_G} \rightleftarrows {\Xscr_{\hh G}}:{j_{\ast}}, \]
such that:
\begin{enumerate}
\item $j^{\ast}$ is implemented by restriction of sections along the inclusion $BG \subset \Oscr_G^{\op}$.
\item We have a factorization
\[ \begin{tikzcd}[row sep=4ex, column sep=6ex, text height=1.5ex, text depth=0.25ex]
\Xscr_G \ar{r}[swap]{j^{\ast}} \ar[bend left]{rr}{\overline{\pi}^{\ast}} & \Xscr_{\hh G} \ar{r}[swap]{\pi^{\ast}} & \Xscr. 
\end{tikzcd} \]
\item If $G = C_p$, then the right adjoint $j_{\ast}$ is implemented by (pointwise) relative right Kan extension along $BG \subset \Oscr_G^{\op}$, hence $j_{\ast}$ is fully faithful and we have the factorization
\[ \begin{tikzcd}[row sep=4ex, column sep=6ex, text height=1.5ex, text depth=0.25ex]
\Xscr_{\hh G} \ar{r}[swap]{j_{\ast}} \ar[bend left]{rr}{\pi^{\ast}} & \Xscr_G \ar{r}[swap]{\overline{\pi}^{\ast}} & \Xscr.
\end{tikzcd} \]
\end{enumerate}
\end{lemma}
\begin{proof} By definition, $\overline{\pi}^{\ast}$ is given by restriction of sections (valued in $\underline{\Xscr}^{\hh G}$) along $\{G/1 \} \subset \Oscr_G^{\op}$, and $\pi^{\ast}$ is given by restriction of sections (valued in $\underline{\Xscr}$) along $\{G/1 \} \subset BG$. If we define $j^{\ast}: \Xscr_G \to \Xscr^{h G} \simeq \Xscr_{\hh G}$ as restriction of sections along $BG \subset \Oscr^{\op}_G$, then we have defined a functor $j^*$ satisfying (1) and (2).

%
% It follows that $\overline{\pi}^{\ast}: \Xscr_G \to \Xscr$ is a $G$-equivariant functor in $\PrL$ with respect to the trivial $G$-action on $\Xscr_G$ and the given $G$-action on $\Xscr$, and $\overline{\pi}^{\ast}$ is then adjoint to the functor $j^{\ast}: \Xscr_G \to \Xscr^{h G} \simeq \Xscr_{\hh G}$ computed as restriction of sections along $BG \subset \Oscr^{\op}_G$, so $\pi^{\ast} \circ j^{\ast} \simeq \overline{\pi}^{\ast}$.

Note that the right adjoint $j_{\ast}$ exists by presentability considerations. However, we observe that the hypotheses of (the dual of) \cite[Proposition~4.3.2.15]{htt} and \cite[Proposition~4.3.1.10]{htt} for the existence of a dotted lift in the commutative diagram
\[ \begin{tikzcd}[row sep=4ex, column sep=4ex, text height=1.5ex, text depth=0.25ex]
BG \ar{r} \ar{d} & \underline{\Xscr}^{\hh} \ar{d} \\
\Oscr_G^{\op} \ar{r}[swap]{=} \ar[dotted]{ru} & \Oscr_G^{\op}
\end{tikzcd} \]
are not satisfied \emph{unless} $G = C_p$, since the left adjoint in a geometric morphism need not preserve infinite limits (if $G = C_p$, then condition (2) in \cite[Proposition~4.3.1.10]{htt} is vacuous; we already used this in the proof of Proposition~\ref{prop:recoll}). Therefore, we do not know if $j_{\ast}$ is computed by the relative right Kan extension in general; in particular, we do not know that $j_{\ast}$ is fully faithful unless $G = C_p$. However, if $G = C_p$, then the final assertions of the lemma follow immediately.
\end{proof}

\begin{lemma}  \label{lem:i-closed} For any finite group $G$, there is a canonical adjunction
\[ i^{\ast}:{\Xscr_{G}} \rightleftarrows {\Xscr^{\hh G}}:i_{\ast}, \]
where $i^{\ast}$ is implemented by restriction of sections along $\{ G/G \} \subset \Oscr_G^{\op}$, such that the diagrams
\[ \begin{tikzcd}[row sep=4ex, column sep=4ex, text height=1.5ex, text depth=0.25ex]
\Xscr_{G} \ar{r}{i^{\ast}} \ar{d}{\overline{\pi}^{\ast}} & \Xscr^{\hh G} \ar{d} \\
\Xscr \ar{r} & \ast
\end{tikzcd}, \quad
\begin{tikzcd}[row sep=4ex, column sep=4ex, text height=1.5ex, text depth=0.25ex]
\Xscr_{G} \ar{d}{\overline{\pi}^{\ast}} & \Xscr^{\hh G} \ar{d} \ar{l}[swap]{i_{\ast}} \\
\Xscr & \ast \ar{l}
\end{tikzcd} \]
commute.
\end{lemma}
\begin{proof} Since the left adjoint of a geometric morphism preserves the terminal object, we see that the right adjoint $i_{\ast}$ to the given functor $i^{\ast}$ exists and is computed by relative right Kan extension along $\{ G/G \} \subset \Oscr^{\op}_G$. The commutativity of the first square is trivial, and the commutativity of the second square follows from the formula for relative right Kan extension.
\end{proof}

% Given two recollements $(\Uscr_1, \Zscr_1)$ of $\Xscr_1$ and $(\Uscr_2, \Zscr_2)$ of $\Xscr_2$, a {\bf lax morphism} of recollements is a functor $f: \Xscr_1 \rightarrow \Xscr_2$ that carries $j_1^*$-equivalences to $j_2^*$-equivalences and $i_1^*$-equivalences to $i_2^*$-equivalences (again, using standard notation for the functors of the recollement adjunctions, but now decorated with appropriate subscripts). $f$ then induces functors $f_{\Uscr}: \Uscr_1 \to \Uscr_2$ and $f_{\Zscr}: \Zscr_1 \to \Zscr_2$ that commute with $f$ and the left adjoints.
% We refer to \cite{quigley-shah}*{\S1} for an extensive discussion and only set up the necessary terminology here.

To proceed further, it will be convenient to refer to the $\infty$-category of recollements. If a functor $f^*: \Xscr \to \Xscr'$ is a morphism of recollements $(\Uscr, \Zscr) \to (\Uscr', \Zscr')$, then we say the morphism is {\bf strict} if $f^*$ also commutes with the gluing functors \cite{quigley-shah}*{Definition 1.4}. The collection of recollements and strict morphisms forms an $\infty$-category $\mathrm{Recoll}_0$. We shall also consider the $\infty$-category $\mathrm{Recoll}_0^{\stab}$ of {\bf strict stable recollements} where the underlying $\infty$-categories are stable and functors are exact. If $S$ is a small $\infty$-category, we say that a functor $S \rightarrow \mathrm{Recoll}^{(\stab)}_0$ is a {\bf (stable) fiberwise recollement}. 

\begin{corollary} \label{cor:recoll} Let $\Xscr$ be a $C_p$-$\infty$-topos. Then the functor $\Xscr_{(-)}$ lifts to a fiberwise recollement given by
\[
(\Xscr_{\hh(-)}, \Xscr^c):\Oscr^{\op}_{\C_p} \rightarrow \mathrm{Recoll}_0.
\]
\end{corollary}

\begin{proof} Lemmas~\ref{lem:i-closed} and~\ref{lem:j-open} demonstrate that the diagram
\[
\begin{tikzcd}
\Xscr_{\hh C_p} \ar[swap]{d}{\pi^*} & \ar[swap]{l}{j^*} \Xscr_{C_p} \ar{d}{\overline{\pi}^*} \ar{r}{i^*} & \Xscr^{\hh C_p} \ar{d} \\
\Xscr  & \Xscr \ar{l} {\id} \ar{r} & *
\end{tikzcd}
\]
specifies a strict morphism of recollements, whence we obtain a functor as desired.
\end{proof}

% Let us also denote the lift of $\Xscr_{(-)}$ to $\mathrm{Recoll}_0$ by $\Xscr_{(-)}$ again, leaving the specific recollement of Corollary \ref{cor:recoll} implicit.

Given a stable fiberwise recollement $\Fscr: \Oscr_G^{\op} \rightarrow \mathrm{Recoll}_0^{\stab}$, we say that $\Fscr$ is a {\bf $G$-stable recollement} if the underlying $G$-$\infty$-category of $\Fscr$ is $G$-stable \cite{quigley-shah}*{Definition 1.39}. This implies that the open and closed parts of $\Fscr$ are $G$-stable $G$-$\infty$-categories and the adjunctions are $G$-exact $G$-adjunctions \cite{quigley-shah}*{Corollary 1.38}.

\begin{proposition} \label{prop:preserves} Let $\Xscr$ be a $C_p$-$\infty$-topos. Then $\underline{\Sp}^{C_p}(\Xscr)$ admits a $C_p$-stable recollement given by $$(\underline{\Sp}^{C_p}(\underline{\Xscr}_{\hh C_p}), \underline{\Sp}^{C_p}(\underline{\Xscr}^c)).$$
% \[ (\underline{\Sp}^{C_p}(\underline{\Xscr}_{\hh C_p}), \underline{\Sp}^{C_p}(\underline{\Xscr}^c)):\Oscr^{\op}_{\C_p} \rightarrow \mathrm{Recoll}_0 \]
\end{proposition}

\begin{proof} Recall (Construction~\ref{rem:nardin}) that genuine stabilization is obtained by taking $G$-commutative monoids and fiberwise stabilization, where by Lemma~\ref{lem:interchange} we may perform these two operations in either order. Fiberwise stabilization preserves recollements by Lemma~\ref{lem:stabilizationOfRecollement}. Since $\underline{\Sp}^{C_p}(\Xscr)$ is $C_p$-stable, it suffices to check that taking $C_p$-commutative monoids on the recollement of Corollary~\ref{cor:recoll} gives us a fiberwise recollement
\[ \left( \underline{\CMon}^{C_p}(\underline{\Sp}(\underline{\Xscr}_{\hh C_p})), \: \underline{\CMon}^{C_p}(\underline{\Sp}(\underline{\Xscr}^c) \right). \]
Since $\Xscr^c(C_p/1) \simeq \ast$, this is obvious over the orbit $C_p/1$, so it suffices to check the recollement conditions over the orbit $C_p/C_p$. For this, first note that if we take $K = \underline{\A}^{\eff}(\Fin_{C_p})$ in \cite{quigley-shah}*{Lemma 1.33}, we obtain a recollement
\[ \begin{tikzcd}[row sep=4ex, column sep=6ex, text height=1.5ex, text depth=0.25ex]
\Fun_{C_p}(\underline{\A}^{\eff}(\Fin_{C_p}),\underline{\Sp}(\underline{\Xscr}_{\hh C_p})) \ar[shift right=1,right hook->]{r}[swap]{j_{*}} & \Fun_{C_p}(\underline{\A}^{\eff}(\Fin_{C_p}),\underline{\Sp}(\underline{\Xscr}_{C_p})) \ar[shift right=2]{l}[swap]{j^{*}} \ar[shift left=2]{r}{i^{*}} & \Fun_{C_p}(\underline{\A}^{\eff}(\Fin_{C_p}),\underline{\Sp}(\underline{\Xscr}^c))\ar[shift left=1,left hook->]{l}{i_{*}}.
\end{tikzcd} \]

Here, we denote by $j_{\ast}$ etc. the functors induced by postcomposition by the same denoted $C_p$-functors for the fiberwise recollement $(\underline{\Sp}(\underline{\Xscr}_{\hh C_p}), \underline{\Sp}(\underline{\Xscr}^c))$ on $\underline{\Sp}(\underline{\Xscr}_{C_p})$. Since the $\infty$-category of $C_p$-Mackey functors in a fiberwise stable presentable $C_p$-$\infty$-category $\C$ is a Bousfield localization of $\Fun_{C_p}(\underline{\A}^{\eff}(\Fin_{C_p}), \C)$, and postcomposition by the $C_p$-left-exact $C_p$-functors $j_{\ast}$, $i_{\ast}$ preserves $C_p$-commutative monoid objects, we have induced adjunctions

\[ \begin{tikzcd}[row sep=4ex, column sep=6ex, text height=1.5ex, text depth=0.25ex]
\CMon^{C_p}(\underline{\Sp}(\underline{\Xscr}_{\hh C_p})) \ar[shift right=1,right hook->]{r}[swap]{\overline{j}_{*}} & \CMon^{C_p}(\underline{\Sp}(\underline{\Xscr}_{C_p})) \ar[shift right=2]{l}[swap]{\overline{j}^{*}} \ar[shift left=2]{r}{\overline{i}^{*}} & \CMon^{C_p}(\underline{\Sp}(\underline{\Xscr}^c)) \ar[shift left=1,left hook->]{l}{\overline{i}_{*}},
\end{tikzcd} \]
in which $\overline{j}_{\ast}, \overline{i}_{\ast}$ are defined by postcomposition by $j_{\ast}, i_{\ast}$ and their left adjoints $\overline{j}^{\ast}, \overline{i}^{\ast}$ are defined by postcomposition by $j^{\ast}, i^{\ast}$ followed by $L_{\sad}$ (Notation~\ref{ntn:sad}). It is thus clear that $\overline{j}_{\ast}, \overline{i}_{\ast}$ remain fully faithful.

Furthermore, we have that $j^{\ast} \overline{\pi}_{\ast} \simeq \pi_{\ast}$, as this may be checked after application of $j_{\ast}$, in which case it follows from Lemma~\ref{lem:j-open}(2). Therefore, $\overline{j}^{\ast}$ is computed already by postcomposition by $j^{\ast}$. We deduce that $\overline{j}^{\ast} \overline{i}_{\ast} \simeq 0$ by the same property for $j^{\ast} i_{\ast}$.

Next, we observe that we have a canonical equivalence $L_{\sad} i_{\ast} \simeq \overline{i}_{\ast} L_{\sad}$. Indeed, it suffices to show that $i_{\ast}$ of an acyclic object $N$ is again acyclic. But for this, note that $$i^{\ast} L_{\sad} i_{\ast} N \simeq L_{\sad} i^{\ast} i_{\ast} N \simeq L_{\sad} N \simeq 0 \quad \text{and} \quad j^{\ast} L_{\sad} i_{\ast} N \simeq L_{\sad} j^{\ast} i_{\ast} N \simeq 0,$$ 
and hence $L_{\sad} N \simeq 0$ by joint conservativity of $j^{\ast}, i^{\ast}$. It follows that the \emph{right adjoint} $i^!$ of $i_{\ast}$ (which exists since we are in the situation of a \emph{stable} recollement) descends to $$\overline{i}\,^!: \CMon^{C_p}(\underline{\Sp}(\underline{\Xscr}_{C_p})) \to \CMon^{C_p}(\underline{\Sp}(\underline{\Xscr}^c)), $$
where $\overline{i}\,^!$ is induced by postcomposition by the $C_p$-right adjoint $i^!: \underline{\Sp}(\underline{\Xscr}_{C_p}) \to \underline{\Sp}(\underline{\Xscr}^c)$.

Finally, note that instead of showing joint conservativity of $\overline{j}^{\ast}, \overline{i}^{\ast}$, in the stable setting it suffices instead to show joint conservativity of $\overline{j}^{\ast}, \overline{i}\,^!$. But this follows immediately from the same property for $j^{\ast}, i^!$.
\end{proof}

In particular, evaluating at $C_p/C_p$ and using the identification of Lemma~\ref{lem:borel-computation}, we have the gluing functor\footnote{Note that we now drop the temporary `overline' decoration for the recollement functors used in the proof of Proposition~\ref{prop:preserves}.}
\begin{equation} \label{eq:g-glue}
i^*j_*:\Sp^{C_p}(\underline{\Xscr}_{\hh C_p}) \simeq \Sp(\Xscr^{hC_p}) \rightarrow \Sp(\Xscr^{\hh C_p}).
\end{equation}
\begin{lemma} \label{lem:killsinduced} The functor $i^*j_*$ vanishes on the stabilization of $\pi_*:\Xscr \rightarrow \Xscr_{\hh C_p}$, i.e., the composite
\[
\Sp(\Xscr) \xrightarrow{\pi_*} \Sp(\Xscr^{hC_p}) \xrightarrow{i^*j_*} \Sp(\Xscr^{\hh C_p})
\]
is nullhomotopic.
\end{lemma}
\begin{proof} By Proposition \ref{prop:preserves} and the $G$-exactness of the $G$-functors in a $G$-stable recollement, we have a commutative diagram
\[
\begin{tikzcd}
\Sp(\Xscr^{hC_p}) \ar[swap]{d}{j_*} & \Sp(\Xscr) \ar[swap]{l}{\pi_*} \ar{d}{=} \\
\Sp^{C_p}(\Xscr_{C_p}) \ar[swap]{d}{i^*} & \Sp(\Xscr) \ar{d} \ar{l}{\overline{\pi}_*} \\
\Sp(\Xscr^{\hh C_p}) & \ar{l} 0.
\end{tikzcd}
\]
% which in turns follows from on the unstable level from Corollary~\ref{cor:recoll}.
\end{proof}

Lemma~\ref{lem:killsinduced} highlights the relevance of the following definition.

\begin{definition} \label{def:topos-induced} Suppose that $\Xscr$ is a $G$-$\infty$-topos. Then the subcategory of {\bf induced objects} in $\Sp(\Xscr_{\hh G})$ is the essential image of the functor $\pi_*: \Sp(\Xscr) \rightarrow \Sp(\Xscr_{\hh G})$.
\end{definition}

\begin{lemma} \label{lem:generated} The adjunction $$\pi^*: \Sp(\Xscr_{\hh G}) \rightleftarrows \Sp(\Xscr): \pi_*$$ is ambidextrous, and $\pi^{\ast}$ exhibits $\Sp(\Xscr_{\hh G})$ as monadic over $\Sp(\Xscr)$. In particular, $\Sp(\Xscr_{\hh G})$ is generated under colimits by the induced objects.
\end{lemma}
\begin{proof} We already produced the ambidexterity equivalence $\pi_! \xto{\simeq} \pi_{\ast}$ in Lemma~\ref{lem:borel-computation}(1). We then note that $\pi^{\ast}$ is conservative as it is given by evaluation of sections on $\ast \in BG$. This confirms the hypotheses of the Barr-Beck-Lurie theorem and thereby shows monadicity of the adjunction $\pi_! \dashv \pi^\ast$. The last statement then follows from the existence of monadic resolutions.
\end{proof}
% Let $\underline{\Sp(C)} \to BG \cong BG^{\op}$ be the cartesian fibration classified by the $G$-action on $\Sp(C)$, so $\Sp(C)^{h G} \simeq \Fun_{/BG}(BG, \underline{\Sp(C)})$. The functor $\pi_{\ast}$ of the lemma is defined to be right adjoint to the functor $\pi^{\ast}$ of restriction of sections along $\ast \to BG$. Using the pointwise formula for relative right Kan extension, we see that $\pi_{\ast}$ is an indexed product functor; more precisely, given $x \in \Sp(C)$, we have $(\pi_{\ast} x)(\ast) \simeq \prod_{G} g^{\ast} x$. Since $G$ is assumed to be finite, the indexed product coincides with the indexed coproduct, and we conclude that $\pi_! \simeq \pi_{\ast}$. The ambidexterity adjunction $\pi_! \dashv \pi^{\ast}$ then satisfies the hypotheses of the Barr-Beck-Lurie theorem. Finally, the lemma follows from the existence of monadic resolutions.

% Hence, Lemma~\ref{lem:killsinduced} tells us that $i^*j_*$ vanishes on induced objects. 
In order to complete the proof of Theorem~\ref{prop:recoll-stab}, we need to identify the gluing functor of~\eqref{eq:g-glue} with the functor $\Theta^{\Tate}$ of Definition~\ref{def:stable-glue}. To facilitate this comparison, we will construct the analogue of {\bf categorical fixed points} in our setting. Recall that if $G$ is a finite group and we identify $\Sp_G$ with the $\infty$-category of spectral Mackey functors $\Fun^{\times}(\Span(\Fin_G), \Sp)$, then the categorical fixed points functor $\Psi^G$ is given by evaluation at the object $G/G$. Our construction will follow a similar pattern.

\begin{theorem} \label{thm:cat-fixed} There is an exact functor 
\[
\Psi^{C_p}: \Sp^{C_p}(\Xscr) \rightarrow \Sp(\Xscr^{\hh C_p})
\]
with the following three properties:
\begin{enumerate}
\item $\Psi^{C_p}$ preserves colimits.
\item $\Psi^{C_p}j_*(-) \simeq (-)^{hC_p} \circ \nu(C_p)^*$ as functors $\Sp(\Xscr_{\hh C_p}) \to \Sp(\Xscr^{\hh C_p})$.
\item $\Psi^{C_p}i_* \simeq \id$ as functors $\Sp(\Xscr^{\hh C_p}) \to \Sp(\Xscr^{\hh C_p})$.
\end{enumerate}
Here, $j_{\ast}$ and $i_{\ast}$ are the functors defined by the recollement on $\Sp^{C_p}(\Xscr)$ of Proposition~\ref{prop:preserves}.
\end{theorem}

\begin{proof} In this proof only, let us distinguish the recollement functors for $(\Sp(\Xscr_{\hh C_p},\Sp(X^{\hh C_p}))$ on $\Sp(X_{C_p})$ (as well as the fiberwise variant for $\underline{\Sp}(\underline{X}_{C_p})$) by the subscript `pre'. First, using Lemma~\ref{lem:interchange} to describe $\Sp^{C_p}(\Xscr)$, we define the functor $\Psi^{C_p}$ by the formula
\[
\Psi^{C_p}: \Sp^{C_p}(\Xscr) \simeq \CMon^{C_p}(\underline{\Sp}(\underline{\Xscr}_{C_p})) \rightarrow \Sp(\Xscr^{\hh C_p}),
\]
\[
\left( X:\underline{\Fin}_{C_p*} \rightarrow \underline{\Sp}(\underline{\Xscr}_{C_p}) \right) \mapsto \left( \{C_p/C_{p+}\} \subset \underline{\Fin}_{C_p*} \xrightarrow{X} \underline{\Sp}(\underline{\Xscr}_{C_p}) \xrightarrow{i^*_{\pre}} \underline{\Sp}(\underline{\Xscr}^c) \right).
\]
In other words, $\Psi^{C_p}$ sends a $C_p$-commutative monoid $X$ to $(i^*_{\pre} X)(C_p/C_{p+})$, an object in the fiber $\underline{\Sp}(\underline{\Xscr}^c)_{C_p/C_p} \simeq \Sp(\Xscr^{\hh C_p})$. It is clear that $\Psi^{C_p}$ is exact. We now verify properties (1)-(3) in turn:

\begin{enumerate}
\item If we regard the above formula as a functor from $\Fun_{C_p}(\underline{\Fin}_{C_p*}, \underline{\Sp}(\underline{\Xscr}_{C_p}))$ to $\Sp(\Xscr^{\hh C_p})$, it obviously preserves colimits as an evaluation functor. Hence, it suffices to prove that the inclusion of the full subcategory
\[
\CMon^{C_p}(\underline{\Sp}(\underline{\Xscr}_{C_p}))  \subset \Fun_{C_p}(\underline{\Fin}_{C_p*}, \underline{\Sp}(\underline{\Xscr}_{C_p}),
\]
preserves colimits. To this end, suppose that we have a diagram $f:K \rightarrow \CMon^{C_p}(\underline{\Sp}(\underline{\Xscr}_{C_p})) $ and let $X$ be the colimit of $f$ as computed in $\Fun_{C_p}(\underline{\Fin}_{C_p*}, \underline{\Sp}(\underline{\Xscr}_{C_p}))$. We need to show that $X$ is $C_p$-semiadditive. For this, it suffices to check that the canonical maps
\[
X(C_p/C_{p+} \coprod C_p/C_{p+}) \rightarrow X(C_p/C_{p+}) \oplus X(C_p/C_{p+}), \: \text{ and}
\]
\[
X(C_p/1_+) \rightarrow \overline{\pi}_*(X(1_+))
\]
are equivalences. The first equivalence follows from stability. For the second equivalence, it suffices to prove that $\overline{\pi}_*$ preserves colimits. But to do so, in view of the conservativity of the recollement functors $$(i^*_{\pre},j^*_{\pre}): \Sp(\Xscr_{C_p}) \rightarrow \Sp(X^{\hh C_p}) \times \Sp(\Xscr_{\hh C_p}),$$ it suffices to check that $i^*_{\pre} \overline{\pi}_*$ and $j^*_{\pre} \overline{\pi}_*$ are colimit-preserving. For the first functor, we have that $i^*_{\pre} \overline{\pi}_* \simeq \nu^*$, which preserves colimits since $\nu^*$ is a left adjoint. For the second functor, we have that $j^*_{\pre} \overline{\pi}_* \simeq \pi_*$, but $\pi_* \simeq \pi_!$ by Lemma~\ref{lem:borel-computation} and $\pi_!$ preserves colimits as a left adjoint.

\item By Lemma~\ref{lem:borel-computation}, fiberwise stabilization for $\underline{\Xscr}_{\hh C_p}$ is already $C_p$-stabilization, so $\Sp^{C_p}(\underline{\Xscr}_{\hh C_p}) \simeq \Sp(\Xscr_{\hh C_p})$. Hence, applying the equivalence~\eqref{eq:denis-observe}, we may canonically regard any $X \in \Sp(\Xscr_{\hh C_p})$ as a $C_p$-commutative monoid $X: \underline{\Fin}_{C_p*} \rightarrow \underline{\Sp}(\underline{\Xscr}_{\hh C_p})$. The functor $j_*$ is then computed on $X$ as the composite
\[
\underline{\Fin}_{C_p*}  \xrightarrow{X} \underline{\Sp}(\underline{X}_{\hh C_p}) \xrightarrow{j^{\pre}_*} \underline{\Sp}(\underline{\Xscr}_{C_p}),
\]
since $j_*^{\pre}$ preserves $C_p$-limits (so that the composite remains $C_p$-semiadditive). Therefore, $\Psi^{C_p}j_*X$ is given by 
\[
\{C_p/C_{p+}\} \subset \underline{\Fin}_{C_p*} \xrightarrow{X} \underline{\Sp}(\underline{X}_{\hh C_p}) \xrightarrow{j^{\pre}_*} \underline{\Sp}(\underline{\Xscr}_{C_p}) \xrightarrow{i^*_{\pre}} \underline{\Sp}(\underline{\Xscr}^c).
\]
$\Psi^{C_p} j_*$ is thus equivalent to $(-)^{hC_p} \circ \nu(C_p)^*$ by the same formula on the unstable level (Proposition~\ref{prop:recoll}).

\item Suppose that $X \in \Sp(\Xscr^{\hh C_p})$. By~\eqref{eq:denis-observe} and Lemma~\ref{lem:interchange}, we may canonically regard $X$ as a $C_p$-commutative monoid $X: \underline{\Fin}_{C_p*} \rightarrow \underline{\Sp}(\underline{\Xscr}^c)$. The functor $i_*$ is then computed on $X$ as the composite
\[
\underline{\Fin}_{C_p*}  \xrightarrow{X} \underline{\Sp}(\underline{\Xscr}^c) \xrightarrow{i_*^{\pre}} \underline{\Sp}(\underline{\Xscr}^{C_p}),
\]
since $i_*^{\pre}$ preserves $C_p$-limits (so that the composite remains $C_p$-semiadditive). Therefore, $\Psi^{C_p}i_*X$ is given by 
\[
\{C_p/C_{p+}\} \subset \underline{\Fin}_{C_p*} \xrightarrow{X} \underline{\Sp}(\underline{\Xscr}^c) \xrightarrow{i_*^{\pre}} \underline{\Sp}(\underline{\Xscr}_{C_p}) \xrightarrow{i^*_{\pre}} \underline{\Sp}(\underline{\Xscr}^c),
\]
which is evidently the identity on $X$ since $i_*^{\pre}$ is fully faithful.
\end{enumerate}
\end{proof}

\begin{corollary} \label{cor:MonadicityForgetful} The forgetful functor $U: \Sp^{C_p}(\Xscr) \simeq \CMon^{C_p}(\underline{\Xscr}_{C_p}) \to \Sp(\Xscr_{C_p})$ preserves all limit and colimits and is conservative. In particular, the adjunction $$\Fr: \Sp(\Xscr_{C_p}) \rightleftarrows \Sp^{C_p}(\Xscr): U$$ is monadic.
\end{corollary}
\begin{proof} We already have that $U$ preserves limits as a right adjoint, and it is clearly conservative. It remains to check that $U$ preserves colimits; monadicity of the adjunction $\Fr \dashv U$ will then follow from the Barr-Beck-Lurie theorem. Note that the functor $\Psi^{C_p}: \Sp^{C_p}(\Xscr) \to \Sp(\Xscr^{\hh C_p})$ of Theorem~\ref{thm:cat-fixed} factors as
\[ \Psi^{C_p}: \Sp^{C_p}(\Xscr) \xrightarrow{U} \Sp(\Xscr_{C_p}) \xrightarrow{i^*_{\pre}} \Sp(\Xscr^{\hh C_p}). \]
Also, the functor $j^*: \Sp^{C_p}(\Xscr) \to \Sp(\Xscr_{\hh C_p})$ factors as
\[ \Sp^{C_p}(\Xscr) \xrightarrow{U} \Sp(\Xscr_{C_p}) \xrightarrow{j^*_{\pre}} \Sp(\Xscr_{\hh C_p}) \]
as we saw in the proof of Proposition~\ref{prop:preserves}. Since $(j^*_{\pre}, i^*_{\pre})$ are jointly conservative and colimit-preserving and $(j^*, \Psi^{C_p})$ are colimit-preserving, the claim follows.
\end{proof}

\begin{remark} For any finite group $G$ and $G$-$\infty$-topos $\Xscr$, we expect $U: \Sp^G(\Xscr) \to \Sp(\Xscr_G)$ to preserve all colimits. For this, it would suffice to show that for every subgroup $H \leq G$, the composite functor $$\Psi^H = \ev_{G/H} \circ U: \Sp^G(\Xscr) \to \Sp(\Xscr_G) \to \Sp(\Xscr^{\hh H}) $$ preserves colimits, or equivalently that either inclusion
\[
\CMon^G(\underline{\Sp}(\underline{\Xscr}_G)) \subset \Fun_G(\underline{\Fin}_{G*}, \underline{\Sp}(\underline{\Xscr}_G)),  \qquad  \Mack^G(\underline{\Sp}(\underline{\Xscr}_G)) \subset  \Fun_G(\underline{\A}^{\eff}(\Fin_G), \underline{\Sp}(\underline{\Xscr}_G))
\]
preserves colimits. For instance, this holds for the base case of genuine $G$-spectra using \cite{barwick}*{Proposition~6.5} or the known compactness of the orbits $\Sigma^{\infty}_+ G/H$ in $\Sp^G$. As the general case is of less significance to us in this paper, we leave the details to the reader.
\end{remark}

\begin{proof} [Proof of Theorem~\ref{prop:recoll-stab}] Consider the following commutative diagram of functors and transformations between them
\[
\begin{tikzcd}
(-)_{hC_p}\nu(C_p)^* \ar{r} \ar{d} & (-)^{hC_p}\nu(C_p)^* \ar{r} \ar{d}{\simeq} & (-)^{tC_p}\nu(C_p)^* \ar{d}\\
\Psi^{C_p} j_! \ar{r} & \Psi^{C_p} j_* \ar{r}  & \Psi^{C_p} i_*i^*j_*,
\end{tikzcd}
\]
in which the bottom row comes from the exact functor $\Psi^{C_p}$ of Theorem~\ref{thm:cat-fixed} applied to the fiber sequence $j_! \to j_{\ast} \to i_{\ast} i^{\ast} j_{\ast}$ associated to the recollement of Proposition~\ref{prop:preserves}.

We first explain how to obtain the vertical arrows. The middle equivalence holds by Theorem~\ref{thm:cat-fixed}(2). Observe then that the resulting composite from the top left corner to the bottom right corner is nullhomotopic, since $i^{\ast} j_{\ast}$ vanishes on induced objects by Lemma~\ref{lem:killsinduced}, the functor $(-)_{hC_p}\nu(C_p)^*$ preserves colimits, and $\Sp(\Xscr_{\hh C_p})$ is generated under colimits by induced objects by Lemma~\ref{lem:generated}. Hence, the right vertical arrow exists by taking cofibers and the left vertical arrow is in turn its fiber.

By Theorem~\ref{thm:cat-fixed}(3) the bottom right corner is equivalent to $i^*j_*$, and thus we get a natural transformation $$\alpha: \Theta^{\Tate} = (-)^{tC_p} \nu(C_p)^{\ast} \Rightarrow i^*j_*.$$ To prove that $\alpha$ is an equivalence, it suffices to check that the natural transformation on fibers
$$\beta: \fib(\Theta \to \Theta^{\Tate}) \simeq (-)_{h C_p} \nu(C_p)^{\ast} \Rightarrow \fib(\Theta \to i^*j_*) \simeq \Psi^{C_p} j_!$$
is an equivalence. For this, note that $(-)_{h C_p} \nu(C_p)^{\ast}$ preserves colimits, and since $\Psi$ preserves colimits by Theorem~\ref{thm:cat-fixed}(1), the functor $\Psi j_!$ preserves colimits as well. Moreover, $\beta$ is an equivalence on induced objects because $(-)^{tC_p}\nu(C_p)^*$ and $i^*j_*$ both vanish on induced objects: for the former claim, we use the commutativity of the diagram
\[ \begin{tikzcd}
\Sp(\Xscr) \ar{r}{\nu^{\ast}} \ar{d}{\pi_{\ast}} & \Sp(\Xscr^{\hh C_p}) \ar{d}{\pi_{\ast}} \\ 
\Sp(\Xscr_{\hh C_p}) \ar{r}{\nu(C_p)^{\ast}} & \Fun(B C_p, \Sp(\Xscr^{\hh C_p})),
\end{tikzcd} \]
which holds in view of the calculation of the vertical functors as $C_p$-indexed products, the left-exactness of the unstable $\nu^{\ast}$, and the definition of $\nu(C_p)^{\ast}$. The conclusion then follows by invoking Lemma~\ref{lem:generated} again.
\end{proof}

\begin{remark} \label{rem:StableRecollementFunctoriality} Let us examine the functoriality of the recollement on $\Sp^{C_p}(\Xscr)$ of Proposition~\ref{prop:preserves} along $C_p$-equivariant geometric morphisms. First note that $\Omega^{\infty}: \Sp^{C_p}(\Xscr) \to \Xscr_{C_p}$ clearly commutes with the recollement right adjoints, since they are induced via postcomposition from their unstable counterparts (as we saw in the proof of \ref{prop:preserves}). Therefore, $\Sigma^{\infty}_+: \Xscr_{C_p} \to \Sp^{C_p}(\Xscr)$ is a morphism of recollements $$\Sigma^{\infty}_+: (\Xscr_{\hh C_p}, \Xscr^{\hh C_p}) \to (\Sp(\Xscr_{\hh C_p}), \Sp(\Xscr^{\hh C_p})).$$ Likewise, we observe the same pattern for the functor $f_*: \Sp^{C_p}(\Xscr) \to \Sp^{C_p}(\Yscr)$ induced by a $G$-equivariant geometric morphism $f_*: \Xscr \to \Yscr$, hence its left adjoint $f^*$ is a morphism of recollements $$f^*: (\Sp(\Yscr_{\hh C_p}), \Sp(\Yscr^{\hh C_p})) \to (\Sp(\Xscr_{\hh C_p}), \Sp(\Xscr^{\hh C_p})).$$
Under Proposition~\ref{prop:recoll} and Theorem~\ref{prop:recoll-stab}, we thus obtain exchange transformations
\[
\alpha: \Sigma^{\infty}_+ \theta \Rightarrow \Theta^{\Tate} \Sigma_+^{\infty}, \qquad \xi: f^* \Theta^{\Tate} \Rightarrow \Theta^{\Tate} f^*
\]
that induce the functors $\Sigma^{\infty}_+$ and $f^*$ above upon taking right-lax limits. We can also understand how $\alpha$ and $\xi$ interact with the factorizations $\theta \simeq (-)^{h C_p} \nu(C_p)^*$ and $\Theta^{\Tate} \simeq (-)^{t C_p} \nu(C_p)^*$. If we factor $\Sigma^{\infty}_+$ as $$\Xscr_{C_p} \xrightarrow{\Sigma^{\infty}_+} \Sp(\Xscr_{C_p}) \xrightarrow{\Fr} \Sp^{C_p}(\Xscr)$$ (so the middle term is given by the right-lax limit of $\Theta$), we obtain a factorization of $\alpha$ as 
\[  \Sigma^{\infty}_+ \theta \simeq \Sigma^{\infty}_+ (-)^{hC_p} \nu(C_p)^{\ast} \xRightarrow{\beta} \Theta \Sigma^{\infty}_+ \simeq (-)^{hC_p} \nu(C_p)^{\ast} \Sigma^{\infty}_+ \xRightarrow{\gamma} \Theta^{\Tate} \Sigma^{\infty}_+ (-)^{tC_p} \nu(C_p)^{\ast} \Sigma^{\infty}_+, \]
where $\beta$ is induced by the adjoint to $\Omega^{\infty} \delta \simeq \delta \Omega^{\infty}$ and $\gamma$ is induced by the usual map $\mu: (-)^{h C_p} \to (-)^{tC_p}$. This may be depicted diagrammatically as follows:
\[ \begin{tikzcd}[row sep=6ex, column sep=8ex, text height=1.5ex, text depth=0.5ex]
\Xscr_{\hh C_p} \ar{r}{\nu(C_p)^*} \ar{d}[swap]{\Sigma^{\infty}_+} \ar[phantom]{rd}{\simeq \SWarrow} & \Fun(B C_p, \Xscr^{\hh C_p}) \ar{d}{\Sigma^{\infty}_+} \ar{r}{(-)^{h C_p}} \ar[phantom]{rd}{\SWarrow} & \Xscr^{\hh C_p} \ar{d}{\Sigma^{\infty}_+} \\
\Sp(\Xscr_{\hh C_p}) \ar{d}[swap]{\simeq} \ar{r}[swap]{\nu(C_p)^*} \ar[phantom]{rd}{\simeq \SWarrow} & \Fun(B C_p, \Sp(\Xscr^{\hh C_p})) \ar{r}[swap]{(-)^{h C_p}} \ar[phantom]{rd}{\mu \SWarrow} \ar{d}{\simeq} & \Sp(\Xscr^{\hh C_p}) \ar{d}{\simeq} \\
\Sp(\Xscr_{\hh C_p}) \ar{r}[swap]{\nu(C_p)^*} & \Fun(B C_p, \Sp(\Xscr^{\hh C_p})) \ar{r}[swap]{(-)^{t C_p}} & \Sp(\Xscr^{\hh C_p}).
\end{tikzcd} \]
Similarly, we have a factorization of $\xi$ as
\[ \begin{tikzcd}[row sep=6ex, column sep=8ex, text height=1.5ex, text depth=0.5ex]
\Sp(\Yscr_{\hh C_p}) \ar{r}{\nu(C_p)^*} \ar{d}[swap]{f^*} \ar[phantom]{rd}{\simeq \SWarrow} & \Fun(B C_p, \Sp(\Yscr^{\hh C_p})) \ar{d}{f^*} \ar{r}{(-)^{t C_p}} \ar[phantom]{rd}{\lambda \SWarrow} & \Sp(\Yscr^{\hh C_p}) \ar{d}{f^*} \\
\Sp(\Xscr_{\hh C_p}) \ar{r}[swap]{\nu(C_p)^*} & \Fun(B C_p, \Sp(\Xscr^{\hh C_p})) \ar{r}[swap]{(-)^{t C_p}} & \Sp(\Xscr^{\hh C_p}) 
\end{tikzcd} \]
in which $\lambda$ is induced by vertically taking cofibers in the commutative square
\[ \begin{tikzcd}[row sep=4ex, column sep=6ex, text height=1.5ex, text depth=0.5ex]
f^* (-)_{h C_p} \ar{r}{\simeq} \ar{d}{f^* \Nm} &  (-)_{h C_p} f^* \ar{d}{\Nm f^*} \\
f^* (-)^{h C_p} \ar{r} & (-)^{h C_p} f^*.
\end{tikzcd} \]
Here, the desired naturality property of the additive norm follows from an elementary diagram chase using its inductive construction (see \cite{higheralgebra}*{Construction 6.1.6.4} or \cite{nikolaus-scholze}*{Construction I.1.7}).

Given Proposition~\ref{prp:UnstableRecollementFunctoriality}, both of these factorization assertions hold in view of the naturality of the comparison map $(-)^{t C_p} \nu(C_p)^* \xrightarrow{\simeq} i^* j_*$ as constructed in the proof of Theorem~\ref{prop:recoll-stab}.
\end{remark}

\subsection{Symmetric monoidal structures}\label{sec:symmon} In this subsection, we apply Theorem~\ref{prop:recoll-stab} to endow $\Sp^{C_p}(\Xscr)$ with a symmetric monoidal structure. Interpreting dualizability as a type of ``finiteness" for an object in a symmetric monoidal $\infty$-category, we then flesh out the way in which the genuine stabilization may be thought of as a ``compactification" of a $C_p$-$\infty$-topos as discussed in the introduction (Corollary~\ref{cor:DualizableImpliesCompact} and Proposition~\ref{prp:InternallyCompactToDualizable}). When $\Xscr$ is taken to be the $C_2$-$\infty$-topos $\widehat{X[i]}_{\et}$, we will later see how the good dualizability properties of $\Sp^{C_2}(\Xscr)$ (as opposed to $\Sp(\Xscr_{C_2})$) intercede in the proof of Thom stability (Lemma~\ref{lem:duals}) and thus in establishing the formalism of six operations.
% When $\Xscr$ is taken to be the $C_2$-$\infty$-topos $\widehat{X[i]}_{\et}$, the dualizability properties of $\Sp^{C_2}(\Xscr)$  will also play a crucial role in the proof of Lemma~\ref{lem:duals}, which underpins Thom stability and thereby the full six functor formalism on $\Sp^{C_2}(\widehat{X[i]}_{\et}) \comp$ (Theorem~\ref{construct:genuine}).
 % In the sequel, we will need these results when establishing the full six functor formalism for $\Sp^{C_2}(\widehat{(-)[i]}_{\et}) \comp$ (valued on a certain class of schemes; cf. Theorem~\ref{construct:genuine}).

Recall that given a left-exact lax symmetric monoidal functor $\phi: \Uscr \to \Zscr$ of symmetric monoidal $\infty$-categories, the right-lax limit $\Xscr = \Uscr \overrightarrow{\times} \Zscr$ may be endowed with the {\bf canonical symmetric monoidal structure} \cite{quigley-shah}*{Definition~1.21} that makes $(\Uscr,\Zscr)$ into a monoidal recollement \cite{quigley-shah}*{Definition~1.19}. Explicitly, given the data of a morphism of $\infty$-operads $\phi^{\otimes}: \Uscr^{\otimes} \to \Zscr^{\otimes}$ lifting $\phi$, we define (cf. Remark~\ref{rem:rlaxlim-param}) $$\Xscr^{\otimes} = \Uscr^\otimes \times_{\Zscr^{\otimes}} (\Zscr^\otimes)^{\Delta^1}.$$

\begin{construction} \label{con:MonoidalStructureGenuineStabilization}Let $\Xscr$ be a $C_p$-$\infty$-topos. To construct the symmetric monoidal structure on $\Sp^{C_p}(\Xscr)$, we note the following two facts:
\begin{enumerate} \item The stabilization functor $\Sp: \PrL_{\infty} \to \PrL_{\infty,\stab}$ is symmetric monoidal with respect to the Lurie tensor product of presentable $\infty$-categories \cite{higheralgebra}*{\S4.8.2}. In particular, given an $\infty$-topos $\Xscr$ equipped with the cartesian symmetric monoidal structure, we obtain a symmetric monoidal structure on its stabilization $\Sp(\Xscr)$. Moreover, for the left-exact left adjoint $f^{\ast}: \Xscr \to \Yscr$ of a geometric morphism, we obtain a symmetric monoidal functor $f^{\ast}: \Sp(\Xscr) \to \Sp(\Yscr)$.
\item Given a stable presentable symmetric monoidal $\infty$-category $\C$ and finite group $G$, the Tate construction $$(-)^{tG}: \Fun(B G, \C) \to \C$$ is lax symmetric monoidal such that the usual map $\mu: (-)^{h G} \to (-)^{t G}$ is the \emph{universal} lax symmetric monoidal natural transformation among those to lax symmetric monoidal functors that annihilate induced objects \cite{nikolaus-scholze}*{Theorem~I.3.1}.\footnote{The cited reference states this fact for $\C = \Sp$, but their proof clearly extends to accommodate the indicated generality as it only uses abstract multiplicative properties of the Verdier quotient.}
\end{enumerate}
We deduce that the stable gluing functor $\Theta^{\Tate} = (-)^{tC_p} \circ \nu(C_p)^{\ast}$ is lax symmetric monoidal. We then regard the right-lax limit $\Sp^{C_p}(\Xscr)$ of $\Theta^{\Tate}$ as a symmetric monoidal $\infty$-category via the canonical symmetric monoidal structure, and we call the resulting tensor product on $\Sp^{C_p}(\Xscr)$ the {\bf smash product}.
\end{construction}

\begin{remark} \label{rem:PresentablySMC} Since the jointly conservative recollement left adjoints of a monoidal recollement are strong symmetric monoidal, we see that the smash product on $\Sp^{C_p}(\Xscr)$ commutes with colimits separately in each variable. Therefore, $\Sp^{C_p}(\Xscr)$ is presentably symmetric monoidal and lifts to an object in $\CAlg(\PrL_{\infty, \stab})$.
\end{remark}

\begin{remark} \label{rem:idempotence} Lurie's construction of the smash product in spectra arises as a ``decategorification'' of the fact that $\Sp$ is an idempotent object in $\PrL_{\infty}$ with respect to the Lurie tensor product of presentable $\infty$-categories \cite{higheralgebra}*{\S 4.8.2}. The analogous picture for $\underline{\Sp}^G$ and its $G$-symmetric monoidal structure (which by design also incorporates the Hill-Hopkins-Ravenel norm functors \cite{HHR}) has been worked out by Nardin in his thesis \cite{denis-thesis}*{\S 3}. We expect that similar methods may be used to construct a $G$-symmetric monoidal structure on $\underline{\Sp}^G(\Xscr)$ for any $G$-$\infty$-topos $\Xscr$. Since the details of such a construction would take us too far afield, we will be content with the simpler and more explicit Construction~\ref{con:MonoidalStructureGenuineStabilization}, which relies on the recollement presentation of $\Sp^{C_p}(\Xscr)$.
% We expect that the smash product on $\underline{\Sp}^G$ and, more generally, on $\underline{\Sp}^{\C_p}(\underline{\Xscr})$ can be constructed in a parallel fashion given a full-fledged theory of a $G$-symmetric monoidal $G$-stable $\infty$-categories. Since this is not yet available in the literature, we contend with the more explicit Construction~\ref{con:MonoidalStructureGenuineStabilization}, which relies on the recollement presentation of $\Sp^{C_p}(\Xscr)$.
\end{remark}

As with the Tate construction, $\Theta^{\Tate}$ has a monoidal universal property. Let $\mu: \Theta \to \Theta^{\Tate}$ also denote the lax symmetric monoidal transformation obtained from $\mu: (-)^{h C_p} \to (-)^{tC_p}$ via precomposition by $\nu(C_p)^*$, and recall our notion of induced objects in $\Sp(\Xscr_{\hh C_p})$ (Definition~\ref{def:topos-induced}).

% we note that
% \[ \Theta \circ \pi_* \simeq \nu^*: \Sp(\Xscr) \to \Sp(\Xscr^{\hh C_p}), \]
% and in particular, $\Theta \circ \pi_*$ preserves colimits.
\begin{proposition} \label{prop:ThetaTateMonoidalUniversalProperty} $\mu: \Theta \to \Theta^{\Tate}$ is the universal lax symmetric monoidal transformation among those to lax symmetric monoidal functors that annihilate induced objects of $\Sp(\Xscr_{\hh C_p})$.
\end{proposition}
\begin{proof} By the theory of the Verdier quotient \cite{nikolaus-scholze}*{\S I.3}, there exists a lax symmetric monoidal transformation $\mu': \Theta \to \widehat{\Theta}$ with the indicated universal property \cite{nikolaus-scholze}*{Theorem~I.3.6}. Moreover, by the formula of \cite{nikolaus-scholze}*{Theorem~I.3.3(ii)} (or rather, its fiber), we see that by construction $\fib(\mu')$ commutes with the canonical monadic resolutions in $\Sp(\Xscr_{\hh C_p})$ exhibiting objects as colimits of induced objects. Since we also have that
\[ \Theta \circ \pi_* \simeq \nu^*: \Sp(\Xscr) \to \Sp(\Xscr^{\hh C_p}), \]
and in particular, $\Theta \circ \pi_*$ preserves colimits, we deduce that $\fib(\mu')$ commutes with all colimits.

As we saw in the proof of Theorem~\ref{prop:recoll-stab}, $\Theta^{\Tate}$ annihilates induced objects, and $\mu$ is lax symmetric monoidal via Construction~\ref{con:MonoidalStructureGenuineStabilization}. By the universal propery of $\mu'$, there exists a lax symmetric monoidal transformation $\alpha:\widehat{\Theta} \to \Theta^{\Tate}$ factoring $\mu$ through $\mu'$. Forgetting monoidal structure (as we may do to check that $\alpha$ is an equivalence), we then take fibers to obtain a commutative diagram
\[ \begin{tikzcd}[row sep=4ex, column sep=6ex]
\fib(\mu') \ar{r}{\Nm'} \ar{d}[swap]{\beta} & \Theta \ar{r}{\mu'} \ar{rd}[swap]{\mu} & \widehat{\Theta} \ar{d}{\alpha} \\
\fib(\mu) \ar{ru}[swap]{\Nm} & & \Theta^{\Tate}.
\end{tikzcd} \]
Since $\Nm$ and $\Nm'$ are equivalences on induced objects, $\beta$ is also an equivalence on induced objects by the two-out-of-three property. Since both $\fib(\mu')$ and $\fib(\mu) \simeq (-)_{h C_p} \nu(C_p)^{\ast}$ commute with all colimits and induced objects generate $\Sp(\Xscr_{\hh C_p})$ under colimits, we deduce that $\beta$ and thus $\alpha$ is an equivalence.
\end{proof}

We will need to perform Construction~\ref{con:MonoidalStructureGenuineStabilization} in families. Intuitively, given a lax commutative square of lax symmetric monoidal functors (with $\phi$ left-exact)
\[ \begin{tikzcd}[row sep=4ex, column sep=4ex, text height=1.5ex, text depth=0.5ex]
\Uscr \ar{r}{\phi} \ar{d}[swap]{f^*} \ar[phantom]{rd}{\eta \SWarrow} & \Zscr \ar{d}{f^*} \\
\Uscr' \ar{r}[swap]{\phi} & \Zscr',
\end{tikzcd} \]
the induced functor $f^*: \Xscr = \Uscr \overrightarrow{\times} \Zscr \to \Xscr' = \Uscr' \overrightarrow{\times} \Zscr'$ canonically lifts to a lax symmetric monoidal functor via the following formula: let $\gamma_F: F(-) \otimes F(-) \to F(- \otimes -)$ generically denote the lax comparison map for a lax symmetric functor $F$ and suppose $x_i = (u_i,z_i,\alpha_i: u_i \to \phi z_i) \in \Xscr$, $i=1,2$. Then we have equivalences
\begin{align*} f^*(x_1) \otimes f^*(x_2) & \simeq (f^*u_1 \otimes f^* u_2, f^* z_1 \otimes f^* z_2, \gamma_{\phi} \circ (\eta(u_1) \otimes \eta(u_2)) \circ (f^* \alpha_1 \otimes f^* \alpha_2)), \\
f^*(x_1 \otimes x_2) & \simeq (f^*(u_1 \otimes u_2), f^*(z_1 \otimes z_2), \eta(u_1 \otimes u_2) \circ f^* \gamma_{\phi} \circ f^*(\alpha_1 \otimes \alpha_2) ),
\end{align*}
with respect to which the lax comparison map \[ \gamma_{f^*}: f^*(x_1) \otimes f^*(x_2) \to f^*(x_1 \otimes x_2) \]
is given by the maps $\gamma_{f^*}$ on components and the commutative diagram
\[ \begin{tikzcd}[row sep=6ex, column sep=8ex, text height=1.5ex, text depth=0.5ex]
f^* z_1 \otimes f^* z_2 \ar{r}{f^* \alpha_1 \otimes f^* \alpha_2} \ar{d}{\gamma_{f^*}} & f^* \phi u_1 \otimes f^* \phi u_2 \ar{r}{\eta(u_1) \otimes \eta(u_2)} \ar{d}{\gamma_{f^*}} \ar{rd} & \phi f^* u_1 \otimes \phi f^* u_2 \ar{r}{\gamma_\phi} \ar{rd} \ar[phantom]{d}{\simeq} & \phi(f^* u_1 \otimes f^* u_2) \ar{d}{\phi \gamma_{f^*}} \\
f^*(z_1 \otimes z_2) \ar{r}{f^*(\alpha_1 \otimes \alpha_2)} & f^*(\phi u_1 \otimes \phi u_2) \ar{r}{f^* \gamma_{\phi}} \ar{r} & f^* \phi (u_1 \otimes u_2) \ar{r}{\eta(u_1 \otimes u_2)} \ar{r} & \phi f^*(u_1 \otimes u_2).
\end{tikzcd} \]

Furthermore, if $f^*: \Uscr \to \Uscr'$ and $f^*: \Zscr \to \Zscr'$ are in addition strong symmetric monoidal, then we see that $f^*: \Xscr \to \Xscr'$ is strong symmetric monoidal.

Using the fibrational perspective, we can make this intuition precise, at least if the vertical functors $f^*$ are in addition symmetric monoidal. Recall that a diagram $$\C_{(-)}: S \to \CAlg(\widehat{\Cat}_{\infty})$$ of symmetric monoidal $\infty$-categories and strong symmetric monoidal functors unstraightens to a {\bf cocartesian $S$-family of symmetric monoidal $\infty$-categories} \cite{higheralgebra}*{Definition~4.8.3.1} $$\C^{\otimes} \to S \times \Fin_*.$$ A lax natural transformation $\phi: \C_{(-)} \Rightarrow \D_{(-)}$ through lax symmetric monoidal functors then corresponds to a functor $\phi: \C^{\otimes} \to \D^{\otimes}$ over $S \times \Fin_*$ such that for every $s \in S$, the restriction $\phi_s: \C_s^{\otimes} \to \D_s^{\otimes}$ is a morphism of $\infty$-operads. By Lemma~\ref{lem:lax-lim} (with $B = S \times \Fin_*$) and \cite{quigley-shah}*{Lemma~B.2}, the fiberwise right-lax limit $\C^{\otimes} \overrightarrow{\times} \D^{\otimes}$ of $\phi$ yields again a cocartesian $S$-family of symmetric monoidal $\infty$-categories. Moreover, unpacking the description of the pushforward functor in Lemma \ref{lem:lax-lim} yields the formula just described. Using the analysis in Remark~\ref{rem:StableRecollementFunctoriality}, we may then deduce:

\begin{theorem} \label{thm:MonoidalityOfGenuineStabilizationFunctor} The functor $\Sp^{C_p}: \Fun(B C_p, \LTop) \to \PrL_{\infty, \stab}$ lifts to a functor
\[ (\Sp^{C_p})^{\otimes}: \Fun(B C_p, \LTop) \to \CAlg(\PrL_{\infty, \stab}). \]
\end{theorem}
\begin{proof} By a Yoneda argument and the naturality of unstraightening, it will suffice to consider a diagram $$\Xscr_{(-)}: S \to \Fun(B C_p, \LTop)$$ of $C_p$-$\infty$-topoi and construct the corresponding diagram $$\Sp^{C_p}(\Xscr)^{\otimes}: S \to \CAlg(\PrL_{\infty, \stab}).$$
Let $\Sp(\Xscr_{\hh C_p})^{\otimes} \to S \times \Fin_*$ and $\Sp(\Xscr^{\hh C_p})^{\otimes} \to S \times \Fin_*$ be the two cocartesian $S$-families of symmetric monoidal $\infty$-categories obtained via unstraightening. As we have just indicated, we need to lift the stable gluing functor $\Theta^{\Tate} = (-)^{t C_p} \circ \nu(C_p)^{\ast}$ to a map
\[ \Sp(\Xscr_{\hh C_p})^{\otimes} \xrightarrow{\nu(C_p)^{\ast}} (\Sp(\Xscr^{\hh C_p})^{\otimes})^{B C_p} \xrightarrow{(-)^{tC_p}} \Sp(\Xscr^{\hh C_p})^{\otimes} \]
over $S \times \Fin_*$ that preserves inert edges. For $\nu(C_p)^*$, this is easy; since $\nu(C_p)^*$ is already strong symmetric monoidal, we may obtain it at the level of cocartesian $S$-families via unstraightening. We now invoke Lemma~\ref{lem:TateMonoidalInFamilies} to handle the more difficult case of $(-)^{t C_p}$. 
\end{proof}

%Technical lemma
\begin{lemma} \label{lem:TateMonoidalInFamilies} Let $\C_{(-)}: S \to \CAlg(\PrL_{\infty,\stab})$ be a functor and let $\C^{\otimes} \to S \times \Fin_{\ast}$ be the corresponding cocartesian $S$-family of symmetric monoidal $\infty$-categories. The collection of lax symmetric monoidal transformations $$ \left \{ \left( \mu_s: (-)^{h G} \Rightarrow (-)^{tG} \right): \C_s^{B G} \to \C_s \: | \: s \in S \right\}$$ of \cite{nikolaus-scholze}*{Theorem~I.3.1} assemble to a natural transformation
\[ \left( \mu: (-)^{h G} \Rightarrow (-)^{t G} \right): (\C^{\otimes})^{B G} \to \C^{\otimes} \]
of functors over $S \times \Fin_*$, which restrict to maps of $\infty$-operads over $\{ s \} \times \Fin_*$ for all $s \in S$.
\end{lemma}
\begin{proof} The diagonal $\delta: \C_{(-)} \Rightarrow (\C_{(-)})^{BG}$ is a natural transformation valued in $\CAlg(\PrL_{\infty,\stab})$, and hence unstraightens to a functor $\delta: \C^{\otimes} \to (\C^{\otimes})^{B G}$ over $S \times \Fin_*$ that preserves cocartesian edges (where $(-)^{B G}$ now indicates the cotensor in (marked) simplicial sets over $S \times \Fin_*$). By \cite{higheralgebra}*{Corollary~7.3.2.7}, $\delta$ admits a relative right adjoint $(-)^{h G}$ over $S \times \Fin_*$, which restricts to a map of $\infty$-operads over $\{ s \} \times \Fin_*$ for all $s \in S$. 

To further assemble the lax symmetric monoidal transformations $\mu_s$ (and, along the way, construct $(-)^{t G}$), we apply an $\infty$-operadic variant of the pairing construction of \cite{htt}*{Corollary~3.2.2.13} (also see \cite{jay-thesis}*{Example~2.24}) in order to speak of a ``family of lax monoidal functors and lax monoidal transformations." Namely, given cocartesian $S$-families of symmetric monoidal $\infty$-categories $A^{\otimes}, B^{\otimes} \to S \times \Fin_*$, let $\leftnat{A^{\otimes}}$ indicate the marking given by the inert edges and consider the span of marked simplicial sets
\[ \begin{tikzcd}[row sep=4ex, column sep=4ex, text height=1.5ex, text depth=0.5ex]
S^{\sharp} & \leftnat{A^{\otimes}} \ar{r}{p} \ar{l}[swap]{q} & (S \times \Fin_*)^{\sharp},
\end{tikzcd} \]
where $q$ is the composition of the structure map $p$ and the projection to $S$. Then we have the functor
$$q_* p^*: s\Set^+_{/(S \times \Fin_*)} \to s\Set^+_{/S}, \qquad \leftnat{B^{\otimes}} \mapsto \widetilde{\Fun}{}^{\otimes,\lax}(A,B) = q_* p^*(\leftnat{B^{\otimes}})$$
that carries cocartesian fibrations over $S \times \Fin_*$ to categorical fibrations over $S$ by \cite{higheralgebra}*{Proposition~B.4.5}, using that $q$ is a flat categorical fibration \cite{higheralgebra}*{Definition~B.3.8} since it is a cocartesian fibration by \cite{higheralgebra}*{Example~B.3.11}. Unwinding the definition of the simplicial set $\widetilde{\Fun}{}^{\otimes,\lax}(A,B)$, we see that the fiber over $s \in S$ is isomorphic to the functor $\infty$-category $\Fun^{\otimes,\lax}_{/\{s\} \times \Fin_*}(A^{\otimes}_s, B^{\otimes}_s)$ of maps of $\infty$-operads $f: A^{\otimes}_s \to B^{\otimes}_s$, and a morphism $f \to g$ in $\widetilde{\Fun}{}^{\otimes,\lax}(A,B)$ covering $\alpha: s \to t$ in $S$ is given by the data of a lax commutative square of lax symmetric monoidal functors
\[ \begin{tikzcd}[row sep=4ex, column sep=4ex, text height=1.5ex, text depth=0.5ex]
A_s \ar{r}{f} \ar{d}[swap]{\alpha^{A}_!} \ar[phantom]{rd}{\SWarrow} & B_s \ar{d}{\alpha^{B}_!} \\
A_t \ar{r}[swap]{g} & B_t.
\end{tikzcd} \]
Moreover, if $B^{\otimes}$ is classified by a functor to $\CAlg(\PrL_{\infty})$ (so that the underlying $\infty$-category is presentable and the tensor product distributes over colimits), then the existence of operadic left Kan extensions\footnote{To invoke the existence theorem (and later on, the adjoint functor theorem), we note that if the fibers of $A^{\otimes}$ are presentable, we also implicitly restrict to the full subcategory of accessible functors.} \cite{higheralgebra}*{Theorem~3.1.2.3} ensures that $\widetilde{\Fun}{}^{\otimes,\lax}(A,B)$ is a \emph{locally cocartesian} fibration\footnote{Note that this sort of functoriality is separate from that articulated in \cite{htt}*{Corollary~3.2.2.13}.} over $S$, with the indicated square corresponding to a locally cocartesian edge if and only if $g$ is a operadic left Kan extension of $\alpha^{B}_! \circ f$ along $\alpha^{A}_!$. Finally, we note that sections of $\widetilde{\Fun}{}^{\otimes,\lax}(A,B)$ over $S$ correspond to functors $A^{\otimes} \to B^{\otimes}$ over $S \times \Fin_*$ that preserve inert edges.

Returning to the task at hand, consider $\widetilde{\Fun}{}^{\otimes,\lax}(\C^{BG}, \C)$ and the full subcategory $\widetilde{\Fun}{}^{\otimes,\lax}_0(\C^{BG}, \C)$ on those lax symmetric monoidal functors $\C_s^{BG} \to \C_s$ that annihilate induced objects. Note that since the $\infty$-category of lax symmetric monoidal functors $\Fun^{\otimes,\lax}(\C,\D)$ is equivalent to $\CAlg(\Fun(\C,\D))$ for the Day convolution symmetric monoidal structure on $\Fun(\C,\D)$ \cite{glasmanDay}*{Proposition~2.12}, it follows that $\Fun^{\otimes,\lax}(\C,\D)$ is presentable if $\D$ is presentable symmetric monoidal by \cite{glasmanDay}*{Lemma~2.13} and \cite{higheralgebra}*{Corollary~3.2.3.5}, and limits in $\Fun^{\otimes,\lax}(\C,\D)$ are created in $\Fun(\C,\D)$ by \cite{higheralgebra}*{Corollary~3.2.2.5}. Since the condition that functors annihilate induced objects is stable under arbitrary limits, it then follows by the adjoint functor theorem that the inclusion $$\Fun^{\otimes, \lax}_0(\C_s^{BG}, \C_s) \subset \Fun^{\otimes,\lax}(\C_s^{BG}, \C_s)$$ admits a left adjoint, which sends $(-)^{h G}$ to $(-)^{t G}$ by \cite{nikolaus-scholze}*{Theorem~I.3.1}. We also see that $\widetilde{\Fun}{}^{\otimes,\lax}_0(\C^{BG}, \C)$ inherits the property of being locally cocartesian over $S$ by applying the (fiberwise in the target) localization functor to the operadic left Kan extension (so the inclusion of $\widetilde{\Fun}_0$ into $\widetilde{\Fun}$ need not preserve locally cocartesian edges). Invoking \cite{higheralgebra}*{Proposition~7.3.2.11}, we deduce that these fiberwise left adjoints assemble to a localization functor
\[ L: \widetilde{\Fun}{}^{\otimes,\lax}(\C^{BG}, \C) \to \widetilde{\Fun}{}^{\otimes,\lax}_0(\C^{BG}, \C). \]
The desired functor $(-)^{tG}: (\C^{\otimes})^{BG} \to \C^{\otimes}$ is now obtained by postcomposition of $(-)^{h G}$ as a section of $\widetilde{\Fun}{}^{\otimes,\lax}(\C^{BG}, \C)$ with $L$, and the natural transformation $\mu$ similarly corresponds to the unit for the localization applied to $(-)^{h G}$.
\end{proof}

\begin{remark} \label{rem:SuspensionMonoidal} The exchange transformation $\alpha: \Sigma^{\infty}_+ \theta \Rightarrow \Theta^{\Tate} \Sigma^{\infty}_+$ of Remark~\ref{rem:StableRecollementFunctoriality} is adjoint to an equivalence of lax symmetric monoidal functors, hence is a lax symmetric monoidal transformation. We deduce that its right-lax limit $\Sigma^{\infty}_+: \Xscr_{C_p} \to \Sp^{C_p}(\Xscr)$ is strong symmetric monoidal with respect to the cartesian product on the source and the smash product on the target. Moreover, with respect to the smash product on $\Sp(\Xscr_{C_p})$, $\Sigma^{\infty}_+$ factors as the composite of strong symmetric monoidal functors
\[ \Xscr_{C_p} \xrightarrow{\Sigma^{\infty}_+} \Sp(\Xscr_{C_p}) \xrightarrow{\Fr} \Sp^{C_p}(\Xscr). \]
\end{remark}

\begin{corollary} \label{cor:DualizableImpliesCompact} If the unit in $\Sp(\Xscr_{C_p})$ is compact, then the unit in $\Sp^{C_p}(\Xscr)$ is compact (and hence any dualizable object in $\Sp^{C_p}(\Xscr)$ is compact).
\end{corollary}
\begin{proof} By Remark~\ref{rem:SuspensionMonoidal} and Corollary~\ref{cor:MonadicityForgetful}, $\Fr: \Sp(\Xscr_{C_p}) \to \Sp^{C_p}(\Xscr)$ is a strong symmetric monoidal functor whose right adjoint $U$ preserves colimits. The claim follows.
\end{proof}

We can also produce dualizable objects in $\Sp(\Xscr_{C_p})$ from $\Sp(\Xscr)$. To show such a result, we first need to establish the projection formula for the monoidal adjunction $$\overline{\pi}^*: \Sp^{C_p}(\Xscr) \rightleftarrows \Sp(\Xscr) : \overline{\pi}_*,$$
where $\overline{\pi}^*$ comes endowed with the structure of a strong symmetric monoidal functor as the composition $$\pi^* j^*: \Sp^{C_p}(\Xscr) \to \Sp(\Xscr_{\hh C_p}) \to \Sp(\Xscr).$$

\begin{proposition} Let $\Xscr$ be a $C_p$-$\infty$-topos and consider the monoidal adjunction $\overline{\pi}^* \dashv \overline{\pi}_*$. Then:
\begin{enumerate}
\item $\overline{\pi}_*$ preserves colimits.
\item $\overline{\pi}_*$ is conservative.
\item For any $A \in \Sp^{C_p}(\Xscr)$ and $B \in \Sp(\Xscr)$, the natural map
$$ \gamma: \overline{\pi}_*(B) \otimes A \to \overline{\pi}_*(B \otimes \overline{\pi}^*(A)) $$
is an equivalence.
\end{enumerate}
Therefore, we have an equivalence $\Sp(\Xscr) \simeq \Mod_{\overline{\pi}_*(\1)} (\Sp^{C_p}(\Xscr))$ of symmetric monoidal $\infty$-categories.
\end{proposition}
\begin{proof} The conclusion will follow from \cite{mnn-descent}*{Theorem~5.29} once we establish the three properties. (1) holds in view of the ambidexterity equivalence $\overline{\pi}_* \simeq \overline{\pi}_!$. For (2) and (3), we first note that for any $X \in \Sp(\Xscr)$, we have the formula $$\overline{\pi}^* \overline{\pi}_* (X) \simeq X \oplus \sigma^*(X),$$
using that $\underline{\Sp}^{C_p}(\Xscr)$ admits finite $C_p$-products and the Beck-Chevalley condition (here, $\sigma^*$ denotes the $C_2$-action on $\Sp(\Xscr)$). It is thus clear that $\overline{\pi}_*$ is conservative. As for the projection formula, since $i^*: \Sp^{C_p}(\Xscr) \to \Sp(\Xscr^{\hh C_p})$ is symmetric monoidal and annihilates induced objects, it suffices to check that $\overline{\pi}^*(\gamma)$ is an equivalence. But for this, under the given formula $\overline{\pi}^*(\gamma)$ becomes
\[ (B \oplus \sigma^* B) \otimes \overline{\pi}^*(A) \to (B \otimes \overline{\pi}^*(A)) \oplus \sigma^* (B \otimes \overline{\pi}^* A),\]
which is an equivalence in view of $\sigma^* \overline{\pi}^* \simeq \overline{\pi}^*$ and distributivity of $\otimes$ over $\oplus$.
\end{proof}

Let $\underline{\Hom}(-,-)$ denote the internal hom in a presentable symmetric monoidal $\infty$-category.

\begin{corollary} \label{cor:Dualizability} Let $\Xscr$ be a $C_p$-$\infty$-topos. Then:
\begin{enumerate} \item For all $A \in \Sp^{C_p}(\Xscr)$ and $B \in \Sp(\Xscr)$, the natural map
\[ \overline{\pi}_* \underline{\Hom}(B, \overline{\pi}^* A) \to \underline{\Hom}(\overline{\pi}_* B,A) \]
is an equivalence.
\item Suppose $E \in \Sp(\Xscr)$ is a dualizable object (with dual $E^{\vee} \simeq \underline{\Hom}(E,\1)$). Then $\overline{\pi}_*(E)$ is dualizable with dual $\overline{\pi}_*(E)^{\vee} \simeq \overline{\pi}_*(E^{\vee})$. In particular, $\overline{\pi}_*(\1)$ is a self-dual $\mathbb{E}_{\infty}$-algebra.
\end{enumerate}
\end{corollary}
\begin{proof} (1) is an easy consequence of the $(\overline{\pi}^*, \overline{\pi}_*)$ projection formula and ambidexterity equivalence $\overline{\pi}^* \simeq \overline{\pi}_!$. For (2), it suffices to show that for any $X \in \Sp^{C_p}(\Xscr)$, the natural map
\[ \underline{\Hom}(\overline{\pi}_*(E),\1) \otimes X \to  \underline{\Hom}(\overline{\pi}_*(E),X) \]
is an equivalence. But under (1) and the projection formula, this map is equivalent to
\[ \overline{\pi}_* (\underline{\Hom}(E,\1) \otimes \overline{\pi}^* X) \to \overline{\pi}_* \underline{\Hom}(E,\overline{\pi}^* X). \]
Since $E$ is dualizable, we have that $E^{\vee} \otimes \overline{\pi}^* X \simeq \underline{\Hom}(E,\overline{\pi}^* X)$, and the claim follows.
\end{proof}

\begin{remark} Suppose $\Xscr$ is a presentable symmetric monoidal $\infty$-category and $(\Uscr, \Zscr)$ is a monoidal recollement of $\Xscr$. Let $X \in \Xscr$. Then the endofunctor $(-) \otimes X : \Xscr \to \Xscr$ is a morphism of recollements that induces $(-) \otimes j^*(X)$ on $\Uscr$ and $(-) \otimes i^*(X)$ on $\Zscr$. Moreover, if $X$ is dualizable, then $(-) \otimes X$ is a \emph{strict} morphism of recollements. This amounts to showing that the $(j^* , j_*)$-projection formula holds, i.e., for every $U \in \Uscr$, the natural map
\[ j_*(U) \otimes X \to j_*(U \otimes j^*X). \]
is an equivalence. Indeed, if we let $Y \in \Xscr$ be any object, then using dualizability of $X$ we see that
\[ \Map(Y,j_*(U) \otimes X) \simeq \Map(j^* Y \otimes j^*(X)^{\vee},U) \simeq \Map(Y,j_*(U \otimes j^*X)). \]
\end{remark}

\begin{remark} \label{rem:FiniteLocalization} We can use Corollary~\ref{cor:Dualizability} to control the smashing localization $i_* i^*$ on $\Sp^{C_p}(\Xscr)$ as a finite localization (in the sense of Miller \cite{Miller}), by presenting the idempotent $\mathbb{E}_{\infty}$-algebra $i_* i^*(\1)$ as a filtered colimit of dualizable objects. Namely, we observe that the dualizable $\mathbb{E}_{\infty}$-algebra object $A = \overline{\pi}_* \overline{\pi}^*(\1)$ fits into the general theory of $A$-complete, $A$-torsion, and $A^{-1}$-local objects in the sense of Mathew-Naumann-Noel \cite{mnn-descent}*{Part~1}. In particular, using that $\Sp(\Xscr_{\hh C_p})$ is monadic over $\Sp(\Xscr)$ (Lemma~\ref{lem:generated}), we see that the full subcategory of $A$-complete objects \cite{mnn-descent}*{Definition~2.15} and the functor $L_A$ of $A$-completion \cite{mnn-descent}*{Definition~2.19} identify with $\Sp(\Xscr_{\hh C_p})$ embedded via $j_*$ and $j_* j^*$. Moreover, the recollement cofiber sequence
\[ j_!j^*(\1) \to \1 \to i_* i^*(\1) \]
identifies with the cofiber sequence of \cite{mnn-descent}*{Construction~3.4}
\[ V_A \to \1 \to U_A, \]
where $U_A$ is defined to be the filtered colimit of the sequence of maps
\[ \1 = U_0 \to U_1 \to U_2 \to U_3 \to \cdots \]
dual to the $A$-based Adams tower for $\1$ (where $I = \fib(\1 \to A)$ and $U_i:=(I^{\otimes i})^{\vee}$)
\[ \1 = I^{\otimes 0} \leftarrow I \leftarrow I^{\otimes 2} \leftarrow I^{\otimes 3} \leftarrow \cdots .\]
Thus, the full subcategory of $A$-torsion objects \cite{mnn-descent}*{Definition~3.1} and the $A$-acyclization functor $\mathrm{AC}_A$ \cite{mnn-descent}*{Construction~3.2} identify with $\Sp(\Xscr_{\hh C_p})$ embedded via $j_!$ and the colocalization $j_! j^*$, while the full subcategory of $A^{-1}$-local objects \cite{mnn-descent}*{Definition~3.10} and the $A^{-1}$-localization functor identify with $\Sp(\Xscr^{\hh C_p})$ embedded via $i_*$ and the smashing localization $i_* i^*$.

We further observe that if $p=2$, $I$ and hence $U_k$ are all \emph{invertible} objects. Indeed, since $I$ is dualizable, it suffices to show that $i^*(I)$ and $\overline{\pi}^*(I)$ are invertible. But $i^*(I) \simeq \1$ since $i^*(A) \simeq 0$, and $\overline{\pi}^*(I) \simeq \Sigma^{-1} \1$ using that the map $\overline{\pi}^*(\1 \to A)$ identifes with the summand inclusion $\1 \to \1 \oplus \1$. This reflects the fact that the cofiber sequence $$A^{\vee} \simeq A \to \1 \to U_1$$ is a categorical avatar of the cofiber sequence (for $\sigma$ the sign $C_2$-real representation) $$(C_2)_+ = S(\sigma)_+ \to S^0 \to S^{\sigma} $$ in equivariant homotopy theory, which recovers it when $\Xscr = \Spc$ with trivial $C_2$-action. We will see another perspective on this in the algebro-geometric setting in Lemma~\ref{lem:duals}.
\end{remark}

We end this subsection by applying Remark~\ref{rem:FiniteLocalization} to prove a partial converse to Corollary~\ref{cor:DualizableImpliesCompact}. For a presentable symmetric monoidal $\infty$-category $\C$, call an object $X \in \C$ {\bf internally compact} if the endofunctor $\underline{\Hom}(X,-): \C \to \C$ preserves filtered colimits. Note that if $X$ is dualizable, then $X$ is internally compact, since $X^{\vee} \otimes (-) \simeq \underline{\Hom}(X,-)$.

\begin{proposition} \label{prp:InternallyCompactToDualizable} Let $\Xscr$ be a $C_p$-$\infty$-topos and $X \in \Sp^{C_p}(\Xscr)$ an internally compact object. Suppose that $j^*(X)$ is dualizable in $\Sp(\Xscr_{\hh C_p})$ (equivalently, $\overline{\pi}^*(X)$ is dualizable in $\Sp(\Xscr)$)\footnote{Note that this latter condition is sometimes easier to check. For instance, if (internally) compact objects are known to be dualizable in $\Sp(\Xscr)$, then we may use that $\overline{\pi}^*$ preserves (internally) compact objects as it participates in an ambidextrous adjunction. However, $j^*$ does not preserve all compact objects, and correspondingly, dualizable objects are almost never compact in $\Sp(\Xscr_{\hh C_p})$ unless they are also induced.} and $i^*(X)$ is dualizable in $\Sp(\Xscr^{\hh C_p})$. Then $X$ is dualizable. 
\end{proposition}
\begin{proof} Let $Y \in \Sp^{C_p}(\Xscr)$ be any object. Without any assumptions on $X$ or $Y$, we note that the natural map
\[ j^* \underline{\Hom}(X,Y) \to \underline{\Hom}(j^*X,j^*Y)
\]
is an equivalence by \cite{quigley-shah}*{Proposition~1.30.5}. Given that $X$ is internally compact, we further note that the natural map
\[ i^* \underline{\Hom}(X,Y) \to \underline{\Hom}(i^* X,i^*Y)
\]
is an equivalence. Indeed, it suffices to check the map is an equivalence upon applying $i_*$, after which it identifies as
\[ \underline{\Hom}(X,Y) \otimes i_* i^*(\1) \to \underline{\Hom}(X,Y \otimes i_* i^*(\1)),
\]
where on the righthand side, we have traded the internal hom in $\Sp(\Xscr^{\hh C_p})$ for that in $\Sp^{C_p}(\Xscr)$ (cf. \cite{quigley-shah}*{1.29}). Now by Remark~\ref{rem:FiniteLocalization}, we have that $i_* i^*(\1) \simeq \colim_{k \in \NN} U_k$, and using that $X$ is internally compact, the map is equivalent to
\[ \colim_{k \in \NN}(f_k) : \colim_{k \in \NN} \underline{\Hom}(X,Y) \otimes U_{k} \to \colim_{k \in \NN} \underline{\Hom}(X, Y \otimes U_{k}). \]
It thus suffices to check that each map $f_k$ is an equivalence, which follows readily from the dualizability of each of the $U_k$. Indeed, we note
\[ \underline{\Hom}(X,Y) \otimes U_{k} \simeq \underline{\Hom}(X \otimes U_k^{\vee},Y) \simeq \underline{\Hom}(X,Y \otimes U_k). \]
We now generically write $(-)^{\vee} = \underline{\Hom}(-,\1)$, so that $(j^* X)^{\vee}$ and $(i^* X)^{\vee}$ are the respective duals of $j^* X$ and $i^* X$, and claim that $X^{\vee}$ is dual to $X$. For this, first note that the above equivalences specialize to show that $j^\ast(X^{\vee}) \simeq (j^\ast X)^{\vee}$ and $i^\ast(X^{\vee}) \simeq (i^\ast X)^{\vee}$. Next, note that the natural maps
\[ X^{\vee} \otimes Y \to \underline{\Hom}(X,Y), \qquad X \otimes Y \to \underline{\Hom}(X^\vee,Y) \]
are sent under $j^*$ and $i^*$ to the maps
\begin{align*} (j^*X)^{\vee} \otimes j^* Y \to \underline{\Hom}(j^*X,j^* Y),& \qquad j^*X \otimes j^* Y \to \underline{\Hom}((j^*X)^\vee,j^* Y), \\
(i^*X)^{\vee} \otimes i^* Y \to \underline{\Hom}(i^*X,i^* Y),& \qquad i^*X \otimes i^* Y \to \underline{\Hom}((i^*X)^\vee,i^* Y),
\end{align*}
which are equivalences by our assumption that $j^* X$ and $i^* X$ are dualizable. We then invoke the joint conservativity of $(j^*, i^*)$ to conclude that $X$ is dualizable.
\end{proof}

\begin{remark} Suppose $X \in \Sp^G$ is a compact object. Then $X$ is internally compact. We can show this fact even without assuming the implication that compact objects in $\Sp^G$ are dualizable. Indeed, using the joint conservativity of the categorical fixed point functors (that preserve colimits in view of the reverse implication that dualizable objects are compact in $\Sp^G$), it suffices to check that for all subgroups $H \leq G$, the composite
\[ \Sp^G \xrightarrow{\underline{\Hom}(X,-)} \Sp^G \xrightarrow{\res^G_H} \Sp^H \xrightarrow{\Hom(\1,-)} \Sp \]
preserves colimits (where $\Hom(-,-)$ denotes the mapping spectrum). But a diagram chase shows this composite to be equivalent to
\[ \Sp^G \xrightarrow{\res^G_H} \Sp^H \xrightarrow{\Hom(\res^G_H X,-)} \Sp,
\]
and since $\res^G_H(X)$ is compact, this functor does preserve colimits. However, we do not expect a similar implication to hold for $\Sp^G(\Xscr)$ in general.
\end{remark}

\section{The six functors formalism for $b$-sheaves with transfers} \label{sect:six}

To set the stage for our work, we first recall the formalism of six operations after the work of Ayoub \cite{ayoub} and as further elaborated upon by Cisinski-D\'eglise \cite{cisinski-deglise} (an $\infty$-categorical framework is explained in Khan's thesis \cite{khan-thesis}). Let $\Sch'$ be a subcategory of schemes such that for any smooth morphism $f: T \rightarrow S$ in $\Sch'$, the pullback of $f$ along any morphism $g: S' \to S$ in $\Sch'$ again exists in $\Sch'$ (for example, we may suppose that $\Sch'$ is an ``adequate category of schemes" as defined in \cite{cisinski-deglise}*{2.0.1}).

\begin{definition}[\cite{khan-thesis}*{Chapter~2, Definition~3.1.2}] A \df{premotivic category of coefficients} (or {\bf premotivic functor}) is a functor\footnote{Some authors do not suppose presentability hypotheses.}
\[
\D^*: (\Sch')^{\op} \rightarrow \CAlg(\PrL_{\infty})
\]
that satisfies the following axioms:
\begin{enumerate}
\item For any morphism $f: T \rightarrow S$, we
have an adjunction\footnote{Of course, this axiom is already implicit since $\D^*$ is valued in $\PrL_{\infty}$.}
\[
f^*:\D^*(S) \rightleftarrows \D^*(T):f_*.
\]
\item For any smooth morphism $f: T \rightarrow S$, we have an adjunction
\[
f_{\sharp}:\D^*(T) \rightleftarrows \D^*(S):f^*.
\]
\item ($\sharp$-base change)  For a cartesian square
\begin{equation} \label{cart}
\begin{tikzcd}
T' \ar{r}{g'} \ar{d}{f'} & T \ar{d}{f} \\
S' \ar{r}{g} & S
\end{tikzcd}
\end{equation}
in which $f$ is smooth, the exchange transformation 
\begin{equation} \label{eq:sharp-star}
f_{\sharp}g^* \rightarrow g'^{*}f'_{\sharp}
\end{equation} is an equivalence.
\item ($\sharp$-projection formula) For any smooth morphism $f: T \to S$ and objects $X \in \D^*(T), Y \in \D^*(S)$, the canonical map 
\begin{equation} \label{eq:proj-formula}
f_{\sharp}(X \otimes f^*Y) \rightarrow f_{\sharp}X \otimes Y
\end{equation}
is an equivalence.
\end{enumerate}
We say that $\D^*$ is {\bf stable} if $\D^*$ is valued in $\CAlg(\PrL_{\infty,\stab})$.
\end{definition}
% \cite{khan-thesis}*{2.2.5}
% We say $f^{\ast}$ is left-adjointable along a smooth morphism \cite{khan-thesis}*{2.2}.

\begin{remark} \label{rem:f^*} We note that the functor $f^*$ is always strong monoidal. Therefore, for any morphism $f: T \rightarrow S$ in $\Sch'$, we have that $\D^*(T)$ is endowed with a canonical $\D^*(S)$-algebra structure, from which the natural transformation~\eqref{eq:proj-formula} is obtained \cite{khan-thesis}*{Chapter 0, Lemma 2.7.7}.
\end{remark}

We next introduce three properties that can be asked of a premotivic functor.

\begin{definition} \label{def:MotivicFunctorProperties} Suppose $\D^*: (\Sch')^{\op} \to \CAlg(\PrL_{\infty,\stab})$ is a stable\footnote{The stability assumption simplifies our discussion of the localization property -- see \cite{khan-thesis}*{3.3.10}. Note that even if we do not assume $D^*$ is stable to begin with, it follows as a corollary of the three properties \cite{khan-thesis}*{Corollary~3.4.20}.} premotivic functor.
\begin{enumerate} \item Let $f:\Escr \rightarrow T$ be a vector bundle in $\Sch'$. Then $\D^*$ satisfies \df{homotopy invariance with respect to $\Escr$} if the functor $$f^*:\D^*(T) \rightarrow \D^*(\Escr)$$ is fully faithful. We say that $\D^*$ satisifies \df{homotopy invariance} if it satisfies homotopy invariance with respect to all vector bundles in $\Sch'$. 
\item Let $i:Z \hookrightarrow T$ be a closed immersion in $\Sch'$ with open complement $j: U \hookrightarrow T$. The counit and unit of the adjunctions $j_{\sharp} \dashv j^*$ and $i^* \dashv i_*$ give a sequence of natural transformations
\begin{equation} \label{ij}
j_{\sharp}j^* \rightarrow \id \rightarrow i_*i^*.
\end{equation}
We say that $\D^*$ satisfies \df{localization with respect to $i$} if $\D^*(\emptyset)$ is the trivial category, $i_{\ast}$ is fully faithful, and \eqref{ij} is a cofiber sequence (or equivalently, $(i^*, j^*)$ are jointly conservative \cite{khan-thesis}*{Lemma 3.3.11}).\footnote{It is a consequence of $\sharp$-base change that $j^{\ast} i_{\ast} \simeq 0$ and that  $j_{\sharp}$ is fully faithful. Therefore $j_{\ast}$ is fully faithful and $(\D^*(U), \D^*(Z))$ defines a recollement of $\D^*(T)$. This holds once we know that $\D^*$ satisfies the full six functors formalism.} We say that $\D^*$ satisfies \df{localization} if it satisfies localization with respect to all closed immersions in $\Sch'$.
\end{enumerate}
If $\D^*$ satisfies localization, then for any closed immersion $g: S' \hookrightarrow S$ we have an adjunction \cite{khan-thesis}*{Lemma 3.3.13}
\[
g_*:\D^*(S') \rightleftarrows \D^*(S): g^!.
\]
In the notation of~\eqref{cart}, we then have the exchange transformation
\[
g^!f_* \rightarrow g'_*f'^{!},
\]
which is necessarily an equivalence.
\begin{enumerate} \setcounter{enumi}{2}
\item Let $p: \Escr \rightarrow T$ be a vector bundle in $\Sch'$ and let $s: T \to \Escr$ be its zero section. Then we have an adjunction
\[
p_{\sharp}s_*: \D^*(T) \rightleftarrows \D^*(T): s^!p^*,
\]
and we define the \df{Thom transformation} associated to $p$ to be
\[
\Th(\Escr) = p_{\sharp}s_*, \qquad \Th(\Escr)^{-1} = s^!p^*.
\]
We say that $\D^*$ has \df{Thom stability with respect to $\Escr$} if the transformation $\Th(\Escr)$ is an equivalence, and $\D^*$ has \df{Thom stability} if it has Thom stability with respect to all vector bundles in $\Sch'$.
\end{enumerate}
\end{definition}

The next theorem is due originally to Ayoub \cite{ayoub}*{Scholie~1.4.2} and was expanded upon by Cisinski-D\'eglise \cite{cisinski-deglise}*{Theorem~2.4.50}. We follow the version found in \cite{khan-thesis}*{Chapter~2, Theorem~3.5.4 and Theorem~4.2.2}.

\begin{theorem} \label{thm:six} [Ayoub, Cisinski-D\'eglise] Let $\Sch'$ be the category of noetherian schemes of finite dimension or, more generally, an adequate category of schemes as in \cite{cisinski-deglise}*{2.0.1}. Suppose $\D^*$ is a premotivic functor that satisfies homotopy invariance, localization, and Thom stability. Then $\D^*$ satisfies the full six functors formalism as in \cite{khan-thesis}*{Chapter~2, Theorem 4.2.2}.
\end{theorem}

In particular, given a separated morphism $f: T \rightarrow S$ in $\Sch'$, we have a transformation
\[
f_! \rightarrow f_*
\]
that is an equivalence whenever $f$ is a proper morphism. The functor $f_!$ participates in an adjunction
\[
f_!: \D^*(T) \rightleftarrows \D^*(S): f^!.
\]
If $f$ is moreover smooth, then we have the {\bf purity equivalences}
\begin{equation} \label{eq:purity}
f_{\sharp} \stackrel{\simeq}{\rightarrow} f_!\Th(\Omega_f), \qquad f^* \stackrel{\simeq}{\rightarrow} \Th(\Omega_f)^{-1}f^!.
\end{equation}

\begin{remark} \label{rem:hoyois-glv} Using the results in \cite{hoyois-glv}*{Appendix C}, one may eliminate the noetherian hypotheses appearing in Theorem~\ref{thm:six}; see, in particular, \cite{hoyois-glv}*{Remark C.14}. We also refer to \cite{hoyois-sixops} for a construction of the six functors formalism in optimal generality, within the more general setting of quotient stacks. We leave it to the reader to formulate the results of this paper in the non-noetherian setting.
\end{remark}

\subsection{Recollement of six functors formalisms} \label{sect:RecollementSixFunctors}
% \begin{definition} Let $\D^*$ and $\C^*$ be two premotivic functors. A {\bf (lax) monoidal morphism of premotivic functors} is a natural transformation $\phi: \D^* \Rightarrow \C^*$ that is objectwise (lax) symmetric monoidal.
% \end{definition}
Suppose $\D^*$ and $\C^*$ are two premotivic functors and let $\phi: \D^* \Rightarrow \C^*$ be a natural transformation valued in $\widehat{\Cat}_{\infty}$. Then for each $T \in \Sch'$, we have a functor $\phi_T: \D^*(T) \rightarrow \C^*(T)$, and for every morphism $f: T \rightarrow S$ in $\Sch'$, we have a canonical equivalence $f^*\phi_S \simeq \phi_Tf^*$. By adjunction, we obtain exchange transformations
\begin{equation} \label{eq:mate1}
\phi_Sf_* \rightarrow f_*\phi_T
\end{equation}
for all $f$, and also
\begin{equation} \label{eq:mate2}
f_{\sharp}\phi_T \rightarrow \phi_S f_{\sharp}
\end{equation}
whenever $f$ is smooth. We further say that $\phi$ is {\bf (lax) symmetric monoidal} if we have the data of a lift of $\phi$ to a functor $(\Sch')^{\op} \times \Delta^1 \to \CAlg_{(\mathrm{lax})}(\widehat{\Cat}_{\infty})$ that restricts to the given functors $\D^*$ and $\C^*$.
% The next theorem asserts that we can take right-lax limits of lax monoidal morphisms of premotivic functors.
% \footnote{If $\phi_T$ satisfies some property (P) or has structure (S) for all $T \in \Sch'$, then we say the same for $\phi$ itself.}

\begin{theorem} \label{thm:abstract} Suppose $\phi: \D^* \Rightarrow \C^*$ is left-exact, accessible, and lax symmetric monoidal. Then its right-lax limit $\D^* \overrightarrow{\times} \C^*$ is canonically a premotivic functor that comes equipped with adjoint natural transformations
\[ j^*: \D^* \overrightarrow{\times} \C^* \rightleftarrows \D^*: j_*, \qquad i^*:\D^* \overrightarrow{\times} \C^* \rightleftarrows \C^*: i_* \]
in which $j^*, i^*$ are strong symmetric monoidal and $j_*, i_*$ are lax symmetric monoidal (and the naturality assertion for $i_*$ only holds if the pullback functors of $\D^*$ preserve the terminal object). Moreover, if $\D^{\ast}$ and $\C^{\ast}$ are stable, then $\D^* \overrightarrow{\times} \C^*$ is stable.
% such that the pair $(i^*, j^*)$ is jointly conservative.
\end{theorem}

\begin{remark} \label{rem:sharp} Most authors \cite{cisinski-deglise}, \cite{framed-loc}*{Remark 18} insist that a morphism $\phi$ of premotivic functors commutes with $f_{\sharp}$ for $f$ smooth. We note that we do not assume this in our definition --- nonetheless, the functor $f_{\sharp}$ is still defined on the recollement for $f$ smooth and is left adjoint to $f^*$. We also note that $\phi$ does not in general commute with $f_*$, and thus $f_*$ does not in general commute with $i^*: \D^* \overrightarrow{\times} \C^* \Rightarrow \C^*$ (but does with $j^*: \D^* \overrightarrow{\times} \C^* \Rightarrow \D^*$).
\end{remark}

To prove Theorem~\ref{thm:abstract}, we first record a few general facts concerning the functoriality of right-lax limits.

\begin{remark} \label{rem:rlaxlim-param} Let $B$ be a $\infty$-category, $\D^*, \C^*: B^{\op} \to \widehat{\Cat}_{\infty}$ be two functors, and $\phi: \D^* \Rightarrow \C^*$ a left-exact natural transformation. It is often useful to consider the right-lax limit $\D^* \overrightarrow{\times} \C^*$ from the ``parametrized'' point of view. Namely, let $\D = \int \D^*$, $\C = \int \C^*$ be the corresponding cartesian fibrations over $B$, and let $\phi = \int \phi$ (abusing notation). Recall that given any (co)cartesian fibration $X \to B$, we have the {\bf fiberwise arrow category} $X^{\Delta^1} \to B$ defined as the cotensor of $X$ with $\Delta^1$ in (marked) simplicial sets over $B$; this is again a (co)cartesian fibration whose fiber over $b \in B$ is isomorphic as a simplicial set to $\Fun(\Delta^1,X_b)$. Let $$\D \overrightarrow{\times} \C = \D \times_{\phi, \C, \ev_1} \C^{\Delta^1}$$ be the fiberwise right-lax limit of $\phi: \D \to \C$ as a map over $B$. Using that the unstraightening functor commutes with pullbacks and cotensors, we then have an equivalence
\[ \int(\D^* \overrightarrow{\times} \C^*) \simeq \D \overrightarrow{\times} \C. \]
We also have a similar result if we take cocartesian fibrations over $B^{\op}$ instead.
\end{remark}

We have the following pair of lemmas that give componentwise formulas for $f_{\sharp}$ resp. $f_{\ast}$ on the right-lax limit. Note that although the two statements could be combined, we separate them out so as to clarify the situation with adjoints $f_{\sharp} \dashv f^* \dashv f_*$.

\begin{lemma} \label{lem:lax-lim} Let $B$ be an $\infty$-category and suppose that $\C, \D$ are cocartesian fibrations over $B$ and $\phi:\D \rightarrow \C$ is a functor over $B$ (not necessarily preserving cocartesian edges). Then the fiberwise right-lax limit $\D \overrightarrow{\times} \C$ of $\phi$ is again a cocartesian fibration, and given a morphism $f: S \to T$ in $B$, the induced functor $f_{\sharp}$ is computed by the formula
\[ f_{\sharp}(X,Y,Y \to \phi X) \simeq (f_{\sharp} X, f_{\sharp} Y, f_{\sharp} Y \to f_{\sharp} \phi X \to \phi f_{\sharp} X). \]
In particular, if $\C, \D$ are also cartesian fibrations over $B$ and $\phi$ preserves cartesian edges, then $\D \overrightarrow{\times} \C \to B$ is a bicartesian fibration with left adjoints as described.
% In order to construct and understand the required adjoints, we use the description of adjunctions in terms of bicartesian fibrations over $\Delta^1$ \cite{htt}*{Definition 5.2.2.1}.
\end{lemma}
\begin{proof} The functor $\ev_1: \C^{\Delta^1} \rightarrow \C$ is clearly a $B$-cocartesian fibration in the sense of \cite{jay-thesis}*{Definition~7.1}, hence is itself a cocartesian fibration \cite{jay-thesis}*{Remark~7.3}. Thus, the functor $\D \overrightarrow{\times} \C \rightarrow \D$ obtained via pullback is a cocartesian fibration. Composing this by the structure map of $\D$, we conclude that $\D \overrightarrow{\times} \C \rightarrow B$ is a cocartesian fibration. The desired formula is then easily seen upon unwinding the definitions.
\end{proof}

\begin{lemma}  \label{lem:lax-lim2} Let $B^{\op}$ be an $\infty$-category, let $\C, \D$ be bicartesian fibrations over $B^{\op}$, and let $\phi:\D \rightarrow \C$ be a functor over $B^{\op}$ (not necessarily preserving cocartesian edges). Suppose that $\D$ admits fiberwise pullbacks. Then the fiberwise right-lax limit $\D \overrightarrow{\times} \C$ of $\phi$ is a bicartesian fibration over $B^{\op}$, and given a morphism $f: S \to T$ in $B$, the induced functor $f_{\ast}$ (encoded by the cartesian functoriality) is computed by 
\[ f_{\ast}(X, Y, Y \rightarrow \phi X) \simeq (f_*X, f_*Y \times_{f_*\phi X} \phi f_*X, f_*Y \times_{f_*\phi X} \phi f_*X \rightarrow \phi f_*X). \]
\end{lemma}
% so that applying the $f_*$ to $(X, Y, Y \rightarrow \phi X)$ gives the object $(f_*X, f_*Y \times_{f_*\phi X} \phi f_*X, f_*Y \times_{f_*\phi X} \phi f_*X \rightarrow \phi f_*X)$ over $A$.
\begin{proof} By Lemma.~\ref{lem:lax-lim} with the base taken to be $B^{\op}$, we have that $\D \overrightarrow{\times} \C \to B^{\op}$ is a cocartesian fibration such that for a morphism $f: S \to T$ in $B$, the induced functor $f^*$ is computed by
\[ f^*(X,Y,Y \to \phi X) \simeq (f^* X, f^* Y, f^*Y \to f^* \phi X \to \phi f^* X).
\] 

It remains to check is that $\D \overrightarrow{\times} \C$ is a cartesian fibration over $B^{\op}$. Suppose that $f:S \rightarrow T$ is a morphism in $B$ and we have an object $Z = (X, Y, Y \rightarrow \phi X) \in (\D \overrightarrow{\times} \C)_S$. On $\D$ and $\C$, denote the cocartesian functoriality by $f^*$ and the cartesian functoriality by $f_*$. Let
$$ f_{\ast} (Z) = (f_*X, f_*Y \times_{f_*\phi X} \phi f_*X, f_*Y \times_{f_*\phi X} \phi f_*X \rightarrow \phi f_*X). $$
To specify the morphism $\gamma: f_{\ast} Z \to Z$ lifting $f$, we may equivalently define the morphism $\epsilon: f^{\ast} f_{\ast} Z \to Z$ in the fiber over $S$ (which will be the counit of the adjunction $f^{\ast} \dashv f_{\ast}$). To this end, we take $\epsilon$ to be
\[
(f^*f_*X \rightarrow X, f^*(f_*Y \times_{f_*\phi X} \phi f_*X) \rightarrow f^*f_*Y \rightarrow Y, \square)
\]
where $\square$ is the diagram
\[
\begin{tikzcd}
f^*(f_*Y \times_{f_*\phi X} \phi f_*X) \ar{r} \ar{d}   & f^* \phi f_*(X) \ar{r} \ar{d} & \phi f^*f_*(X)\ar{dd}\\
f^*f_*Y  \ar{r} \ar{d} & f^*f_*\phi(X) \ar{d} & \\
Y \ar{r} & \phi(X) \ar{r}{=} & \phi(X).
\end{tikzcd}
\]
To complete the proof, it now suffices to check that $\gamma$ is a \emph{locally} cartesian edge over $f$, since we already know that $\D \overrightarrow{\times} \C$ is a cocartesian fibration over $B^{\op}$. This follows by a simple diagram chase.
\end{proof}

\begin{corollary} \label{cor:FullyFaithfulAdjointsRecollement} Suppose we are in the situation of Lemma~\ref{lem:lax-lim2} and let $f: S \to T$ be a morphism in $B$.
\begin{enumerate} \item If $f_{\ast}: \D_S \to \D_T$ and $f_{\ast}: \C_S \to \C_T$ are fully faithful, then $f_{\ast}: (\D \overrightarrow{\times} \C)_S \to (\D \overrightarrow{\times} \C)_T$ is fully faithful.
\item If $f^{\ast}: \D_T \to \D_S$ and $f^{\ast}: \C_T \to \C_S$ are fully faithful, then $f^{\ast}: (\D \overrightarrow{\times} \C)_T \to (\D \overrightarrow{\times} \C)_S$ is fully faithful.
\end{enumerate}
\end{corollary}
\begin{proof} This follows immediately from Lemma~\ref{lem:lax-lim2} in view of the componentwise formulas for $f^{\ast}$ and $f_{\ast}$.
\end{proof}

\begin{proof}[Proof of Theorem \ref{thm:abstract}] \label{proof:thm.abstract} Let $\rlax.\lim: \Fun(\Delta^1,\widehat{\Cat}_{\infty}) \rightarrow \widehat{\Cat}_{\infty}$ denote the right-lax limit functor. First note that $\rlax.\lim(F)$ of an accessible left-exact functor $F: \A \to \B$ of presentable $\infty$-categories is again presentable \cite{quigley-shah}*{Corollary~1.36}, and $\rlax.\lim(\gamma)$ of a natural transformation $\gamma: F \Rightarrow F'$ through colimit-preserving functors again preserves colimits, using that $\rlax.\lim(\gamma)$ is a morphism of recollements \cite{quigley-shah}*{1.7} and we have on the target $\rlax.\lim(F')$ the jointly conservative colimit-preserving functors $(i'^{\ast}, j'^{\ast})$. Also, it is clear that $\rlax.\lim$ of an exact functor of stable $\infty$-categories is again stable, and $\rlax.\lim$ of an exact natural transformation thereof is an exact functor.

% In fact, the reference constructs this functor at the level of marked simplicial sets over $\Fin_*$, and we may then pass to the underlying $\infty$-categories.
As for the symmetric monoidal structure, in \cite{quigley-shah}*{1.28} we constructed a lift of $\rlax.\lim$ to a functor\footnote{We could also use the construction described before Theorem~\ref{thm:MonoidalityOfGenuineStabilizationFunctor} (note that the cited reference only establishes functoriality through strictly commutative squares, but this is all that we need here).}
\[ \rlax.\lim: \Fun(\Delta^1,\CAlg_{\lax}(\widehat{\Cat}_{\infty})) \to \CAlg_{\lax}(\widehat{\Cat}_{\infty}), \]
where we endow $\rlax.\lim(F)$ with its canonical symmetric monoidal structure (recalled in \S\ref{sec:symmon}). Moreover, $\rlax.\lim$ sends natural transformations through strong symmetric monoidal functors to strong symmetric monoidal functors, and the tensor product on $\rlax.\lim(F)$ commutes with colimits separately in each variable if the same is true for the domain and codomain of $F$ (Remark~\ref{rem:PresentablySMC}).

Combining the above facts, we thus obtain a functor $$\D^* \overrightarrow{\times} \C^*: (\Sch')^{\op} \to \CAlg(\PrL_{\infty}), \qquad T \mapsto \D^*(T) \overrightarrow{\times} \C^*(T),$$
which lands in $\CAlg(\PrL_{\infty,\stab})$ if $\D^*, \C^*$ are stable. Furthermore, if we write objects of $\D^*(T) \overrightarrow{\times} \C^*(T)$ as tuples $(X, Y, Y \rightarrow \phi_T X)$ for $X \in \D^*(T)$ and $Y \in \C^*(T)$, then for a morphism $f: T \rightarrow S$ in $\Sch'$, we may express the pullback functor $f^{\ast}$ as given by
\begin{equation} \label{eq:pull2}
f^{\ast}(X, Y, Y \rightarrow \phi_S X) \simeq (f^*X, f^*Y, f^*Y \rightarrow f^*\phi_S X \simeq \phi_Tf^*X),
\end{equation} 
whereas its right adjoint $f_{\ast}$ is given by
\begin{equation} \label{eq:push}
f_*(X, Y, Y \rightarrow \phi_T X) \simeq (f_*X, f_*Y \times_{f_*\phi_T X} \phi_S f_*X, f_*Y \times_{f_*\phi_T X} \phi_S f_*X \rightarrow \phi_S f_*X)
\end{equation}
in view of Remark~\ref{rem:rlaxlim-param} and Lemma~\ref{lem:lax-lim2} with $B = \Sch'$, $\C = (\int \C^*)^{\vee}$, $\D = (\int \D^*)^{\vee}$, and $\phi = (\int \phi)^{\vee}$. This reflects the fact that $\rlax.\lim$ carries natural transformations through symmetric monoidal functors to \emph{strict} morphisms of monoidal recollements \cite{quigley-shah}*{1.7 and Lemma~1.23}, and thus $\D^* \overrightarrow{\times} \C^*$ comes equipped with $j^*$, $j_*$, $i^*$ as indicated. Moreover, $i_*$ commutes with $f^*$ if $f^*$ preserves the terminal object \cite{quigley-shah}*{1.3}.

Now suppose $f: T \to S$ is a smooth morphism. By Remark~\ref{rem:rlaxlim-param} and Lemma~\ref{lem:lax-lim} with $B = \Sch'$, $\C = \int \C^*$, $\D = \int \D^*$, and $\phi = \int \phi$, we have an adjunction
\[ f_{\sharp}: \D^*(T) \overrightarrow{\times} \C^*(T) \rightleftarrows \D^*(S) \overrightarrow{\times} \C^*(S): f^*  \]
in which the left adjoint $f_{\sharp}$ is given by the formula
\begin{equation} \label{eq:sharppush}
f_{\sharp} (X, Y, Y \rightarrow \phi_TX) \simeq (f_{\sharp}X, f_{\sharp}Y, f_{\sharp}Y \rightarrow f_{\sharp}\phi_TX \rightarrow  \phi_Sf_{\sharp}X).
\end{equation}

We now prove the $\sharp$-base change formula. Suppose that we have a cartesian square of schemes as in~\eqref{cart} with $f$ smooth. Since we have constructed the appropriate adjoints for $\D^* \overrightarrow{\times} \C^*$, we get an exchange transformation as in~\eqref{eq:sharp-star}. Using the formula for $f^*$~\eqref{eq:pull2} and $f_{\sharp}$~\eqref{eq:sharppush} and the joint conservativity of $(j^*, i^*)$, we get that this transformation is invertible since the corresponding exchange transformations for $\D^*$ and $\C^*$ are. The argument for the $\sharp$-projection formula is similar, using also that $j^*, i^*$ are strong symmetric monoidal.
\end{proof}

Now suppose that $\phi: \D^* \Rightarrow \C^*$ is an exact accessible lax symmetric monoidal morphism of stable premotivic functors, so that $\D^* \overrightarrow{\times} \C^*$ is a stable premotivic functor by Theorem~\ref{thm:abstract}. We next show that, under certain hypotheses, $\D^* \overrightarrow{\times} \C^*$ inherits the properties enumerated in Definition~\ref{def:MotivicFunctorProperties} from $\D^*$ and $\C^*$. 

\begin{lemma} \label{lem:PropertyHomotopyInvariance} Let $\Escr \rightarrow T$ be a vector bundle and suppose that $\D^*$ and $\C^*$ satisfy homotopy invariance with respect to $\Escr$. Then $\D^* \overrightarrow{\times} \C^*$ satisfies homotopy invariance with respect to $\Escr$. Therefore, if $\D^*$ and $\C^*$ satisfy homotopy invariance, then $\D^* \overrightarrow{\times} \C^*$ satisfies homotopy invariance.
\end{lemma}
\begin{proof} This follows immediately from Corollary~\ref{cor:FullyFaithfulAdjointsRecollement}.2.
\end{proof}

\begin{lemma} \label{lem:PropertyLocalization} Let $f: Z \hookrightarrow T$ be a closed immersion and suppose that $\D^*$ and $\C^*$ satisfy localization with respect to $f$. Then:
\begin{enumerate} \item The exchange transformation
\[ \chi: \phi_T f_* \to f_* \phi_Z \]
is an equivalence.
\item $\D^* \overrightarrow{\times} \C^*$ satisfies localization with respect to $f$. 
\end{enumerate}
Therefore, if $\D^*$ and $\C^*$ satisfy localization, then $\D^* \overrightarrow{\times} \C^*$ satisfies localization.
\end{lemma}
\begin{proof} Let $g: U \hookrightarrow T$ be the open immersion complementary to $f$. For (1), by the joint conservativity of $(g^*,f^*)$ it suffices to check that $g^* \chi$ and $f^* \chi$ are equivalences. For the former, using that $g^* f_* \simeq 0$, we have
\[ g^* \chi \simeq 0: g^* \phi_T f_* \simeq \phi_U g^* f_* \simeq 0 \to g^* f_* \phi_Z \simeq 0. \]
For the latter, using that $f_*$ is fully faithful (so $f^* f_* \simeq \id$), we have 
\[ f^* \chi \simeq \id: f^* \phi_T f_* \simeq \phi_Z f^* f_* \simeq \phi_Z \to f^* f_* \phi_Z \simeq \phi_Z. \]
For (2), we verify the three defining conditions in turn. It is obvious that $\D^* \overrightarrow{\times} \C^* (\emptyset) \simeq \ast$. The full faithfulness of $f_*$ follows from Corollary~\ref{cor:FullyFaithfulAdjointsRecollement}.1. Finally, it remains to check that the sequence
\[ g_{\sharp}g^* \rightarrow \id \rightarrow f_*f^* \]
is a cofiber sequence, for which we wish to apply the joint conservativity of $(j^*, i^*)$ to reduce to the known cofiber sequences on $\D^*(T)$ and $\C^*(T)$. But we have generically that both $j^*$ and $i^*$ commute with $g_{\sharp}$, $g^*$, $f^*$, and $j^*$ commutes with $f_*$, whereas (1) implies that $i^*$ commutes with $f_*$. The claim then follows.
\end{proof}

The most subtle property is that of Thom stability. Let us begin with a warning:

\begin{warning} \label{warn:invertible} If a premotivic functor $\D^*$ satisfies the full six functors formalism, then we have the purity equivalences~\eqref{eq:purity}. In particular, if $S$ is a scheme and $p:X \rightarrow S$ is a smooth proper morphism, then we have the {\bf ambidexterity equivalence} \cite{hoyois-sixops}*{Theorem 6.9} $$p_{\sharp}\Sigma^{-\Omega_p} \simeq p_*$$ that leads to {\bf Atiyah duality}: the dual of $p_!p^!\1_S$ is $p_{\sharp}\Sigma^{-\Omega_p}\1_X$ \cite{hoyois-sixops}*{Corollary 6.13}. Hence, any premotivic functor satisfying the full six functors formalism has a large supply of dualizable objects whose duals are computed by a shift of the conormal bundle. We emphasize that Thom stability is essential here; indeed, one needs Thom stability to even formulate the negative shifts occuring in the statement of Atiyah duality.

In light of this, we can produce explicit counterexamples to the naive expectation that if $\D^*$ and $\C^*$ both satisfy the full six functors formalism, then the right-lax limit of $\phi: \D^* \Rightarrow \C^*$ again satisfies the full six functors formalism. Indeed, we can use the gluing functor $\Theta$ from~\eqref{eq:small-theta} to construct a premotivic functor that assigns to $\Spec \RR$ the $\infty$-category of naive $C_2$-spectra $\Fun(\Oscr^{\op}_{C_2}, \Sp)$. As we have already mentioned, Atiyah duality does not hold in this $\infty$-category; for instance, if $V$ is the regular real $C_2$-representation, then $\Sigma^{\infty}_+ S^V$ fails to be dualizable in $\Fun(\Oscr^{\op}_{C_2}, \Sp)$.
\end{warning}

The next lemma gives a sufficient condition for Thom stability to be inherited by the right-lax limit. Assume that for any $X \in \Sch'$, we have that $\GG_m \times X \in \Sch'$. For a stable premotivic functor $\D^*$ and any scheme $S \in \Sch'$, let $\pi_S:\GG_m \times S \rightarrow S$ denote the projection map and consider the counit $\pi_{S\sharp}\pi_S^*\1 \rightarrow \1$ in $\D^*(S)$. Define the {\bf Tate motive} in $\D^*(S)$ to be
\[
\1_S^{\D^*}(1) = \mathrm{fib}(\pi_{S\sharp}\pi_S^*\1 \rightarrow \1)[-1].
\]
% We formulate an additional condition for a premotivic functor:
\begin{definition} \label{def:enough} We say that $\D^*$ is {\bf Tate-dualizable over $S$} if $\1_S^{\D^*}(1) \in \D^*(S)$ is a dualizable object. If $\D^*$ is Tate-dualizable over $S$ for all $S \in \Sch'$, we say that $\D^*$ is {\bf Tate-dualizable}.
\end{definition}

% Then $\D^* \overrightarrow{\times} \C^*$ has Thom stability with respect to $\Escr$.
% Now for any premotivic functor $\D^*$ with localization,
%Since $\D^* \overrightarrow{\times} \C^*$ is furthermore symmetric monoidal by the proof of Theorem~\ref{thm:abstract}, it suffices to verify that $\1_S^{\D^* \overrightarrow{\times} \C^*}(1)$ is invertible. 
\begin{lemma} \label{lem:PropertyThomStability} Suppose $\D^*$ and $\C^*$ satisfy localization and assume that $\D^* \overrightarrow{\times} \C^*$ is Tate-dualizable. Then if $\D^*$ and $\C^*$ have Thom stability, $\D^* \overrightarrow{\times} \C^*$ has Thom stability.
\end{lemma}
\begin{proof} By Lemma~\ref{lem:PropertyLocalization}, $\D^* \overrightarrow{\times} \C^*$ has localization, so the lemma is well-posed. By \cite{cisinski-deglise}*{Corollary 2.4.14},\footnote{For the reference, we note that localization implies (wLoc) and (Zar-sep).} Thom stability holds for $\D^* \overrightarrow{\times} \C^*$ if and only if for any scheme $S$, the Tate motive $\1_S^{\D^* \overrightarrow{\times} \C^*}(1)$ is invertible. Since we have assumed that $\D^* \overrightarrow{\times} \C^*$ is Tate-dualizable and $j^*, i^*$ are jointly conservative and symmetric monoidal, it suffices to prove that the Tate motive is invertible after applying $j^*$ and $i^*$. But since $j^*,i^*$ commute with $\pi_{S\sharp}$ by Lemma~\ref{lem:lax-lim}, we have
\[ i^*\1_S^{\D^* \overrightarrow{\times} \C^*}(1) \simeq \1_S^{\D^*}(1), \qquad j^*\1_S^{\D^* \overrightarrow{\times} \C^*}(1) \simeq \1_S^{\D^*}(1), \]
and these are invertible objects by assumption (and appealing once more to \cite{cisinski-deglise}*{Corollary 2.4.14}).
\end{proof}

%\begin{lemma} \label{lem:PropertyThomStability} Suppose $\D^*$ and $\C^*$ satisfy localization. Let $\Escr \to T$ be a vector bundle and suppose that $\D^*$ and $\C^*$ have Thom stability with respect to $\Escr$. Then $\D^* \overrightarrow{\times} \C^*$ has Thom stability with respect to $\Escr$. Therefore, if $\D^*$ and $\C^*$ have Thom stability, then $\D^* \overrightarrow{\times} \C^*$ has Thom stability.
%\end{lemma}
%\begin{proof} By Lemma~\ref{lem:PropertyLocalization}, $\D^* \overrightarrow{\times} \C^*$ has localization, so the lemma is well-posed. The claim then follows by the same type of joint conservativity argument with $(j^*,i^*)$ as was used in the proof of Lemma~\ref{lem:PropertyLocalization}.
%\end{proof}

We may now invoke Theorem~\ref{thm:six} to conclude:

\begin{theorem} \label{thm:abstract2} Let $\Sch'$ be an adequate category of schemes, and let $\D^*, \C^*$ be stable premotivic functors, each of which satisfy the full six functors formalism (or equivalently, homotopy invariance, localization, and Thom stability). Let $\phi: \D^* \Rightarrow \C^*$ be an exact accessible lax symmetric monoidal morphism. Then its right-lax limit $\D^* \overrightarrow{\times} \C^*$ is canonically a stable premotivic functor that satisfies homotopy invariance and localization. Moreover, if $\D^* \overrightarrow{\times} \C^*$ is Tate-dualizable, then $\D^* \overrightarrow{\times} \C^*$ also satisfies Thom stability and thus has the full six functors formalism. 

In addition, $\D^* \overrightarrow{\times} \C^*$ comes equipped with symmetric monoidal morphisms
\[
j^*: \D^* \overrightarrow{\times} \C^* \Rightarrow \D^*, \qquad i^*:\D^* \overrightarrow{\times} \C^* \Rightarrow \C^*
\]
and lax symmetric monoidal morphisms
\[
j_*: \D^* \Rightarrow \D^* \overrightarrow{\times} \C^*, \qquad i_*: \C^* \Rightarrow \D^* \overrightarrow{\times} \C^*
\]
that are objectwise the recollement adjunctions.
\end{theorem}

\subsection{Genuine stabilization and the six functors formalism} \label{genuine-six}

% By the assumption that $\tfrac{1}{2} \in \Oscr_X$, the map $\pi:X[i] \rightarrow X$ is a Galois cover, whence we can compute the homotopy orbits topos of $\widetilde{X[i]}_{\et}$. Indeed, it is computed as the limit in $\widehat{\Cat}_{\infty}$:
% \[
% (\widetilde{X[i]}_{\et})_{\hh C_2} \simeq \lim_{BC_2} \widetilde{X[i]}_{\et} \simeq  \widetilde{X}_{\et},
% \]
% where the last equivalence follows from Galois descent. On the other hand, the homotopy fixed topos of $\widetilde{X[i]}_{\et}$ is, via Scheiderer's result recalled in Example~\ref{ex:involute}, canonically equivalent to real \'etale sheaves on $X$, i.e., we have an equivalence
% \[
% \widetilde{X[i]}_{\et}^{\hh C_2} \simeq \widetilde{X}_{\ret}.
% \]
Let $X$ be a scheme with $\tfrac{1}{2} \in \Oscr_X$, and consider the $C_2$-$\infty$-topos $\Xscr = \widetilde{X[i]}_{\et}$ with $C_2$-action as induced by the involution $\sigma: X[i] \rightarrow X[i]$.

\begin{definition} \label{dfn:OfficialBSheavesWithTransfers} The $\infty$-category of {\bf (hypercomplete) $b$-sheaves of spectra with transfers} on $X$ is $$\Sp^{C_2}(\widetilde{X}_b) = \Sp^{C_2}(\Xscr), \quad \text{resp.} \quad \Sp^{C_2}_b(X) = \Sp^{C_2}(\widehat{X}_b) = \Sp^{C_2}(\widehat{\Xscr}).$$
\end{definition}

\begin{remark} \label{rem:SameBaseChangeObviousRemark} By Example~\ref{ex:b}, we have that $\widetilde{X}_b \simeq \Xscr_{C_2}$. Moreover, using the $C_2$-Galois cover $\pi: X[i] \to X$ and the involution $\sigma$, we may naturally extend $\widetilde{X}_b$ to a $C_2$-$\infty$-category $$\underline{\widetilde{X}}_b = \left( \widetilde{X}_b \xrightarrow{\pi^*} \widetilde{X[i]}_b \simeq \widetilde{X[i]}_{\et} \right),$$ such that $\underline{\widetilde{X}}_b \simeq \underline{\Xscr}_{C_2}$. In this light, Definition~\ref{dfn:OfficialBSheavesWithTransfers} is thus sensible.

% By Examples~\ref{exmp:et-cover} and \ref{ex:ret},
Furthermore, the stable gluing functor of Definition~\ref{def:stable-glue} has the form
\[
\Theta^{\Tate}_X: \Sp(\Xscr_{\hh C_2}) \simeq  \Sp(\widetilde{X}_{\et}) \rightarrow  \Sp(\Xscr^{\hh C_2}) \simeq \Sp(\widetilde{X}_{\ret}),
\]
and by Theorem~\ref{prop:recoll-stab}, the right-lax limit of $\Theta^{\Tate}_X$ is equivalent to $\Sp^{C_2}(\widetilde{X}_{b})$.
\end{remark}

% This is the $\infty$-category of $b$-sheaves of spectra with transfers as introduced in Definition~\ref{def:b-shvtr} and we write 
% \[
% \Sp_b^{C_2}(X):= \Sp^{C_2}(\widetilde{X[i]}_{\et}).
% \] 

\begin{remark} \label{rem:mackey} We feel justified in referring to objects of $\Sp^{C_2}(\widetilde{X}_b)$ as ``$b$-sheaves of spectra with transfers'' in view of the Mackey description of the $C_2$-stabilization. Roughly speaking, to specify an object of $\Sp^{C_2}(\widetilde{X}_b)$ is to give a $b$-sheaf of spectra on $\Et_X$ together with a single transfer map along $\pi: X[i] \rightarrow X$, subject to certain compatibilities; see \S\ref{b-rigid} for more details.
\end{remark}

% Recalling the details of work in \S\ref{sect:gen-stab}, we note that the assignment of a $G$-$\infty$-topos to its genuine stabilization is functorial in the \emph{left} adjoint of a geometric morphism.
We will use results of Bachmann together with those in \S\ref{sect:RecollementSixFunctors} to promote the assignment $X \mapsto \Sp_b^{C_2}(X) \comp$ to a premotivic functor admitting the full six functors formalism (on a suitable category of schemes).

\begin{theorem} \label{prop:sixret} Let $S$ be a noetherian scheme of finite Krull dimension. The functor
\[
\Sp(\widetilde{(-)}_{\ret}):(\Sch_S^{\mathrm{fin.dim},\mathrm{noeth}})^{\op} \rightarrow  \CAlg(\PrL_{\infty,\stab})
\]
is a premotivic functor satisfying the full six functors formalism.
\end{theorem}
\begin{proof} This follows from \cite{bachmann-ret}*{Theorem 35}.
\end{proof}

Let $\Sch_S^{p\mbox{-}\fin}$ denote the subcategory of $\Sch_S$ on the {\bf locally $p$-\'etale finite} $S$-schemes in the sense of \cite{bachmann-et}. These are the $S$-schemes that admit an \'etale cover by {\bf $p$-\'etale finite $S$-schemes}, which in turn are those $S$-schemes $X$ such that for any finite type $X$-scheme $Y$, there exists an $n$ such that for every finitely presented, qcqs \'etale $Y$-scheme $Z$, we have that $\mathrm{cd}_p(Z_{\et}) \leq n$. Let $\Sch_S\pinv$ be the subcategory of $S$-schemes $X$ for which $p$ is invertible in $\Oscr_X$, and let
\[
\Sch^{p\mbox{-}\fin}_S\pinv = \Sch^{p\mbox{-}\fin}_S \cap \Sch_S\pinv.
\]

\begin{theorem} \label{prop:sixet} Let $p$ be a prime and suppose that $S$ is a locally $p$-\'etale finite scheme. Then the functor
\[
\Sp(\widehat{(-)}_{\et})\comp:(\Sch^{p\mbox{-}\fin,\mathrm{noeth}}_{S}\pinv)^{\op} \rightarrow  \CAlg(\PrL_{\infty,\stab})
\]
is a premotivic functor satisfying the full six functors formalism.
\end{theorem}
\begin{proof} It is a theorem of Ayoub that $\SH_{\et}$ satisfies the hypotheses of Theorem~\ref{thm:six} (cf. \cite{bachmann-et}*{Theorem~5.1}). The theorem now follows from \cite{bachmann-et}*{Theorem 6.6}.
\end{proof}

% \begin{lemma} The functor $\Theta: \Sp(\widetilde{X}_{\et}) \rightarrow \Sp(\widetilde{X}_{\ret})$ coincides with the stabilization of the functor:
% \[
% \widetilde{X}_{\et} \xrightarrow{i_{\et}} \widetilde{X}_b \xrightarrow{L_{\ret}} \widetilde{X}_{\ret}.
% \]
% \end{lemma}
% \begin{proof} This is the stabilization of Example~\ref{ex:et-to-ret}.
% \end{proof}

 % resp. $\widehat{\theta}_X$,resp. $\widehat{X[i]}_{\et}$
 % , \qquad \widehat{\theta}_X \simeq L_{\ret} i_{\et}: \widehat{X}_{\et} \to \widehat{X}_{\ret}
Now let $\theta_X$ be the unstable gluing functor for $\widetilde{X[i]}_{\et}$ (Definition~\ref{def:glue}), and recall by Example~\ref{ex:et-to-ret} that $$\theta_X \simeq L_{\ret} i_{\et}: \widetilde{X}_{\et} \to \widetilde{X}_{\ret}.$$

\begin{proposition} \label{prop:base-change} For any morphism of schemes $f: X \rightarrow Y$, the canonical exchange transformation of Proposition~\ref{prp:UnstableRecollementFunctoriality}
\[
\chi: f_{\ret}^* \theta_Y \Rightarrow \theta_X f_{\et}^*
\]
is $\infty$-connective. Therefore, the hypercompletion $\widehat{\chi}$ is an equivalence, or if $X$ is of finite Krull dimension, then $\chi$ is an equivalence.
\end{proposition}

\begin{proof} Let $\Fscr \in \widetilde{Y}_{\et}$ and consider the map
\[
\chi_{\Fscr}: f_{\ret}^* \theta \Fscr \rightarrow \theta f_{\et}^*\Fscr
\]
as a morphism in $\widetilde{X}_{\ret}$. We claim that $\chi_{\Fscr}$ is an equivalence on real \'etale points. To verify this, let $\alpha: \Spec k \rightarrow X$ be a morphism for $k$ a real closed field. Then on the stalk $\alpha$ we have
\[
(\chi_{\Fscr})_{\alpha}: (f^*_{\ret}\theta\Fscr)_{\alpha} \simeq  (\theta\Fscr)_{\alpha \circ f} \simeq  \Gamma(\alpha^*_{\et}f_{\et}^*\Fscr) \simeq (\theta f^*_{\et}\Fscr)_{\alpha},
\]
where the second and third equivalences follow from Lemma~\ref{lem:stalk-compute}.

To conclude, we note that the real \'etale site is bounded (since it is $1$-localic) and coherent (since the schemes in question are quasicompact). By Deligne's completeness theorem \cite{sag}*{Theorem A.4.0.5}, it follows that $\alpha_{\Fscr}$ is $\infty$-connective and thus an equivalence after hypercompletion. For the last statement, we note that if $X$ is of finite Krull dimension, then $\widetilde{X}_{\ret}$ is hypercomplete by Theorem~\ref{thm:hyper}.
\end{proof}

% \begin{remark} If we instead consider the non-hypercomplete variant $\theta_X = L_{\ret} i_{\et}: \widetilde{X}_{\et} \to \widetilde{X}_{\ret}$, then Proposition~\ref{prop:base-change} still applies to show the map $\chi$ is $\infty$-connective. Therefore, if $X$ is of finite Krull dimension, then $\chi$ is an equivalence since $\widetilde{X}_{\ret}$ is hypercomplete by Theorem~\ref{thm:hyper}.
% \end{remark}

Since we are considering the hypercomplete version $\Sp^{C_2}_b(X)$ of $b$-sheaves with transfers, we note that $\Theta^{\Tate}_X$ in the next statement refers to the stable gluing functor with respect to $\Xscr = \widehat{X[i]}_{\et}$.

% such that the map $\chi$ of Proposition~\ref{prop:base-change} is an equivalence
\begin{corollary} \label{cor:tate-stable} Let $f: X \rightarrow Y$ be a morphism of schemes. Then the canonical exchange transformation of Remark~\ref{rem:StableRecollementFunctoriality}
\[
\xi: f^*\Theta^{\Tate}_Y \Rightarrow \Theta^{\Tate}_Xf^*
\]
is an equivalence.
\end{corollary}
\begin{proof} As we saw in Remark~\ref{rem:StableRecollementFunctoriality}, we have a commutative diagram
\[ \begin{tikzcd}[row sep=4ex, column sep=6ex, text height=1.5ex, text depth=0.25ex]
f^* (-)_{h C_2} \nu(C_2)^* \ar{r} \ar{d} & f^* \Theta \ar{d}{\chi'} \ar{r} \ar{r} & f^* \Theta^{\Tate} \ar{d}{\xi} \\
 (-)_{h C_2} \nu(C_2)^* f^* \ar{r} & \Theta f^* \ar{r} & \Theta^{\Tate} f^*
\end{tikzcd} \]
in which the rows are cofiber sequences and $\chi'$ is the exchange transformation induced by stabilizing $\widehat{\chi}$ in Proposition~\ref{prop:base-change}. Since $\widehat{\chi}$ is an equivalence and the unstable $f^*$ is left-exact, we deduce that $\chi'$ is an equivalence. We also note that the lefthand vertical arrow is always an equivalence. Indeed, we may check commutation of the right adjoints, in which case the claim is obvious. It follows that $\xi$ is an equivalence.
\end{proof}

We will also need that $(\Sp^{C_2}_b)\comp$ is Tate-dualizable. To this end, given a premotivic functor $\D^*$, we first recall the functoriality of the assignment 
\[
\left( (p_X:X \rightarrow S) \in \Sm_S \right) \mapsto \left( p_{X\sharp}\1 \in \D^*(S) \right).
\] Suppose that $p_X: X \to S$ and $p_Y: Y \to S$ are two smooth $S$-schemes and $f:X \rightarrow Y$ is an $S$-morphism (which is not necessarily smooth). Then $f$ induces a map in $\D^*(S)$
\[
[f]: p_{X\sharp}\1 \rightarrow p_{Y\sharp}\1,
\]
defined to be the adjoint of the map $\1 \simeq f^*\1 \xrightarrow{f^*\eta} f^*p_Y^* p_{Y\sharp}\1 \simeq p_X^* p_{Y \sharp} \1$. One may check that this construction assembles to a functor
% f: X \rightarrow Y \mapsto p_{X\sharp}\1 \rightarrow p_{Y\sharp}\1.
\[ \Sm_S \rightarrow \D^*(S), \qquad (p_X: X \to S) \mapsto (p_{X \sharp} \1),\]
as a special case of a ``geometric section" in the sense of \cite{cisinski-deglise}*{\S 1.1.34}.

Let $1: S \rightarrow \GG_m \times S$ be the map that classifies the unit $1$, so we obtain the map $[1]: \1 \rightarrow \pi_{S\sharp}\1$ in $\D^*(S)$. The counit map $\epsilon: \pi_{S\sharp}\1 \rightarrow \1$ coincides with the map induced by the projection $\GG_m \times S \rightarrow S$, and thus the composite $\1 \xrightarrow{[1]} \pi_{S\sharp}\1 \xrightarrow{\epsilon} \1$ is homotopic to the identity. Therefore, we get a splitting
\begin{equation} \label{eq:tate-splits}
\pi_{S\sharp}\1 \simeq \1 \oplus  \1^{\D^*}_S(1)[1].
\end{equation}
On the other hand, the map $-1:S \rightarrow \GG_m \times S$ classifying the unit $-1$ induces a map $[-1]: \1 \rightarrow \pi_{S\sharp}\1$. Projecting to $\1^{\D^*}_S(1)[1]$ under the splitting~\eqref{eq:tate-splits} then yields a map
\[
\rho:\1 \rightarrow \1^{\D^*}_S(1)[1]
\]
that we think of as a premotivic avatar of the usual map $\rho$ in $\SH(S)$ (and which recovers it if $\D^* = \SH$).
% If $\D^* = \SH$, then this map coincides with the map usually called $\rho$ in $\SH(S)$.

\begin{lemma} \label{lem:duals} Let $p$ be a prime and suppose that $S$ is a locally $p$-\'etale finite scheme on which $2$ and $p$ are invertible. Then $(\Sp_b^{C_2})\comp$ is Tate-dualizable over $S$.
\end{lemma}
\begin{proof} Consider the $C_2$-Galois cover $p_S: S[i] \rightarrow S$. We also write $\1_S(1)\comp$, $\1_S^{\et}(1)\comp$, and $\1_S^{\ret}(1)$ for $\1_S^{\Sp_b^{C_2}(-)\comp}(1)$, $\1_S^{\Sp(\widehat{-}_{\et})\comp}(1)$, and $\1^{\Sp(\widetilde{-}_{\ret})}_S(1)$, respectively. From the first part of Theorem~\ref{thm:abstract2}, $(\Sp_b^{C_2})\comp$ is a premotivic functor, at least whenever $\Sp(\widehat{-}_{\et})\comp$ is. From the discussion above, we have the sequence of maps in $\Sp^{C_2}_b(S)\comp$
\begin{equation} \label{eq:c-to-gm}
p_{S\sharp}p_S^*(\1_S)\comp \rightarrow (\1_S)\comp \xrightarrow{\rho} \1_S(1)[1] \comp.
\end{equation}
We claim that this is a cofiber sequence. Given this, by Corollary~\ref{cor:Dualizability} (and noting that the adjunction $p_S^* \dashv p_{S*} \simeq p_{S \sharp}$ is the same as the adjunction written there as $\overline{\pi}^* \dashv \overline{\pi}_*$ for $\Xscr = \widehat{S[i]}_{\et}$; see Remark~\ref{rem:SameBaseChangeObviousRemark}), it will then follow that $\1_S(1)[1] \comp$ (and thus $\1_S(1)\comp$) is dualizable. To prove the claim, it suffices to show that~\eqref{eq:c-to-gm} is a cofiber sequence after applying $i^*$ and $j^*$.

Upon applying $i^*$ and using that $i^*$ annihilates induced objects, we get the $p$-completion of the sequence of maps in $\Sp(\widetilde{S}_{\ret})$:
\[ \begin{tikzcd}
0 \ar{r}  & \1_S \ar{r} & \1^{\ret}_S(1)[1].
\end{tikzcd} \]
But the second map is invertible even before $p$-completion, since it corresponds to the motivic $\rho$ under the equivalence~\eqref{eq:x-ret}, and $\rho$ is invertible in $\SH_{\ret}(S) \simeq \SH(S)[\rho^{-1}]$ by Bachmann's theorem. This proves exactness after applying $i^*$.

On the other hand, we claim that~\eqref{eq:c-to-gm} is exact after applying $j^*$. Since all functors in sight commutes with base change and we are working with \'etale sheaves, it suffice to check this after base change along the Galois cover $S[i] \rightarrow S$. Since we have an isomorphism of $S$-schemes $S[i] \times_S S[i] \simeq S[i] \coprod S[i]$, the sequence~\eqref{eq:c-to-gm} then becomes
\[ \begin{tikzcd}
(\1_{S[i]} \oplus \1_{S[i]})\comp \ar{r}{\nabla} & (\1_{S[i]})\comp \ar{r} & (\1^{\et}_{S[i]}(1)[1])\comp,
\end{tikzcd} \]
in $\Sp(\widehat{S[i]}_{\et})\comp$, where $\nabla$ is the fold map.

To proceed, recall that for any scheme $X$ with $\tfrac{1}{p} \in \Oscr_X$ we have Bachmann's ``twisting spectrum" 
\[
\hat{\1}_p(1)_X \in \Sp(\widehat{X}_{\et})\comp.
\] More precisely, this is the object that maps under the change-of-site functor $\nu^*:\Sp(\widehat{X}_{\et})\comp \rightarrow \Sp(\widehat{X}_{\mathrm{pro}\et})\comp$ to the spectrum denoted by $\hat{\1}_p(1)_X$ in \cite{bachmann-et} that satisfies  \cite{bachmann-et}*{Theorem 3.6.1-3} and exists by \cite{bachmann-et}*{Theorem 3.6.4}. Whenever $X$ is also locally $p$-\'etale finite, we have a canonical equivalence $\hat{\1}_p(1)_X \simeq (\1^{\et}_X(1))\comp$ \cite{bachmann-et}*{Proposition 4.5}. If $X$ contains all $p^n$-th roots of unity for all $n$, we furthermore have a canonical equivalence $\hat{\1}_p(1)_{X} \simeq  (\1_{X})\comp$ by \cite{bachmann-et}*{Theorem 3.6}.

To conclude, we note that by \cite{bachmann-et}*{Corollary 5.12} the pullback functor to algebraically closed fields form a conservative family. Under the assumption that $p$ is invertible in $\Oscr_S$, we need only pullback to those fields with characteristics prime to $p$ hence, since they are furthermore algebraically closed, contains all $p^n$-th roots of unity. In this case, when $S$ is the spectrum of such a field, $\1_{S[i]}(1)[1]\comp$ is equivalent to $\1_S[1]\comp$, which is exactly the cofiber of the fold map. 

%Therefore, if $S$, and thus $S[i]$, has all $p^n$-th roots of unity, we get that $\1_{S[i]}(1)[1]\comp$ is equivalent to $\1_S[1]\comp$, which is exactly the cofiber of the fold map. We can now reduce to this situation since for all $n \geq 1$, the map $S[i][\zeta_{p^n}] \rightarrow S[i]$ is an \'etale cover under our assumption that $p$ is invertible in $\Oscr_S$.
%
%Hence to prove the claim, we need to show that the induced map $ (\1_{S[i]})\comp[1] \rightarrow \hat{\1}_p(1)_{S[i]}[1]$ is an equivalence. Since $\tfrac{1}{p} \in \Oscr_{S}$, recall that adjoining $p^n$-th roots of unity are \'etale covers for all $n \geq 1$. Hence, we may assume that $S$ contains all $p^n$-th roots of unity\footnote{Technically one needs to pass to a pro-\'etale cover to accomplish this; but note that the resulting equivalence is an equivalence between objects which lie in the image of $\Sp(\widehat{X}_{\et})\comp \rightarrow \Sp(\widehat{X}_{\mathrm{pro}-\et})\comp$.} in which case $\hat{\1}_p(1)_{S[i]} \simeq  (\1_{S[i]})\comp[1]$ by \cite{bachmann-et}*{Theorem 3.6}.
\end{proof}

\begin{remark} \label{rem:c2} The strategy of the proof of Lemma~\ref{lem:duals} is to show that the ``premotivic'' $\rho: \1 \to \1(1)[1]$ and the ``categorical'' $\rho: \1 \to U_1$ of Remark~\ref{rem:FiniteLocalization} coincide after $p$-completion (in effect, also dispensing with the Tate twist). We do not expect such a result to hold in $\SH(-)$ itself. For example, in \cite{behrens-shah}, we observed that in $\SH(\RR)$ the two maps are \emph{not} equivalent even after $p$-completion \cite{behrens-shah}*{Remark~8.4}, but \emph{are} equivalent after $p$-completion and cellularization \cite{behrens-shah}*{Proposition~8.3}.
% In $C_2$-equivariant homotopy theory, the sequence~\eqref{eq:c-to-gm} is analogous\footnote{When $S$ is the spectrum of a real closed field, it will in fact be equivalent after $p$-completion under the identification of the next section.} to the cofiber sequence
% \[
% C_2/1_+ \rightarrow \1 \rightarrow S^{\mathrm{sgn}}.
% \]
% Here, $S^{\mathrm{sgn}}$ is the sign-representation sphere. While the above is still a cofiber sequence in naive $C_2$-spectra, the spectrum $C_2/1_+$ is \emph{not} dualizable and thus $S^{\mathrm{sgn}}$ is not invertible. 
\end{remark}

We are now ready to prove the main theorem of this section.

\begin{theorem} \label{construct:genuine} Let $p$ be a prime and suppose that $S$ is a noetherian locally $p$-\'etale finite scheme. Then 
\[
(\Sp_b^{C_2})\comp: (\Sch^{p\mbox{-}\fin,\mathrm{noeth}}_{S}[\tfrac{1}{p}, \tfrac{1}{2}])^{\op} \rightarrow \CAlg(\PrL_{\infty,\stab})
\]
is a premotivic functor that satisfies the full six functors formalism.
\end{theorem}
\begin{proof} We note that we have a functor valued in $\CAlg(\PrL_{\infty,\stab})$ by Theorem~\ref{thm:MonoidalityOfGenuineStabilizationFunctor}. Under Corollary~\ref{cor:tate-stable}, the theorem then follows from Bachmann's theorems \ref{prop:sixret} and \ref{prop:sixet} along with Theorem~\ref{thm:abstract2} and the fact that $(\Sp_b^{C_2})\comp$ is Tate-dualizable by Lemma~\ref{lem:duals}.
\end{proof}

\begin{example} \label{ex:ft} Fix a prime $p$. By \cite{bachmann-et}*{Corollary 2.13}, if $S$ is locally $p$-\'etale finite, then so is any finite type $S$-scheme. Therefore, the premotivic functor $(\Sp_b^{C_2})\comp$ satisfies the full six functors formalism when restricted to $\Sch^{\mathrm{ft}}_{\ZZ[\tfrac{1}{p}, \tfrac{1}{2}]}$. However, we note that the $\infty$-category $\Sp(\widetilde{X}_{\ret})$ is nonzero only if one of its residue fields is an orderable field, which necessitates characteristic zero. Hence, the premotivic functor $(\Sp_b^{C_2})\comp$ is most interesting when restricted to schemes whose residue fields are characteristic zero (otherwise, it coincides with the $\infty$-category of $p$-completed hypercomplete \'etale sheaves).

In light of this, Theorem~\ref{construct:genuine} produces a genuinely new six functors formalism when $(\Sp_b^{C_2})\comp$ is restricted to $\Sch^{\mathrm{ft}}_{\QQ}$. In fact, letting $p$ vary across all primes and using the fact that we are in characteristic zero, the profinite-completed version of $\Sp_b^{C_2}$ assembles into a premotivic functor
\[
(\Sp_b^{C_2})\comfin: (\Sch^{\mathrm{ft}}_{\QQ})^{\op} \rightarrow \CAlg(\PrL_{\infty,\stab})
\]
that satisfies the full six functors formalism.
\end{example}

\section{Motivic spectra over real closed fields and genuine $C_2$-spectra}  \label{sec:field-case}

Let $S$ be a scheme. We define the {\bf $b$-topology} on the large site $\Sm_S$ to be the intersection of the real \'etale and \'etale topologies on $\Sm_S$. The $b$-topology is finer than the Nisnevich topology, so we may consider the full subcategory $\SH_b(S) \subset \SH(S)$ of $b$-local motivic spectra as a Bousfield localization. We find it apposite to then make the following definition.

\begin{definition} \label{def:sch} The $\infty$-category of {\bf Scheiderer motivic spectra} over $S$ is $\SH_b(S)$.
\end{definition}

Our primary aim for the remainder of this paper is to relate Scheiderer motivic spectra and $b$-sheaves of spectra with transfers via a parametrized and semialgebraic analogue of the $C_2$-Betti realization functor $\Be^{C_2}: \SH(\RR) \to \Sp^{C_2}$ of Heller-Ormsby \cite{heller-ormsby}. We first consider the simplest nontrivial example with $S$ taken to be the spectrum of a real closed field. To this end, fix a real closed field $k$ and let $C = k[i]$ be an algebraic closure of $k$, so that $C/k$ is a Galois extension of degree $2$. We will define a {\bf $C_2$-semialgebraic} or {\bf Delfs-Knebusch realization} functor
\[
\DK^{C_2}:\SH(k) \rightarrow \Sp^{C_2}
\]
and then show its right adjoint $\Sing^{C_2}$ is fully faithful after $p$-completion with essential image $\SH_b(k)\comp$ (Theorem~\ref{thm:sing-ff}),  thereby locating $C_2$-equivariant homotopy theory as a concept in motivic homotopy theory. The relationship with $C_2$-Betti realization of \cite{heller-ormsby} is addressed in Remark~\ref{rem:salgagrees}. For the proof, we imitate the strategy of \cite{behrens-shah}, in which Behrens and the second author studied a similar question concerning $C_2$-Betti realization and cellular real motivic spectra; apart from semialgebraic topology, the only novel aspect to our proof here is the identification of \'etale motivic spectra over $k$ with Borel $C_2$-spectra under $\Sing^{C_2}$ after $p$-completion (Corollary~\ref{cor:iet}), which permits us to eliminate cellularity hypotheses (but also see Remark~\ref{compare-behrens-shah}). Along the way, we additionally establish the symmetric monoidal monadicity of the adjunction $\DK^{C_2} \dashv \Sing^{C_2}$ (Theorem~\ref{thm:MonadicityofRealization}), and hence of $(L_b)\comp \dashv (i_b)\comp$ as well under Theorem~\ref{thm:sing-ff}.

\subsection{$C_2$-equivariant semialgebraic realization} \label{betti-c2}

Suppose that $X$ is a smooth $k$-variety. Recall that we may endow the set $X(k)$ of $k$-points of $X$ with the {\bf semialgebraic topology}, thereby making it into a semialgebraic space (cf. \cite{salg} for the relevant definitions and \cite{scheiderer}*{Chapter~15}, \cite{geom-rel} for textbook references). This procedure assembles to a functor
\[
\Sm_{k+} \rightarrow \Top_{\bullet}, \qquad X_+ \mapsto X(k)_+.
\]
Furthermore, if we suppose instead that $Y$ is a smooth quasiprojective $C$-variety, then we may endow $Y(C)$ with the semialgebraic topology under the isomorphism
\[
R_{C/k}Y(k) \cong Y(C)
\]
that identifies $Y(C)$ with the $k$-points of its Weil restriction $R_{C/k} Y$ (cf. \cite{scheiderer}*{15.2}).\footnote{In fact, as Scheiderer notes, given any $C$-variety $Y$, we may obtain the semialgebraic topology on $Y(C)$ from the above procedure by gluing on affine open pieces.} Then for a smooth quasiprojective $k$-variety $X$, the induced conjugation $C_2$-action on $X(C)$ is seen to be continuous. We thereby obtain a functor 
\[
\Sm\QP_{k+} \rightarrow \Top_{\bullet}^{BC_2}, \qquad X_+ \mapsto X(C)_+
\]
into the category of pointed topological spaces with $C_2$-action, such that the natural map $X(k)_+ \to X(C)_+$ induces a homeomorphism $X(k)_+ \cong X(C)^{C_2}_+$. Passing to the $\infty$-category of $C_2$-spaces, we get a functor
\begin{equation} \label{eq:c2-formula}
\DK^{C_2}:\Sm\QP_{k+} \rightarrow \Spc_{C_2\bullet} \simeq \Pre(\Oscr_{C_2})_{\bullet},
\end{equation}
such that under Elmendorf's theorem,\footnote{Note that we could define $\DK^{C_2}$ into pointed presheaves $\Pre(\Oscr_{C_2})_{\bullet}$ directly by means of the formula and avoid Elmendorf's theorem altogether.} $\DK^{C_2}$ is given by\footnote{To ward off any potential confusion, we note here a clash of terminology between the total singular complex functor $\Sing$ and the right adjoint to ($C_2$-)Betti or semialgebraic realization, also denoted $\Sing^{(C_2)}$.}
\[
 X_+ \mapsto  (\Sing(X(C))_+ \leftarrow \Sing(X(k))_+).
\]

Similarly, we have the functor
\[
\DK_C: \Sm\QP_{C+} \rightarrow \Spc_{\bullet}, \qquad Y_+ \mapsto \Sing(Y(C))_+.
\]

With respect to the symmetric monoidal structure on the presheaf category $\Pre(\Oscr_{C_2})_{\bullet}$ that is given by the pointwise smash product, the next lemma is immediate.

\begin{lemma} \label{lem:sym-mon} The functors $\DK^{C_2}$ and $\DK_C$ are strong symmetric monoidal.
\end{lemma}

The symmetric monoidal structure on $\Pre_{\Sigma}(\Sm\QP_{k+})$ in which we will eventually invert $\TT$ to obtain $\SH(k)$ is given by Day convolution (and likewise for $C$). By the universal property of $\Pre_{\Sigma}$ and Day convolution \cite{higheralgebra}*{Proposition 4.8.1.10}, we get sifted colimit-preserving strong symmetric monoidal extensions
\begin{align*}
\DK^{C_2} &: \Pre_{\Sigma}(\Sm\QP_{k+}) \simeq \Pre_{\Sigma}(\Sm\QP_k)_{\bullet} \rightarrow \Spc_{C_2\bullet}, \\
\DK_C &: \Pre_{\Sigma}(\Sm\QP_{C+}) \simeq \Pre_{\Sigma}(\Sm\QP_C)_{\bullet} \rightarrow \Spc_{\bullet}.
\end{align*}

\begin{lemma} \label{lem:nis} The functors $\DK^{C_2}$ and $\DK_C$ convert Nisnevich sieves and $\AA^1$-homotopy equivalences to equivalences.
\end{lemma}

\begin{proof} We prove the claim for $\DK^{C_2}$, since the claim for $\DK_C$ is both similar and easier. The assertion regarding Nisnevich sheaves is essentially \cite{dugger-isaksen}*{Theorem 5.5} and proceeds as follows. By standard arguments, we may reduce to the case of a Nisnevich sieve generated by a single map $X' \rightarrow X$. In this case, it suffices to prove that the \v{C}ech nerves of the maps
 \[ \Sing(X'[i](C))_+ \rightarrow  \Sing(X(C))_+, \qquad \Sing(X')(k)_+ \rightarrow \Sing(X)(k)_+ \]
induce equivalences
\[ |\check{C}_{\bullet}(\Sing(X'[i](C)))_+| \rightarrow  (\Sing(X(C)))_+, \qquad |\check{C}_{\bullet}(\Sing(X')(k))_+| \rightarrow \Sing(X)(k)_+. \]
We note that since equivalences are detected pointwise in presheaf categories, we do not have to consider the $C_2$-action on $\Sing(X'[i](C))_+$. The claim then follows from the fact that \'etale morphisms are converted to generalized space covers in the sense of \cite{dugger-isaksen}*{4.8} under taking semialgebraization by \cite{dk-semi}*{Example 5.1} and \cite{dugger-isaksen}*{Proposition 4.10}. 

The claim about $\AA^1$-homotopy equivalences follows from the fact that $\DK^{C_2}$ is strong symmetric monoidal and that $\DK^{C_2}(\AA^1)$ defines homotopies in the semialgebraic category (cf. \cite{delfs-homotopy}). 
\end{proof}

Recall that for any scheme $S$, $\Shv_{\Nis}(\Sm\QP_S) \simeq \Shv_{\Nis}(\Sm_S)$ since any smooth $S$-scheme admits a Zariski cover by smooth quasiprojective ones. Thus, by Lemma~\ref{lem:nis} we further get sifted colimit-preserving strong symmetric monoidal functors
\[
\DK^{C_2}: \H(k)_{\bullet} \rightarrow  \Spc_{C_2 \bullet}, \qquad \DK_C: \H(C)_{\bullet} \rightarrow \Spc_{\bullet}.
\]

\begin{lemma} \label{lem:bec2-cpct} The functors $\DK^{C_2}$ and $\DK_C$ preserve compact objects.
\end{lemma}
\begin{proof} It suffices to prove that $\Sing(X(C))$ and $\Sing(X(k))$ have the homotopy type of finite CW-complexes for a $k$-variety $X$. This follows from the fact that both $X$ and $X[i]$ are, by assumption, finite type and hence the semialgebraic spaces $X(k)$ and $X(C)$  may be covered by finitely many semialgebraic subsets. A theorem of Delfs \cite{salg}*{Theorem 4.1} then asserts the existence of a finite triangulation of these semialgebraic spaces and thus they have the homotopy types of finite CW-complexes.
\end{proof}

Let $V$ be a real $C_2$-representation. Then we can view $S^V$, the $1$-point compactification of $V$, as an object in $\Pre(\Oscr_{C_2})_{\bullet}$. 

\begin{lemma} \label{lem:bec2} We have the following natural equivalences in $\Spc_{C_2{\bullet}}$, resp. $\Spc_{\bullet}$
\[
\DK^{C_2}( (\GG_m, 1)) \simeq S^{\sigma}, \qquad \DK_C( (\GG_m, 1)) \simeq S^1, 
\]
\[
\DK^{C_2}( (S^1,1) ) \simeq S^1, \qquad \DK_C((S^1,1)) \simeq S^1,
\]
and
\[
\DK^{C_2}( (\PP^1, \infty)) \simeq \Sigma S^{\sigma} \simeq S^{\sigma+1}, \qquad \DK_C( (\PP^1, \infty)) \simeq \Sigma S^1 \simeq S^2.
\]
\end{lemma}
\begin{proof} The first two pairs of equivalences are standard. The last pair of equivalences holds in view of the Zariski cover $\{ \AA^1 \rightarrow  \PP^1, \AA^1 \rightarrow \PP^1 \}$ and Lemma~\ref{lem:nis}.
\end{proof}

Stabilization is now standard procedure. The representation sphere $S^{\sigma}$ is invertible in $\Sp^{C_2}$; this is either true by fiat if $\Sp^{C_2}$ is constructed out of $\Spc_{C_2 \bullet}$ from inverting representation spheres (see \cite{bachmann-hoyois}*{\S9.2} for a treatment in our language) or a consequence of the comparison between $C_2$-spectra and spectral Mackey functors (a theorem of Guillou-May \cite{guillou-may}). We also refer the reader to Nardin's proof of the comparison \cite{denis-stab}*{Theorem~A.4} which is native to the language of this paper (also see \cite{bachmann-hoyois}*{Proposition 9.10}).
% (but note that the invertibility of $S^{\sigma}$ has a more straightforward argument as explained to us by Naumann \cite{naumann-invert}) 

Using \cite{bachmann-hoyois}*{Lemma 4.1}, which is a variant of Robalo's theorem \cite{robalo}, we get the colimit-preserving symmetric monoidal functors of $C_2$ (resp. $C$)-{\bf semialgebraic realization}
\[
\DK^{C_2}: \SH(k) \rightarrow \Sp^{C_2}, \qquad \DK_C: \SH(C) \rightarrow \Sp
\]
that commute with the suspension functors and the unstable $\DK^{C_2}$, resp. $\DK_C$. We denote their right adjoints by
\[
\Sing^{C_2}: \Sp^{C_2} \rightarrow \SH(k), \qquad \Sing_C: \Sp \rightarrow \SH(C).
\]

Note that the following diagram commutes by construction:
\[ \begin{tikzcd}
\SH(k) \ar{r}{\DK^{C_2}} \ar{d}{\pi^{\ast}} & \Sp^{C_2} \ar{d}{\res^{C_2}} \\
\SH(C) \ar{r}{\DK_C} & \Sp,
\end{tikzcd} \]
where $\pi: \Spec C \to \Spec k$ denotes the Galois cover.

\begin{remark} \label{rem:c} There is yet another adjunction that connects $C_2$-equivariant homotopy theory and motivic spectra over real closed fields that was studied by Heller-Ormsby \cite{heller-ormsby}. It takes the form
\[
c: \Sp^{C_2} \rightleftarrows \SH(k): u.
\] By definition, the functor $\DK^{C_2}$ left splits the constant functor $c: \Sp^{C_2} \rightarrow \SH(k)$ of \cite{heller-ormsby}*{\S4}, so that the composite $\Sp^{C_2} \stackrel{c}{\rightarrow} \SH(k) \stackrel{\DK^{C_2}}{\rightarrow} \Sp^{C_2}$ is equivalent to the identity functor (cf. \cite{bachmann-hoyois}*{Proposition 10.6}). It follows that $c$ is a faithful functor and we have an equivalence
\begin{equation} \label{eq:sing}
u\circ\Sing^{C_2} \simeq \id
\end{equation}
by uniqueness of adjoints. The functor $c$ is further proved to be fully faithful in \cite{heller-ormsby2}.
\end{remark}

Our next result is the semialgebraic version of \cite{behrens-shah}*{Lemma~8.13}.

\begin{theorem} \label{thm:MonadicityofRealization} The symmetric monoidal adjunctions
\[
\DK^{C_2}:\SH(k) \rightleftarrows \Sp^{C_2}: \Sing^{C_2}, \qquad \DK_C:\SH(C) \rightleftarrows \Sp: \Sing_C
\]
satisfy the following conditions:
\begin{enumerate}
\item The functors $\DK^{C_2}$ and $\DK_C$ preserve compact objects.
\item The functors $\Sing^{C_2}$ and $\Sing_C$ preserve colimits.
\item The projection formula holds, i.e., the canonical maps
\begin{align*} \Sing^{C_2}(K) \otimes E & \rightarrow \Sing^{C_2}(K \otimes \DK^{C_2}(E)), \\
\Sing_C(L) \otimes F & \rightarrow \Sing_C(L \otimes \DK_C(F)),
\end{align*}
are equivalences for all $K \in \Sp^{C_2}, E \in \SH(k)$, resp. $L \in \Sp, F \in \SH(C)$.
\item The functors $\Sing^{C_2}$ and $\Sing_C$ are conservative.
\end{enumerate}
Consequently, we have equivalences of symmetric monoidal $\infty$-categories
\begin{align*} \Sp^{C_2} \simeq \Mod_{\Sing^{C_2}\DK^{C_2}(\1)}(\SH(k)), \qquad \Sp \simeq \Mod_{\Sing_C \DK_C(\1)}(\SH(C))
\end{align*}
as $\SH(k)$-algebras, resp. $\SH(C)$-algebras.
% $\Sp^{C_2}$, resp. $\Sp$ are equivalent to $\Mod_{\Sing^{C_2}\DK^{C_2}(\1)}(\SH(k))$ as a $\SH(k)$-algebra, resp. $\Mod_{\Sign_C \DK_C(\1)}(\SH(C))$ as a $\SH(C)$-algebra.
\end{theorem}

\begin{proof} We give the proof for $\DK^{C_2} \dashv \Sing^{C_2}$ as the proof for $\DK_C \dashv \Sing_C$ is the same. We check: 
\begin{enumerate}
\item This follows from the unstable statement Lemma~\ref{lem:bec2-cpct} and the fact that $\DK^{C_2}$ commutes with the appropriate suspension functors.
\item Since the functor $\Sing^{C_2}$ is exact, we need only prove that $\Sing^{C_2}$ preserves filtered colimits. Let $K_{(-)}:I \rightarrow \Sp^{C_2}$ be a filtered diagram with colimit $K = \colim_i K_i$. We have a comparison map $$\colim_i \Sing^{C_2}(K_i) \rightarrow \Sing^{C_2}(K)$$ that we claim is an equivalence. It suffices to prove this after applying $[\Sigma^{p,q}X_+, -]$ where $X \in \Sm_k$ and $p, q\in \ZZ$. The result now follows from the fact that $\DK^{C_2}$ is strong symmetric monoidal, and that $X_+$ and $\DK^{C_2}(X_+)$ are compact objects in $\SH(k)$ and $\Sp^{C_2}$ respectively.
\item In this case, since $k$ is characteristic zero, $\SH(k)$ is generated by dualizable objects \cite{yang-zhao-levine}*{Proposition B.1}. The argument follows by \cite{e-kolderup}*{Lemma 5.1}.
\item This is a consequence of the faithfulness of the $\Sing^{C_2}$ functor thanks to~\eqref{eq:sing}.
\end{enumerate}
The consequence now follows from the monoidal Barr-Beck theorem \cite{mnn-descent}*{Theorem~5.29} (also see \cite{e-kolderup}*{Corollary 5.3}).
\end{proof}

% \begin{remark} \label{rem:cell} Theorem~\ref{thm:MonadicityofRealization} is the obvious generalization of \cite{behrens-shah}*{Lemma~8.13} from $\RR$ to $k$.
% \end{remark}

% We will later need the following compatibility with Betti realization. Recall that if we choose an embedding $k[i] \subset \CC$, then we have the Betti realization functor $\Be:\SH(k[i]) \rightarrow \Sp$. On the equivariant side, we have the restriction functor $\mathrm{Res}_{C_2}:\Sp^{C_2} \rightarrow \Sp$.
% \begin{proposition} \label{prop:compatible} Let $k$ be a real closed field and consider the map $\pi: \Spec k[i] \rightarrow \Spec k$. The following diagram commutes
% \begin{equation} \label{prop:betti-res}
% \begin{tikzcd}
% \SH(k) \ar[swap]{d}{\pi^*} \ar{r}{\Be^{C_2}} & \Sp^{C_2} \ar{d}{\mathrm{res}^{C_2}}\\
% \SH(k[i]) \ar{r}{\Be}  & \Sp.
% \end{tikzcd}
% \end{equation}
% \end{proposition}
% \begin{proof} In view of the compatibility of $\Be^{(C_2)}$ with $\Sigma^{\infty}_+$, this follows from the commutativity of the diagram
% \begin{equation} \label{prop:betti-res-sm}
% \begin{tikzcd}
% \Sm_{k+} \ar[swap]{d}{\pi^*} \ar{r}{\Be^{C_2}} & \Spc_{C_2\bullet} \ar{d}{\mathrm{ev}_{C_2/1_+}}\\
% \Sm_{k[i]+} \ar{r}{\Be}  & \Spc_{\bullet},
% \end{tikzcd}
% \end{equation}
% which is evident from the formula in~\eqref{eq:c2-formula}.
% \end{proof}
\begin{remark} We also have a {\bf $k$-semialgebraic realization} functor, specified by extending the functor
\[ \DK_k: \Sm\QP_{k \bullet} \rightarrow \Spc_{\bullet}, \qquad X_+ \mapsto  \Sing(X(k))_+ \]
to $\DK_k: \SH(k) \to \Sp$ by the same methods as above. However, this realization functor is already present in the above picture. Namely, if we let $i^*: \Sp^{C_2} \to \Sp$ denote the functor of geometric $C_2$-fixed points, then by construction we have an equivalence $\DK_k \simeq i^* \DK^{C_2}$.
\end{remark}

\begin{remark} \label{rem:salgagrees} When $k = \RR$, we note that $C_2$-semialgebraic realization $\DK^{C_2}$ is equivalent to $C_2$-Betti realization $\Be^{C_2}$. To see this, we first compare the two as functors $\Sm\QP_{\RR} \to \Spc_{C_2}$. Recall that $\Be^{C_2}$ is defined by sending a smooth quasiprojective $\RR$-scheme $X$ to the $C_2$-space $X(\CC)^{\mathrm{an}}$, where we endow $X(\CC)$ with the analytic topology. Since an analytic space is locally homeomorphic to $\CC^n$ with the standard topology, the space $X(\CC)^{\mathrm{an}}$ is given by endowing $X(\CC) \cong R_{\CC/\RR}X_{\CC}(\RR)$ with the Euclidean or strong topology on the $\RR$-points of an $\RR$-variety, noting that $R_{\CC/\RR}\AA^n \simeq \AA^{2n}$. On the other hand, the semialgebraic topology on $X(\CC)$ is finer than the analytic topology, so if we denote the set $X(\CC)$ equipped with the semialgebraic topology as $X(\CC)^{\mathrm{salg}}$, then the identity map on sets defines a continuous map of spaces $\alpha^{\pre}_X: X(\CC)^{\mathrm{salg}} \rightarrow X(\CC)^{\mathrm{an}}$, which is clearly $C_2$-equivariant. Upon passing to $\Spc_{C_2}$, we obtain a natural transformation $\alpha: \DK^{C_2} \Rightarrow \Be^{C_2}$. From the existence of triangulations for semialgebraic spaces \cite{salg}*{Theorem 4.1} together with \cite{tenyears}*{\S10, Second Comparison Theorem}, we get that $\alpha^{\pre}_X$ and $(\alpha^{\pre}_X){}^{C_2}$ induce isomorphisms of homotopy groups, whence $\alpha^{\pre}_X$ is a weak homotopy equivalence of $C_2$-topological spaces. We conclude that $\alpha$ is an equivalence. It follows that all successive extensions of these functors are equivalent. Similarly, we have that $\DK_{\CC}$ is equivalent to complex Betti realization $\Be_{\CC}$ and $\DK_{\RR}$ is equivalent to real Betti realization $\Be_{\RR}$.
\end{remark}

\subsection{The real \'etale part} In motivic stable homotopy theory, we have a map $\rho: S^{0,0} \rightarrow S^{1,1}$ in $\SH(S)$ for any scheme $S$, defined as the stabilization of the map of schemes $S \rightarrow \GG_m$ that classifies the unit $-1$. We also have a map $\rho$ in $\Sp^{C_2}$ that is the Euler class for the $C_2$-sign representation $\sigma$, given by the stabilization of the unique $C_2$-equivariant inclusion of $C_2$-spaces $S^0 \rightarrow S^{\sigma}$. By the construction of $\DK^{C_2}$ and the identification of Lemma~\ref{lem:bec2}, we note:

\begin{lemma} \label{lem:rho} The $C_2$-semialgebraic realization functor $\DK^{C_2}: \SH(k) \to \Sp^{C_2}$ sends $\rho$ to $\rho$.
\end{lemma}

% In view of Lemma~\ref{lem:rho}, we are entitled to denote both of the maps in the motivic and equivariant contexts by $\rho$, which we now do for the rest of the paper.
Since $\Sper(k) = \pt$, we have that $\SH_{\ret}(k) \simeq \Sp(\widetilde{k_{\ret}}) = \Sp$ by Bachmann's theorem~\eqref{eq:x-ret}, which moreover identifies the localization endofunctor $i_{\ret} L_{\ret}$ on $\SH(k)$ with $\rho$-inversion (thereby embedding $\Sp$ as $\rho$-inverted objects). We wish to now identify $L_{\ret}$ and $i^* \DK^{C_2}$. To this end, we first make a simple observation.

\begin{lemma} \label{lem:SimpleAbstractNonsense} Suppose $\C$ is a stable presentable symmetric monoidal $\infty$-category and $F: \C \to \Sp$ is a colimit-preserving symmetric monoidal functor. Suppose $\rho: \1 \to E$ is a map in $\C$ such that we have an abstract equivalence $\gamma: \Sp \xrightarrow{\simeq} \C[\rho^{-1}]$ of symmetric monoidal $\infty$-categories. Then if $F(\rho)$ is an equivalence in $\Sp$, the functor $F$ descends to an equivalence $\overline{F}: \C[\rho^{-1}] \xrightarrow{\simeq} \Sp$ of symmetric monoidal $\infty$-categories. Therefore, if we let $R$ denote the right adjoint to $F$, the adjunction $F \dashv R$ is a smashing localization with the essential image of $R$ given by the $\rho$-inverted objects.
\end{lemma}
\begin{proof} By assumption and using the universal property of $\rho$-inversion, $F$ descends to a colimit-preserving symmetric monoidal functor $\overline{F}: \C[\rho^{-1}] \to \Sp$. Because the $\infty$-category of colimit-preserving, symmetric monoidal endofunctors of $\Sp$ is contractible (e.g., by \cite{thomas-yoneda}*{Corollary 6.9}), we see that $\overline{F}\circ \gamma$ and $\id$ are homotopic. By the two-out-of-three property of equivalences, we deduce that $\overline{F}$ is an equivalence. Therefore, the unit map $\eta_{\1}: \1 \to RF(\1)$ is homotopic to the idempotent object $\1 \to \1[\rho^{-1}]$ that determines the smashing localization $\C[\rho^{-1}] \subset \C$ (cf. \cite{higheralgebra}*{Proposition~4.8.2.10}), so the conclusion follows.
\end{proof}

\begin{proposition} \label{prop:ret-compare} The adjunction
\[ i^* \DK^{C_2}: \SH(k) \rightleftarrows \Sp: \Sing^{C_2} i_*  \]
is a smashing localization with the essential image of $\Sing^{C_2} i_*$ given by the $\rho$-inverted objects (or equivalently, the $\ret$-local objects).
\end{proposition}
\begin{proof} We note that $\DK^{C_2}$ is symmetric monoidal by Theorem~\ref{thm:MonadicityofRealization} and carries the motivic $\rho$ to the $C_2$-equivariant $\rho$ by Lemma~\ref{lem:rho}, while $i^*$ is the functor of geometric $C_2$-fixed points and is thus symmetric monoidal and given by inverting the $C_2$-equivariant $\rho$ (see \cite{greenlees-may}*{Proposition 3.20} for a classical reference). By Lemma~\ref{lem:SimpleAbstractNonsense} and under Bachmann's theorem \eqref{eq:x-ret}, the conclusion follows.
\end{proof}

\subsection{The \'etale part} Since the Galois group of a real closed field is $C_2$, we have that $\widetilde{k_{\et}} \simeq \Spc^{BC_2}$. In particular, the $\infty$-topos of \'etale sheaves on $\Spec k$ is already hypercomplete. Bachmann's theorem \eqref{eq:x-et} thus furnishes equivalences
\[
\Sp(\widetilde{k_{\et}})\comp \simeq \SH_{\et}(k)\comp, \qquad \Sp(\widetilde{C_{\et}})\comp \simeq \SH_{\et}(C)\comp.
\]

\begin{proposition} \label{prop:et-compare} For any prime $p$, the following diagrams commute
\[
\begin{tikzcd}
\SH(k) \ar{r}{\DK^{C_2}} \ar[swap]{d}{(L_{\et})\comp} & \Sp^{C_2} \ar{d}{(j^*)\comp}\\
\SH_{\et}(k)\comp \ar{r}{\simeq} & (\Sp^{BC_2})\comp
\end{tikzcd}, \qquad
\begin{tikzcd}
\SH(C) \ar{r}{\DK_C} \ar[swap]{d}{(L_{\et})\comp} & \Sp \ar{d}{(-)\comp}\\
\SH_{\et}(C)\comp \ar{r}{\simeq} & \Sp\comp.
\end{tikzcd}
\]
\end{proposition}

\begin{proof} Note that the equivalences $\widetilde{k_{\et}} \simeq \Spc^{BC_2}$ and $\widetilde{C_{\et}} \simeq \Spc$ are implemented by sending an \'etale sheaf $\Fscr$ over $k$, resp. $C$ to the $C_2$-space $\Fscr(C)$, resp. space $\Fscr(C)$; we implicitly make these identifications for the rest of the proof. Let $\gamma^*$ denote the change of site functor. We first define  lax commutative squares
\[
\begin{tikzcd}
\Sm_{k+} \ar[hookrightarrow]{d}{y} \ar{r}{\DK^{C_2}} \ar[phantom]{rd}{\NEarrow}  & \Spc_{\bullet}^{C_2} \ar{d}{j^*}\\
\Shv_{\et}(\Sm_{k})_{\bullet} \ar{r}{\gamma^*}  & \Spc_{\bullet}^{BC_2}
\end{tikzcd}, \qquad
\begin{tikzcd}
\Sm_{C+} \ar[hookrightarrow]{d}{y} \ar{r}{\DK_C} \ar[phantom]{rd}{\NEarrow}  & \Spc_{\bullet} \ar{d}{=}\\
\Shv_{\et}(\Sm_{C})_{\bullet} \ar{r}{\gamma^*}  & \Spc_{\bullet}
\end{tikzcd} 
\]
as specified by natural transformations
\[ \zeta_k: \gamma^*y \Rightarrow j^*\DK^{C_2}, \qquad \zeta_C: \gamma^* y \Rightarrow \DK_C. \]
For $\zeta_k$, we take the transformation that on pointed $k$-varieties $X_+$ evaluates to the canonical map $X(C)^{\delta}_+ \rightarrow X(C)_+$ of topological spaces that forgets to the identity map on sets (where the decoration $(-)^{\delta}$ indicates that we take the discrete topology), and $\zeta_C$ is defined similarly.

We further note that $\zeta_{k}$ and $\zeta_C$ are strong monoidal transformations between strong monoidal functors. By the universal properties of the functors involved, we thereby obtain lax commutative squares
\[
\begin{tikzcd}[column sep=8ex]
\SH(k) \ar{d}{L_{\et}} \ar{r}{\DK^{C_2}}  \ar[phantom]{rd}{\zeta_k \NEarrow} & \Sp^{C_2} \ar{d}{j^*}\\
\SH_{\et}(k) \ar{r}{\gamma^*}& \Sp^{BC_2}
\end{tikzcd}, \qquad
\begin{tikzcd}[column sep=8ex]
\SH(C) \ar{d}{L_{\et}} \ar{r}{\DK_C}  \ar[phantom]{rd}{\zeta_C \NEarrow} & \Sp \ar{d}{=}\\
\SH_{\et}(k) \ar{r}{\gamma^*}& \Sp.
\end{tikzcd}
\]

Our goal is to prove that the homotopies of these lax squares are $p$-adic equivalences. For $\zeta_k$, since its codomain lies in Borel $C_2$-spectra, it suffices to check that the map is an equivalence after forgetting the $C_2$-action. But using the compatibility of both sides with base change, we note that the underlying map of $\zeta_k$ is equivalent to 
\[  \zeta_C \circ \pi^*: \gamma^* L_{\et} \pi^* \Rightarrow \DK_{C} \pi^*,  \]
where $\pi^*: \SH(k) \to \SH(C)$ is base change along the Galois cover. It thus suffices to show that $\zeta_C$ is a $p$-adic equivalence.

Since all functors in sight preserve colimits and are strong symmetric monoidal after $p$-completion,\footnote{The functor $\gamma^*$ is only lax symmetric monoidal, but becomes an strongly so after $p$-completion since it induces an equivalence of symmetric monoidal $\infty$-categories.} it suffices to show that for quasiprojective $X \in \Sm_{C}$ and all $n \geq 0$, the induced map
\[
H^n(X(C); \FF_p) \rightarrow H^n(\gamma^*L_{\et}X; \FF_p) 
\]
is an isomorphism.

Under the rigidity equivalence $\SH_{\et}(C)\comp \simeq \Sp(\widetilde{C_{\et}})\comp = \Sp\comp$, the spectrum $H\FF_p \in \Sp\comp$ coincides with the spectrum representing \'etale cohomology (with coefficients in $\mu_p \simeq \FF_p$) in $\SH_{\et}(C)\comp$, essentially because of \cite{etalemotives}*{Proposition 3.2.3} and the evident compatibility of the rigidity equivalences in \cite{bachmann-et} and \cite{etalemotives}. Hence, for the right hand side, we have an isomorphism
\[
H^*(\gamma^*L_{\et}X; \FF_p) \cong H^*_{\et}(X, \ZZ/p),
\]
noting that all Tate twists disappear since we are working over an algebraically closed field. 

On the other hand, $H^n(X(C); \FF_p)$ computes the semialgebraic cohomology $H^n_{\mathrm{sa}}(X(C); \FF_p)$ for the constant sheaf $\FF_p$, as this is just sheaf cohomology by the discussion in \cite{scheiderer}*{15.2}. The claim now follows from the \'etale-semialgebraic comparison theorem of Roland Huber \cite{scheiderer}*{Theorem~15.2.1}.
\end{proof}

\begin{remark} If $k = \RR$ and $C = \CC$, then the conclusion of Proposition~\ref{prop:et-compare} holds (with the same proof) if one instead considers $C_2$ and complex Betti realization and uses the Artin comparison theorem.
\end{remark}

Passing to the right adjoints of the functors in Proposition~\ref{prop:et-compare}, we deduce:

\begin{corollary} \label{cor:iet} For any prime $p$, the following diagrams commute
\begin{equation} \label{eq:i-et}
\begin{tikzcd}
\SH(k)\comp   & (\Sp^{C_2})\comp \ar[swap]{l}{\Sing^{C_2}}\\
\SH_{\et}(k)\comp \ar[hookrightarrow]{u}{i_{\et}} & (\Sp^{BC_2})\comp \ar[hookrightarrow]{u}[swap]{j_*} \ar{l}{\simeq}
\end{tikzcd}, \qquad
\begin{tikzcd}
\SH(C)\comp   & \Sp\comp \ar[swap]{l}{\Sing_C}\\
\SH_{\et}(C)\comp \ar[hookrightarrow]{u}{i_{\et}} & \Sp\comp \ar{u}[swap]{=} \ar{l}{\simeq}.
\end{tikzcd}
\end{equation}
In particular, the functors $\Sing^{C_2} j_*$ and $\Sing_C$ are fully faithful after $p$-completion, with essential image given by the $p$-complete \'etale-local objects.
\end{corollary}

\begin{remark} \label{rem:expand} More precisely, we have the following commutative diagram displaying the interaction with $p$-completion
\begin{equation} \label{eq:i-et-expand}
\begin{tikzcd}
\SH(k)   & \Sp^{C_2} \ar[swap]{l}{\Sing^{C_2}}\\
\SH_{\et}(k) \ar{u}{i_{\et}} & \Sp^{BC_2}  \ar{u}{j_*} \\
\SH_{\et}(k)\comp \ar{u}{i_p} & (\Sp^{BC_2})\comp \ar[swap]{u}{i_p} \ar{l}{\simeq}.
\end{tikzcd}
\end{equation}
In other words, on $p$-complete objects the functors $\Sing^{C_2}j_*$ and $i_{\et}$ coincides.
\end{remark}

\subsection{The $C_2$-Tate construction in algebro-geometric terms}

The goal for the rest of this section is to identify the $C_2$-Tate construction in terms of motivic homotopy theory, and thereby identify Scheiderer motives over $k$ with genuine $C_2$-spectra (all after $p$-completion).

\begin{theorem} \label{thm:mot-v-tate} Under the equivalences $(\Sp^{BC_2})\comp \stackrel{\simeq}{\rightarrow} \SH_{\et}(k)\comp$ and $\Sp\comp \stackrel{\simeq}{\rightarrow} \SH_{\ret}(k)\comp$, we have a canonical equivalence of lax symmetric monoidal functors
\[
(-)^{tC_2} \simeq L_{\ret}i_{\et}:  \SH_{\et}(k)\comp \rightarrow \SH_{\ret}(k)\comp.
\]
\end{theorem}

To prove Theorem~\ref{thm:mot-v-tate}, we begin with some preliminary identifications that exactly parallel \cite{behrens-shah}*{Lemma~8.19-20} which was done in the context of Betti realization. For the convenience of the reader, we transcribe those proofs into our setting. Recall that the functor $$j^*i_*: \Sp \rightarrow \Sp^{C_2} \rightarrow \Sp^{BC_2}$$ is nullhomotopic. This remains true after inserting $\DK^{C_2}\Sing^{C_2}$ in the middle.

\begin{lemma} \label{lem:zero} Let $k$ be a real closed field. The composite of functors
\[
\Sp \xrightarrow{i_*} \Sp^{C_2} \xrightarrow{\Sing^{C_2}} \SH(k) \xrightarrow{\DK^{C_2}} \Sp^{C_2} \xrightarrow{j^*} \Sp^{BC_2},
\]
is nullhomotopic.
\end{lemma}

\begin{proof} In Proposition~\ref{prop:ret-compare}, we saw that the essential image of $\Sing^{C_2} i_*$ was given by the $\rho$-inverted objects. Since $\DK^{C_2}(\rho) \simeq \rho$, the same holds for $\DK^{C_2} \Sing^{C_2} i_*$. Since $j^*$ annihilates this subcategory (as it is the essential image of $i_*$), the claim follows.
\end{proof}

\begin{lemma} \label{lem:glue} The natural transformation
\[
i^*\DK^{C_2}\Sing^{C_2}j_* \Rightarrow i^*j_*,
\]
induced by the counit $\DK^{C_2}\Sing^{C_2} \Rightarrow \id$, is invertible.

\end{lemma}

\begin{proof} The functor $i_*$ is fully faithful and is thus conservative, and the endofunctor $i_*i^*$ is given by tensoring with $\widetilde{EC_2} \simeq S^0[\rho^{-1}]$. Therefore, it suffices to prove that the canonical map
\[
\widetilde{EC_2} \otimes \DK^{C_2}\Sing^{C_2}j_*(X) \rightarrow \widetilde{EC_2} \otimes j_*(X)
\]
is an equivalence for all $X \in \Sp^{BC_2}$. By Lemma~\ref{lem:rho}, we have that $\widetilde{EC_2} \simeq \DK^{C_2}(\1[\rho^{-1}])$, and thus the projection formula from Theorem~\ref{thm:MonadicityofRealization} yields the equivalence
\[
\widetilde{EC_2} \otimes \DK^{C_2}\Sing^{C_2}j_*(X) \simeq \DK^{C_2}\Sing^{C_2}(j_*(X) \otimes \widetilde{EC_2})
\]
under which $\widetilde{EC_2} \otimes \epsilon_{j_*(X)}$ identifies with $\epsilon_{\widetilde{EC_2} \otimes j_*X}$. Therefore, if we let $Z = \widetilde{EC_2} \otimes j_*X$, it suffices to show that for all $Z \in \Sp$, the counit map induces an equivalence
\[
\DK^{C_2}\Sing^{C_2}(i_*Z) \rightarrow i_* Z.
\]
But for this, note that we have the recollement cofiber sequence
\[
j_!j^*\DK^{C_2}\Sing^{C_2}(i_*Z) \rightarrow \DK^{C_2}\Sing^{C_2}(i_*Z) \rightarrow i_*i^*\DK^{C_2}\Sing^{C_2}(i_*Z),
\]
and $j_!j^*\DK^{C_2}\Sing^{C_2}(i_*Z) \simeq 0$ by Lemma~\ref{lem:zero}, while $i_*i^*\DK^{C_2}\Sing^{C_2}(i_*Z) \simeq i_*Z$ since $\Sing^{C_2}i_*$ is fully faithful by Proposition~\ref{prop:ret-compare}.
\end{proof}

\begin{proof} [Proof of Theorem~\ref{thm:mot-v-tate}] We need to prove that on all $p$-complete objects, the functors $(-)^{tC_2} \simeq i^* j_*$ and $L_{\ret}i_{\et}$ agree as lax symmetric monoidal functors. To show this, we use the sequence of equivalences
\begin{eqnarray*}
L_{\ret}i_{\et}i_p & \simeq & L_{\ret}\Sing^{C_2}j_*i_p \\
& \simeq & i^*\DK^{C_2}\Sing^{C_2}j_*i_p \\
& \simeq & i^*j_*i_p.
\end{eqnarray*}

The first equivalence follows from Remark~\ref{rem:expand} and Corollary~\ref{cor:iet}, the second from Proposition~\ref{prop:ret-compare}, and the last equivalence from Lemma~\ref{lem:glue}.
\end{proof}

%\begin{theorem} \label{thm:mot-tate}  Under the equivalences $(\Sp^{BC_2})\comtwo \stackrel{\simeq}{\rightarrow} \SH_{\et}(k)\comtwo$ and $(\Sp)\comtwo \stackrel{\simeq}{\rightarrow} \SH_{\ret}(k)\comtwo$ the (2-completion of the) cofiber sequence
%\begin{equation}\label{eq:equiv}
%(-)_{hC_2} \rightarrow (-)^{hC_2} \rightarrow (-)^{tC_2}
%\end{equation}
%is equivalent to the cofiber sequence
%\begin{equation} \label{eq:mot}
%(-)^{\mot}_{hC_2} \rightarrow \map(\1, i_{\et}(-)) \rightarrow \Psi_k^{\mot}.
%\end{equation}
%\end{theorem}

We now apply Theorem~\ref{thm:mot-v-tate} to prove the main theorem of this section.

\begin{theorem} \label{thm:sing-ff} Let $p$ be any prime. The functor $\Sing^{C_2}: (\Sp^{C_2})\comp \rightarrow \SH(k)\comp$ is fully faithful with essential image given by the $p$-complete Scheiderer motivic spectra, and we have a canonical equivalence of symmetric monoidal $\infty$-categories
\[ \SH_b(k)\comp \simeq (\Sp^{C_2})\comp. \]
Furthermore, $\SH_b(k)\comp$ is a smashing localization of $\SH(k)\comp$.
\end{theorem}

\begin{proof} Following the proof of \cite{behrens-shah}*{Theorem 8.22}, we need only verify the hypotheses (1)-(3) of \cite{behrens-shah}*{Lemma 5.1} for the full faithfulness assertion. (1) holds by Lemma~\ref{lem:glue}, and (2) holds by Lemma~\ref{lem:zero}. The second part of (3) holds by Proposition~\ref{prop:ret-compare}, while the first part of (3) (and the only part requiring $p$-completion) holds by Corollary~\ref{cor:iet}. The identification of the essential image and the resulting equivalence then follow from Example~\ref{ex:recoll-b}, Theorem~\ref{thm:mot-v-tate}, and the uniqueness of symmetric monoidal structures on a monoidal recollement as determined by the lax symmetric monoidal gluing functor \cite{quigley-shah}*{Proposition~1.26}. We then deduce the last assertion by combining full faithfulness with Theorem~\ref{thm:MonadicityofRealization} (which persists after $p$-completion by \cite{behrens-shah}*{Lemma~3.7}).
\end{proof} 

\begin{remark} \label{compare-behrens-shah} Let $k = \RR$. We note that taking the cellularization of $\Sing^{C_2}$ still yields a fully faithful functor on $p$-complete objects by \cite{behrens-shah}*{Theorem 8.22}. Since there is no reason for the cellularization of a fully faithful functor to remain fully faithful, Theorem~\ref{thm:sing-ff} does not imply that result, which has a substantially different proof and results in a distinct embedding of $(\Sp^{C_2})\comp$ into $\SH(\RR)$.
\end{remark}

\section{Scheiderer motives versus $b$-sheaves with transfers}

We now aim to promote the equivalence of Theorem~\ref{thm:sing-ff} to a larger class of schemes. We begin with some preliminaries. We always have the localization adjunction
\[ L_b: \SH(X) \rightleftarrows \SH_b(X): i_b. \]
By Example~\ref{ex:recoll-b}, we also have a monoidal stable recollement
\begin{equation} \label{eq:b-recoll-sh} \begin{tikzcd}[row sep=4ex, column sep=6ex, text height=1.5ex, text depth=0.25ex]
\SH_{\et}(X) \ar[shift right=1,right hook->]{r}[swap]{j_{\ast}} & \SH_{b}(X) \ar[shift right=2]{l}[swap]{j^{\ast}} \ar[shift left=2]{r}{i^{\ast}} & \SH_{\ret}(X) \ar[shift left=1,left hook->]{l}{i_{\ast}}.
\end{tikzcd} \end{equation}
We denote the gluing functor of this recollement by
\begin{equation} \label{eq:mot}
 \Theta^{\mot}_{X} =L_{\ret}i_{\et}:\SH_{\et}(X) \rightarrow \SH_{\ret}(X)
\end{equation}
and call $\Theta^{\mot}$ the {\bf motivic gluing functor}. We also have the (motivic) unstable\footnote{Note that we do not know if $\theta^{\mot}_{X}$ is the gluing functor for a recollement on $\H_{b}(X)$, since we do not know if the unstable $L_{\ret}: \H_{b}(X) \to \H_{\ret}(X)$ is left-exact (Warning~\ref{warn:UnstableMotivicGluingFunctor}).} and $S^1$-stable variants
\begin{align*}
\theta^{\big}_{X} & = L_{\ret} i_{\et}: \Shv_{\et}(\Sm_X) \rightarrow \Shv_{\ret}(\Sm_X), \\
\theta^{\mot}_{X} & = L_{\ret}i_{\et}:\H_{\et}(X) \rightarrow \H_{\ret}(X), \\
\Theta^{S^1\mot}_{X} & = L_{\ret}i_{\et}:\SH^{S^1}_{\et}(X) \rightarrow \SH^{S^1}_{\ret}(X).
\end{align*}
As we observe in Appendix~\ref{app:sh-top}, these fit together into a lax commutative diagram
\[ \begin{tikzcd}[row sep=6ex, column sep=8ex]
\Shv_{\et}(\Sm_Y) \ar{r}{L_{\AA^1}} \ar{d}{\theta^{\big}_X} \ar[phantom]{rd}{\NEarrow} & \H_{\et}(X) \ar{r}{\Sigma^{\infty}_+} \ar{d}{\theta^{\mot}_{X}} \ar[phantom]{rd}{\NEarrow} & \SH^{S^1}_{\et}(X) \ar{r}{\Sigma^{\infty}_{\GG_m}} \ar{d}{\Theta^{S^1\mot}_{X}} \ar[phantom]{rd}{\NEarrow} & \SH_{\et}(X) \ar{d}{\Theta^{\mot}_{X}} \\
\Shv_{\ret}(\Sm_X) \ar{r}{L_{\AA^1}} & \H_{\ret}(X) \ar{r}{\Sigma^{\infty}_+} & \SH^{S^1}_{\ret}(X) \ar{r}{\Sigma^{\infty}_{\GG_m}} & \SH_{\ret}(X).
\end{tikzcd} \]
In more detail, $\theta^{\mot}_X \simeq L_{\AA^1} \theta^{\big}_X$ upon inclusion into $\ret$-sheaves in view of the discussion at the start of \S\ref{sec:AppendixMotivicPart} (or rather its unstable counterpart), we have that $\Theta^{S^1\mot}_{X}$ is the stabilization\footnote{Note that we do not need to assume the unstable $L_{\ret}: \H_{b}(X) \to \H_{\ret}(X)$ is left-exact. More precisely, we have that $L_{\ret} \simeq \Sp(L_{\ret})$ in the sense for \emph{left} adjoints, so that $\Sigma^{\infty}_+ L_{\ret} \simeq L_{\ret} \Sigma^{\infty}_+$, and $i_{\et} \simeq \Sp(i_{\et})$ in the sense for \emph{left-exact} functors, so we have an exchange transformation $\Sigma^{\infty}_+ i_{\et} \Rightarrow i_{\et} \Sigma^{\infty}_+$.} of $\theta^{\mot}_{X}$, and $\Theta^{\mot}_{X}$ is obtained via monoidal inversion of $\GG_m$ from $\Theta^{S^1\mot}_{X}$ in the sense of Lemma~\ref{lem:InvertedRecollement}.

We next address the behavior of the motivic gluing functor under base change by proving the ``big site" analogue of Proposition~\ref{prop:base-change}. Given a morphism $f: Y \rightarrow X$, we have the ``big site" analogues of the adjunctions displayed in \S\ref{subsec:glue-small}:
\[ f_{\pre}^*: \Pre(\Sm_{X}) \rightleftarrows \Pre(\Sm_Y): f_{*\pre} \]
\[ f_{\ret}^*: \Shv_{\ret}(\Sm_{X}) \rightleftarrows \Shv_{\ret}(\Sm_Y): f_{*\ret} \] 
\[ f_{\et}^*: \Shv_{\et}(\Sm_{X}) \rightleftarrows \Shv_{\et}(\Sm_Y): f_{*\et} \]
and the induced adjunctions on the level of $\H$, $\SH^{S^1}$, and $\SH$. 

% such that $Y$ has finite Krull dimension
\begin{proposition} \label{prop:base-change-big} Let $f: X \rightarrow Y$ be a morphism of schemes. Then the canonical exchange transformations
\begin{align*}
(\chi_0: f_{\ret}^* \widehat{\theta^{\big}} \Rightarrow \widehat{\theta^{\big}} f_{\et}^*) &: \widehat{\Shv}_{\et}(\Sm_Y)  \rightarrow \widehat{\Shv}_{\ret}(\Sm_X) , \\
(\chi: f_{\ret}^*\theta^{\mot} \Rightarrow \theta^{\mot} f_{\et}^*) &: \H_{\et}(Y)  \rightarrow \H_{\ret}(X) , \\
(\chi': f_{\ret}^* \Theta^{S^1\mot} \Rightarrow \Theta^{S^1\mot} f_{\et}^*) &: \SH^{S^1}_{\et}(Y)  \rightarrow \SH^{S^1}_{\ret}(X), \\
(\chi'': f_{\ret}^* \Theta^{\mot} \Rightarrow \Theta^{\mot} f_{\et}^*) &:\SH_{\et}(Y)  \rightarrow \SH_{\ret}(X),
\end{align*}
are equivalences.
\end{proposition}

\begin{proof}
We first prove the claim for $\chi_0$ by a stalkwise argument. For any $X' \in \Sm_X$ and any $\alpha: \Spec k \rightarrow X'$ where $k$ is a real closed field, Lemma~\ref{lem:stalk-compute} gives equivalences:
\[
(\chi_{0\Fscr})_{\alpha}: (f^*_{\ret}\theta^{\big}\Fscr)_{\alpha} \simeq  (\theta^{\big}\Fscr)_{\alpha \circ f} \simeq  \alpha^*_{\et}f_{\et}^*\Fscr(1_\alpha) \simeq (\theta^{\big} f^*_{\et}\Fscr)_{\alpha}.
\]
By the same reasoning as in Proposition~\ref{prop:base-change}, $\chi_0$ is then an equivalence.
% Just like the argument in Proposition~\ref{prop:base-change}, for any $X' \in \Sm_X$ and any $f: \Spec k \rightarrow X'$ where $k$ is a real closed field, Lemma~\ref{lem:stalk-compute} gives equivalences:
% \[
% (\chi_{0\Fscr})_{\alpha}: (f^*_{\ret}\theta^{\big}\Fscr)_{\alpha} \simeq  (\theta^{\big}\Fscr)_{\alpha \circ f} \simeq  \alpha^*_{\et}f_{\et}^*\Fscr(1_\alpha) \simeq (\theta^{\big} f^*_{\et}\Fscr)_{\alpha}.
% \]
% By the same reasoning as in Proposition~\ref{prop:base-change}, the claim holds for $\chi_0$.

%We first prove the claim for $\chi_0$ by reducing to Proposition~\ref{prop:base-change} via a change-of-site argument. For any smooth $X$-scheme $p:T \rightarrow X$ and $\tau \in \{ \et, \ret \}$, denote by
%\[ u^*_{\tau,T}: \Shv_{\tau}(\Sm_T) \rightarrow \Shv_{\tau}(\Et_T) \]
%the functor of restriction to the small site as discussed in Appendix~\ref{bigsmall}. This comes equipped with a fully faithful left adjoint $u^{\tau}_{!T}$ as in Lemma~\ref{lem:big-to-small}. It then suffices to prove that for all $T \in \Sm_X$, the exchange transformation
%\[
%u^*_{\ret,T}p_{\ret}^*f_{\ret}^*\theta^{\big} \Rightarrow u^*_{\ret,T}p_{\ret}^*\theta^{\big} f_{\et}^*
%\]
%is invertible. Without loss of generality, we may assume $X = T$. Now, $L_{\ret}$  and $i_{\et}$ commutes with $u^*_{\ret,Y}$ by Lemma~\ref{lem:big-to-small}, whence the claim holds by Proposition~\ref{prop:base-change}.

The claim for $\chi$ follows since $L_{\AA^1}(\chi_0) \simeq \chi$. The stable claims now follow in succession from the unstable ones. Indeed, by Remark~\ref{rem:PrespectraAndSpectra} all functors in question are computed first pointwise on the level of prespectra and then by applying a localization functor $L_{\mathrm{sp}}$ to get to either $S^1$-spectra in $\H_{\ret}(X)$ or $\GG_m$-spectra in $\SH^{S^1}_{\ret}(X)$. Since equivalences at the level of prespectra are preserved by $L_{\mathrm{sp}}$, if $\chi$ is an equivalence then so is $\chi'$, and likewise if $\chi'$ is an equivalence then so is $\chi''$.
% and $\SH$ (resp. $\SH^{S^1}$) are obtained from the stable $\infty$-category of $\TT$-(resp. $S^1$-)prespectra by inverting the stable equivalences (see ); in particular equivalences on the level of prespectra are preserved.
\end{proof}

\begin{corollary} \label{cor:SixFunctorsFormalismSHb} Let $S$ be a noetherian finite-dimensional base scheme and let $\Sch'_{S}$ be an adequate category of finite-dimensional noetherian $S$-schemes. Then $\SH_b: (\Sch'_{S})^{\op} \to \CAlg(\PrL_{\infty,\stab}) $ satisfies the full six functors formalism.
\end{corollary}
\begin{proof} First note that $\SH_{\et}$ and $\SH_{\ret}$ are known to satisfy the full six functors formalism (the former by Ayoub's work \cite{bachmann-et}*{Theorem~5.1}, and the latter using the finite-dimensional hypothesis and Bachmann's theorem~\eqref{eq:x-ret}). Since $\SH_b(X)$ is the right-lax limit of $\Theta^{\mot}$ (as a symmetric monoidal $\infty$-category), we may conclude by Theorem~\ref{thm:abstract2} if we know that $\Theta^{\mot}$ commutes with base change and that $\SH_b(X)$ is Tate-dualizable for all $X \in \Sch'_S$. We just showed the former in Proposition~\ref{prop:base-change-big}, and the latter is true by design since we invert $\GG_m$ in $\SH_b(X)$ and for the map $\pi_X: \GG_m \times X \to X$, $\pi_{X \sharp}(\1) \simeq \Sigma^{\infty}_+ (\GG_m)$ by construction.
\end{proof}

We now proceed to construct a comparison functor between $\Sp^{C_2}_b$ and $\SH_b$. From hereon, we suppose that our schemes $X$ have $\tfrac{1}{2} \in \Oscr_X$. Let $p: X[i] \to X$ denote the $C_2$-Galois cover.
% From hereon, we suppose our schemes $X$ have finite Krull dimension and that $\tfrac{1}{2} \in \Oscr_X$. 

\begin{definition} An {\bf \'etale induced object} in $\SH_{\et}(X)$ is an object in the essential image of the functor $p_{\sharp}: \SH_{\et}(X[i]) \rightarrow \SH_{\et}(X)$.
% \[
% f_{\sharp}: \SH(X[i]) \rightarrow \SH(X), \quad \text{ resp. } \quad 
% f_{\sharp}: \SH_{\et}(X[i]) \rightarrow \SH_{\et}(X).
% \]
\end{definition}

% We note that since $\SH_{\ret}(X) \simeq \Sp(\widetilde{X}_{\ret})$ and $\widetilde{X[i]}_{\ret} \simeq *$, we get that $\SH_{\ret}(X[i]) \simeq 0$.
We note that since $\Shv_{\ret}(\Sm_{Y}) \simeq \ast$ if $\sqrt{-1} \in Y$, we have that $\SH_{\ret}(X[i]) \simeq \ast$.

\begin{lemma} \label{lem:kills-ind} The functor $\Theta^{\mot}:\SH_{\et}(X) \rightarrow \SH_{\ret}(X)$ vanishes on \'etale induced objects. 
\end{lemma}

% \todo{do we need this? I thought this wasn't true until we p-complete for SH_et}
\begin{proof} Since we have the ambidexterity equivalence $p_* \simeq p_{\sharp}$ for any premotivic functor with the full six functors formalism, we get
\begin{eqnarray*}
\Theta^{\mot} p_{\sharp}E & = & L_{\ret}i_{\et} p_{\sharp}E\\
 & \simeq & L_{\ret}i_{\et} p_*E \\
& \simeq & L_{\ret}p_*i_{\et}E\\
& \simeq & L_{\ret}p_{\sharp}i_{\et}E.
\end{eqnarray*}
The last spectrum is zero. Indeed, this follows from the commutative diagram of left adjoints
\[
\begin{tikzcd}
\SH(X[i]) \ar[swap]{d}{p_{\sharp}} \ar{r}{L_{\ret}} & \SH_{\ret}(X[i])  \simeq \ast \ar{d}{p_{\sharp}}\\
\SH(X) \ar{r}{L_{\ret}} & \SH_{\ret}(X).
\end{tikzcd}
\]
%AsThe top right corner is zero that since $\SH_{\ret}(X) \simeq \Sp(\widetilde{X}_{\ret})$ by the equivalence~\eqref{eq:x-ret}, and $\widetilde{X[i]}_{\ret} \simeq *$, i.e., it is equivalent to the initial $\infty$-topos. This proves that $L_{\ret}$ vanishes on induced objects. The claim follows.
\end{proof}

\begin{remark} Stabilizing Construction~\ref{construct:borel-orb} applied to Example~\ref{exmp:et-cover} gives us a functor $\pi_*:\Sp(\widehat{X[i]}_{\et}) \rightarrow \Sp(\widehat{X}_{\et})$. After $p$-completion and applying Bachmann's theorem~\eqref{eq:x-et}, the functor $\pi_*$ agrees with the functor $p_{\sharp}$. In this way, the notion of \'etale induced objects is consistent with our earlier notion of induced objects (Definition~\ref{def:topos-induced}); in fact, we have an integral compatibility as recorded by Lemma~\ref{lem:toposvshind}. Lemma~\ref{lem:kills-ind} is then the motivic analogue of the vanishing result of Lemma~\ref{lem:killsinduced}.
% This will be an input to Construction~\ref{construct:compare}. 
\end{remark}

% Let $X$ be a scheme with $\tfrac{1}{2} \in \Oscr_X$. Consider the functor
\begin{lemma} \label{lem:toposvshind} If we let $\pi_*:\Sp(\widehat{X[i]}_{\et}) \rightarrow \Sp(\widehat{X}_{\et})$ be the functor of Construction~\eqref{construct:borel-orb}, then the diagram
\[
\begin{tikzcd}
\Sp(\widehat{X[i]}_{\et}) \ar{r}{\pi_*} \ar{d}[swap]{\Sigma^{\infty}_{\GG_m,\et} u^{\et}_{X[i]!}} &  \Sp(\widehat{X}_{\et}) \ar{d}{\Sigma^{\infty}_{\GG_m,\et} u^{\et}_{X!}}\\
\SH_{\et}(X[i]) \ar{r}{p_{\sharp}} & \SH_{\et}(X)
\end{tikzcd}
\]
commutes. In other words, $\Sigma^{\infty}_{\GG_m,\et} u^{\et}_{X!}$ carries induced objects to \'etale induced objects.
\end{lemma}

\begin{proof} All functors in sight commute with colimits, and the claim is obvious on the level of the representable objects that generate under colimits.
\end{proof}

\begin{construction} \label{construct:compare} Consider the following concatenation of lax commutative squares
\begin{equation} \label{eq:big-small}
\begin{tikzcd}[column sep=6ex]
\Sp(\widehat{X}_{\et}) \ar{r}{L_{\ret}i_{\et}} \ar[swap]{d}{u^{\et}_{X!}} \ar[phantom]{rd}{\SWarrow} & \Sp(\widehat{X}_{\ret}) \ar{d}{u^{\ret}_{X!}}  \\
\SH^{S^1}_{\et}(X) \ar{r}{L_{\ret}i_{\et}} \ar[swap]{d}{\Sigma^{\infty}_{\GG_m,\et}} \ar[phantom]{rd}{\SWarrow} & \SH^{S^1}_{\ret}(X) \ar{d}{\Sigma^{\infty}_{\GG_m,\ret}}   \\
\SH_{\et}(X) \ar{r}{L_{\ret}i_{\et}} & \SH_{\ret}(X).
\end{tikzcd}
\end{equation}
Here, the top lax square is given by stabilizing Construction~\ref{construct:laxbigsmall} and the bottom lax square is given by Lemma~\ref{lem:InvertedRecollement}. Let $\gamma^{\et}_X = \Sigma^{\infty}_{\GG_m,\et}u^{\et}_{X!}$ and $\gamma^{\ret}_X = \Sigma^{\infty}_{\GG_m,\ret}u^{\ret}_{X!}$. We thus obtain a lax symmetric monoidal transformation (using Remark~\ref{rem:ushriek} for strong monoidality of the $u_!$ functors)
\begin{equation} \label{eq:s1-t}
\gamma^{\ret}_X \Theta \Rightarrow \Theta^{\mot} \gamma^{\et}_X.
\end{equation}
Moreover, $\Theta^{\mot} \gamma^{\et}_X$ vanishes on induced objects by Lemma~\ref{lem:toposvshind} and Lemma~\ref{lem:kills-ind}. If $X$ has finite Krull dimension, so that $\gamma^{\ret}_X$ is an equivalence by Bachmann's theorem~\eqref{eq:x-ret}, then Proposition~\ref{prop:ThetaTateMonoidalUniversalProperty} directly applies to produce a comparison lax symmetric monoidal transformation $T_X: \Theta^{\Tate} \Rightarrow \Theta^{\mot} \gamma^{\et}_X$. In the general case, we note that the logic of the proof also applies to show preservation of the universal property upon whiskering with a colimit-preserving symmetric monoidal functor such as $\gamma^{\ret}_X$. We thereby obtain a lax symmetric monoidal transformation

% We claim that precomposing this transformation with the norm map 
% \[
% \Sigma^{\infty}_{\GG_m,\ret}u^{\ret}_{X!}(-)_{hC_2}\nu(C_2)^* \Rightarrow \Sigma^{\infty}_{\GG_m,\ret}u^{\ret}_{X!}\Theta
% \] is nullhomotopic. Indeed this is the case on \'etale induced objects by Lemma~\ref{lem:kills-ind} and~\ref{lem:toposvshind}. But now, $\Sigma^{\infty}_{\GG_m,\ret}u^{\ret}_{X!}(-)_{hC_2}\nu(C_2)^*$ preserves colimits  since it is a composite of left adjoints, and $\Sp(\widehat{X}_{\et})$ is generated under colimits by induced objects by Lemma~\ref{lem:generated}, hence the transformation is indeed nullhomotopic.

% Using the universal property of $\Theta^{\Tate}$ from Proposition~\ref{prop:ThetaTateMonoidalUniversalProperty} and the strong monoidality of the $u_!$ functors from Remark~\ref{rem:ushriek}, we obtain a lax monoidal transformation\footnote{Unique up to a contractible space of choices.}
\begin{equation} \label{eq:sh-t}
T_X: \gamma^{\ret}_X \Theta^{\Tate} \Rightarrow \Theta^{\mot} \gamma^{\et}_X, 
\end{equation}
or equivalently, a lax commutative square
\[
\begin{tikzcd}
\Sp(\widehat{X}_{\et}) \ar{r}{\Theta^{\Tate}} \ar{d}[swap]{\gamma^{\et}_X} \ar[phantom]{rd}{\SWarrow T_X} & \Sp(\widehat{X}_{\ret}) \ar{d}{\gamma^{\ret}_X}  \\
\SH_{\et}(X) \ar{r}{\Theta^{\mot}} & \SH_{\ret}(X).
\end{tikzcd}
\]
In view of the discussion below Proposition~\ref{prop:ThetaTateMonoidalUniversalProperty} and using that the vertical functors are strong symmetric monoidal, upon taking right-lax limits we obtain a strong symmetric monoidal functor
\[
C_X: \Sp^{C_2}_b(X) \rightarrow \SH_b(X),
\]
where we have identified the right-lax limit of $\Theta^{\mot}$ by Example~\ref{ex:recoll-b} and the right-lax limit of $\Theta^{\Tate}$ by Theorem~\ref{prop:recoll-stab}.

Now let $f:X \to Y$ be a morphism of schemes and consider the lax commutative diagrams
\[ \begin{tikzcd}[column sep=6ex]
\Sp(\widehat{Y}_{\et}) \ar{r}{\Theta^{\Tate}} \ar{d}[swap]{\gamma_Y^{\et}} \ar[phantom]{rd}{\SWarrow T_Y} & \Sp(\widehat{Y}_{\ret}) \ar{d}{\gamma_Y^{\ret}} \\
\SH_{\et}(Y) \ar{d}[swap]{f^*} \ar{r}{\Theta^{\mot}} \ar[phantom]{rd}{\SWarrow \simeq} & \SH_{\ret}(Y) \ar{d}{f^*}  \\
\SH_{\et}(X) \ar{r}{\Theta^{\mot}} & \SH_{\ret}(X)
\end{tikzcd}, \qquad
\begin{tikzcd}[column sep=6ex]
\Sp(\widehat{Y}_{\et}) \ar{r}{\Theta^{\Tate}} \ar{d}[swap]{f^*} \ar[phantom]{rd}{\SWarrow \simeq} & \Sp(\widehat{Y}_{\ret}) \ar{d}{f^*} \\
\Sp(\widehat{X}_{\et}) \ar{r}{\Theta^{\Tate}} \ar{d}[swap]{\gamma_X^{\et}} \ar[phantom]{rd}{\SWarrow T_X} & \Sp(\widehat{X}_{\ret}) \ar{d}{\gamma_X^{\ret}} \\
\SH_{\et}(X) \ar{r}{\Theta^{\mot}} & \SH_{\ret}(X),
\end{tikzcd} \]
where for the equivalences we use that $\Theta^{\Tate}$ and $\Theta^{\mot}$ are both stable under base change by Corollary~\ref{cor:tate-stable} and Proposition~\ref{prop:base-change-big} respectively. We then claim that the lax symmetric monoidal transformations
\begin{align*} f^* \gamma^{\ret}_Y \Theta^{\Tate} \Rightarrow \Theta^{\mot} f^* \gamma^{\et}_Y, \qquad \gamma^{\ret}_X f^* \Theta^{\Tate} \Rightarrow \Theta^{\mot} \gamma^{\et}_X f^*
\end{align*}
are canonically homotopic under the equivalences $f^* \gamma^{\ret}_Y \simeq \gamma^{\ret}_X f^*$, $f^* \gamma^{\et}_Y \simeq \gamma^{\et}_X f^*$. Indeed, this again follows from the universal property of $\Theta \Rightarrow \Theta^{\Tate}$ and its stability under whiskering with a colimit-preserving strong symmetric monoidal functor, as well as the fact that $f^*: \Sp(\widehat{Y}_{\et}) \to \Sp(\widehat{X}_{\et})$, resp. $f^*: \SH_{\et}(Y) \to \SH_{\et}(X)$ preserves induced objects, resp. \'etale induced objects. Upon taking right-lax limits, we then obtain a commutative square of strong symmetric monoidal functors
\[ \begin{tikzcd}[row sep=4ex, column sep=6ex]
\Sp_b^{C_2}(Y) \ar{r}{C_Y} \ar{d}{f^*} & \SH_b(Y) \ar{d}{f^*} \\
\Sp_b^{C_2}(X) \ar{r}{C_X} & \SH_b(X).
\end{tikzcd} \]
\end{construction}

% remove adjectives  `adequate' and `premotivic'
% \item $C$ is a premotivic transformation.
% Let $S$ be a finite-dimensional base scheme such that $\tfrac{1}{2} \in \Oscr_S$, and let $\Sch^{\mathrm{fd}}_S$ be the category of finite-dimensional $S$-schemes.
\begin{theorem} \label{cor:sch-1} The functors $C_X$ of Construction~\ref{construct:compare} assemble to a strong symmetric monoidal transformation \[ C:\Sp^{C_2}_b \Rightarrow \SH_b \] of functors $(\Sch[\tfrac{1}{2}])^{\op} \rightarrow \CAlg(\PrL_{\infty,\stab})$. Moreover, if $p$ is a prime and $X$ is a locally $p$-\'etale finite scheme with $1/p \in \Oscr_X$, then $(C_X)\comp$ is an equivalence.
% \[
% (C:\Sp^{C_2}_b \Rightarrow \SH): (\Sch'_S)^{\op} \rightarrow \CAlg\PrL_{\infty,\stab}
% \]
% \begin{enumerate}
% \item Let $p$ be a prime and suppose that $X$ is a locally $p$-\'etale finite $S$-scheme with $1/p \in \Oscr_X$. Then $C\comp$ is an equivalence.
% \end{enumerate}
 \end{theorem}

\begin{proof} For this proof, we denote $S:=(\Sch[\tfrac{1}{2}])^{\op}$. To assemble the transformation $C$, we employ the same strategy as in Lemma~\ref{lem:TateMonoidalInFamilies} and we assume the notation introduced in that proof. The pairing construction furnishes a locally cocartesian fibration
\[
\widetilde{\Fun}{}^{\otimes,\lax}(\Sp(\widehat{-}_{\et}),\SH_{\ret}(-)) \rightarrow S,
\]
equipped with sections
\[
\gamma^{\ret}\Theta, \gamma^{\ret}\Theta^{\Tate}, \Theta^{\mot}\gamma^{\et}
\]
which fiberwise restricts to the functors $\Sp(\widehat{X}_{\et}) \rightarrow \SH_{\ret}(X)$ of the same name. We first construct a section
\[
\begin{tikzcd}[column sep=6ex]
S \times \Delta^1 \ar{r}{q} \ar{rd}[swap]{\mathrm{pr}_S} & \widetilde{\Fun}{}^{\otimes,\lax}(\Sp(\widehat{-}_{\et}),\SH_{\ret}(-))\ar{d}\\
 & S,\\
\end{tikzcd}
\]
classifying a transformation
\[
\gamma^{\ret}\Theta \Rightarrow \Theta^{\mot}\gamma^{\et},
\]
which is fiberwise given by the lax commutative square~\eqref{eq:big-small}. We begin by noting that have the following commutative diagram of cocartesian fibrations over $S \times \Fin_{\ast}$
\[
\begin{tikzcd}
\Sp(\widehat{-}_{b})^{\otimes} \ar{r}{L_{\tau}}\ar[swap]{d}{\gamma_b} & \Sp(\widehat{-}_{\tau})^{\otimes} \ar{d}{\gamma_{\tau}}\\
\SH_b(-)^{\otimes} \ar{r}{L_{\tau}} & \SH_{\tau}(-)^{\otimes},
\end{tikzcd}
\]
where $\tau = \ret, \et$. The theory of relative adjunctions (specifically, using its formulation in terms of unit and counit maps \cite{higheralgebra}*{Proposition 7.3.2.1}) yields a lax commutative mate square
\[
\begin{tikzcd}
\Sp(\widehat{-}_{b})^{\otimes} \ar[swap]{d}{\gamma_b}  \ar[phantom]{rd}{\SEarrow} & \Sp(\widehat{-}_{\tau})^{\otimes} \ar{d}{\gamma_{\tau}} \ar[swap]{l}{i_{\tau}}\\
\SH_b(-)^{\otimes} & \SH_{\tau}(-)^{\otimes} \ar{l}{i_{\tau}},
\end{tikzcd}
\]
where the maps $i_{\tau}$ are fiberwise only lax symmetric monoidal. Concatenating the strict commutative square for $\tau=\ret$ with the lax commutative square for $\tau=\et$ yields a lax commutative square over $S \times \Fin_{\ast}$ that globalizes~\eqref{eq:big-small}.

Now, we have the full subcategory $\iota: \widetilde{\Fun}{}^{\otimes,\lax}_0(\Sp(\widehat{-}_{\et}),\SH_{\ret}(-)) \subset \widetilde{\Fun}{}^{\otimes,\lax}(\Sp(\widehat{-}_{\et}),\SH_{\ret}(-))$ on those lax symmetric monoidal functors that annihilate induced objects. As in the proof of Lemma~\ref{lem:TateMonoidalInFamilies}, this restricts fiberwise over $X \in S$ to a Bousfield localization (bearing in mind that we implicitly restrict to accessible functors) that admits a left adjoint
\[ L_X: \Fun^{\otimes, \lax}(\Sp(\widehat{X}_{\et}),\SH_{\ret}(X)) \to \Fun^{\otimes,\lax}_0(\Sp(\widehat{X}_{\et}),\SH_{\ret}(X)). \]
Moreover, the locally cocartesian pushforward functor for $\widetilde{\Fun}{}^{\otimes,\lax}(\Sp(\widehat{-}_{\et}),\SH_{\ret}(-))$ over $f: X \to Y$ is given by the composite of left adjoints
\[ \begin{tikzcd}
\Fun^{\otimes, \lax}(\Sp(\widehat{Y}_{\et}),\SH_{\ret}(Y)) \ar{r}{f^*_{\ret} \circ } & \Fun^{\otimes, \lax}(\Sp(\widehat{Y}_{\et}),\SH_{\ret}(X)) \ar{r}{(f^*_{\et})_!} & \Fun^{\otimes, \lax}(\Sp(\widehat{X}_{\et}),\SH_{\ret}(X))
\end{tikzcd}
 \]
where $f^*_{\ret} \circ$ is postcomposition by $f^*_{\ret}: \SH_{\ret}(Y) \to \SH_{\ret}(X)$ and is left adjoint to $f_{*\ret} \circ$, and $(f^*_{\et})_!$ is operadic left Kan extension along $f^*_{\et}: \Sp(\widehat{Y}_{\et}) \to \Sp(\widehat{X}_{\et})$ and is left adjoint to restriction $(f^*_{\et})^*$ along $f^*_{\et}$. Since $f^*_{\et}$ preserves induced objects, it follows that $(f_{*\ret} \circ) \circ (f^*_{\et})^* $ preserves $L$-local objects and hence $(f^*_{\et})_! \circ (f^*_{\ret} \circ)$ preserves $L$-equivalences. Therefore, the relative adjoint functor theorem \cite{higheralgebra}*{Proposition~7.3.2.11} applies to show that the inclusion $\iota$ admits a relative left adjoint
\[ L: \widetilde{\Fun}{}^{\otimes,\lax}(\Sp(\widehat{-}_{\et}),\SH_{\ret}(-)) \to \widetilde{\Fun}{}^{\otimes,\lax}_0(\Sp(\widehat{-}_{\et}),\SH_{\ret}(-)). \]
We then see that $L \circ q$ corresponds to the desired transformation
\[ \gamma^{\ret} \Theta^{\Tate} \Rightarrow \Theta^{\mot} \gamma^{\et}. \]
Indeed, by definition the fiber over $X \in S$ of this transformation yields the transformation $T_X$ of Construction~\ref{construct:compare}. Formation of right-lax limits in the parametrized sense of Remark~\ref{rem:rlaxlim-param} now yields a functor $C^{\otimes}: \Sp^{C_2}_b(-)^{\otimes} \to \SH_b(-)^{\otimes}$ over $S \times \Fin_*$ that preserves cocartesian edges, and the straightening of $C^{\otimes}$ defines the transformation $C: \Sp^{C_2}_b \Rightarrow \SH_b$ of interest.

% \todo{Use prior strategy to correctly assemble functors.}
% Let $B = \Spec \ZZ\twoinv$ and consider $T_B$ in Construction~\ref{construct:compare}. For any scheme $S$ with $\tfrac{1}{2} \in \Oscr_S$, we have the structure morphism $p_S: S \rightarrow B$. We then consider 
% \[
% p_S^*T_B: f^*(\Sigma^{\infty}_{\GG_m,\ret}u^{\ret}_{B!}\Theta_B^{\Tate} \Rightarrow \Theta_B^{\mot}\Sigma^{\infty}_{\GG_m}u^{\et}_{B!}).
% \]
% Because the functors $\Theta^{\Tate}$ and $\Theta^{\mot}$ are both stable under base change by Corollary~\ref{cor:tate-stable} and Proposition~\ref{prop:base-change-big} respectively, we get that the transformation above takes the form
% \[
% p_S^*T_B: \Sigma^{\infty}_{\GG_m,\ret}u^{\ret}_{S!}\Theta_S^{\Tate} \Rightarrow \Theta_S^{\mot}\Sigma^{\infty}_{\GG_m}u^{\et}_{S!},
% \]
% and we have a canonical equivalence $p_S^*T_B \simeq T_S$ as lax symmetric monoidal transformations in view of the universal property of $T_S$ as explicated in Construction~\ref{construct:compare}. Upon taking right-lax limits, $p_S^*T_B$ thus defines the strong symmetric monoidal functor 
% \[
% C_S: \Sp^{C_2}_b(S) \rightarrow \SH_b(S),
% \]
% which is stable under base change (since $p_S^*T_B$ is defined via pullback from $B$). Using the unstraightening technique around Theorem~\ref{thm:MonoidalityOfGenuineStabilizationFunctor}, the compatible collection $\{p_S^* T_B\}$ then defines the desired $C$.

Next, suppose $X$ is a locally $p$-\'etale finite scheme with $\tfrac{1}{p} \in \Oscr_X$. Restricting the transformation $T_X$ to $p$-complete objects, we get a diagram
\[
\begin{tikzcd}[row sep=6ex,column sep=8ex]
\Sp(\widehat{X}_{\et})\comp \ar[swap]{d}{\simeq} \ar{r}{(\Theta^{\Tate})\comp} \ar[phantom]{rd}{\SWarrow (T_X)\comp} & \Sp(\widetilde{X}_{\ret})\comp \ar{d}{\simeq} \\ 
\SH_{\et}(X)\comp \ar{r}[swap]{(\Theta^{\mot})\comp} & \SH_{\ret}(X)\comp,
\end{tikzcd}
\]
where we have applied~\eqref{eq:x-ret} and~\eqref{eq:x-et} to deduce the vertical equivalences. Therefore, to prove the result, it suffices to prove that $(T_X)\comp$ is invertible. For this, let $E \in \Sp(\widehat{X}_{\et})\comp$ and consider the map 
\[
T_X(E)\comp:(\Theta^{\Tate})\comp(E) \rightarrow (\Theta^{\mot})\comp(E).
\]
% Let $C_E$ denote the cofiber of $T_X(E)\comp$. , which we can regard as an object in $\SH_{\ret}(X)$
Since $\SH_{\ret}(X) \simeq \Sp(\widetilde{X_{\ret}})$ and $\widetilde{X_{\ret}}$ is hypercomplete, it suffices to check that $T_X(E)\comp$ is an equivalence after base change to all real closed points $\alpha: \Spec k \to X$. But since $T\comp$ is stable under base change, we have $\alpha^* T_X(E)\comp \simeq T_{k}(\alpha^* E)\comp$. Appealing once more to the universal property of $T_k$, it also identifies upon $p$-completion with the canonical equivalence $((-)^{tC_2})\comp \simeq (L_{\ret} i_{\et})\comp$ of Theorem~\ref{thm:mot-v-tate}, so we are done.
% Since $\SH_{\ret}$ satisfies the six functors formalism, it suffices to check that for all $\alpha:\Spec k \rightarrow S$, the pullback $\alpha^*C_E$ is null (see, for example, \cite{bachmann-ret}*{Corollary 14}). We can assume that $k$ is of characteristic zero since otherwise $\SH_{\ret}(k) \simeq \Sp(k_{\ret})$ is trivial. By real \'etale descent, if $k$ is characteristic zero field and $k'$ is its real \'etale closure, the functor $\SH_{\ret}(k) \rightarrow \SH_{\ret}(k')$ is fully faithful. Hence we may further assume that $k$ is real closed. Now, since the transformation $T$ is stable under base change, the transformation $T\comp$ is as well and so $\alpha^*C_E \simeq C_{\alpha^*E}$. This spectrum is now null from the case of real closed fields as proved in Theorem~\ref{thm:mot-v-tate}.
\end{proof}

\begin{remark} We note that Theorem~\ref{cor:sch-1} in conjunction with Corollary~\ref{cor:SixFunctorsFormalismSHb} gives another proof that $(\Sp^{C_2}_b)\comp$ satisfies the full six functors formalism when restricted to $(\Sch^{p\mbox{-}\fin,\mathrm{noeth}}_{S}[\tfrac{1}{p}, \tfrac{1}{2}])^{\op}$ for $S$ a noetherian locally $p$-\'etale finite scheme (Theorem~\ref{construct:genuine}), which in particular bypasses the explicit verification of Tate-dualizability for $(\Sp^{C_2}_b)\comp$ (Lemma~\ref{lem:duals}). However, the proof of Lemma~\ref{lem:duals} still yields more information (cf. Remark~\ref{rem:c2}).
\end{remark}

\section{Applications} \label{sec:apps}

% We collect a few quick applications of the theory developed in this paper. 
% Fix a prime $p$ and let $X$ be a locally $p$-\'etale finite scheme with $\tfrac{1}{2}, \tfrac{1}{p} \in \Oscr_X$.

\subsection{Spectral $b$-rigidity} \label{b-rigid}

 % --- indeed, they are different in precisely the sense that naive $C_2$-equivariant stable homotopy theory differs from genuine $C_2$-equivariant stable homotopy theory.
% Consider the small $b$-site $(\Et_X, b)$ and the associated $\infty$-category of sheaves of spectra on it. Unlike the situations studied in \cite{bachmann-ret} and \cite{bachmann-et}, it is \emph{not} the case that the $p$-completion of $\Sp(\widehat{X}_{b})$ is equivalent to the $p$-completion of $\SH_b(X)$. Rather, the distinction is a matter of adjoining certain transfers to $b$-sheaves on the small site. To formulate our version of $b$-rigidity, note that if $\tfrac{1}{2} \in X$, then the $C_2$-$\infty$-category $(\widetilde{X[i]}_{\et})_{(-)}$ evaluates to $\widetilde{X}_b$, the $\infty$-topos of $b$-sheaves on the $\Et_X$ on the object $C_2/C_2 \in \Oscr_{C_2}$. This follows by Scheiderer's result recalled in Example~\ref{ex:b}. We set
% \[
% \underline{\widetilde{X}_b}:= (\widetilde{X[i]}_{\et})_{(-)}.
% \]
% We fiber-stabilize this to $\underline{\Sp}(\underline{\widetilde{X}}_b)$ and take its hypercomplete version $\underline{\Sp}(\underline{\widehat{X}}_b)$. The next theorem is our version of $b$-rigidity.

% and in the context of \'etale motives by Suslin-Voevodsky (ref), Ayoub (ref), and Cisinski-D\'eglise (ref)
We explain how to interpret Theorem~\ref{cor:sch-1} as a spectral ``rigidity'' result for the $b$-topology, along the lines studied in the \'etale and real \'etale settings by Bachmann~\cite{bachmann-et, bachmann-ret}. Recall the notation $\underline{\widehat{X}}_b$ from \S\ref{genuine-six}, and the equivalence $\underline{\widehat{X}}_b \simeq \underline{\Xscr}_{C_2}$ for $\Xscr = \widehat{X[i]}_{\et}$ (the hypercomplete version of Remark~\ref{rem:SameBaseChangeObviousRemark}).

\begin{theorem} \label{thm:sch-c2} Let $X$ be a locally $p$-\'etale finite scheme with $\tfrac{1}{2}, \tfrac{1}{p} \in \Oscr_X$. Then there is a canonical equivalence of $\infty$-categories\footnote{We have not explicated a construction of the symmetric monoidal structure on parametrized Mackey functors in this paper.}
% This will follow from forthcoming work of the second author and Nardin on the parametrized Day convolution in the context of $G$-$\infty$-operads.
\[
\Fun^{\times}_{\C_2}(\underline{\A}^{\eff}(\Fin_{C_2}), \underline{\Sp}(\underline{\widehat{X}}_b))\comp \simeq \SH_b(X)\comp.
\]
\end{theorem}

\begin{proof} After Lemma~\ref{lem:interchange} and Theorem~\ref{prop:recoll-stab}, the result follows immediately from Theorem~\ref{cor:sch-1}.
\end{proof}

The meaning of Theorem~\ref{thm:sch-c2} is this: after $p$-completion and with mild finiteness hypotheses, the sole distinction between Scheiderer motives and $b$-sheaves of spectra (on the small site!) lies in adjoining certain transfer maps, which serve to repair the obvious obstruction presented by the failure of Tate-dualizability. Let us now elaborate the structure of these transfers.
% The next two remarks unpack the lefthand side of the equivalence in Theorem~\ref{thm:sch-c2} and explains why the above theorem is a form of ``rigidity."

\begin{remark} \label{rem:concrete} An object in $\Fun^{\times}_{\C_2}(\underline{\A}^{\eff}(\Fin_{C_2}), \underline{\Sp}(\widehat{X}_b))$ is concretely given as follows: it is the data of a hypercomplete $b$-sheaf of spectra $E \in \Sp(\widehat{X}_b)$, together with single transfer map
\[
\mathrm{tr}: \overline{\pi}_*\overline{\pi}^*E \rightarrow E,
\]
and the data of relations as encoded by the effective Burnside category. Here, the adjunction
\[
\overline{\pi}^*: \Sp(\widehat{X}_b) \rightleftarrows  \Sp(\widehat{X[i]}_b) \simeq \Sp(\widehat{X[i]}_{\et}): \overline{\pi}_*
\]
follows our customary notation (first introduced around Lemma~\ref{lem:j-open}), but also coincides with the pullback-pushforward adjunction for sheaves.
\end{remark}

\begin{remark} \label{rem:concrete2} We can further unwind the transfer map $\mathrm{tr}: \overline{\pi}_*\overline{\pi}^*E \rightarrow E$ in terms of the real \'etale and the \'etale parts of a $b$-sheaf. Let us place ourselves in the context of the (stabilization of the) recollement given in Proposition~\ref{prop:recoll}. Given $E \in \Sp(\widehat{X}_b)$, consider the \'etale and real \'etale parts
\[
F := j^*E \in  \Sp(\widehat{X}_{\et}), \qquad G:=i^*E \in \Sp(\widetilde{X}_{\ret}).
\]
We also have an ambidextrous adjunction
\[
\pi^*:\Sp(\widehat{X}_{\et}) \rightleftarrows \Sp(\widehat{X[i]}_{\et}):\pi_*.
\]

Now, using the ambidexterity equivalence $\pi_! \simeq \pi_*$, the transfer restricted to the \'etale site coincides with the counit map
\[
\epsilon:\pi_*\pi^*F \simeq \pi_!\pi^*F \rightarrow F.
\]
In particular, the \'etale component of the transfer is no additional data. By contrast, the real \'etale component is specified by a map
\[
\mathrm{tr}_{\ret}:\Theta\pi_*\pi^*F  \rightarrow G,
\]
subject to the commutativity constraint
\[
\begin{tikzcd}
\Theta\pi_*\pi^*F \ar{d}[swap]{\mathrm{tr}_{\ret}} \ar{dr}{\Theta(\epsilon)} & \\
G \ar{r} & \Theta F,
\end{tikzcd}
\]
where the bottom arrow arises from the gluing datum of the $b$-sheaf $E$ (in terms of the recollement, this is given by the transformation $i^* E \rightarrow i^*j_*j^* E$). Note also that $\Theta \pi_* \simeq \nu^*$.

In the simplest case of $X = \Spec k$ for a real closed field $k$, so that $E$ is a spectral presheaf on $\Oscr_{C_2}$ that underlies a genuine $C_2$-spectrum, we see that the map $\mathrm{tr}_{\ret}: \nu^* \pi^* F \to G$ is precisely the usual transfer map $t: E^{e} \to E^{C_2}$, the map $G \to \Theta F$ is the inclusion of fixed points $r: E^{C_2} \to E^{e}$, and the constraint amounts to the relation $r \circ t \simeq \sigma + \id$ for $\sigma$ the $C_2$-action on $E^{e}$.
\end{remark}

\subsection{Parametrized $C_2$-realization functor} \label{realization} We append the next theorem to the long list of realization functors already constructed in stable motivic homotopy theory. Again, let $X$ be a locally $p$-\'etale finite scheme with $\tfrac{1}{2}, \tfrac{1}{p} \in \Oscr_X$.

\begin{theorem} \label{thm:realization} There is a colimit preserving and strong symmetric monoidal {\bf parametrized $C_2$-Delfs-Knebusch realization functor}
\[
\DK_X^{C_2}\comp: \SH(X) \rightarrow \Sp^{C_2}_b(X)\comp,
\]
which agrees with the $p$-completion of the $C_2$-Delfs-Knebusch realization functor $\DK^{C_2}$ of \S\ref{betti-c2} whenever $X$ is the spectrum of a real closed field.
\end{theorem}
% The functor $\DK_X^{C_2}\comp$ admits a right adjoint $\Sing\comp$.

\begin{proof} After Corollary~\ref{thm:sch-c2}, this functor is given by applying $L_b$ and then $p$-completing, both of which are strong symmetric monoidal functors.
\end{proof}

\begin{remark} The functor of Theorem~\ref{thm:realization} defines the $p$-complete part of a functor whose existence was conjectured by Bachmann-Hoyois \cite{bachmann-hoyois}*{Remark 11.8}. We defer the construction of an integral version of this functor, as well as an appropriately compatible norm structure in their sense, to a future work.
\end{remark}

% In a future work, we aim to describe the functor $\DK_X^{C_2}\comp$ as given by inverting both the element $\rho$ (which gives real \'etale descent by \cite{bachmann-ret}) and the ``Bott element" $\tau$ (which gives \'etale descent by the forthcoming \cite{beo}).
\begin{remark} In \cite{behrens-shah}, Behrens and the second author gave an explicit formula for $\Be^{C_2}\comp$ when restricted to cellular real motivic spectra, as a certain combination of the operations of $\rho$-inversion and $\tau$-inversion on the $\rho$-completion, where $\tau$ is the ``spherical Bott element".\footnote{More precisely, we must take an inverse limit of $C(\rho^n)[\tau_N^{-1}]$ for self-maps $\tau_N$ on $C(\rho^n)$ to define $\1\comrho[\tau^{-1}]$; $\tau$ as a self-map on $\1\comrho$ itself does not exist as the periodicity lengths of the self-maps $\tau_N$ go to $\infty$ as $n \to \infty$.} In \cite{beo}, Bachmann, \O stva\ae r, and the first author will understand \'etale descent in terms of $\tau$-inversion. In a future work, we aim to unite these two perspectives to give an explicit formula for $\DK_X^{C_2}\comp$ in terms of $\rho$ and $\tau$. In particular, such a description would enable us to show that the right adjoint $\Sing^{C_2}_X \comp$ to $\DK_X^{C_2}\comp$ preserves colimits. Conditional on this, one can then enhance the equivalence of Theorem~\ref{thm:MonadicityofRealization} after $p$-completion using \cite{e-kolderup}*{Theorem 5.5}: if $X$ is a regular scheme over a field $k$ of characteristic zero, then we expect a natural equivalence
\[
\Sp^{C_2}_b(X)\comp \simeq \Mod_{\Sing^{C_2}_X \comp \DK_X^{C_2}\comp(\1)}(\SH(X)\comp).
\]
Since $\Sp^{C_2}_b(X)\comp$ also embeds fully faithfully into $\SH(X)\comp$ as $\SH_b(X)\comp$, we would then deduce that $\SH_b(X)\comp$ is a smashing localization of $\SH(X)\comp$, thereby extending Theorem~\ref{thm:sing-ff}.
\end{remark}

%\begin{remark} \label{rem:realization} In general, there is a zig-zag:
%\[
%\SH(X) \rightarrow \SH_b(X) \leftarrow \Sp^{C_2}(X),
%\]
%and our main result says that under suitable hypothesis on $X$, the second arrow is an equivalence after $p$-completion. Suppose that $X$ is variety over $\RR$. The $p$-completion of $\Sp_b^{C_2}(X)$ should be the $p$-complete part of a theory of ``genuine $C_2$-spectra" over the topological space $X(\CC)^{\mathrm{an}}$, $\Sp^{C_2}(X(\CC)^{\mathrm{an}})$. 
%\end{remark}

\subsection{The Segal conjecture and the Scheiderer sphere} \label{segal}
Let us now place ourselves in the context of the recollement in Proposition~\ref{prop:preserves} for $p=2$; specializing to the case of real closed fields, this yields the recollement $(\Sp^{BC_2}, \Sp)$ of $\Sp^{C_2}$. Recall the following formulation of the Segal conjecture for the group $C_2$, which is a theorem of Lin \cite{lin-1,lin}.

\begin{theorem}[Lin] \label{thm:lin} In $\Sp^{C_2}$, the canonical map $\1 \rightarrow j_*j^*\1$ is an equivalence after $2$-completion, i.e, the $C_2$-sphere spectrum is Borel complete after $2$-completion.
\end{theorem}

We have the following translation of the Segal conjecture into the motivic setting.

% \footnote{This refers to the smallest stable subcategory containing $\1\comtwo$ and contains all retracts.}
\begin{corollary} \label{lem:rcf} Suppose that $k$ is a real closed field and let $E \in \SH_b(k)\comtwo$ be in the thick subcategory generated by $\1\comtwo$. Then $E$ satisfies \'etale descent.
\end{corollary}
\begin{proof} After Theorem~\ref{thm:sing-ff}, this is a consequence of Theorem~\ref{thm:lin}.
\end{proof}

\begin{question} Can one give an independent proof of the Segal conjecture under the connection furnished by Theorem~\ref{thm:sing-ff}? See \cite{hahn-wilson} for another new perspective on the Segal conjecture.
\end{question}

The Segal conjecture further implies the following descent result for the Scheiderer sphere over any scheme of finite Krull dimension.

\begin{theorem} \label{thm:segal} Let $X$ be a finite-dimensional scheme and let $E \in \SH_b(X)\comtwo$ be in the thick subcategory generated by the $2$-complete Scheiderer sphere spectrum $\1_b\comtwo$. Then $E$ satisfies \'etale descent.
\end{theorem}

\begin{proof} It suffices to prove the result for $E = \1\comtwo$. We are in the $2$-completion of the recollement situation of~\eqref{eq:b-recoll-sh}. To prove that $\1\comtwo$ has \'etale descent, we need only prove that the canonical map $\1\comtwo \rightarrow j_*j^* \1\comtwo$ is an equivalence, for which it suffices to show that $i^!\1\comtwo \simeq 0$. To prove this, consider the recollement fiber sequence
\[ i^! \rightarrow i^* \rightarrow i^* j_* j^* = \Theta^{\mot} j^*.\]
Both $j^* = L_{\et}$ and $\Theta^{\mot}$ are stable under base change (the former is obvious, and the latter is Proposition~\ref{prop:base-change-big}), so $i^!$ is also stable under base change. Now by the finite-dimensionality assumption on $X$, we have that $\SH_{\ret}(X) \simeq \Sp(\widetilde{X_{\ret}})$ and $\widetilde{X_{\ret}}$ is hypercomplete, so to prove that $i^!\1\comtwo$ is zero, it suffices to show that $i^!\1\comtwo$ is zero after base change along all $\alpha: \Spec k \to X$ for $k$ a real closed field. We thereby reduce to Corollary~\ref{lem:rcf}.
% , or equivalently, that $i^!\1\comtwo \simeq 0$
%
% Suppose that $U_{\bullet} \rightarrow X'$ is an \'etale hypercover in $\Sm_X$. Consider the cofiber of the map $\Sigma^{\infty}U_{\bullet} \rightarrow \Sigma^{\infty}X'$, which we denote by $K$. We claim that for all $i, j \in \ZZ$, we have that $[K, \1\comtwo(i)[j]] = 0$. We may assume that $(i)[j] = 0$. Furthermore, we have that $\Maps(K, \1\comtwo) \simeq K \otimes (\1\comtwo)^{\vee} \simeq K \otimes \1\comtwo$. Since $\SH_b(X)\comtwo$ has the full six functors formalism, we may need only prove that for all morphisms $\Spec k \rightarrow X'$ where $k$ is a field, $K \otimes \1\comtwo \simeq 0$. 

\end{proof}

\appendix

\section{Motivic spectra over various topologies} \label{app:sh-top}
\tikzcdset{arrow style=tikz, diagrams={>=stealth}}

The goal of this appendix is to show that recollements on sheaves induced by intersecting Grothendieck topologies descend to recollements of motivic spectra. 

\subsection{Generalities} We first work in greater generality. Let $\B$ be an $\infty$-category and suppose that $\tau_U$ and $\tau_Z$ are two Grothendieck topologies on $\B$. Let $\tau = \tau_U \cap \tau_Z$, so that $\tau$ is the finest topology on $\B$ coarser than both $\tau_U$ and $\tau_Z$. Explicitly, for every $X \in \B$, a sieve $R \subset \B_{/X}$ is a $\tau$-covering sieve if and only if $R$ is both a $\tau_U$ and $\tau_Z$-covering sieve. The recollement on the level of motivic spaces then fits into the format of the next construction:

\begin{construction} \label{construct:recoll-a1}
Let $\omega$ be any Grothendieck topology on $\B$.
\begin{itemize}
\item We let $\Kscr_{\omega}$ be the small set of morphisms $\{ R \hookrightarrow h_X \}$ in $\Pre(\B)$ ranging over $X \in \B$ and $\omega$-covering sieves $R \subset \B_{/X}$.
\item Given another small set $\Lscr$ of morphisms in $\Pre(\B)$, let $\Shv_{\omega}^{\Lscr}(\B) \subset \Pre(\B)$ be the full subcategory of $(\Lscr \cup \Kscr)$-local objects.
\end{itemize} 
Let $\Shv_{\omega}^{\Lscr}(\B,\Escr) = \Shv_{\omega}^{\Lscr}(\B) \otimes \Escr$ be the tensor product in $\PrL_{\infty}$. Appealing to the theory of Bousfield localization \cite{htt}*{Proposition 5.5.4.15}, we have the two adjunctions
\[ \begin{tikzcd}[row sep=4ex, column sep=6ex, text height=1.5ex, text depth=0.25ex]
\Shv_{\tau_U}^{\Lscr}(\B, \Escr) \ar[shift right=1,right hook->]{r}[swap]{j_{\ast}} &  \Shv_{\tau}^{\Lscr}(\B, \Escr) \ar[shift right=2]{l}[swap]{j^{\ast}} \ar[shift left=2]{r}{i^{\ast}} & \Shv_{\tau_Z}^{\Lscr}(\B, \Escr) \ar[shift left=1,left hook->]{l}{i_{\ast}}.
\end{tikzcd} \]
\end{construction}

% We may suppose that $\Escr = \Spc$ without loss of generality.
\begin{lemma} \label{lem:conservativity} In the notation of Construction~\ref{construct:recoll-a1}, the functors $i^*$ and $j^*$ are jointly conservative.
\end{lemma}
\begin{proof} This follows from the observation that $\Kscr_{\tau} \cup \Lscr = (\Kscr_{\tau_U} \cup \Lscr) \cap (\Kscr_{\tau_Z} \cup \Lscr)$.
\end{proof}

The next lemma introduces an asymmetry that distinguishes the ``open" from the ``closed" part of the recollement.

% \todo{Check whether this really works for all E or just spaces?}
\begin{lemma} \label{lem:asymmetry} Suppose for every $X \in \B$, there exists a $\tau_U$-covering sieve $R$ such that for all $Y \in R$, the empty sieve is $\tau_Z$-covering for $Y$. Then the composite functor
\[ \Shv_{\tau_Z}(\B) \xto{i_*} \Shv_{\tau}(\B) \xto{j^*} \Shv_{\tau_U}(\B) \]
is constant at the terminal object.
\end{lemma}
\begin{proof} Let $F$ be a $\tau_Z$-sheaf. It suffices to show that the $\tau_U$-sheafification of $F$ is trivial. For this, we will use the formula for the $\tau_U$-sheafification functor as a filtered colimit of $\tau_U$-{\bf plus constructions} $F^{\dagger}$ that was established by Lurie in the proof of \cite{htt}*{Proposition 6.2.2.7}. The presheaf $F^{\dagger}$ is defined in \cite{htt}*{Remark 6.2.2.12} by the formula
\[
F^{\dagger}(X) = \colim_{R \subset C_{/X}} \lim_{Y \in R} F(Y),
\]
where the colimit is taken over the poset of $\tau_U$-covering sieves of $X$. Let $\kappa$ be the regular cardinal specified in the proof of \cite{htt}*{Proposition 6.2.2.7}, and let 
\begin{itemize}
\item $T_0 = \id$, 
\item $T_{\beta+1}(F) = T_{\beta}(F)^{\dagger}$ for $\beta < \kappa$, and
\item  $T_\gamma(F) = \colim_{\beta < \gamma} T_{\beta}(F)$ for every limit ordinal $\gamma \leq \kappa$,
\end{itemize}
so that $\tau_U$-sheafification $j^*$ is computed as $j^* F \simeq T_{\kappa} F$. We claim:
\begin{itemize} \item[($\ast$)] For all $\gamma \leq \kappa$, $(T_{\gamma} F)(Y) \simeq \ast$ for all $Y \in \B$ such that the empty sieve in $\B_{/Y}$ is $\tau_Z$-covering.
\end{itemize}

The proof is by ordinal induction on $\gamma$. For the base case $\gamma = 0$, this holds since $F$ is a $\tau_Z$-sheaf, whence $F(Y)$ is equivalent to the limit over the empty diagram. Suppose we have proven the claim for all $\beta < \gamma$. If $\gamma$ is a limit ordinal, then we have 
\[
(T_{\gamma} F)(Y) \simeq \colim_{\beta<\gamma} (T_{\beta})(Y) \simeq \colim_{\beta<\gamma} \ast \simeq \ast,
\] 
where for the last equivalence, we use that a filtered category is weakly contractible \cite{htt}*{Lemma~5.3.1.20}. For the successor ordinal case, let us write $\gamma = \beta+1$. Note that for every $Z \to Y$ in $\B$, the empty sieve is also $\tau_Z$-covering for $Z$ by the stability of covering sieves under pullback. Therefore, we may compute 
\[
(T_{\beta} F)^{\dagger}(Y) \simeq \colim_{R \subset \B_{/Y}} \lim_{Z \in R} (T_{\beta} F)(Z) \simeq \colim_{R \subset \B_{/Y}} \ast \simeq \ast, 
\]
using for the last equivalence that the poset of covering sieves is filtered \cite[Rem.~6.2.2.11]{htt}, hence weakly contractible. This proves claim $(\ast)$.

Next, we claim:
\begin{itemize} \item[($\ast\ast$)] For all infinite ordinals $\gamma \leq \kappa$, $(T_{\gamma} F)(X) \simeq \ast$ for all $X \in \B$.
\end{itemize}

It suffices to consider $\gamma = \omega$. We note that for every $X \in \B$, we have a factorization
$$(T_n F)(X) \to \ast \to (T_{n} F)^{\dagger}(X)$$
via selecting the $\tau_U$-covering sieve $R \subset \B_{/X}$ given by our assumption, where we use the previous claim to see that $(T_{n}F)(Y) \simeq \ast$ for all $Y \in R$. It then follows that $(T_{\omega} F)(X) \simeq \ast$, and we are done.
\end{proof}

Given this, we deduce:

\begin{proposition} \label{prp:genericToposRecollementIntersectingTopologies} With assumptions as in Lemma~\ref{lem:asymmetry}, the pair \[
(\Shv_{\tau_U}(\B), \Shv_{\tau_Z}(\B))
\] constitutes a recollement of $\Shv_{\tau}(\B)$. Moreover, this recollement is monoidal with respect to the cartesian monoidal structure.
\end{proposition}
\begin{proof} Appealing to the theory of sheafification, we have that the localization functors $j^*$ and $i^*$ are left-exact. Since $j^*, i^*$ are jointly conservative by Lemma~\ref{lem:conservativity} and $j^* i_* \simeq \ast$ by Lemma~\ref{lem:asymmetry}, the claim follows.
\end{proof}

\begin{remark} In view of \cite{higheralgebra}*{Proposition A.8.15}, Proposition~\ref{prp:genericToposRecollementIntersectingTopologies} is equivalent to the obvious $\infty$-categorical extension of \cite{scheiderer}*{Proposition~2.2} (in the (ii)$\Rightarrow$(i) direction), although our proof differs from his. 
\end{remark}

We now pass to stabilizations. Recall the following paradigm: 
\begin{construction} Suppose that we have a left-exact functor $G: \C \to \D$, then we have have an induced functor \[
\overline{G}: \Sp(\C) = \Exc_{\ast}(\Spc^{\fin}_{\bullet},\C) \to \Sp(\D) = \Exc_{\ast}(\Spc^{\fin}_{\bullet},\D). \qquad (f: \Spc_{\bullet}^{\fin} \to \C) \mapsto (G \circ f).
\] Moreover, if $G$ admits a left-exact left adjoint $F: \D \to \C$, then $\overline{G}$ admits a left adjoint \[
\overline{F}:  \Sp(\D) \to \Sp(\C)
\] defined by postcomposition by $F$.
\end{construction}

\begin{lemma} \label{lem:stabilizationOfRecollement} Suppose we have a recollement \[ \begin{tikzcd}[row sep=4ex, column sep=6ex, text height=1.5ex, text depth=0.25ex]
\Uscr \ar[shift right=1,right hook->]{r}[swap]{j_{*}} & \Xscr \ar[shift right=2]{l}[swap]{j^{*}} \ar[shift left=2]{r}{i^{*}} & \Zscr \ar[shift left=1,left hook->]{l}{i_{*}}.
\end{tikzcd} \]
Then the induced adjunctions
\[ \begin{tikzcd}[row sep=4ex, column sep=6ex, text height=1.5ex, text depth=0.25ex]
\Sp(\Uscr) \ar[shift right=1,right hook->]{r}[swap]{\overline{j}_{*}} & \Sp(\Xscr) \ar[shift right=2]{l}[swap]{\overline{j}^{*}} \ar[shift left=2]{r}{\overline{i}^{*}} & \Sp(\Zscr) \ar[shift left=1,left hook->]{l}{\overline{i}_{*}}.
\end{tikzcd} \]
together yield a stable recollement.
\end{lemma}
\begin{proof} Since $j^{*} i_{*}$ is constant at the terminal object, we have that $\overline{j}^{*} \overline{i}_{*} \simeq 0$ in view of the description of this functor as postcomposition by $j^{*} i_{*}$ in terms of spectrum objects. Likewise, the joint conservativity of $\overline{j}^{*}$ and $\overline{i}^{*}$ holds since equivalences of spectrum objects are detected levelwise.
\end{proof}

The next corollary is now immediate.

\begin{corollary} \label{cor:RecollementSheavesSpectra} With assumptions as in Lemma~\ref{lem:asymmetry}, the pair \[
(\Shv_{\tau_U}(\B,\Sp), \Shv_{\tau_Z}(\B, \Sp))
\] constitutes a stable recollement of $\Shv_{\tau}(\B,\Sp)$. Moreover, this recollement is monoidal with respect to the smash product symmetric monoidal structure.
\end{corollary}

Finally, we explain how to pass recollements through monoidal inversion of objects. For this, recall from \cite{robalo}*{Corollary 2.22} that if $\C$ is a presentable symmetric monoidal $\infty$-category and $X \in \C$ is a symmetric object (i.e., for some $n>1$ the cyclic permutation of $X^{\otimes n}$ is homotopic to the identity), then the presentable symmetric monoidal $\infty$-category $C[X^{-1}]$ is computed as the filtered colimit of presentable $\infty$-categories
\[ \C[X^{-1}] \simeq \colim (\C \xrightarrow{- \otimes X} \C \xrightarrow{- \otimes X} \cdots). \]

\begin{lemma} \label{lem:InvertedRecollement} Let
\[ \begin{tikzcd}[row sep=4ex, column sep=6ex, text height=1.5ex, text depth=0.25ex]
\Uscr \ar[shift right=4,right hook->]{r}[swap]{j_{*}} \ar[shift left=4,right hook->]{r}{j_{!}} & \Xscr \ar[shift left=1]{l}[swap]{j^{*}} \ar[shift left=4]{r}{i^{*}} \ar[shift right=4]{r}[swap]{i^!} & \Zscr \ar[left hook->,shift left=1]{l}[swap]{i_{*}}
\end{tikzcd} \]
be a stable\footnote{Stability is used in an essential way in our proof.} monoidal recollement of presentable stable symmetric monoidal $\infty$-categories and let $E \in \Xscr$ be a symmetric object. Then we may descend the recollement adjunctions to obtain a stable monoidal recollement
\[ \begin{tikzcd}[row sep=4ex, column sep=6ex, text height=1.5ex, text depth=0.25ex]
\Uscr[(j^{*} E)^{-1}] \ar[shift right=4,right hook->]{r}[swap]{\overline{j}_{*}} \ar[shift left=4,right hook->]{r}{\overline{j}_{!}} & \Xscr[E^{-1}] \ar[shift left=1]{l}[swap]{\overline{j}^{*}} \ar[shift left=4]{r}{\overline{i}^{*}} \ar[shift right=4]{r}[swap]{\overline{i}\,^!} & \Zscr[(i^{*} E)^{-1}] \ar[left hook->,shift left=1]{l}[swap]{\overline{i}_{*}}.
\end{tikzcd} \]
\end{lemma}
\begin{proof} Let us write 
\[
\Sigma^{\infty}_E: \Xscr \rightleftarrows \Xscr[E^{-1}]: \Omega^{\infty}_E
\]
%\[ \adjunct{\Sigma^{\infty}_e}{\Xscr}{\Xscr[E^{-1}]}{\Omega^{\infty}_e} \] 
for the defining adjunction, and likewise for $\Uscr$ and $\Zscr$. Also let $E, j^{*} E, i^{*} E$ denote the images of the objects in $\Xscr[E^{-1}]$, etc. We first explain how the various overlined functors are defined. Since $j^{*}$ and $i^{*}$ are symmetric monoidal, we have the commutative diagram
\[ \begin{tikzcd}[row sep=4ex, column sep=6ex, text height=1.5ex, text depth=0.25ex]
\Uscr \ar{d}[swap]{- \otimes j^{*} E} & \Xscr \ar{r}{i^{*}} \ar{l}[swap]{j^{*}} \ar{d}{- \otimes E} & \Zscr \ar{d}{- \otimes i^{*} E} \\
\Uscr & \Xscr \ar{r}{i^{*}} \ar{l}[swap]{j^{*}} & \Zscr,
\end{tikzcd} \]
hence $\overline{j}^{*}$ and $\overline{i}^{*}$ may be taken to be the filtered colimit of the system of functors $\{ j^{*} \}$ and $\{ i^{*} \}$, and we have the commutativity relations
\[  \overline{j}^{*} \circ (E^{-n} \otimes \Sigma^{\infty}_E) \simeq ((j^{*} E)^{-n} \otimes \Sigma^{\infty}_{j^{*} E}) \circ j^{*}, \quad \overline{i}^{*} \circ (E^{-n} \otimes \Sigma^{\infty}_E) \simeq ((i^{*} E)^{- n} \otimes \Sigma^{\infty}_{i^{*} E}) \circ i^{*}.\] 
By adjunction, $\overline{j}^{*}$ and $\overline{i}^{*}$ admit right adjoints $\overline{j}_{*}$ and $\overline{i}_{*}$ that satisfy adjoint commutativity relations involving $\Omega^{\infty}_{-}$. To aid understanding, it may be helpful to note that $\Sigma^{\infty}_E X$ is given as a $E$-spectrum object by $\{ X \otimes E^n \}$, and $\Omega^{\infty}_E \{ X_n \} \simeq X_0$ (cf. Remark~\ref{rem:PrespectraAndSpectra}).

We next observe that the projection formulas 
\[
j_!(U \otimes j^{*} X) \simeq j_!(U) \otimes X \qquad i_{*}(Z) \otimes X \simeq i_{*}(Z \otimes i^{*} X)
\] established in \cite{quigley-shah}*{Proposition 1.30} imply that $\overline{j}_!$ and $\overline{i}_{*}$ may be defined by the same colimit procedure as was done with $\overline{j}^{*}$ and $\overline{i}^{*}$. In particular, $\overline{j}^{*}$ commutes with $\Omega^{\infty}_{-}$ and $\overline{i}_{*}$ commutes with $\Sigma^{\infty}_{-}$ and admits a right adjoint $\overline{i}\,^!$ that commutes with $\Omega^{\infty}_{-}$ (and the same is true for shifts thereof by $E^n$). 

We next establish the full faithfulness of $\overline{j}_!, \overline{j}_{*}$, and $\overline{i}_{*}$. For $\overline{j}_{*}$, note that the counit $\overline{j}^{*} \overline{j}_{*} \to \id$ is sent to an equivalence by $\Omega^{\infty}_{j^{*} E}(E^{n} \otimes -)$, hence is an equivalence. For $\overline{i}_{*}$, the counit $\overline{i}^{*} \overline{i}_{*} \to \id$ is a natural transformation of colimit-preserving functors and is evidently an equivalence on objects in the image of $(i^{*} E)^{-n} \otimes \Sigma^{\infty}_{i^{*} E}$; since such objects generate $\Zscr[(i^{*} E)^{-1}]$ under colimits, we may conclude. The case of $\overline{j}_!$ follows by the same reasoning. Finally, also note that the $\overline{j}_! \dashv \overline{j}^{*}$ and $\overline{i}^{*} \dashv \overline{i}_{*}$ projection formulas follow by the same reasoning. 

Now let $A = i_{*} i^{*} \1$, so that $A$ is the idempotent $\mathbb{E}_{\infty}$-algebra in $\Xscr$ that determines the smashing localization $\Zscr \simeq \Mod_{A}(\Xscr)$. Invoking the monoidal Barr-Beck theorem \cite{mnn-descent}*{Theorem 5.29}, we may descend this equivalence to obtain $$\Zscr[(i^{*} E)^{-1}] \simeq \Mod_{\Sigma^{\infty}_E A}(\Xscr[E^{-1}]).$$

We thus see that $\Zscr[(i^{*} E)^{-1}]$ is the smashing localization of $\Xscr[E^{-1}]$ determined by the idempotent $\mathbb{E}_{\infty}$-algebra $\Sigma^{\infty}_E A$. In this situation, we recall from \cite[1.32]{quigley-shah} that to establish the recollement conditions, it suffices to show 
\begin{itemize}
\item[($\ast$)] the essential image of $\overline{j}_{*}$ coincides with the full subcategory of objects $X$ such that $\Map(\overline{i}_{*} Z,X) \simeq \ast$ for all $Z \in \Zscr[(i^{*} E)^{-1}]$.
\end{itemize} To prove this, first note that $\overline{j}^{*} \overline{i}_{*} \simeq 0$ using that for all $n \in \ZZ$, $$\Omega^{\infty}_{j^{*} E}((j^{*} E)^{n} \otimes \overline{j}^{*} \overline{i}_{*}(-)) \simeq j^{*} i_{*} \Omega^{\infty}_{i^{*} E} ((i^{*} E)^{n} \otimes -) \simeq 0.$$
Therefore, given $U \in \Uscr[(j^{*} E)^{-1}]$, we have that $$\Map(\overline{i}_{*} Z,\overline{j}_{*} U) \simeq \Map(\overline{j}^{*} \overline{i}_{*} Z,U) \simeq \Map(0,U) \simeq *.$$
Conversely, suppose $X \in \Xscr[E^{-1}]$ is such that $\Map(\overline{i}_{*} Z,X) \simeq *$ for all $Z \in \Zscr[(i^{*} E)^{-1}]$. By adjunction, we have that $\overline{i}\,^!(X) \simeq 0$. It follows that $\Omega^{\infty}_E (E^{n} \otimes X)$ lies in the essential image of $j_{*}$ for all $n$, and thus $X$ itself lies in the essential image of $\overline{j}_{*}$. This proves the claim ($\ast$) and hence the lemma is proved.
\end{proof}

\begin{remark} \label{rem:PrespectraAndSpectra} For sake of reference, we record some basic facts about prespectrum and spectrum objects in the $\infty$-categorical setting, fleshing out some details in \cite{elso}*{\S 4.0.6}, \cite{hoyois-cdh}*{\S 3}. Let $\C$ be a presentable symmetric monoidal $\infty$-category, let $E \in \C$ be a symmetric object, and let $\Sigma^E$, resp. $\Omega^E$ denote the endofunctors $E \otimes (-)$, resp. $\underline{\Hom}(E,-)$ of $\C$. Let $p: \ZZ_{\geq 0}^{\op} \to \widehat{\Cat}_{\infty}$ be the functor given by
\[ \begin{tikzcd}[column sep=8ex]
\C & \C \ar{l}[swap]{\Omega^E} & \C \ar{l}[swap]{\Omega^E} & \cdots \ar{l}[swap]{\Omega^E}
\end{tikzcd} \]
Because $E$ is a symmetric object, we have an equivalence of $\infty$-categories $\C[E^{-1}] \simeq \lim(p)$. Moreover, if we let $\widehat{\C} = \int p \to \ZZ_{\geq 0}$ be the cartesian fibration classified by $p$, then $\Sect^{\cart}(\widehat{\C}) \simeq \lim(p)$.\footnote{We could lift $\C[E^{-1}] \simeq \Sect^{\cart}(\widehat{\C})$ to an equivalence of symmetric monoidal $\infty$-categories by means of the Day convolution, but we will not need this.} Using that the inclusion of the spine\footnote{The \emph{spine} of $\ZZ_{\geq 0}$ is the sub-simplicial set given by $\{0 < 1\} \cup_{\{1\}} \{ 1 < 2\} \cup_{\{2\}} \{2 < 3\} \cup \cdots $.} of $\ZZ_{\geq 0}$ into $\ZZ_{\geq 0}$ is inner anodyne, we see that the objects of $\Sect^{\cart}(\widehat{\C})$ may be described as tuples $\{ X_n \in \C, \alpha_n: X_n \xrightarrow{\simeq} \Omega^E X_{n+1} \}$ with further coherences determined essentially uniquely, and similarly for morphisms, so we are entitled to refer to $\Sect^{\cart}(\widehat{\C})$ as the $\infty$-category of $E$-spectra in $\C$.

Moreover, the cartesian fibration $\widehat{\C} \to \ZZ_{\geq 0}$ is also a cocartesian fibration with pushforward functors the left adjoints $\Sigma^E$. Thus, if we consider $\Sect(\widehat{\C})$ instead, its objects may be described as tuples $\{ X_n \in \C, \beta_n: \Sigma^E X_n \rightarrow X_{n+1}\}$, so $\Sect(\widehat{\C})$ is the $\infty$-category of $E$-prespecta in $\C$. We have the inclusion functor $\Sect^{\cart}(\widehat{\C}) \subset \Sect(\widehat{\C})$ that preserves limits and is accessible (since the same is true for $\Omega^E$), so admits a left adjoint $L_{\mathrm{sp}}$ that exhibits $E$-spectra as a Bousfield localization of $E$-prespectra. We also have the equivalence $\Sect^{\mathrm{cocart}}(\widehat{\C}) \xrightarrow{\simeq} \C$ implemented by evaluation at the initial object $0 \in \ZZ_{\geq 0}$, and the inclusion $\Sect^{\cocart}(\widehat{\C}) \subset \Sect(\widehat{\C})$ preserves colimits (as $\Sigma^E$ does) and thus admits a right adjoint which also identifies with evaluation at $0$.

% Using the identification $\C \simeq \Sect^{\cocart}(\widehat{\C})$ (with inverse given by evaluation at $0$) and $\C[E^{-1}] \simeq \Sect^{\cart}(\widehat{\C})$.
Using these identifications, we then may factor the adjunction $\Sigma^{\infty}_E: \C \rightleftarrows \C[E^{-1}] : \Omega^{\infty}_E$ as
\[ \begin{tikzcd}[row sep=4ex, column sep=6ex]
\C \ar[shift left=1, right hook->]{r} & \Sect(\widehat{\C}) \ar[shift left=2]{l}{\ev_0} \ar[shift left=1]{r}{L_{\mathrm{sp}}}  & \C[E^{-1}] \ar[shift left=2, left hook->]{l}.
\end{tikzcd} \]

Now suppose $i^*: \C \to \D$ is a strong symmetric monoidal colimit-preserving functor of presentable symmetric monoidal $\infty$-categories, and let $\overline{i}^*: \C[E^{-1}] \to \D[(i^*E)^{-1}]$ be the induced functor. We may then factor the equivalence $\Sigma^{\infty}_{i^* E} i^* \simeq \overline{i}^* \Sigma^{\infty}_E$ as
\[ \begin{tikzcd}[row sep=4ex, column sep=6ex]
\C \ar[right hook->]{r} \ar{d}{i^*} & \Sect(\widehat{\C}) \ar{r}{L_{\mathrm{sp}}} \ar{d}{i^*_{\pre}} & \C[E^{-1}] \ar{d}{\overline{i}^*} \\
\D \ar[right hook->]{r} & \Sect(\widehat{\D}) \ar{r}{L_{\mathrm{sp}}} & \D[(i^* E)^{-1}],
\end{tikzcd} \]
where $i^*_{\pre}$ is obtained by postcomposition by $\widehat{i}^*: \widehat{\C} \to \widehat{\D}$, the unstraightening of the natural transformation
\[ \begin{tikzcd}[row sep=4ex, column sep=6ex]
\C \ar{r}{\Sigma^E} \ar{d}{i^*} & \C \ar{r}{\Sigma^E} \ar{d}{i^*} & \C \ar{r}{\Sigma^E} \ar{d}{i^*} & \cdots \\
\D \ar{r}{\Sigma^{i^*E}} & \D \ar{r}{\Sigma^{i^*E}} & \D \ar{r}{\Sigma^{i^* E}} & \cdots
\end{tikzcd} \]
In particular, we see that $\overline{i}^*$ is computed by inclusion into $E$-prespectra and then $L_{\mathrm{sp}} \circ i^*_{\pre}$.
\end{remark}

\subsection{Motivic homotopy theory} \label{sec:AppendixMotivicPart}
Let us now specialize to the situation where $\B = \Sm_{S}$. In general, if we have two topologies $\tau'$ and $\tau$ on $\Sm_S$ with $\tau'$ finer than $\tau$, then the localization functor $L_{\tau'}: \SH^{S^1}_{\tau}(S) \to \SH^{S^1}_{\tau'}(S)$ always preserves the set of morphisms $\{ \Sigma^{\infty}_+ (X \times \AA^1) \to \Sigma^{\infty}_+ X \}$ that generate the $\AA^1$-local equivalences, using that sheafification preserves products. Moreover, $L_{\tau'}$ preserves hypercovers as a sheafification. Therefore, we get a commutative diagram\footnote{Note that the localization functor $L_{\AA^1}$ also imposes hyperdescent according to our conventions.}
\[ \begin{tikzcd}[row sep=4ex, column sep=4ex]
\Shv_{\tau}(\Sm_S,\Sp) \ar{r}{L_{\tau'}} \ar{d}{L_{\AA^1}} & \Shv_{\tau'}(\Sm_S,\Sp)  \ar{d}{L_{\AA^1}} \\
\SH^{S^1}_{\tau}(S) \ar{r}{L_{\tau'}} & \SH^{S^1}_{\tau'}(S).
\end{tikzcd} \]
Using the criterion of \cite{higheralgebra}*{Example 2.2.1.7}, we also see that this is a diagram of localizations compatible with the various symmetric monoidal structures in the sense of \cite{higheralgebra}*{Definition 2.2.1.6}. Returning to the setting of two topologies $\tau_U$ and $\tau_Z$ with $\tau = \tau_U \cap \tau_Z$, we thus obtain a commutative diagram of symmetric monoidal left adjoints
\begin{equation} \label{eq:mot-recolls1} \begin{tikzcd}[row sep=4ex, column sep=6ex]
\Shv_{\tau_U}(\Sm_S,\Sp)\ar{d}{L_{\AA^1}} & \Shv_{\tau}(\Sm_S,\Sp)\ar{d}{L_{\AA^1}} \ar{l}[swap]{j^{*}} \ar{r}{i^{*}} & \Shv_{\tau_Z}(\Sm_S,\Sp)) \ar{d}{L_{\AA^1}} \\
\SH^{S^1}_{\tau_U}(S) & \SH^{S^1}_{\tau}(S) \ar{l}[swap]{j^{*}} \ar{r}{i^{*}} & \SH^{S^1}_{\tau_Z}(S)
\end{tikzcd} 
\end{equation}
with fully faithful right adjoints $j_{*}$, $i_{*}$, and $\iota_{\AA^1}$.
\begin{proposition} \label{prp:S1motivesRecollement} The pair $(\SH^{S^1}_{\tau_U}(S),\SH^{S^1}_{\tau_Z}(S))$ constitute a stable monoidal recollement of $\SH^{S^1}_{\tau}(S)$.
\end{proposition}
\begin{proof} For clarity, we decorate the lower horizontal functors of~\eqref{eq:mot-recolls1} with an $S^1$ in this proof. By Lemma~\ref{lem:conservativity}, $j^{*}_{S^1}$ and $i^{*}_{S^1}$ are jointly conservative. As for the composite functor $j^{*}_{S^1} i_{*}^{S^1}$, this follows from Corollary~\ref{cor:RecollementSheavesSpectra} via the following computation:
\[ j^{*}_{S^1} i_{*}^{S^1} \simeq L_{\AA^1} j^{*} \iota_{\AA^1} i_{*}^{S^1} \simeq L_{\AA^1} j^{*} i_{*} \iota_{\AA^1} \simeq L_{\AA^1} (0) \simeq 0. \]
\end{proof}

\begin{warning} \label{warn:UnstableMotivicGluingFunctor} On the unstable level of (pointed) motivic spaces, one likewise has adjunctions
\[ \begin{tikzcd}[row sep=4ex, column sep=6ex, text height=1.5ex, text depth=0.25ex]
\H_{\tau_U}(S)_{(\bullet)} \ar[shift right=1,right hook->]{r}[swap]{j_*} & \H_{\tau}(S)_{(\bullet)} \ar[shift right=2]{l}[swap]{j^*} \ar[shift left=2]{r}{i^*} & \H_{\tau_Z}(S)_{(\bullet)} \ar[shift left=1,left hook->]{l}{i_*}.
\end{tikzcd} \]
such that $j^*,i^*$ are jointly conservative and $j^* i_* \simeq \ast$. However, we do not know how to show that $i^*$ is left-exact (although $j^*$ is left-exact since we may descend the adjunction $j_! \dashv j^*$).
\end{warning}

To pass from $\SH^{S^1}_{\tau}(S)$ to $\SH_{\tau}(S)$, we may invert $\GG_m$. We then have:

\begin{theorem} \label{thm:MainMotivesRecollement} The pair $(\SH_{\tau_U}(S), \SH_{\tau_Z}(S))$ constitutes a stable monoidal recollement of $\SH_{\tau}(S)$.
\end{theorem}
\begin{proof} Since $\GG_m$ is a symmetric object for any topology, the theorem follows from Proposition~\ref{prp:S1motivesRecollement} and Lemma~\ref{lem:InvertedRecollement}.
\end{proof}

\begin{example} \label{ex:recoll-b} Let $\tau_U$ be the \'etale topology and $\tau_Z$ be the real \'etale topology, so that $\tau$ is Scheiderer's b-topology. By \cite[2.4]{scheiderer}, the criterion of \ref{lem:asymmetry} is satisfied in this case. Therefore, Theorem~\ref{thm:MainMotivesRecollement} applies to decompose $\SH_b(S)$ as a stable monoidal recollement of $\SH_{\et}(S)$ and $\SH_{\ret}(S)$.
\end{example}

\section{The real \'etale and $b$-topologies}

We begin with some recollections on the real \'etale topology.

\begin{definition} Suppose that $f:Y \rightarrow X$ is a morphism of schemes. We say that $f$ is a {\bf real cover} if for any morphism $\alpha: \Spec k \rightarrow X$ where $k$ is a real closed field, there exists an extension of real closed fields $k'/k$ and a morphism $\alpha':\Spec k' \rightarrow Y$ such that the following diagram commutes
\[
\begin{tikzcd}
\Spec k' \ar{d} \ar{r}{\alpha'} & Y \ar{d} \\
\Spec k \ar{r}{\alpha} & X. 
\end{tikzcd}
\]
\end{definition}

\begin{example} \label{ex:nontrivialRealClosedFieldsExtension} For an example of a non-trivial extension of real closed fields, consider the field $\RR_{\mathrm{alg}}$ of real algebraic numbers, which is the real closure of $\QQ$. Then the inclusion $\RR_{\mathrm{alg}} \subset \RR$ is an extension of real closed fields.
\end{example}

\begin{lemma} \label{lem:real-cover} Let $f:Y \rightarrow X$ be a morphism of schemes. The following conditions on $f$ are equivalent:
\begin{enumerate}
\item $f: Y \rightarrow X$ is a real cover.
\item The induced map on real spectra $f_r: Y_r \rightarrow X_r$ is surjective.
\item For any point $x \in X$ with an ordering on $k(x)$, there exists a point $y \in Y$ with an ordering on $k(y)$ such that $f(y) = x$ and $k(y)/k(x)$ is an extension of ordered fields.
\end{enumerate}
\end{lemma}

\begin{proof} The equivalence of (2) and (3) follows by tracing through the definition of the real spectrum. The equivalence of (1) and (2) follows from the description of the points of $X_r$ as classes of morphisms $\Spec k \rightarrow X$ where $k$ is a real closed field up to the equivalence relation specified in \cite{scheiderer}*{0.4.3}.
\end{proof}

\begin{remark} \label{rem:ret} A real cover need not be surjective on any other field points. For example, the morphism $\Spec \RR \rightarrow \Spec \RR[x,y]/(x^2+y^2)$ classifying the $\RR$-point ``$(0,0)$" is a real cover. However, on $\CC$-points, it is the inclusion of a single point into a conic and thus is not surjective.
\end{remark}

\begin{definition} \label{def:ret} Let $X$ be a scheme. The {\bf real \'etale topology on $\Et_X$} is the coarsest topology such that the sieves generated by families $\{ U_{\alpha} \rightarrow U \}$ for which $\coprod U_{\alpha} \rightarrow U$ is a real cover are covering sieves. Similarly, we can also consider the {\bf real \'etale topology on $\Sm_X$} or $\Sch_X$ as the coarsest topology such that the sieves generated by families $\{ U_{\alpha} \rightarrow U \}$ for which $U_{\alpha} \rightarrow U$ is \'etale and $\coprod U_{\alpha} \rightarrow U$ is a real cover are covering sieves. 
\end{definition}

Note that if $k \subset k'$ is an algebraic extension of ordered fields and $k$ is real closed, then $k = k'$. Hence, if we assume that a real cover $f: Y \rightarrow X$ is \'etale (so that the singleton $\{Y \rightarrow X\}$ generates a real \'etale covering sieve), we have the following refinement of Lemma~\ref{lem:real-cover}.

\begin{lemma} \label{lem:real-cover-et} Let $f:Y \rightarrow X$ be a morphism of schemes such that for any $y \in Y$ the induced extension of residue fields $f:k(f(y)) \subset k(y)$ is algebraic (e.g., $f$ is \'etale). The following are equivalent:
\begin{enumerate}
\item $f: Y \rightarrow X$ is a real cover.
\item For any real closed field $k$, the induced map $Y(k) \rightarrow X(k)$ is surjective.
\end{enumerate}
\end{lemma}

\begin{example} \label{ex:nis-et} Recall that a morphism $f:Y \rightarrow X$ is {\bf completely decomposed} if it is surjective on $k$-points. Hence, one might call condition (2) of Lemma~\ref{lem:real-cover-et} {\bf real completely decomposed}, since we only demand surjectivity on real closed points. Since the Nisnevich topology is given by \'etale morphisms that are jointly surjective on $k$-points, we see that any Nisnevich cover is a real \'etale cover.
%
%
%Tautologically, a completely decomposed cover is a real cover.  condition 3 of Lemma~\ref{lem:real-cover} is satisfied with $k(x) = k(y)$. Hence, the real \'etale topology is \emph{finer} than the Nisnevich topology.
\end{example}
We now discuss points of the real \'etale topology. 

%Finally, although we do not need it, let us recall the points of the real \'etale topology.

%\begin{definition} A {\bf strictly real henselian} ring is a henselian local ring with a real closed residue field.
%\end{definition}

% ; it automatically is contained in it
Suppose that $X = \Spec A$ is an affine scheme and $x \in \Sper A$ is a point corresponding to a pair $(\mathfrak{p}, \leq_{\mathfrak{p}})$ where $\mathfrak{p}$ is a prime ideal of $A$ and $\leq_{\mathfrak{p}}$ is an ordering on $\kappa(\mathfrak{p})$. Let $\mathrm{Neib}_{\ret}(X,x)$ be the filtered category of tuples $(B,y \in \Sper B)$, where $B$ is an \'etale $A$-algebra and $y = (\mathfrak{q}, \leq_{\mathfrak{q}})$ is a point over $x$, i.e., $\mathfrak{q}$ lies over $\mathfrak{p}$ and $k(\mathfrak{q})/k(\mathfrak{p})$ is an ordered extension. Unlike the situation of separably closed fields, we note that the real closures of ordered fields are unique (see, for example, \cite{sander}), whence we do not have to make a further choice of an embedding of $k(\mathfrak{q})$ into the real closure $k$ of $k(\mathfrak{p})$; it follows that to define $\mathrm{Neib}_{\ret}(X,x)$, we may have equivalently specified factorizations of $\Spec k \to \Spec A$ through \'etale $A$-schemes. We then have the notion of a {\bf strictly real henselian local ring}, defined to be a henselian local ring with real closed residue field, and the {\bf strict real henselization} of $X$ at $x = (\mathfrak{p}, \leq_{\mathfrak{p}})$, which is the initial strictly real henselian local ring $A^{\ret h}_x$ equipped with a local homomorphism $A_{\mathfrak{p}} \rightarrow A^{\ret h}_x$ preserving the ordering on residue fields (see \cite{strictreal}*{Definition 3.1.1} or \cite{scheiderer}*{Definition~3.7.3}). By \cite{strictreal}*{Proposition 3.1.3}, the strict real henselization of $X$ at $x$ exists and is computed by the usual formula:
\begin{equation} \label{eq:rhensel}
A^{\ret h}_x = \colim_{\mathrm{Neib}_{\ret}(X,x)} B.
\end{equation}

\begin{lemma} \label{lem:splits-ret} Let $R$ be a commutative ring. Then the following are equivalent: 
\begin{enumerate}
\item The ring $R$ is strictly real henselian.
\item For any real \'etale cover $\{ \phi_{\alpha}: \Spec R_{\alpha} \rightarrow \Spec R \}$, one of the morphisms $\phi_{\alpha}$ admits a splitting.
\end{enumerate}
\end{lemma} 

\begin{proof} This is in \cite{costes}. 
\end{proof}

% We also have

\begin{lemma} \label{lem:approx} Suppose that $S$ is the limit of a cofiltered diagram of schemes\footnote{Recall that, by convention, schemes are all qcqs.} $S_{\alpha}$. Then we have equivalences of $\infty$-categories
\[
\widetilde{S}_{\ret} \simeq \lim_{\alpha} \widetilde{S_{\alpha}}_{\ret}, \qquad \Sp(\widetilde{S}_{\ret}) \simeq \lim_{\alpha} \Sp(\widetilde{S_{\alpha}}_{\ret}).
\]
\end{lemma}

\begin{proof} Since \'etale morphisms are Zariski-locally finitely presented \cite{stacks}*{Tag 02GR} and the real \'etale topology is finer than Zariski, we have that $\widetilde{X}_{\ret} \simeq \Shv_{\ret}(\Et_X^{\mathrm{fp}})$ in general. Following the same logic as in \cite{hoyois-glv}*{Proposition C.7(3)}, the first claim follows from the fact that a finitely presented real-\'etale surjection in $\Et^{\mathrm{fp}}_{S}$ is pulled back from one in $\Et^{\mathrm{fp}}_{S_{\alpha}}$; this is the case for \'etale maps by \cite{stacks}*{Tag 07RP} and clear for surjectivity on real closed points. The stable case follows from the unstable case since $\Sp$ preserves limits of $\infty$-categories with finite limits (using the formula for stabilization in \cite{higheralgebra}*{Proposition 1.4.2.24}).
\end{proof}

\subsection{Hypercompleteness of the real \'etale site} \label{subsec:hyp}

In this section, we will prove that if $X$ is of finite Krull dimension, then the real \'etale site of $X$ is hypercomplete. One use of this result in the main body of the paper is to show that certain gluing functors are stable under base change --- an argument that appeals to ``checking on stalks." If $X$ is a topological space, we write $\widetilde{X}$ for the $\infty$-category of sheaves on spaces on $X$, i.e., that of functors
\[
\mathrm{Open}(X)^{\op} \rightarrow \Spc
\] 
satisfying descent with respect to open covers. We first state a fundamental theorem in the subject of real \'etale cohomology, which in the setting of sheaves of sets is due originally to Coste-Roy and Coste (cf. \cite{scheiderer}*{1.4}); we will follow Scheiderer's more functorial formulation, which may be readily promoted to a statement involving $\infty$-topoi.
% We will use this theorem to prove hypercompleteness results for real \'etale topoi.

\begin{theorem}[Scheiderer] \label{thm:et-ret} Let $X$ be a scheme. There is a canonical equivalence of $\infty$-topoi
\[
\widetilde{X}_{\ret} \simeq \widetilde{X_r}.
\]
\end{theorem}

We can deduce this theorem from \cite{scheiderer}*{Theorem 1.3} after some topos-theoretic preliminaries. Recall from \cite{htt}*{Definition 6.4.1.1} that for $0 \leq n \leq \infty$, an {\bf $n$-topos} is an $\infty$-category $\Xscr$ that is a left-exact accessible localization of $\Pre(\C, \Spc_{\leq n-1})$ for $\C$ a small $\infty$-category. For example, a $1$-topos is a left-exact accessible localization of a category of set-valued presheaves. A {\bf geometric morphism between $n$-topoi $\Xscr$ and $\Yscr$} is a functor $f_*:\Xscr \rightarrow \Yscr$ that admits a left-exact left adjoint.

Examples of $n$-topoi and geometric morphisms thereof are furnished by the following construction. Suppose that $1\leq m \leq n \leq \infty$. An {\bf $m$-site} $(\C, \tau)$ is an $m$-category $\C$ equipped with a Grothendieck topology $\tau$. If $(\C, \tau)$ and $(\D, \tau')$ are $m$-sites with finite limits, then a {\bf morphism of sites} is a functor $f:\C \rightarrow \D$ that preserves finite limits and enjoys the following property:
\begin{itemize}
\item[($\ast$)] For every family of morphisms $\{ U_{\alpha} \rightarrow X \}$ that generates a $\tau$-covering sieve, the collection $\{ f(U_{\alpha}) \rightarrow f(X) \}$ generates a $\tau'$-covering sieve.
\end{itemize}
Note that this recovers the usual notion of a morphism of $1$-sites.\footnote{In \cite[III Definition 1.1]{sga4-1}, this functor is called {\bf continuous}. A better name might be {\bf topologically continuous}, in line with our usage of topologically cocontinuous in Appendix~\ref{bigsmall}.} Given $f: \C \to \D$ a morphism of sites, by \cite{DAGV}*{Lemma 2.4.7} restriction along $f$ induces a geometric morphism of $n$-topoi\footnote{We warn the reader that $f^*$ here is the right adjoint of the geometric morphism, in conflict with our usual notation and that in \cite{DAGV}*{Lemma 2.4.7}. We will also write $f_!$ for the left-exact left adjoint, which is induced by the functor of left Kan extension along $f$ at the level of presheaves.}
\[
f^*: \Shv_{\tau'}(\D, \Spc_{\leq n-1}) \rightarrow \Shv_{\tau}(\C, \Spc_{\leq n-1})
\]
fitting into the following commutative diagram
\[
\begin{tikzcd}
\Shv_{\tau'}(\D, \Spc_{\leq n-1}) \ar{r}{f^*} \ar[hook]{d} &  \Shv_{\tau}(\C, \Spc_{\leq n-1}) \ar[hook]{d}\\
\Pre(\D, \Spc_{\leq n-1}) \ar{r}{f_{\pre}^*} & \Pre(\C, \Spc_{\leq n-1}),
\end{tikzcd}
\] where $f^*_{\pre}$ is the restriction functor. In other words, $f^*_{\pre}$ preserves sheaves. 
% (see \S\ref{bigsmall} for generalities)
%This generalizes the fact that if $\C, \D$ are $1$-categories, then the induced functor $f^*:\Shv_{\tau'}(\D, \Sets) \rightarrow \Shv_{\tau}(\C, \Sets)$ is a geometric morphism of $1$-topoi. 

We write $\RTop_n \subset \widehat{\Cat}_{\infty}$ for the subcategory spanned by $n$-topoi and geometric morphisms thereof. Using the results of \cite{htt}*{\S6.4.5}, we have that passage to $(n-1)$-truncated objects defines a functor
\[
\tau_{\leq n-1}:\RTop \rightarrow \RTop_n,
\]
equipped with a fully faithful right adjoint $\RTop_n \hookrightarrow \RTop$. If $\Xscr$ is an $n$-topos, we call its image in $\RTop$ the {\bf associated $n$-localic $\infty$-topos}. We now have the following lemma, which allows us to promote equivalences of $n$-topoi to the level of $\infty$-topoi.

% (so, in particular, $\C, \D$ have finite limits and $f$ preserves finite limits)
\begin{lemma} \label{lem:fin-lim} Let $1\leq m \leq n\leq \infty$. Suppose that $f: (\C, \tau) \rightarrow (\D, \tau')$ is a morphism of $m$-sites such that the induced geometric morphism $f^*:\Shv_{\tau'}(\D, \Spc_{\leq n-1}) \rightarrow \Shv_{\tau}(\C, \Spc_{\leq n-1})$ is an equivalence of $n$-topoi. Then the induced geometric morphism
\[
f^*: \Shv_{\tau'}(\D) \rightarrow \Shv_{\tau}(\C)
\]
is an equivalence of $\infty$-topoi.\footnote{We thank Peter Haine for alerting us to this result.}
\end{lemma}

\begin{proof} According to \cite{htt}*{Proposition 6.4.5.6}, under the hypotheses on $(\C, \tau)$ and $(\D, \tau')$, the $n$-localic $\infty$-topoi associated to the $n$-topoi $\Shv_{\tau'}(\D, \Spc_{\leq n-1})$ and $\Shv_{\tau}(\C, \Spc_{\leq n-1})$ are given by passing to sheaves of spaces, whence the claim follows.
\end{proof}

\begin{remark} \label{rem:1-top} Suppose that $(\C, \tau)$ is an $m$-site and consider the $n$-topos $\Shv_{\tau}(\C, \Spc_{\leq n-1})$. Without the assumption on the existence of finite limits, passing to sheaves of spaces need not compute the associated $n$-localic $\infty$-topos; instead, one only expects an equivalence upon hypercompletion. See the discussion in \cite{nlab} and also \cite{asai-shah}*{Remark~2.6}.
\end{remark}

\begin{proof} [Proof of Theorem~\ref{thm:et-ret}] The equivalence of \cite{scheiderer}*{Theorem 1.3} is implemented by a zig-zag of equivalences through geometric morphisms arising from a zig-zag of morphisms of sites, which we first recall from \cite{scheiderer}*{1.9}. Define a site $(\mathrm{Pairs}_X, \mathrm{aux})$ in the following way: the objects are pairs $(U, W)$ where $U \in \Et_X$ and $W \subset U_r$ is an open subset of the real spectrum of $U$, and a morphism $f: (U,W) \to (U',W')$ is a morphism $f: U \to U'$ of $X$-schemes such that $f_r(W) \subset W'$. A family $\{ f_{\alpha}:(U_{\alpha}, W_{\alpha}) \rightarrow (U, W) \}$ is an $\mathrm{aux}$-cover whenever $W = \cup f_{\alpha,r}(W_{\alpha})$. This category evidently has finite limits. By design, we then have a span of morphisms of sites:
\[
(\Et_X, \ret) \stackrel{\phi}{\leftarrow} (\mathrm{Pairs}_X, \mathrm{aux}) \stackrel{\psi}{\rightarrow} (\mathrm{Open}(X_r), \mathrm{usual}).
\]
It is easy to see that $\phi$ preserves finite limits, while $\psi$ preserves finite limits by \cite{scheiderer}*{Corollary 1.7.1}. Lemma~\ref{lem:fin-lim} thus applies to promote \cite{scheiderer}*{Theorem 1.3} to the claimed theorem.
\end{proof}

We now prove:

\begin{theorem} \label{thm:hyper} Let $X$ be a scheme of finite Krull dimension. Then the $\infty$-topos $\widetilde{X}_{\ret}$ is hypercomplete.
\end{theorem}

\begin{proof} Let $A$ be a commutative ring. By \cite{geom-rel}*{Corollaire 7.1.16}, $\Sper A$ is a spectral space. Furthermore, the Krull dimension of a spectral space is given by the supremum over lengths of chains of specializations. With this, the Krull dimension of $\Sper A$ is then bounded above by the Krull dimension of $\Spec A$ as the former is computed using specializations of real prime ideals in $A$. In this case, \cite{clausen-mathew}*{Theorem 3.12} tells us that if the Krull dimension of $\Spec A$ is $d$, then the homotopy dimension of $\widetilde{\Sper A}$ is less than $d$. Hence, for any scheme $X$ of finite Krull dimension $d$, the $\infty$-topos $\widetilde{X_r}$ is locally of homotopy dimension $\leq d$, whence hypercomplete by \cite{htt}*{Corollary 7.2.1.12}. Theorem~\ref{thm:et-ret} then implies that $\widetilde{X}_{\ret}$ is hypercomplete.
\end{proof}

\subsection{The gluing functor} \label{subsec:glue-small}

 In this section, we will compute the stalks of the gluing functor. To begin with, we need to introduce some notation. We have the gluing functor from Example~\ref{ex:et-to-ret}
\[
\theta = L_{\ret} i_{\et}: \widetilde{X}_{\et} \rightarrow \widetilde{X}_{\ret}.
\]
More generally, suppose that $\Sch'_X \subset \Sch_X$ is a subcategory such that:
\begin{enumerate}
\item For any $Y \in \Sch'_X$, if $U \rightarrow Y$ is \'etale then $U \in \Sch'_X$.
\item $\Sch'_X \subset \Sch_X$ is closed under limits.
\end{enumerate}
The only example that will concern us is $\Sch'_X = \Sm_X$. In this case, we can define the \'etale and real \'etale topology on $X$, and thus we can define the gluing functor in this level of generality:
\[
\theta' = L_{\ret}i_{\et}:\Shv_{\et}(\Sch'_X) \rightarrow \Shv_{\ret}(\Sch'_X).
\]

To proceed further, we now appeal to the results of Appendix~\ref{bigsmall} to relate $\theta$ and $\theta'$. Namely, if
$$u^*_{\tau}:\Shv_{\tau}(\Sch'_X) \rightarrow \Shv_{\tau}(\Et_X), \; \tau \in \{\et, \ret \}$$
is the restriction functor, then $u^*_{\ret}\theta' \simeq \theta u^*_{\et}$ by Lemma~\ref{lem:big-to-small}. Furthermore, if $f:Y \rightarrow X$ is a morphism, then we have adjunctions
 \[
f_{\pre}^*: \Pre(\Sch'_X) \rightleftarrows \Pre(\Sch'_Y): f_{*\pre}
\]
 \[
f_{\ret}^*:\Shv_{\ret}(\Sch'_X) \rightleftarrows \Shv_{\ret}(\Sch'_Y): f_{*\ret}
\] 
\[
f_{\et}^*: \Shv_{\et}(\Sch'_X) \rightleftarrows \Shv_{\et}(\Sch'_Y): f_{*\et}
\]

The next lemma is a simple enhancement of \cite[Proposition~3.7.2]{scheiderer} to a slightly more general context. 
\begin{lemma} \label{lem:stalk-compute} Let $\Sch'_X \subset \Sch_X$ be as above. Suppose that $\Fscr: (\Sch'_X)^{\op} \rightarrow \Spc$ is an \'etale sheaf, $Y \in \Sch'_X$, and we have a morphism $\alpha: \Spec k \rightarrow Y$ where $k$ is a real closed field. Then we have an equivalence
\[
(\theta'\Fscr)_{\alpha} \simeq (\alpha_{\et}^*\Fscr)(1_{\alpha}),
\]
where $1_{\alpha} \in \Shv_{\et}(k)$ is the terminal object.
\end{lemma}

%
%\begin{lemma} \label{lem:stalk-compute} Let $\Fscr: \Et^{\op}_X \rightarrow \Spc$ be an \'etale sheaf. Suppose that $\alpha: \Spec k \rightarrow X$ is a morphism where $k$ is a real closed field. Then we have an equivalence
%\[
%(\theta\Fscr)_{x} \simeq \Gamma(\alpha_{\et}^*\Fscr),
%\] where $\Gamma$ is the global sections functor $\Gamma: \widetilde{\Spec k}_{\et} \rightarrow \Spc$.
%\end{lemma}

\begin{proof} Without loss of generality, we may assume that $X = Y$. For this proof we follow the notation in Lemma~\ref{lem:big-to-small}. By~\eqref{eq:rhensel}, taking stalks at $\alpha$ is given by the formula\footnote{The strict real henselization may not exist as an object in $\Sch'_X$ although it is an object in $\mathrm{Pro}(\Sch'_X)$, hence the stalk is defined via left Kan extension along the functor $(\Sch'X)^{\op} \rightarrow \mathrm{Pro}(\Sch'_X)^{\op}$.}
\[
(\theta'\Fscr)_{\alpha} \simeq \colim_{T \in \mathrm{Neib}_{\ret}(Y,\alpha)} (\theta'\Fscr)(T).
\]
Since each $T$ in $\mathrm{Neib}_{\ret}(Y,\alpha)$ is, in particular, an \'etale $Y$-scheme, we have the following equivalences 
\begin{eqnarray*}
 \colim_{T \in \mathrm{Neib}_{\ret}(Y,\alpha)} \theta' \Fscr(T) & \simeq &  \colim_{T \in \mathrm{Neib}_{\ret}(Y,\alpha)} \Map(L_{\ret} u_{\pre !} T, \theta' \Fscr) \\
 & \simeq & \colim_{T \in \mathrm{Neib}_{\ret}(Y,\alpha)} \Map(u_{\ret !} L_{\ret_{\mathrm{small}}} T, \theta' \Fscr) \\
 & \simeq &  \colim_{T \in \mathrm{Neib}_{\ret}(Y,\alpha)} \Map(L_{\ret_{\mathrm{small}}} T, u^*_{\ret}\theta'\Fscr)\\
 & \simeq & (\alpha_{\ret_{\mathrm{small}}}^*u^*_{\ret}\theta'\Fscr)(1_{\alpha}),
\end{eqnarray*}
where $\alpha_{\ret_{\mathrm{small}}}^*: \widetilde{X}_{\ret} \rightarrow \widetilde{\Spec k}_{\ret} \simeq \Spc$ is the pullback functor on the small site and we use Lemma~\ref{lem:big-to-small}.A. 

By definition of the gluing functor and the fact that sheafifications always commute with pullbacks and Lemma~\ref{lem:big-to-small}.B, we have equivalences
\begin{eqnarray*}
(\alpha_{\ret_{\mathrm{small}}}^*u^*_{\ret}\theta'\Fscr)(1_{\alpha}) & = & (\alpha_{\ret_{\mathrm{small}}}^*u^*_{\ret}L_{\ret}i_{\et}\Fscr)(1_{\alpha})\\
& \simeq & (\alpha_{\ret_{\mathrm{small}}}^*L_{\ret_{\mathrm{small}}}u^*_{\pre}i_{\et}\Fscr)(1_{\alpha})\\
& \simeq &  (L_{\ret_{\mathrm{small}}}\alpha_{\pre}^*u^*_{\pre}i_{\et}\Fscr)(1_{\alpha}).
\end{eqnarray*}
But now, we note that there are no nontrivial real \'etale covers of a real closed field, so since $\Fscr$ was an \'etale sheaf and hence, in particular, transforms finite coproducts to finite products, the real \'etale sheafification of $\alpha^*_{\pre}u^*_{\pre}i_{\et}\Fscr$ does nothing to $1_{\alpha}$:
\[
(L_{\ret_{\mathrm{small}}}\alpha_{\pre}^*u^*_{\pre}i_{\et}\Fscr)(1_{\alpha}) \simeq (\alpha_{\pre}^*u^*_{\pre}i_{\et}\Fscr)(1_{\alpha}).
\]
To finish the proof, it suffices to check that the canonical map, induced by sheafification,
\[
(\alpha_{\pre}^*u^*_{\pre}i_{\et}\Fscr)(1_{\alpha}) \simeq (\alpha_{\pre}^*i_{\et_{\mathrm{small}}} u^*_{\et}\Fscr)(1_{\alpha}) \rightarrow (\alpha_{\et}^*u^*_{\et}\Fscr)(1_{\alpha}),
\]
is an equivalence. This follows as in the proof of \cite{scheiderer}*{Lemma~3.5.1}; for completeness, we now recall the argument. Let $\mathrm{Neib}_{\et}(Y,\alpha)$ be the filtered category whose objects are factorizations $\Spec k \to U \to Y$ of $\alpha$ with $U$ \'etale over $Y$, and let $Y^{\alpha} = \lim_{\mathrm{Neib}_{\et}(Y,\alpha)} U$, so that $Y^{\alpha}$ is affine and its ring $R$ of global sections is the strict real henselization of $Y$ at $\alpha$. Let $k'$ be the residue field of the local ring $R$, so $k'$ is a real closed field (but we need not have that the inclusion $k' \subset k$ is an equality; cf. Example~\ref{ex:nontrivialRealClosedFieldsExtension}). We then have a factorization of $\alpha$ as
\[ \begin{tikzcd}[row sep=4ex, column sep=4ex, text height=1.5ex, text depth=0.25ex]
\Spec k \ar{r} \ar{rd}[swap]{\alpha} & \Spec k' \ar{r} \ar{d}{\alpha''} & Y^{\alpha} \ar{ld}{\alpha'} \\
& Y &
\end{tikzcd} \]

The claim now follows from the chain of equivalences:
\begin{eqnarray*}
(\alpha_{\pre}^*i_{\et_{\mathrm{small}}} u^*_{\et}\Fscr)(1_{\alpha}) & \simeq & \colim_{U \in \mathrm{Neib}_{\et}(Y,\alpha)} u^*_{\et}\Fscr(U) \\
 & \simeq & \alpha_{\et}'^{*}u^*_{\et}\Fscr(R)\\
 & \simeq & \alpha_{\et}''^{*}u^*_{\et}\Fscr(k')\\
 & \simeq & \alpha_{\et}^{*}u^*_{\et}\Fscr(k) = \alpha_{\et}^{*}u^*_{\et}\Fscr(1_{\alpha}).
\end{eqnarray*}

Here, the first equivalence follows by the left Kan extension formula for $\alpha^*_{\pre}$. The second equivalence follows from the analogue of Lemma~\ref{lem:approx} for the small \'etale site. The third equivalence follows from the fact that the category of finite \'etale morphisms $V \to Y^{\alpha} = \Spec R$ is equivalent to the category of finite \'etale morphisms $V' \to \Spec k'$ (\cite{stacks}*{Tag 04GK}). The fourth equivalence follows from the fact that the morphism $\Spec k \to \Spec k'$ induces an equivalence $\widetilde{k'}_{\et} \xrightarrow{\simeq} \widetilde{k}_{\et}$.

\end{proof}
%
%\begin{example} \label{ex:small} If $\Sch'_X = \Et_X$ then Lemma~\ref{lem:stalk-compute} furnishes an equivalence:
%\[
%(\theta\Fscr)_{x} \simeq \Gamma(\alpha_{\et}^*\Fscr).
%\]
%\end{example}

\section{Big versus small sites} \label{bigsmall}

\renewcommand{\small}{\mathrm{small}}

In this appendix, we record some facts about big and small sites which are surely well-known to experts. First, let us fix some notation that we have already used above. Suppose that $\Escr$ is a presentable $\infty$-category and $f: \C \rightarrow \D$ is a functor of small $\infty$-categories. Then we have adjunctions

\[
\begin{tikzcd}
\Pre(\C, \Escr)\ar[rr, "f^{\pre}_!", bend left] \ar[rr, "f^{\pre}_*"' , bend right]&  & \Pre(\D,\Escr) \ar[ll, "f_{\pre}^*"'].
\end{tikzcd}
\] 

\begin{warning} \label{warn:convention} This convention differs from the one adopted in \cite{DAGV}, but agrees with the conventions in \cite{sag}*{Definition 20.6.1}.
\end{warning}

We have already recalled the definition of a morphism of sites in \S\ref{subsec:hyp}. The next notion is usually called ``cocontinuous", but we adopt the terminology in \cite{khan-thesis} to avoid conflict with the notion of a functor that preserves colimits.

\begin{definition} \label{def:top-cocont} Suppose that $(C, \tau), (\D, \tau')$ are sites. A functor $f:\C \rightarrow \D$ is said to be {\bf topologically cocontinuous (with respect to $\tau$ and $\tau'$}) if for each $Y \in \C$ and $\tau'$-covering sieve $R' \hookrightarrow h_{f(Y)}$, the sieve $f^*R \times_{f^*h_{f(Y)}} h_Y \hookrightarrow h_Y$ is a $\tau$-covering sieve.
\end{definition}

Unpacking the above definition, this is equivalent to the following condition:
\begin{itemize}
\item[($\ast$)] For each $\tau'$-covering sieve $R' \hookrightarrow h_{f(Y)}$, the sieve of $h_{Y}$ generated by morphisms $Y' \rightarrow Y$ such that $h_{f(Y')} \rightarrow h_{f(Y)}$ factors through $R'$ is a $\tau$-covering sieve.
\end{itemize}

\begin{lemma} \label{lem:cocont} Suppose that $(C, \tau), (\D, \tau')$ are sites and $f: \C \rightarrow \D$ is a topologically cocontinuous functor. Let $\Escr$ be a presentable $\infty$-category. Then $f_{\pre}^*:\Pre(\D, \Escr) \rightarrow \Pre(\C, \Escr)$ sends $\tau'$-local equivalences to $\tau$-local equivalences. Therefore, there exists a functor $f^*: \Shv_{\tau}(\D, \Escr) \rightarrow \Shv_{\tau}(\C, \Escr)$ rendering the following diagram commutative
\[
\begin{tikzcd}
\Shv_{\tau'}(\D, \Escr) \ar{r}{f^*}  &  \Shv_{\tau}(\C, \Escr) \\
\Pre(\D, \Escr) \ar{r}{f_{\pre}^*} \ar{u}{L_{\tau'}} & \Pre(\C, \Escr) \ar[swap]{u}{L_{\tau}}.
\end{tikzcd}
\]
\end{lemma}

\begin{proof} In case $\Escr = \Spc$, the first statement is \cite{gepner-heller}*{Lemma 2.23}, and the proof for a general $\Escr$ is the same. The conclusion follows since the $\infty$-category of sheaves is obtained by inverting the local equivalences.
\end{proof} 

We now specialize to the comparison between small and big sites. Let $X$ be a scheme and suppose that we have $\Sch'_X \subset \Sch_X$ a full subcategory of $X$-schemes which is stable under pullbacks and contains \'etale $X$-schemes; the most prominent example will be $\Sch'_X = \Sm_X$. Given a Grothendieck topology $\tau$, we also have the following condition:
% We are also given $\tau$ a Grothendieck topology on $\Sch'_X$ subject to the following condition:
\begin{enumerate}[leftmargin=1cm]
\item[\textsf{Small}] For each $\tau$-covering sieve $R \hookrightarrow h_Y$, there exists a collection of \'etale morphisms $\{ U_{\alpha} \rightarrow Y \}$ which belongs to $R$.
\end{enumerate}

\begin{remark} \label{rem:small} This condition is inspired by the theory of {\bf geometric sites} in the sense of \cite{sag}*{\S20.6}, tailormade for our purposes. Examples of Grothendieck sites satisfying \textsf{Small} include the \'etale, real \'etale, and Nisnevich sites on $\Sm_X$.
\end{remark}

Consider the inclusion $u: \Et_X \subset \Sch'_X$.

\begin{lemma} \label{lem:big-to-small} In the situation above, suppose that $\tau$ is a Grothendieck topology on $\Sch'_X$. There exists a Grothendieck topology $\tau_{\small}$ on $\Et_X$ satisfying the condition {\sf Small} such that
\begin{enumerate}
\item The inclusion $u:\Et_X \subset \Sch'_X$ determines a morphism of sites
\[
(\Et_X, \tau_{\small}) \rightarrow (\Sch'_X, \tau).
\]
\item $\tau_{\small}$ is the coarsest topology on $\Et_X$ for which the inclusion functor $u$ induces a morphism of sites.
\item If $\tau$ further satisfies the condition {\sf Small}, then the functor $\Et_X \subset \Sm_X$ is topologically cocontinuous.
\end{enumerate}
Furthermore, let $\Escr$ be a presentable $\infty$-category and denote by $\iota_{\tau}:\Shv_{\tau}(\Sch'_X,\Escr) \subset \Pre(\Sch'_X, \Escr)$ and $\iota_{\tau_{\small}}:\Shv_{\tau_{\small}}(\Et_X,\Escr) \subset \Pre(\Et_X, \Escr)$ the obvious inclusions. Then:
\begin{enumerate}
\item[\rm (A)] We have an adjunction
\[
u_!: \Shv_{\tau_{\small}}(\Et_X, \Escr) \rightleftarrows \Shv_{\tau}(\Sch'_X, \Escr):u^*
\]
satisfying:
\[
u^*_{\pre}\iota_{\tau} \simeq \iota_{\tau_{\small}}u^*,
\]
and
\[
L_{\tau}u^{\pre}_! \simeq u_!L_{\tau_{\small}}.
\]
\item[\rm (B)] We have an adjunction
\[
u^*: \Shv_{\tau}(\Sch'_X, \Escr) \rightleftarrows  \Shv_{\tau_{\small}}(\Et_X, \Escr) :u_*
\]
satisfying:
\[
L_{\tau_{\small}}u^*_{\pre}\simeq u^*L_{\tau},
\]
and
\[
u^{\pre}_*\iota_{\tau_{\small}} \simeq \iota_{\tau}u_*.
\]
\item[\rm (C)] The functors $u_*$ and $u_!$ are fully faithful.
\end{enumerate}
\end{lemma}

\begin{proof} The Grothendieck topology $\tau_{\small}$ is the one given by \cite{sag}*{Proposition 20.6.1.1}. Explictly, for each \'etale morphism $Y \rightarrow X$, a sieve $R \hookrightarrow h_Y$ is a $\tau_{\small}$-covering sieve if and only if $u_!R \hookrightarrow h_{u(Y)}$ is a $\tau$-covering sieve. (A) then follows from \cite{sag}*{Proposition 20.6.1.3}. Now, {\sf Small} ensures that for $\Et_X$ equipped with the Grothendieck topology $\tau_{\small}$, the inclusion functor is topologically cocontinuous (which is (3)). Hence, Lemma~\ref{lem:cocont} gives (B). 

We now prove (C). Clearly, $u^{\pre}_!$ (resp. $u^{\pre}_*$) is fully faithful as it is computed as left (resp. right) Kan extension along a fully faithful functor. To see that $u_!$ is fully faithful, it suffices to show that the unit map $\id \rightarrow u^*u_!$ is an equivalence. But this follows from the computation:
\[
u^*u_! \simeq u^*L_{\tau}u_!^{\pre}\iota_{\small} \simeq L_{\tau_{\small}}u^*_{\pre}u_!^{\pre}\iota_{\small} \simeq \id,
\]
where the last equivalence follows from the presheaf-level statement. The argument for $u_*$ is similar. 
\end{proof}

\begin{construction} \label{construct:laxbigsmall} Let $X$ be a scheme. Then we have the following lax commutative square

\[
\begin{tikzcd}[column sep=6ex]
\widetilde{X}_{\et} \ar[swap]{d}{u_{\et!}} \ar{r}{i_{\et_{\small}}} \ar[phantom]{rd}{\SWarrow} & \widetilde{X}_{\Nis} \ar{d}{u_{\Nis!}} \\
\Shv_{\et}(\Sm_X) \ar{r}{i_{\et}}  & \Shv_{\Nis}(\Sm_X),
\end{tikzcd}
\]
obtained as
\[
u_{\Nis!}i_{\et_{\small}} \Rightarrow u_{\Nis!}i_{\et_{\small}}u^*_{\et}u_{!\et} \simeq u_{\Nis!}u^*_{\Nis}i_{\et}u_{!\et} \Rightarrow i_{\et}u_{!\et},
\]
where the middle equivalence is furnished by Lemma~\ref{lem:big-to-small}. Composing with equivalence $L_{\ret}u^{\Nis}_! \simeq u_{\Nis!}L_{\ret_{\small}}$ from the same lemma then gives us a lax commutative square
\[
\begin{tikzcd}[column sep=6ex]
\widetilde{X}_{\et} \ar[swap]{d}{u_{\et!}} \ar{r}{L_{\ret}i_{\et_{\small}}} \ar[phantom]{rd}{\SWarrow} & \widetilde{X}_{\ret} \ar{d}{u_{\ret!}} \\
\Shv_{\et}(\Sm_X) \ar{r}{L_{\ret}i_{\et}}  & \Shv_{\ret}(\Sm_X),
\end{tikzcd}
\]
comparing big versus small sites.

%Clearly, $u^{\pre}_!$ (resp. $u^{\pre}_*$) is fully faithful as it is computed as left Kan extension along a fully faithful functor. To see that $u_!$ is fully faithful, consider the transformation $\id \rightarrow u^*u_!$ which we claim to be an equivalence. This follows from the computation: 
%\[
%u^*u_! \simeq u^*L_{\tau}u_!^{\pre}\iota_{\small} \simeq L_{\tau_{\small}}u^*_{\pre}u_!^{\pre}\iota_{\small} \simeq \id,
%\]
%where the last equivalence follows from the presheaf-level statement. The argument for $u_*$ follows similarly. 

\end{construction}

\begin{remark} \label{rem:ushriek} Since $u$ is a morphism of sites, the functor $u_{\tau!}: \Shv_{\tau_{\small}}(\Et_X) \rightarrow \Shv_{\tau}(\Sch'_X)$ is the left-exact left adjoint of a geometric morphism of $\infty$-topoi. Therefore, its stabilization is strong symmetric monoidal with respect to the induced smash product symmetric monoidal structure on the stabilization of an $\infty$-topos.
% We note further the following fact concerning monoidality of $u_{\tau!}$. Since \'etale $X$-schemes are closed under pullbacks, the inclusion functor $\Et_X \subset \Sch'_X$ in particular preserves products, whence the functor $u_!:\Pre(\Et_X) \rightarrow \Pre(\Sch'_X)$ preserves products. Therefore, since sheafification is left exact, for any topology $\tau$ on $\Sch'_X$, we conclude from Lemma~\ref{lem:big-to-small}.A that the functor $u_{\tau!}: \Shv_{\tau_{\small}}(\Et_X) \rightarrow \Shv_{\tau}(\Sch'_X)$ is symmetric monoidal for the Cartesian symmetric monoidal structure. \todo{reference for cartesian monoidal gets converted to strong monoidal?}
\end{remark}

%We have also used the following lemma concerning pullbacks
%
%\begin{lemma} \label{lem:f-star-hyp} 
%\end{lemma}

\vspace{20pt}
\scriptsize
\noindent
Elden Elmanto\\
Harvard University\\
Department of Mathematics\\
1 Oxford Street\\
Cambridge, MA 02138\\
USA\\
\texttt{elmanto@math.harvard.edu}

\vspace{10pt}
\noindent
Jay Shah\\
University of M\"{u}nster\\
Department of Mathematics\\
Orl\'{e}ans-Ring\\
M\"{u}nster 48149\\
Germany\\
\texttt{jayhshah@gmail.com}

\end{document}